\documentclass[10.5pt,reqno]{amsart}
\usepackage{amsmath,amssymb,mathrsfs,amsthm,amsfonts}
\usepackage[inline]{enumitem} 
\usepackage[usenames,dvipsnames]{xcolor}
\usepackage{hyperref}
\hypersetup{%
	colorlinks=true, linkcolor=blue, 
	citecolor=ForestGreen
}
\usepackage{esint}		
\usepackage{graphicx}	
\usepackage{subcaption}
\usepackage{psfrag}

\usepackage[normalem]{ulem} 

\makeatletter
\DeclareFontEncoding{LS2}{}{\@noaccents}
\makeatother
\DeclareFontSubstitution{LS2}{stix}{m}{n}
\DeclareSymbolFont{arrows3}{LS2}{stixtt}{m}{n}
\DeclareMathSymbol{\triangletimes}{\mathbin}{arrows3}{"AC}
\usepackage[paper=letterpaper,margin=1in]{geometry}
\newcommand*{\Cdot}{{\raisebox{-0.5ex}{\scalebox{1.8}{$\cdot$}}}} 

\graphicspath{{./figs/}}
\newtheorem{theorem}{Theorem}[section]
\newtheorem{lemma}[theorem]{Lemma}
\newtheorem{proposition}[theorem]{Proposition}
\newtheorem{innercustomprop}{Proposition}
\newenvironment{customprop}[1]
{\renewcommand\theinnercustomprop{#1*}\innercustomprop}
{\endinnercustomprop}
\newtheorem{corollary}[theorem]{Corollary}
\theoremstyle{definition}

\newtheorem{remark}[theorem]{Remark}

\numberwithin{equation}{section}
\allowdisplaybreaks

\usepackage{acronym}
\acrodef{LDP}{Large Deviation Principle}
\acrodef{TASEP}{Totally Asymmetric Simple Exclusion Process}
\acrodef{LPP}{Last Passage Percolation}
\acrodef{6VM}{Six Vertex Model}
\acrodef{CGM}{Corner Growth Model}

\newcommand{\Ex}{ \mathbf{E} }				
\renewcommand{\Pr}{ \mathbf{P} }			
\newcommand{\Prr}{ {\mathbf{Q}} }			
\newcommand{\Pois}{\text{Pois}}		
\newcommand{\dist}{d_{C(\R)}}		
\newcommand{\distt}{d_\text{S}}		
\newcommand{\gen}{L}				
\newcommand{\genn}{\widetilde{L}}	

\newcommand{\ind}{ \mathbf{1} }		
\newcommand{\set}[1]{ {\{#1\}} }	
\newcommand{\edge}{\!\!-\!\!}
\newcommand{\prf}{\psi}					
\newcommand{\prff}{\overline{\psi}}		
\newcommand{\rf}{I^{(2)}}				
\newcommand{\rfup}{I^{(1)}}				
\newcommand{\rff}{\widetilde{I}}		
\newcommand{\lrf}{J^{(2)}}				
\newcommand{\lrfup}{J^{(1)}}			
\newcommand{\lrfupa}{J^{(1)}_{a}}		
\newcommand{\hlf}{\theta}						
\newcommand{\HLf}{\Theta}						
\newcommand{\hl}[2]{\mathscr{G}[#1,#2]}			
\newcommand{\hll}[3]{\mathscr{G}_{#3}[#1,#2]}	
\newcommand{\moll}{\omega}				
\newcommand{\Speed}{\Lambda}					
\newcommand{\Speedd}{\widetilde{\Lambda}}		
\newcommand{\speed}{S}				
\newcommand{\doob}{q}					
\newcommand{\mob}{\phi}					
\newcommand{\mobb}{\Phi^{(2)}}				
\newcommand{\mobbup}{\Phi^{(1)}}			
\newcommand{\filt}{\mathscr{F}}				
\newcommand{\maxlam}{\overline{\lambda}}	
\newcommand{\minlam}{\underline{\lambda}}	

\newcommand{\Sp}{\mathscr{D}}				
\newcommand{\Splip}{\mathscr{E}}			
\newcommand{\Spd}{\mathscr{D}_\text{d}}		
\newcommand{\Spf}{\mathscr{E}_\mathbb{Z}}	
\newcommand{\Sph}{\mathscr{E}_\mathbb{Z}}	
\newcommand{\DE}{E^\text{d}}				
\newcommand{\VE}{E^\text{v}}				
\newcommand{\RTri}{\Sigma^{\times}}			
\newcommand{\de}{{e^\text{d}}}			
\newcommand{\ve}{{e^\text{v}}}			
\newcommand{\slab}{\mathcal{S}}						
\newcommand{\tz}{\mathcal{T}}						
\newcommand{\rTri}{{\scriptstyle\triangletimes}}	
\newcommand{\rtri}{{\scriptscriptstyle\triangletimes}}	
\newcommand{\buffer}{\mathcal{B}}			
\newcommand{\intm}{\mathcal{I}}				
\newcommand{\resd}{\mathcal{R}}				

\newcommand{\stripp}{\mathcal{N}}			
\newcommand{\prtn}{\mathscr{Z}}				
\newcommand{\prtnU}{\mathscr{X}}			
\newcommand{\prtnC}{\widehat{\mathscr{Z}}}	
\newcommand{\region}{\mathcal{Z}}			
\newcommand{\ske}{\text{Ske}(\prtnU)}		
\newcommand{\skeE}{\mathcal{E}}				
\newcommand{\trapD}{\mathcal{D}}			
\newcommand{\scl}{\ell}
\newcommand{\scll}{m}
\newcommand{\sclll}{n}

\newcommand{\ling}{g}									
\newcommand{\lingi}{\Gamma}								
\newcommand{\lingiC}{\widehat{\Gamma}}				
\newcommand{\linkap}{{\kappa}_\triangle}		
\newcommand{\linrho}{{\rho}_\triangle}			
\newcommand{\linlam}{{\lambda}_\triangle}		

\newcommand{\PN}{\mathbf{P}_N}				
\newcommand{\QN}{\mathbf{Q}_N}				
\newcommand{\h}{\mathsf{h}}				
\newcommand{\f}{\mathsf{f}}				
\newcommand{\hN}{\mathsf{h}_N}			
\newcommand{\g}{\mathsf{g}}				
\newcommand{\gN}{\mathsf{g}_N}			
\newcommand{\barg}{\overline{\mathsf{g}}}	
\newcommand{\undg}{\underline{\mathsf{g}}}	
\newcommand{\bargic}{\overline{\mathsf{g}}^\text{ic}}	
\newcommand{\undgic}{\underline{\mathsf{g}}^\text{ic}}	
\newcommand{\barenv}{\overline{\mathsf{b}}}		
\newcommand{\undenv}{\underline{\mathsf{b}}}	

\newcommand{\hic}{\mathsf{h}^\text{ic}}		
\newcommand{\hIC}{h^\text{ic}}				
\newcommand{\gic}{\mathsf{g}^\text{ic}}		
\newcommand{\gIC}{g^\text{ic}}				
\newcommand{\ic}{\text{ic}}	
\newcommand{\icint}{\mathcal{V}}
\newcommand{\undt}{\underline{t}}
\newcommand{\bart}{\overline{t}}

\newcommand{\N}{\mathbb{N}}
\newcommand{\R}{\mathbb{R}}
\newcommand{\T}{\mathbb{T}}
\newcommand{\TN}{\frac{1}{N}\mathbb{T}_N}
\newcommand{\Tn}{\mathbb{T}_N}

\newcommand{\Z}{\mathbb{Z}}

\newcommand{\calA}{\mathcal{A}}
\newcommand{\calB}{\mathcal{B}}
\newcommand{\calC}{\mathcal{C}}

\newcommand{\calE}{\mathcal{E}}
\newcommand{\calF}{\mathcal{F}}

\newcommand{\calK}{\mathcal{K}}

\newcommand{\calO}{\mathcal{O}}

\newcommand{\calU}{\mathcal{U}}
\newcommand{\calV}{\mathcal{V}}
\newcommand{\calW}{\mathcal{W}}

\newcommand{\calX}{\mathcal{X}}

\renewcommand{\d}{\delta}
\newcommand{\e}{\varepsilon}
\renewcommand{\hat}{\widehat}
\newcommand{\til}{\widetilde}
\renewcommand{\bar}{\overline}
\newcommand{\und}{\underline}

\newcommand{\diri}{\mathfrak{D}}
\newcommand{\probm}{\mathcal{M}_1}
\newcommand{\young}{\mathbb{Y}}
\newcommand{\Spdm}{\mathscr{D}^\text{m}}

\newcommand{\bareta}{\bar{\eta}}

\title[Speed-$ N^2 $ LDP of the TASEP]
{Exceedingly Large Deviations\\{}of the Totally Asymmetric Exclusion Process}

\author[S.\ Olla]{Stefano Olla}
\address{Stefano Olla,
	CEREMADE, UMR CNRS 7534, 
          Universit\`{e} Paris--Dauphine, PSL Research University,
	\newline\hphantom{\quad \ Stefano Olla}
	75016 Paris, France
	}
\email{olla@ceremade.dauphine.fr}

\author[L.-C.\ Tsai]{Li-Cheng Tsai}
\address{Li-Cheng Tsai,
	Departments of Mathematics, Columbia University,
	\newline\hphantom{\quad \ Li-Cheng Tsai}
	2990 Broadway, New York, NY 10027
	}
\email{lctsai.math@gmail.com}
\subjclass[2010]{%
Primary 60F10, 		
Secondary 82C22. 	
}
\keywords{
Large deviations.
Exclusion processes.
Totally asymmetric.
Corner Growth Model.
Variational formula.
}
\begin{document}
\begin{abstract}
Consider the \ac*{TASEP} on the integer lattice $ \Z $. 
We study the functional Large Deviations of the integrated current
$ \h(t,x) $ under the hyperbolic scaling of space and time by $ N $,
i.e., $ \hN(t,\xi) := \frac{1}{N}\h(Nt,N\xi) $.
As hinted by the asymmetry in the upper- and lower-tail large deviations 
of the exponential Last Passage Percolation, 
the TASEP exhibits two types of deviations. 
One type of deviations occur with probability $ \exp(-O(N)) $, referred to as speed-$ N $;
while the other with probability $ \exp(-O(N^{2})) $, referred to as speed-$ N^2 $. 
In this work we study the speed-$ N^2 $ functional \ac{LDP} of the TASEP, 
and establish (non-matching) large deviation upper and lower bounds.
%
\end{abstract}

\maketitle

\section{Introduction}
\label{sect:intro}

In this article we study the large deviations of 
two equivalent models, the \ac{CGM} and the \ac{TASEP}.
The \ac{CGM} is a stochastic model of surface growth in one dimension.
The state space 
\begin{align}
	\label{eq:Spf}
	\Spf:= \big\{ \f: \Z\to\Z \, : \, \f(x+1)-\f(x) \in \{0,1\}, \, \forall x\in\Z \big\}
\end{align}
consists of $ \Z $-valued height profiles defined on the integer lattice $ \Z $, 
with discrete gradient being either $ 0 $ or $ 1 $.
Starting from a given initial condition $ \h(0,\Cdot) = \hic(\Cdot)\in\Spf $,
the process $ \h(t,\Cdot) $ evolves in $ t $ as a Markov process according to the following mechanism.
At each site $ x\in\Z $ sits an independent Poisson clock of unit rate,
and, upon ringing of the clock, 
the height at $ x $ increases by $ 1 $ if $ \h(t,x+1)-\h(t,x)=1 $ and 
$ \h(t,x)-\h(t,x-1)=0 $.
Otherwise $ \h $ stays unchanged.
On the other hand,
the \ac{TASEP} is an interacting particle system~\cite{liggett05},
consisting of indistinguishable particles occupying the half-integer lattice $ \frac12+\Z $.
Each particle waits an independent Poisson clock of unit rate,
and, upon ringing of the clock, attempts to jump one step to the \emph{left}, 
under the constraint that each site holds at most one particle.
With 
\begin{align*}
	\eta(y) = \left\{\begin{array}{l@{,}l}
		1	&\text{ if the site } y \text{ is occupied},
		\\
		0	&\text{ if the site } y \text{ is empty}
	\end{array}\right.
\end{align*}
denoting the occupation variables,
the \ac{TASEP} is a Markov process with state space $ \{0,1\}^{\frac12+\Z} $,
where each $ (\eta(y))_{y\in\frac12+\Z} \in \{0,1\}^{\frac12+\Z} $ represents a particle configuration on $ \frac12+\Z $.
Given a \ac{CGM} with height process $ \h(t,x) $, 
we identify each slope $ 1 $ segment of $ \h(t,\Cdot) $ with a particle
and each slope $ 0 $ segment of $ \h(t,\Cdot) $ with an empty site,
i.e., 
\begin{align}
	\label{eq:CGMtoTASEP}
	\h(t,y+\tfrac12)-\h(t,y-\tfrac12) =: \eta(t,y);
\end{align}
see Figure~\ref{fig:cgm}.
One readily check that, under such an identification, the resulting particles evolves as the \ac{TASEP}.
Conversely, given the \ac{TASEP}, the integrated current
\begin{align}
	\label{eq:TASEPtoCGM}
	\h (t,x) :=  \#\Big\{ \text{particles crossing } (x-\tfrac12,x+\tfrac12)
	\text{ within } [0,t]\Big\}
	+ \text{sign}(x) \sum_{y\in(0,|x|)} \eta(0,y)
\end{align}
defines an $ \Spf $-valued process that evolves as the \ac{CGM}.
Associated to a given height profile $ \f\in\Splip $ and $ x\in\Z $ 
is the \textbf{mobility function}, defined as
\begin{align}
	\label{eq:mob}
	\mob(\f,x) &:= (\f(x+1)-\f(x))(1-\f(x)+\f(x-1)) 
	=
	\ind_\set{f(x+1)-\f(x)=1,\f(x)-\f(x-1)=0}
\\
	\notag
	& =\eta(x+\tfrac12)(1-\eta(x-\tfrac12)),
	\quad
	\text{where }\eta(y) := \f(y+\tfrac12)-\f(y-\tfrac12).
\end{align}
\begin{remark}
The standard terminology for $ \mob(\f,x) $ in the literature is instantaneous current.
We adopt a different term for $ \mob(\f,x) $ here (i.e., mobility function) to avoid confusion with other terms (i.e., instantaneous flux)
we use in the following.
\end{remark}
\noindent
Formally speaking,
with $ \f^{x} := \f+\ind_{x} $ denoting the profile obtained 
by increasing the value of $ \f $ by $ 1 $ at site $ x $, 
the \ac{CGM} is a Markov process with state space $ \Spf $,  characterized by the generator
\begin{align}
	\label{eq:gen}
	\gen F(\f) := \sum_{x\in\Z} \mob(\f,x) (F(\f^{x}) - F(\f)).
\end{align}
Given this map between the \ac{CGM} and \ac{TASEP},
throughout this article we will operate in both the languages of surface growths and of particle systems. 
To avoid redundancy, hereafter we will refer \emph{solely} to the \ac{TASEP} as our working model,
and associate the height process $ \h $ to the \ac{TASEP}.
\begin{figure}[h]
\includegraphics[width=.75\textwidth]{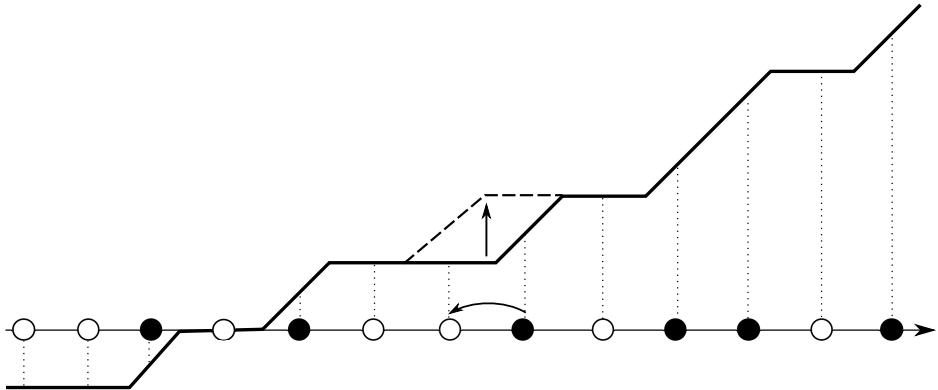}
\caption{The \ac{CGM} and \ac{TASEP}}
\label{fig:cgm}
\end{figure}

The \ac{TASEP} is a special case of exclusion processes
that is connected to a host of physical phenomena.
In addition to surface growth mentioned previously,
the \ac{TASEP} serves as a simple model of traffic, fluid and queuing,
and is linked to last passage percolation, non-intersecting line ensembles and random matrix theory.
Furthermore, the \ac{TASEP} owns rich mathematical structures, which has been the ground of intensive research:
to name a few,
the exact solvability via Bethe ansatz \cite{schuetz97}; 
the relation to the Robinson--Schensted--Knuth correspondence \cite{johansson00};
a reaction-diffusion type identity \cite{balazs10};
and being an attractive particle system \cite{rezakhanlou91}.  

Among known results on the \ac{TASEP} is its hydrodynamic limit.
Let $ N $ denote a scaling parameter that tends to $ \infty $,
and, for any $t\ge 0$ and $\xi\in \R$,
consider the hyperbolic scaling $ \hN(t,\xi) := \frac{1}{N} \h(Nt,N\xi) $ of the height process.
Through this article
we linearly interpolate $ \hN $ in the variable $ \xi\in\frac{1}{N}\Z $
to obtain a process $ \hN(t,\xi) $ defined for all $ \xi\in\R $.
It is well-known \cite{rost81,rezakhanlou91,seppaelaeinen98a} that, as $ N\to\infty $,
$ \hN $ converges to a deterministic function $ h $, given by the unique entropy solution of 
the integrated, \emph{inviscid} Burgers equation:
\begin{align}
	\label{eq:burgers}
	h_t = h_\xi(1-h_\xi).
\end{align}
The limiting equation~\eqref{eq:burgers}, being nonlinear and hyperbolic,
exhibits non-differentiability due to the presence of shock waves.
This is in sharp contrast with the diffusive behavior of the symmetric exclusion processes, which,
under the diffusive space time rescaling $ \hN(t,\xi) := \frac{1}{N} \h(N^2t,N\xi) $,
converge to the linear heat equation.

A natural question following hydrodynamic limit concerns the corresponding large deviations.
At this level, the \ac{TASEP} continues to exhibit drastic difference with
its reversible counterpart, symmetric exclusion processes.
The \ac{LDP} for symmetric exclusion processes is obtained in \cite{kipnis89},
and the typical large deviations have speed $ N $, i.e, of probability $ \exp(-O(N)) $,
and are characterized by solutions of parabolic conservative PDEs.
On the other hand,
under the wedge initial condition,
one-point large deviations of the \ac{TASEP} exhibits \emph{asymmetric} tails:
the lower tail of $ \h(N,0) $ has probability $ \exp(-O(N)) $ 
while the upper tail has probability $ \exp(-O(N^2)) $, i.e.,
\begin{align}
	\label{eq:lowertails}
	&\frac{1}{N}\log\Pr\Big( \tfrac1N\h(N,0) < h(1,0) -\alpha \Big) \rightarrow I^{\text{lw}}(\alpha),
	\quad
	\alpha \in (0,h(1,0)),
\\
	\label{eq:uppertails}
	&
	\frac{1}{N^2}\log\Pr\Big( \tfrac1N\h(N,0) > h(1,0) + \alpha \Big) \rightarrow I^{\text{up}}(\alpha),
	\quad
	\alpha >0.
\end{align}
The lower tail large deviations with an exact rate function as in~\eqref{eq:lowertails}
was obtained in \cite{seppaelaeinen98a} using coupling techniques;
for a different but closely related model,
the complete one-point large deviations as in~\eqref{eq:lowertails}--\eqref{eq:uppertails}
was obtained \cite{deuschel99} using combinatorics tools.
These results show that the \ac{TASEP} in general exhibits two levels of deviations,
one of speed $ N $ and the other of speed $ N^2 $.
We note in the passing that similar two-scale behaviors are also
observed in random matrix theory (e.g., the asymmetric tails of the Tracy-Widom distributions, c.f., \cite{tracy94}),
and in stochastic scalar conservation laws \cite{mariani10}.

The existence of two-scale large deviations can be easily understood in the context of the exclusion processes.
Recall that, for the \ac{TASEP}, $ \h(t,x) $ records the total number of particle passing through $ x $.
The lower deviation $ \frac1N\h(N,0) < h(1,0) -\alpha $ (with $ \alpha>0 $) can be achieved by slowing down the Poisson clock at $ x=0 $.
Doing so creates a blockage and decelerates particle flow across $ x=0 $.
Such a situation involves slowdown of a \emph{single} Poisson clock for time $ t\in[0,N] $, and occurs with probability $ \exp(-O(N)) $.
On the other hand, for the upper deviation $ \frac1N\h(N,0) > h(1,0)+\alpha $,
isolated accelerations have no effect due to the nature of exclusion.
Instead, one needs to speed up the Poisson clocks \emph{jointly} at $ N $ sites, which happens with probability $ (\exp(-O(N)))^N=\exp(-O(N^2)) $.

The speed-$ N $ \emph{functional} large deviations has been studied by
Jensen and Varadhan \cite{jensen00,varadhan04}.
It is shown therein that, up to probability $ \exp(-O(N)) $,
configurations concentrate around weak, generally non-entropy, 
solutions of the Burgers equation~\eqref{eq:burgers}.
The rate function in this case essentially measures how `non-entropic' the given solution is.
In more broad terms,
the speed-$ N $ large deviations 
of asymmetric exclusion processes have captured much attentions,
partly due to their connection with the Kardar--Parisi--Zhang universality
and their accessibility via Bethe ansatz.
In particular, much interest has been surrounding the problems
of open systems with boundaries in contact with stochastic reservoirs,
where rich physical phenomena emerge.
We mention \cite{derrida98,derrida03,bodineau06} and the references therein
for a non-exhaustive list of works in these directions.

In this article, we study the speed-$ N^2 $ functional large deviations
that corresponds to the upper tail in~\eqref{eq:uppertails}.
These deviations are larger than those considered in \cite{jensen00,varadhan04},
and stretch beyond weak solutions of the Burgers equation.
Furthermore, the speed-$ N^2 $ deviations studied here have interpretation in terms of tiling models.
As noted in \cite{borodin16}, asymmetric simple exclusion processes (and hence the \ac{TASEP})
can be obtained as a continuous-time limit of the stochastic \ac{6VM}.
The \ac{6VM} is a model of random tiling on $ \Z^2 $, with six ice-type tiles,
and the stochastic \ac{6VM} is specialization where
tiles are updated in a Markov fashion \cite{gwa92,borodin16}.
Associated to these tiling models are height functions.
Due to the strong geometric constraints among tiles,
the height functions exhibit intriguing shapes reflecting the influence of a prescribed boundary condition. 
A preliminary step toward understanding these shapes is to establish the corresponding 
variational problem via the speed-$ N^2 $ large deviations.
For the \ac{6VM} at the free fermion point, or equivalently the dimer model,
much progress has been obtained thanks to the determinantal structure.
In particular, the speed-$ N^2 $ \ac{LDP} of the dimer model is established in \cite{cohn01}.
%
%

\subsection{Statement of the Result}
We begin by setting up the configuration space and topology.
Consider the space
\begin{align}
\label{eq:Splip}
\begin{split}
	\Splip &:= \big\{ 
		f\in C(\R) :  
		0\leq f(\xi)-f(\zeta) \leq \xi-\zeta, \ \forall \zeta\leq\xi\in\R
	\big\}
\\
	 &=
	\big\{ f\in C(\R) : \text{ Lipschitz, } f'\in[0,1] \text{ a.e.} \big\}
\end{split}
\end{align}
of Lipschitz functions with $ [0,1] $-valued derivatives.
Hereafter, `a.e.'\ abbreviates `almost everywhere/every with respect to Lebesgue measure'.
Indeed, for any height profile $ \f\in\Spf $,
the corresponding scaled profile $ \f_N (\xi) = \frac1N \f (N\xi)$ is $ \Splip $-valued 
(after the prescribed linear interpolation).
Endow the $ \Splip $ with the uniform topology over compact subsets of $ \R $.
More explicitly, 
writing $ \Vert f \Vert_{C[-r,r]} := \sup_{[-r,r]}|f| $ for the uniform norm restricted to $ [-r,r] $,
on $ C(\R) \supset \Splip $ we define the following metric
\begin{align}
	\label{eq:dist}
	\dist(f^1,f^2) 
	:=
	\sum_{k=1}^\infty 2^{-k} \big( \Vert f^1-f^2 \Vert_{C[-k,k]} \wedge 1 \big).
\end{align}
Having defined the configuration space $ \Splip $ and its topology, we turn to the path space.
To avoid technical sophistication regarding topology,
we fix a finite time horizon $ [0,T] $, $ T\in(0,\infty) $ hereafter.
Adopt the standard notation $ D([0,T],\Splip) $
for the space of right-continuous-with-left-limits paths $ t\mapsto h(t,\Cdot) \in \Splip $.
We define the following path space:
\begin{align}
	\label{eq:sp1}
	\Sp := \big\{ h\in D([0,T],\Splip) :
	h(s,\xi) \leq h(t,\xi), \ \forall s\leq t\in T, \xi\in\R \}.
\end{align}
Throughout this article, we endow the space $ \Sp $ with Skorokhod's $ J_1 $ topology.

We say a function $ h\in\Sp $ has (first order) derivatives if,
for some Borel measurable functions $ h^1,h^2:[0,T]\times\R\in[0,\infty) $,
\begin{align}
	\label{eq:ht}
	h(t',\xi) - h(t,\xi) &= \int_t^{t'} h^1(s,\xi) ds,
	\
	\text{for all } t<t' \in[0,T], \quad \text{for a.e. } \xi\in\R,
	\\
	\label{eq:hx}
	h(t,\xi') - h(t,\xi) &= \int_{\xi}^{\xi'} h^2(t,\zeta) d\zeta,
	\
	\text{for all } \xi<\xi' \in\R, \quad \text{for a.e. } t\in[0,T].
\end{align}
For a given $ h\in\Sp $, if such functions $ h^1, h^2 $ exist,
they must be unique up to sets of Lebesgue measure zero.
We hence let $ h^1 $ and $ h^2 $ be denoted by $ h_t $ and $ h_\xi $, respectively,
and refer to them as \emph{the} $ t $- and $ \xi $-derivatives of $ h $.
Set
\begin{align}
	\label{eq:Spd}
	\Spd := \{ h\in\Sp: \ h \text{ has derivatives in the sense of \eqref{eq:ht}--\eqref{eq:hx}} \}.
\end{align}
Referring back to \eqref{eq:Splip}, we see that each $ h\in\Sp $
\emph{automatically} has $ \xi $-derivative in the sense of~\eqref{eq:hx}, so
\begin{align}
	\label{eq:Spd:}
	\Spd = \{ h\in\Sp: \ h \text{ has }  t\text{-derivatives in the sense of \eqref{eq:ht}} \}.
\end{align}
Recall from~\eqref{eq:TASEPtoCGM} that $ h $ has the interpretation of integrated current of particles.
Under such an interpretation,
$ h_t\in[0,\infty) $ corresponds to the (instantaneous) flux, 
and $ h_\xi\in[0,1] $ represents the (local) density of particles.

Next, consider the large deviation rate function of Poisson variables:
\begin{align}
	\label{eq:prf}
	\prf(\lambda|u) := \lambda\log(\tfrac{\lambda}{u})-(\lambda-u).
\end{align}
More precisely, recall from \cite[(4.5)]{seppaelaeinen98} that, for $ X_N\sim\Pois(Nu) $, we have
\begin{align*} 
	\lim_{\e\to 0}\lim_{N\to\infty} \frac{1}{N} \log \Pr\big( X_N \in (-N\e+N\lambda,N\lambda+N\e) \big) 
	= -\prf(\lambda|u).
\end{align*}
When $ u=1 $, we write $ \prf(\lambda):=\prf(\lambda|1) $ to simplify notations.
%
Consider the truncated function 
$ \prff(\lambda) := \prf(\lambda\vee 1) = \prf(\lambda)\ind_\set{\lambda \geq 1} $.
We define
\begin{align}
	\label{eq:lrfup}
	&
	\lrfup: [0,\infty)\times[0,1]\to [0,\infty],
	\quad
	\lrfup(\kappa,\rho) := \big(\rho\wedge(1-\rho)\big) 
		\prff\big( \tfrac{\kappa}{\rho\wedge(1-\rho)} \big),
\\
	\label{eq:lrf}
	&\lrf: [0,\infty)\times[0,1]\to [0,\infty],
	\quad
	\lrf(\kappa,\rho) := \rho(1-\rho) \prff\big( \tfrac{\kappa}{\rho(1-\rho)} \big),
\end{align}
under the convention that $ J^{(i)}(\kappa,0) := \lim_{\rho \downarrow 0} J^{(i)}(\kappa,\rho) $
and $ J^{(i)}(\kappa,1) := \lim_{\rho \uparrow 1} J^{(i)}(\kappa,\rho) $.
More explicitly, $ J^{(i)}(\kappa,0)|_{\kappa>0}:=\infty $,
$ J^{(i)}(\kappa,1)|_{\kappa>0}:=\infty $ and $ J^{(i)}(0,0):=0, J^{(i)}(0,1):=0 $.
%
To simplify notations,
for processes such as $ \h(t,x) $, $ h(t,\xi) $,
in the sequel we often write $ \h(t):=\h(t,\Cdot) $, $ h(t):=h(t,\Cdot) $
for the corresponding fixed-time profiles.
Hereafter throughout this particle, we fix a macroscopic initial condition $ \hIC\in\Splip $.
Under these notations, we define
\begin{align}
	\label{eq:rfup}
	&
	\rfup(h)	
	:= 
	\left\{\begin{array}{l@{}l}
		\displaystyle
		\int_0^T \int_{\R} \lrfup(h_t,h_\xi) \; dtd\xi, 
		&\quad\text{if } h\in\Spd \text{ and } h(0)=\hIC,
		\\
		~
		\\
		\infty,	&\quad\text{if } h\notin\Spd \text{ or } h(0)\neq\hIC,
	\end{array}\right.
\\
	\label{eq:rf}
	&
	\rf(h)	
	:= 
	\left\{\begin{array}{l@{}l}
		\displaystyle
		\int_0^T \int_{\R} \lrf(h_t,h_\xi) \; dtd\xi, 
		&\quad\text{if } h\in\Spd \text{ and } h(0)=\hIC,
		\\
		~
		\\
		\infty,	&\quad\text{if } h\notin\Spd \text{ or } h(0)\neq\hIC.
	\end{array}\right.
\end{align}

With the macroscopic initial condition $ \hIC $ fixed as in the preceding,
we fix further a \emph{deterministic} microscopic initial condition $ \hic\in\Spf $ of the \ac{TASEP} such that,
with $ \hic_N(\frac{x}{N}) := \frac1N\hic(x) $ 
(and linearly interpolated onto $ \R $),
\begin{align}
	\label{eq:ic:cnvg}
	\lim_{N\to\infty}
	\dist(\hic_N, \hIC ) = 0.
\end{align}
\begin{remark}
\label{rmk:ic:cnvt}
We \emph{allow} $ \hic $ to depend on $ N $ as long as \eqref{eq:ic:cnvg} holds,
but \emph{omit} such a dependence in the notation.
This is to avoid confusion with subscripts in $ N $, such as $ \hic_N $, which denote scaled processes.
\end{remark}

With the initial condition $ \hic \in\Spf $ being fixed,
throughout this article we let $ \h(t,x) $ and $ \hN(t,\xi) = \frac 1N \h(Nt, N\xi)$
denote the micro- and macroscopic height processes starting from $ \hic $,
and write $ \PN $ for the law of the \ac{TASEP}.
The following is our main result:

\begin{theorem}
\label{thm:main}
Let $ \hIC\in\Sp $ and $ \hN $ be given as in the preceding.
\begin{enumerate}[label=(\alph*),leftmargin=5ex]
\item \label{enu:upb} For any given closed $ \calC\subset \Sp $,
\begin{align}
	\label{eq:upb}
	\limsup_{N\to\infty} \frac{1}{N^2} \log \PN(\h_N\in\calC) \leq -\inf_{h\in \calC} \rfup(h).
\end{align}
\item \label{enu:lwb} 
For any given open $ \calO\subset \Sp $, we have
\begin{align}
	\label{eq:lwb}
	\liminf_{N\to\infty} \frac{1}{N^2} \log \PN(\hN\in\calO) \geq -\inf_{h\in \calO } \rf(h).
\end{align}
\end{enumerate}
\end{theorem}

After posting of this article, the recent work \cite{de18} gives an explicit characterization of the rate function for the five-vertex model.
This is done by taking the $ N\to\infty $ limit of the free energy obtained from the Bethe ansatz.
Since the the Bethe roots of the \ac{TASEP} and the five-vertex model exhibits very similar structures,
the result~\cite{de18} points to a way of obtaining the rate function of the \ac{TASEP}.


\subsection{A heuristic of Theorem~\ref{thm:main}\ref{enu:lwb}}
\label{sect:intro:heu}
Here we give a heuristic of Theorem~\ref{thm:main}\ref{enu:lwb}.
As mentioned previously, the \ac{TASEP} is a degeneration of the stochastic \ac{6VM}.
The latter, as a tiling model, enjoys the Gibbs conditioning property.
That is, given a subset $ \calA\subset \Z^2 $, 
conditioned on the tiles along the boundary of $ \calA $,
the tiling within $ \calA $ is independent of the tiling outside of $ \calA $.
Such a property suggests a rate function of the form $ I(h) = \int J(h_t,h_\xi) dtd\xi $.
To see this, take a triangulation of $ \R^2 $, and approximate $ h $ by a \emph{linear} function within each triangle.
Thanks to the Gibbs property, the rate function is approximated by the sum of the rates on each triangle.
The latter, since $ h $ is approximated by a linear function on each triangle,
should take the form $ J(h_t,h_x,h)|\triangle| $, where $ |\triangle| $ denotes the area of the triangle.
Further, since the probability law of the \ac{6VM} is invariant under height shifts ($ \h\mapsto \h+\alpha $),
$ J $ should not depend on $ h $, suggesting $ J=J(h_t,h_x) $.

The \ac{TASEP}, being a degeneration of the stochastic \ac{6VM},
should also possess a rate function of the aforementioned form
$$ 
I(h) = \int_0^T\int_\R J(h_t,h_\xi) dt d\xi .
$$
Next, consider a linear deviation $ h^*(t,\xi) = \alpha + \kappa t + \rho \xi $, 
and consider all possible probability laws $ \QN $ on $ \Sp $ such that, 
under $ \QN $, the resulting process $ \hN $ approximates $ h^* $.
The rate $ I(h^*) $ should then be the infimum of the relative entropy $ \frac{1}{N^2} H(\QN|\PN) $ among all such $ \QN $.
Put it differently, we seek the most entropy-cost-effective fashion
of perturbing the law of the \ac{TASEP},
under the constraint that the resulting process $ \hN $ approximates $ h^* $.

Let $ \lambda:=\frac{\kappa}{\rho(1-\rho)} $.
The linear function $ h^* $ is an entropy solution of the equation
$
	h^*_t = \lambda h^*_\xi(1-h^*_\xi).
$
This is the Burgers equation~\eqref{eq:burgers} 
with a time-rescaling $ h(t,\xi) \mapsto h(\lambda t,\xi) $.
In view of the aforementioned hydrodynamic limit result of the \ac{TASEP},
one possible candidate of $ \QN^\lambda $, 
is to change the underlying Poisson clocks to have rate $ \lambda $ instead of unity. 
Equivalently, $ \QN^\lambda $ is obtained by
rescaling entire process $ \h(t) \mapsto \h(\lambda t) $ by a factor $ \lambda $.
This being the case, $ \hN $ necessarily converges to $ h^* $ under $ \QN^\lambda $.
We next calculate the cost of $ \QN^\lambda $.
Recall the definition of the mobility function $ \mob(\f,x) $ from~\eqref{eq:mob}.
Roughly speaking, the cost per site $ x\in\Z $ per unit amount of time is $ \prf(\lambda)\mob(\h(Nt),x) dt $.
This is accounted by the rate $ \prf(\lambda) $ of perturbing each Poisson clock,
modulated by the mobility function $ \mob(\h(Nt),x) $, since disallowed jumps are irrelevant. 
Since, under $ \QN^\lambda $, we expect $ \hN $ to approximate the targeted function $ h^* $,
referring back to the expression~\eqref{eq:mob},
we \emph{informally} approximate $ \mob(\h(Nt),x) $ by $ (1-h^*_\xi)h^*_\xi $.
Such an informal calculation gives
\begin{align*}
	\frac{1}{N^2} H(\QN^\lambda|\PN) 
	\approx
	\int\int \hat\lrf(\kappa,\rho) dtd\xi,
	\quad
	\hat\lrf(\kappa,\rho) 
	:= 
	h^*_\xi(1-h^*_\xi) \prf(\lambda)
	=
	\rho(1-\rho)\prf(\tfrac{\kappa}{\rho(1-\rho)}).	
\end{align*}
Of course, the last integral is infinite, but our discussion here focuses on the density $ \hat\lrf(\kappa,\rho) $.

The aforementioned $ \QN^\lambda $ being a candidate for the law $ \QN $,
we must have $ J(\kappa,\rho) \leq \hat\lrf(\kappa,\rho) $.
As it turns out, for $ \lambda < 1 $, we can device another choice of law such that the cost is \emph{zero}.
To see this, consider an axillary parameter $ \d\downarrow 0 $. 
Our goal is to maintain a constant flux $ \kappa $, 
\emph{lower} than the hydrodynamic value $ \rho(1-\rho) $,
together with the constant density $ \rho $, in the most cost-effective fashion.
Instead of slowing down the Poisson clocks uniformly by $ \lambda $,
let us slow down only in windows $ \calW_i $ of macroscopic width $ \d^2 $,
 every distance $ \d(1-\d) $ apart;
see Figure~\ref{fig:interm}.
We refer to this as the `intermittent construction'.
\begin{figure}[h]
\psfrag{W}{$ \color{blue}\calW_i $}
\psfrag{A}[b]{\color{blue} rate$ =\lambda $}
\psfrag{B}[b]{rate$ =1 $}
\psfrag{D}[b]{$ \color{blue}\d^2 $}
\psfrag{E}[b]{$ \d(1-\d) $}
\includegraphics[width=.8\textwidth]{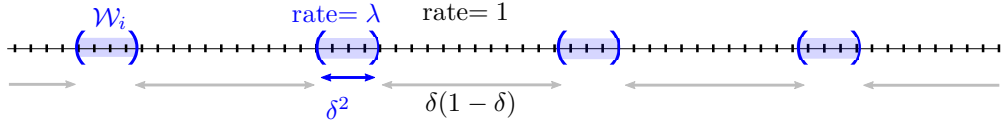}
\caption{The Intermittent construction. The ticks represent the scaled lattice $ \frac{1}{N}(\frac12+\Z) $
where particles reside.}
\label{fig:interm}
\end{figure}
Even though slow-down is only enforced on the $ \calW_i $'s,
since particles cannot jump ahead of each other,
this construction achieves an overall constant flux $ \kappa $ through \emph{blocking}.
More explicitly, it is conceivable that, under the intermittent construction,
particles exhibit the macroscopic stationary density profile as depicted in Figure~\ref{fig:intermDen}.
In between the windows $ \calW_i $, the density takes two values $ \rho_1,\rho_2 $,
with $ \rho_1 > \rho > \rho_2 $, as a result of blocking.
Even though the density varies among the values $ \rho,\rho_1,\rho_2 $, referring to Figure~\ref{fig:intermDen},
we see that as $ \d\downarrow 0 $ the density profile converges to $ \rho $ in an average sense. 
As for the cost, since the region $ \{\calW_i\}_i $ has fraction $ \d $,
as $ \d\downarrow 0 $ the cost in entropy per unit length (in $ \xi\in\R $) goes to zero.
This suggests that the intermittent construction gives approximately zero cost for $\lambda<1$.
\begin{figure}[h]
\psfrag{R}{$ \rho $}
\psfrag{S}{$ \rho_1 $}
\psfrag{T}{$ \rho_2 $}
\psfrag{L}[c]{$ \color{blue}\lambda $}
\psfrag{U}[c]{rate$ =1 $}
\psfrag{D}[t]{$ \color{blue}\d^2 $}
\psfrag{E}[c]{$ \scriptstyle r_1 $}
\psfrag{F}[c]{$ \scriptstyle r_2 $}
\psfrag{G}[t]{$ \d(1-\d) $}
\includegraphics[width=.9\textwidth]{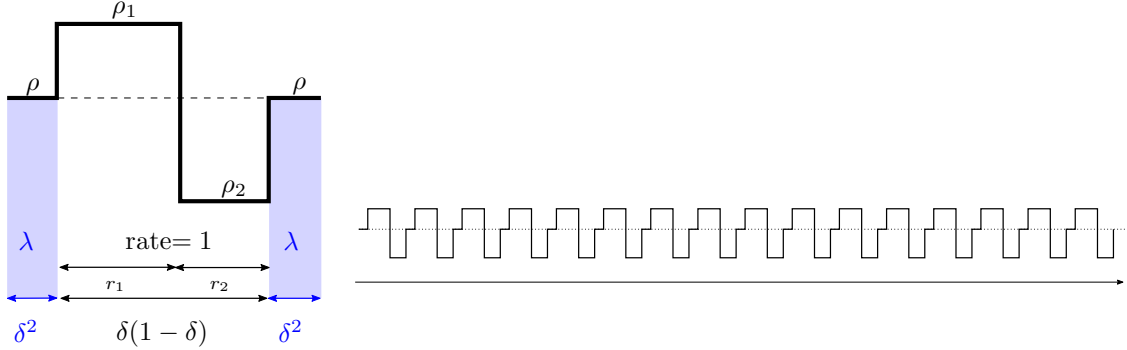}
\caption{%
	Expected macroscopic density under the intermittent construction.
	Here $ \rho_1>\rho_2\in[0,1] $ are the unique solutions of the equation $ \rho_i(1-\rho_i)=\kappa $,
	and $ r_1,r_2 $ are such that $ r_1+r_2=\d(1-\d) $, $ r_1\rho_1+r_2\rho_2 = (r_1+r_2)\rho $.
}
\label{fig:intermDen}
\end{figure}

Combining the preceding discussions for the cases $ \lambda \geq 1 $ and $ \lambda < 1 $,
we have then
\begin{align*}
	J(\kappa,\rho)
	\leq
	\left\{\begin{array}{l@{,}l}
		(1-\rho)\rho\prf(\tfrac{\kappa}{\rho(1-\rho)})	&	\text{ if } \lambda \geq 1
	\\
		0												&	\text{ if } \lambda < 1
	\end{array}\right\}
	=
	(1-\rho)\rho\prff(\tfrac{\kappa}{\rho(1-\rho)}) = \lrf(\kappa,\rho).
\end{align*}
This heuristic gives an upper bound $ J^{(2)} $ on the rate function $ J $,
which corresponds to a large deviation \emph{lower} bound (i.e., lower bound on the probability) as in Theorem~\ref{thm:main}\ref{enu:lwb}.
On the other hand, for the lower bound on $ J $ (i.e., large deviation upper bound, as in Theorem~\ref{thm:main}\ref{enu:upb}), 
we are only able to prove $ J^{(1)} \le J $, 
obtained by bounding the mobility by $ \rho\wedge(1-\rho) $.
The bounds $ J^{(1)}, J^{(2)} $ do not match, and finding the actual rate function remains an open question.

It follows from our result that deviations that are subsolutions of Burgers equation, i.e. $ h_t \le (1-h_\xi)h_\xi $, 
have probability larger than $e^{-cN^2}$. On the other hand, by the Jensen-Varadhan large deviation result\cite{varadhan04}, 
they are smaller than  $e^{-cN}$.
It remains an open problem to determine the order on these subsolution as well as other intermediate deviations. In this direction 
in appendix \ref{append:smalldevi} we investigate large deviations in a finer topology that track the oscillations of $h_\xi$ through Young measures, 
in the spirit of \cite{mariani10}.

\subsection*{Acknowledgements}
SO thanks Lorenzo Bertini, Bernard Derrida and  S.\ R.\ Srinivasa Varadhan 
for useful discussions.
LCT thanks Ivan Corwin, Amir Dembo, Fraydoun Rezakhanlou, Timo Sepp{\"a}l{\"a}inen and S.\ R.\ Srinivasa Varadhan
for useful discussions.

SO's research is supported by ANR-15-CE40-0020-01 grant LSD.
LCT's research was partially supported by a Junior Fellow award from the Simons Foundation,
and by the NSF through DMS-1712575. 

\subsection*{Outline}
In Section~\ref{sect:rf} we establish some useful properties of the functions $ I^{(i)} $ and $ J^{(i)} $.
The lower semi-continuity is not used in the rest of the article, but we include it as a useful property for future reference.
In Section~\ref{sect:pfmian},
we prove Theorem~\ref{thm:main}\ref{enu:upb} and \ref{enu:lwb}
assuming Propositions~\ref{prop:ent:lw} and \ref{prop:ent:up}, respectively.
These propositions concern bounds on relative entropies.
We settle Proposition~\ref{prop:ent:lw} in Section~\ref{sect:pfentlw},
and then devote the rest of the article, Sections~\ref{sect:ITASEP}--\ref{sect:speed},
to showing Proposition~\ref{prop:ent:up}.

\subsection*{Convention}
Throughout this article,
$ x,i,j,k,\ell,m,n\in\Z $ (and similarly for $ x_1,i' $, etc.) denote integers,
and $ \xi,\zeta\in\R $ denote real numbers.
The letters $ s,t $ always denote time variables, with either $ s,t\in [0,T] $ or $ [0,NT] $;
We use $ \h $, $ \g $, etc, to denote un-scaled, \ac{TASEP} height processes,
with $ \hN $, $ \gN $ being the corresponding scaled processes.
The same convention applies also for the initial conditions $ \hic $, $ \gic $, $ \hic_N $, $ \gic_N $
of these processes.

\section{Properties of Functions $ I^{(i)} $ and $ J^{(i)} $}
\label{sect:rf}

Recall that $ \prff(\lambda) := \prff(1\wedge\lambda) $
denote the truncated rate function for Poisson variables.
Under this convention, we still have that $ \lambda\mapsto \prff(\lambda) $ is convex,
and that $ \prff'(\lambda) = \log(\lambda\vee 1) $.

The functions $ \lrfup, \lrf $, defined in~\eqref{eq:lrfup}--\eqref{eq:lrf}, take infinite value at $ \rho=0,1 $.
This property posts undesirable technical issues for our analysis,
and hence we consider the following truncations.
Let $ \mobbup(\rho):= \rho \wedge(1-\rho) $ and $ \mobb(\rho):=\rho(1-\rho) $.
For small $ a\in(0,\frac12) $, define the following truncations
\begin{align}
	&
	\label{eq:fluxeup}
	\mobbup_a(\rho) := (1-a^2) \mobbup(\rho) + a^2,
\\
	&
	\label{eq:fluxe}
	\mobb_a(\rho) :=
	\left\{\begin{array}{l@{,}l}
		\rho(1-\rho) &\text{ when } \rho\in[a,1-a],
	\\
		a(1-a)+(1-2a)(\rho-a)	&\text{ when } \rho\in[0,a),
	\\
		a(1-a)+(2a-1)(\rho-(1-a))	&\text{ when } \rho\in(1-a,1].		
	\end{array}\right.
\end{align}
From this construction, it is clear that $ \Phi^{(i)}_a\geq\Phi^{(i)} $, $ \Phi^{(i)}_a \geq a^2 >0 $
and that $ \rho\mapsto\Phi^{(i)}_a(\rho) $ is concave.
We then define 
\begin{align}
	\label{eq:lrfeup}
	&\lrfupa(\kappa,\rho) := \mobbup_a(\rho) \prff\big( \tfrac{\kappa}{\mobbup_a(\rho)} \big),
\\
	\label{eq:lrfe}
	&\lrf_a(\kappa,\rho) := \mobb_a(\rho) \prff\big( \tfrac{\kappa}{\mobb_a(\rho)} \big).
\end{align}
A straightforward differentiation
$ \frac{d~}{d\xi}(\xi\prff(\frac{\kappa}{\xi})) = -(\frac{\kappa}{\xi}-1)_+ \leq 0 $ shows that
\begin{align}
	\label{eq:nonincr:mob}
	\xi\mapsto \xi\prff(\tfrac{\kappa}{\xi}) \text{ is nonincreasing, } \forall \text{ fixed } \kappa\in[0,\infty).
\end{align}
so in particular $ \lrfupa(\kappa,\rho) \leq a^2 \prff(\kappa a^{-2}) <\infty. $

\begin{lemma}
\label{lem:lrf:cnvx}
The functions $ J^{(i)}: [0,\infty)\times[0,1] \to [0,\infty] $ and
$ J^{(i)}_a: [0,\infty)\times[0,1] \to [0,\infty) $ are convex, for $ i=1,2 $.
\end{lemma}
\begin{proof}
It is straightforward to verify that
\begin{align*}
	J^{(i)}(\kappa,\rho) = \sup_{\alpha \geq 0} \big\{ \kappa \alpha - \Phi^{(i)}(\rho)(e^\alpha-1) \big\},
	\quad
	J^{(i)}_a(\kappa,\rho) = \sup_{\alpha \geq 0} \big\{ \kappa \alpha - \Phi^{(i)}_a(\rho)(e^\alpha-1) \big\}.
\end{align*}
Using these expressions and the concavity of $ \rho\mapsto \Phi^{(i)}(\rho) $ and  $ \rho\mapsto \Phi^{(i)}_a(\rho) $
gives the desired result.
\end{proof}

We next establish a few technical results.
To setup notations, let $ \{\sigma^n_i:=\frac{iT}{2^n}\}_{i=0}^{2^n} $ 
be an equally spaced partition of $ [0,T] $, dyadic in $ n $.
Define, for $ h\in\Sp $, the following quantities
\begin{align}
	\label{eq:rffn}
	\rff_n(h,\xi) 
	&:= 
	\sum_{i=1}^{2^n} \frac{T}{2^n} 
	\prff\Big(\frac{h(\sigma^n_{i},\xi)-h(\sigma^n_{i-1},\xi)}{\sigma^n_{i}-\sigma^n_{i-1}} \Big),
\\	
	\label{eq:rff}
	\rff(h) 
	& :=
	\sup_{n}\int_{\R} \rff_n(h,\xi) d\xi.
\end{align}
\begin{lemma}
\label{lem:rf<inf}
For any $ h\in\Sp $, if $ \rff(h)<\infty $ then $ h\in\Spd $.
\end{lemma}
\begin{proof}
Recall that, we say a function $ f:[0,T]\to \R $ is absolutely continuous if, for any given $ \e>0 $,
there exists $ \d>0 $ such that for any
finite sequence of pairwise disjoint subintervals 
$ \{[s_0,t_0], [s_1 t_1], \ldots, [s_n, t_n]\} $ of $ [0,T] $
with $ \sum_{i} (t_i-s_{i}) \leq \d $,
we always have $ \sum_{i=1}^n |f(t_{i})-f(s_{i})| \leq \e $.
Recall from~\eqref{eq:Spd:} that,
to show $ h\in\Spd $, it suffices to show the existence of $ t $-derivate of $ h $, 
in the sense of~\eqref{eq:ht}.
This, by standard theory of real analysis,
is equivalent to showing
\begin{align}
\label{eq:ac:goal}
	t\mapsto h(t,\xi) \text{ is absolutely continuous },
	\text{for a.e. } \xi\in \R.
\end{align}

With $ \rff_n(h,\xi) $ defined in~\eqref{eq:rff},
by the convexity of $ \lambda\mapsto \prff(\lambda) $, we have that
\begin{align}
	\label{eq:Ginc}
	\rff_n(h,\xi) \leq \rff_m(h,\xi), \quad \forall n<m.
\end{align}
Fix an arbitrary radius $ r<\infty $.
Alongside with the partition $ \{\sigma^n_i\}_{i=0}^{2^n} $ of time,
we consider also the equally spaced, dyadic partition 
$ \{\xi^n_j := \frac{jr}{2^n}\}_{j=-2^n}^{2^n} $ of $ [-r,r] $.
Let $ U^n_j := [\xi^n_{j-1},\xi^n_{j}) $, $ j=-2^n,\ldots, 2^n-1 $,
denote the intervals associated with the partition. 
Fixing arbitrary $ 0<\e<1 $,
we inductively construct sets $ \calU(n)\subset [-r,r] $ as follows.
Set $ \calU(0)=\emptyset $, and, for $ n\geq 1 $, let
\begin{align}
\label{eq:Dn}
	\calU(n) 
	:= 
	\bigcup\Big\{ 
		U^{n}_j :  
		\sup_{ \xi\in U^n_j\setminus \calU(n-1) } \rff_n(h,\xi) \geq 
		\e^{-1}
	\Big\}
\end{align}
denote the union of intervals $ U^n_j $ on which the function $ \rff_n(h,\xi) $ 
exceeds the threshold $ \e^{-1} $,
\emph{excluding} those points from the previous iteration $ \calU(n-1) $.
Since $ h(t)\in\Splip $,
we have that $ |h(t,\xi_1)-h(t,\xi_2)| \leq |\xi_1-\xi_2| \leq r2^{-n} $,
for all $ \xi_1,\xi_2\in U^n_j $.
This gives
\begin{align}
\label{eq:slopecmp}
	\Big|
	\frac{h(\sigma^n_{i},\xi_1)-h(\sigma^n_{i-1},\xi_1)}{\sigma^n_{i}-\sigma^n_{i-1}}
	-
	\frac{h(\sigma^n_{i},\xi_2)-h(\sigma^n_{i-1},\xi_2)}{\sigma^n_{i}-\sigma^n_{i-1}}
	\Big|
	\leq
	\frac{r2^{-n+1}}{\sigma^n_{i}-\sigma^n_{i-1}}
	=
	\frac{2r}{T},
	\quad
	\forall \xi_1,\xi_2\in U^n_j.
\end{align}
That is, the argument of $ \prff(\Cdot) $ in~\eqref{eq:rffn}
differs by at most $ \frac{2r}{T} $ as $ \xi $ varies among $ U^n_j $.
With $ \prff'(\lambda) = \log(\lambda\vee 1) $,
it is straightforward to verify that, for all $ \lambda_1<\lambda_2\in [0,\infty) $
with $ |\lambda_2-\lambda_1| \leq \frac{2r}{T} $, we have
\begin{align*}
	\big| \prff(\lambda_2) - \prff(\lambda_1)| 
	\leq 
	\tfrac{2r}{T} \log((\lambda_1\wedge\lambda_2+\tfrac{2r}{T})\vee 1) 
	\leq
	c \big( \prff(\lambda_1)\wedge \prff(\lambda_2) + 1 ),
\end{align*}
for some constant $ c<\infty $ depending only on $ \frac{2r}{T} $.
In particular, for such $ \lambda_1,\lambda_2 $,
the maximal and minimal of $ \prff(\lambda_1) $ and $ \prff(\lambda_2) $
are comparable in the following sense: 
\begin{align}
	\label{eq:lambda12}
	\prff(\lambda_1) \vee \prff(\lambda_2)
	\leq
	(c+1)\big( \prff(\lambda_1)\wedge \prff(\lambda_2) \big) + c.
\end{align}
In view of~\eqref{eq:slopecmp},
we apply \eqref{eq:lambda12} with 
$ \lambda_i = \frac{h(\sigma^n_{i},\xi_1)-h(\sigma^n_{i-1},\xi_i)}{\sigma^n_{i}-\sigma^n_{i-1}} $,
for all $ \xi_1,\xi_2\in U^n_i $, 
to obtain
\begin{align}
\label{eq:Dn:bd:}
	\inf_{\xi\in U^n_j} \rff_n(h,\xi)
	\geq
	\frac{1}{c+1}\Big(\sup_{\xi\in U^n_j} \rff_n(h,\xi) \Big) -1
	\geq
	\frac{\e^{-1}}{c+1}-1,
	\quad
	\forall U^n_j \subset \calU(n).	
\end{align}
Sum the inequality~\eqref{eq:Dn:bd:} over all $ U^n_j \subset \calU(n) $,
and multiply both sides by $ |\calU(n)| $.
We then obtain
$	
	\int_{\calU(n)} \rff_n(h,\xi) d\xi
	\geq
	(\frac{\e^{-1}}{c+1}-1)|\calU(n)|.
$
From this and~\eqref{eq:Ginc}, we further deduce
\begin{align}
\label{eq:Dn:bd}
	\Big( \frac{\e^{-1}}{c+1}-1 \Big) |\calU(n)| 
	\leq 
	\int_{\calU(n)} \rff_n(h,\xi) d\xi
	\leq
	\int_{\calU(n)} \rff_m(\xi) d\xi,
	\quad
	\forall n\leq m.
\end{align}
Referring back to~\eqref{eq:Dn}, the sets~$ \calU(1),\calU(2),\ldots $ are disjoint.
Under this property, we let $ \calF(m) := \cup_{n=1}^m \calU(m) $ denote the union of the first $ m $ sets,
and sum~\eqref{eq:Dn:bd} over $ n=1,\ldots, m $ to obtain
\begin{align}
\label{eq:Dm:bd}
	\Big( \frac{\e^{-1}}{c+1}-1 \Big)  |\calF(m)| 
	\leq 
	\int_{\calF(m)} \rff_m(\xi) d\xi
	\leq
	\rff(h).
\end{align}
Set $  \calF_*:= \cup_{n=1}^\infty \calU(n) $.
Letting $ m\to\infty $ in \eqref{eq:Dm:bd} gives
$
	|\calF_*| \leq \rff(h)( \frac{\e^{-1}}{c+1}-1 ).
$
Now, with\\
$ \calF_* \supset \{ \xi\in [r,-r) : \sup_n \rff_n (h,\xi) \geq \e^{-1} \} $,
further letting $ \e \downarrow 0 $, we arrive at
\begin{align}
	\label{eq:E**}
	\Big|\Big\{ \xi\in [r,-r) : \sup_n \rff_n (h,\xi) =\infty \Big\}\Big| = 0.
\end{align}

With the properties $ \prff \geq 0 $ and $ \lim_{\lambda\to\infty}\frac{\prff(\lambda)}{\lambda} =\infty $,
it is standard to show that $ \sup_{n} \rff_n(h,\xi)<\infty $
implies the absolute continuity of $ t \mapsto h(t,\xi) $.
This together with~\eqref{eq:E**} 
shows that $ t \mapsto h(t,\xi) $ is absolutely continuous for a.e.\ $ \xi\in [-r,r) $.
As $ r<\infty $ is arbitrary,
taking a sequence $ r_n\uparrow \infty $ concludes the desired result~\eqref{eq:ac:goal}.
\end{proof}

\begin{lemma}
\label{lem:rff<rf}
For all $ h\in\Sp $, we have $ \rff(h) \leq I^j(h) $, $ j=1,2 $.
\end{lemma}
\begin{proof}
Assume without lost of generality $ h\in\Spd $ and $ h(0)=\hIC $, 
otherwise $ \rff(h)=\infty $.
Since $ \Phi^j(\rho) \leq 1 $ for all $ \rho\in[0,1] $,
by~\eqref{eq:nonincr:mob},
\begin{align}
	\label{eq:nonincr:mob:}
	J^j(h_t,h_\xi) 
	= \Phi^{(i)}(h_\xi) \prff(\tfrac{h_t}{\Phi^j(h_\xi)}) 
	\geq 
	\xi\prff(\tfrac{h_t}{\xi})|_{\xi=1}
	= 
	\prff(h_t).
\end{align}
Integrating this inequality over $ [0,T]\times\R $ gives
\begin{align}
	\label{eq:rff<rf1}
	\int_{\R} \Big( \int_0^T \prff(h_t) dt \Big) d\xi \leq I^j(h).
\end{align}
By the convexity of $ \lambda \mapsto \prff(\lambda) $,
we have that
\begin{align}
	\label{eq:rff<rf2}
	\frac{T}{2^n} \prff\Big(\frac{h(\sigma^n_{i},\xi)-h(\sigma^n_{i-1},\xi)}{\sigma^n_{i}-\sigma^n_{i-1}} \Big)
	\leq
	\int_{\sigma^n_{i-1}}^{\sigma^n_i} \prff(h_t(t,\xi)) dt.	
\end{align}
Summing the inequality~\eqref{eq:rff<rf2} over $ i=1,\ldots,2^n $,
gives $ \rff_n(h,\xi) \leq \int_0^T \prff(h_t) dt $.
Integrate this inequality over $ \xi\in\R $,
combine the result with~\eqref{eq:rff<rf1},
and take the supremum over $ n $.
We thus conclude the desired result $ \rff(h) := \sup_n \int_\R \rff_n(h,\xi) d\xi \leq I^j(h) $.
\end{proof}

The next result concerns local approximation of the derivatives $ h_t,h_\xi $ 
of a given deviation $ h\in\Spd $.
To setup the notations,
for given $ r<\infty $ and $ \ell<\infty $, 
we consider a partition
\begin{align}
	\label{eq:Rprtn}
	R_\ell(r)
	:=
	\big\{ 
		\square = [\tfrac{(i-1)T}{\ell},\tfrac{iT}{\ell}]\times[\tfrac{(j-1)r}{\ell},\tfrac{jr}{\ell}]:
		i=1,\ldots,\ell,
		j=-\ell+1,\ldots,\ell
	\big\}
\end{align}
of $ [0,T]\times[-r,r] $ into equal rectangles.
Write $ \fint_A fdtd\xi := \frac{1}{|A|} \int_A fdtd\xi $ for the average over a set $ A $.

\begin{lemma}
\label{lem:local:apprx}
For any fixed $ h\in\Spd $, we have that
\begin{align}
	\label{eq:rect:apprx}
	\limsup_{(r,a)\to(\infty,0)} 
	\limsup_{\ell\to\infty}
	\Bigg\{
		\sum_{\square\in R_\ell(r)} 
		|\square| \
		\lrfupa\Big({\textstyle\fint_{\square}} h_tdtd\xi \,,\, {\textstyle\fint_{\square}} h_\xi dtd\xi\Big)
	\Bigg\}
	\geq
	\int_0^T \int_\R \lrfup(h_t,h_\xi) dtd\xi.
\end{align}
\end{lemma}
\begin{proof}
Fix arbitrary $ \kappa_*,r<\infty $, $ a>0 $ and $ \e>0 $.
Recall the definition of the truncated rate density $ \lrfupa $ from~\eqref{eq:lrfeup}.
We begin by proving the following statement:
there exists $ \ell_*<\infty $ such that, for all $ \ell\geq\ell_* $,
\begin{align}
	\label{eq:local:apprx:}
	&\big| \bigcup\{ \square\in R_\ell(r) : E_\square(h) \geq \e \} \big| \leq \e,
\\
	\label{eq:local:apprx}
	&
	\quad\quad
	\text{where }
	E_\square(h) 
	:=
	\Big| 	
			\lrfupa\Big(\kappa_*\wedge{\textstyle\fint_{\square}} h_tdtd\xi \,,\, {\textstyle\fint_{\square}} h_\xi dtd\xi\Big) 
			- \fint_{\square} \lrfupa(\kappa_*\wedge h_t,h_\xi)dtd\xi 
	\Big|.
\end{align}

Given that $ \lrfupa(\kappa_*\wedge\Cdot,\Cdot) $ is bounded and Borel-measurable,
the statement~\eqref{eq:local:apprx:}
follows from standard real analysis, similarly to the proof of \cite[Lemma~2.2]{cohn01}.
We given a formal proof here for the sake of completeness.
In addition to $ \ell_* $, we consider an axillary parameter $ \ell_{**}\ $. 
Both $ \ell_* $ and $ \ell_{**} $ will be specified in the sequel.
Write $ h_t\wedge\kappa_* =: h^{\kappa_*}_t $ to simplify notations.
Regard the pair of derivatives $ F := (h^{\kappa_*}_t,h_\xi) $
as a measurable map $ F: [0,T]\times[-r,r] \to [0,\kappa^*]\times [0,1] $.
Partition the range $ [0,\kappa_*]\times[0,1] $ of $ F $
into subsets $ U_1,\ldots,U_n $, each of diameter at most $ \frac{1}{\ell_{**}} $.
We let $ V_i := F^{-1}(U_i) $ be the preimage of $ U_i $.
With $ B_b(t,\xi)\subset\R^2 $ denoting the ball of radius $ b $ centered at $ (t,\xi) $,
by the theory of measure density (see, e.g., \cite[Section~7.12]{rudin87}),
we have that
\begin{align*}
	\lim_{b\downarrow 0} \frac{|B_b(t,\xi)\cap V_i|}{|B_b(t,\xi)|}
	=
	1	\text{ for a.e. } (t,\xi) \in V_i,
	\quad
	i=1,\ldots,n.
\end{align*}
This being the case, 
there exists a compact set $ K_i \subset V_i $, with $ |K_i| \geq |V_i|-\frac{1}{2\ell_{*}} $, such that
\begin{align*}
	\lim_{b\downarrow 0} \frac{|B_b(t,\xi)\cap V_i|}{|B_a(t,\xi)|}
	=
	1
	\quad
	\text{for every } (t,\xi)\in K_i.
\end{align*}
From this and the compactness of $ K_i $,
we further constructed  a finite union of open balls $ O_i\supset K_i $, 
such that
\begin{align}
	\label{eq:OiVi}
	|O_i| \geq |K_i|-\tfrac{1}{2\ell_{*}}  \geq |V_i| - \tfrac{1}{\ell_{*}}.
\end{align}
Now, for a fix $ O_i $, we classify rectangles $ \square\in R_{\ell}(r) $
that intersects with $ O_i $ (i.e., $ \square\cap O_i\neq\emptyset $) into three types:
\begin{itemize}[leftmargin=5ex]
	\item Desired rectangles: $ \square\subset O_i $ with $ |\square\cap V_i| \geq (1-\frac{1}{\ell_{**}})|\square| $;
	\item Undesired rectangles: $ \square\subset O_i $ with $ |\square\cap V_i| < (1-\frac{1}{\ell_{**}})|\square| $;
	\item Boundary rectangles: $ \square\cap O_i \neq \emptyset $ and $ \square\cap O_i^c \neq \emptyset $.
\end{itemize}
Let $ \calA^i_\text{des} $,
$ \calA^i_\text{und} $ and $ \calA^i_\text{bdy} $
denote the respective sets of desired, undesired, and boundary rectangles with respect to $ O_i $,
and let $ A^i_\text{des} $, $ A^i_\text{und} $ and $ A^i_\text{bdy} $
denote the areas (i.e., Lebesgue measure) of the union of rectangles 
in $ \calA^i_\text{des} $, $ \calA^i_\text{und} $ and $ \calA^i_\text{bdy} $, respectively.
First, for each of the desired rectangle $ \square\in\calA^i_\text{des} $, 
\begin{align}
	\label{eq:local:desired}
	&|V_i\cap\square| \geq (1-\tfrac{1}{\ell_{**}})|\square|,
\\
	\label{eq:local:desired:}
	&|h^{\kappa_*}_t(t,\xi)-h^{\kappa_*}_t(t',\xi')|,\ |h_\xi(t,\xi)-h_\xi(t',\xi')| \leq \tfrac{1}{\ell_{**}},
	\quad
	\forall (t,\xi), \ (t',\xi')\in V_i\cap\square.
\end{align}
Recall the definition of $ E_\square(h) $ from~\eqref{eq:local:apprx}.
Since $ h^{\kappa_*}_t $ and $ h_\xi $ are bounded,
and since $ (\kappa,\rho) \mapsto \lrfupa(\kappa,\rho) $ is continuous,
for some large enough $ \ell_{**}\in \N $,
the condition~\eqref{eq:local:desired}--\eqref{eq:local:desired:} implies
\begin{align}
	\label{eq:desired}
	E_\square(h)
	\leq
	\e,
	\quad
	\forall\, \square \in \bigcup_{i=1}^n \calA^i_\text{des}.
\end{align}
Next, since each $ O_i $ is finite union of open balls, 
and since the rectangles $ \square\in R_\ell(r) $ in $ R_\ell(r) $ shrinks uniformly as $ \ell\to\infty $,
there exists $ \ell_{*}\in\Z\cap[3n\ell_{**},\infty) $ such that
\begin{align}
	\label{eq:Abdy}
	\sum_{i=1}^n A^i_\text{bdy} \leq \frac{\e}{3\ell_{**}},
	\quad
	\forall \ell \geq \ell_*.
\end{align}
Moving onto undesirable rectangles. 
From the preceding definition of undesirable rectangles, 
we have $ A^i_\text{und}(1-\frac{1}{\ell_{**}}) + A^i_\text{des} + A^i_\text{bdy} \geq |V_i| $.
Combining this with \eqref{eq:OiVi} gives
\begin{align}
	\label{eq:Aundesired:}
	A^i_\text{und}(1-\tfrac{1}{\ell_{**}}) + A^i_\text{des} + A^i_\text{bdy}
	\geq 
	|O_i|-\tfrac{1}{\ell_{*}}
	\geq
	(A^i_\text{und} + A^i_\text{des} )-\tfrac{1}{\ell_{*}}.
\end{align}
Rearrange terms in~\eqref{eq:Aundesired:} and sum over $ i $ to obtain 
\begin{align}
	\label{eq:Aundesired}
	\sum_{i=1}^n A^i_\text{und}
	\leq 
	\sum_{i=1}^n \ell_{**}\Big( \frac{1}{\ell_{*}}+A^i_\text{bdy}\Big)
	\leq
	\frac{n\ell_*}{\ell_{**}}+\frac{\e}{3}
	\leq
	\frac{2\e}{3}.
\end{align}
Combining \eqref{eq:desired}--\eqref{eq:Abdy} and \eqref{eq:Aundesired}, we  conclude~\eqref{eq:local:apprx:}.

Having established~\eqref{eq:local:apprx:},
we now let $ \ell\to\infty $ in~\eqref{eq:local:apprx:} to get
\begin{align*}
	\limsup_{\ell\to\infty}
	\Bigg|
		\sum_{\square\in R_\ell(r)} 
		|\square| \
		\lrfupa\Big(\kappa_*\wedge{\textstyle\fint_{\square}} h_tdtd\xi \,,\, {\textstyle\fint_{\square}} h_\xi dtd\xi\Big)
		-
		\int_0^T \int_{-r}^r &\lrfupa(\kappa_*\wedge h_t,h_\xi)dtd\xi
	\Bigg|	
\\
	&\leq
	2rT\e + \Vert \lrfupa(\kappa_*\wedge\Cdot,\Cdot) \Vert_\infty\e.
\end{align*}
As $ \e>0 $ is arbitrary, further letting $ \e\downarrow 0 $ gives
\begin{align}
	\label{eq:rect:apprx:}
	\lim_{\ell\to\infty}
	\Bigg\{
		\sum_{\square\in R_\ell(r)} 
		|\square| \
		\lrfupa\Big(\kappa_*\wedge{\textstyle\fint_{\square}} h_tdtd\xi \,,\, {\textstyle\fint_{\square}} h_\xi dtd\xi\Big)
	\Bigg\}	
	=
	\int_0^T \int_{-r}^r \lrfupa(\kappa_*\wedge h_t,h_\xi)dtd\xi.
\end{align}
Indeed, $ \lrfupa(h_t\wedge\kappa_*,h_\xi) $ increases as $ \kappa^* $ increases.
We then remove~$ \kappa_*\wedge\Cdot $ on the l.h.s.\ of~\eqref{eq:rect:apprx:} to make the resulting quantity larger,
and let $ \kappa_*\to\infty $ using the monotone convergence theorem on the r.h.s.
This gives
\begin{align*}
	\limsup_{\ell\to\infty}
	\Bigg\{
		\sum_{\square\in R_\ell(r)} 
		|\square| \
		\lrfupa\Big({\textstyle\fint_{\square}} h_tdtd\xi \,,\, {\textstyle\fint_{\square}} h_\xi dtd\xi\Big)
	\Bigg\}	
	\geq
	\int_0^T \int_{-r}^r \lrfupa(h_t,h_\xi)dtd\xi.
\end{align*}
Further letting $ (r,a)\to(\infty,0) $,
using the monotone convergence theorem on the r.h.s.\ ($ \lrfupa $ increases as $ a $ decrease),
%
%
%
we conclude the desired result~\eqref{eq:rect:apprx}.
\end{proof}

\section{Proof of Theorem~\ref{thm:main}}
\label{sect:pfmian}
\subsection{Upper bound}
\label{sect:pfupb}
We begin by establishing the exponential tightness of $ \PN $.
To this end, consider, for $ h\in\Sp $, $ n,r<\infty $, the following modulo of continuity
\begin{align}
	\label{eq:w'}
	w'(h,n,r) := \sup_{i=1,\ldots, n} \Vert h(\tfrac{iT}{n})-h(\tfrac{(i-1)T}{n}) \Vert_{C[-r,r]}.
\end{align}
Note that for $ h\in\Sp $, we have 
$ h(\tfrac{iT}{n},\xi)-h(\tfrac{(i-1)T}{n},\xi) = |h(\tfrac{iT}{n},\xi) -h(\tfrac{(i-1)T}{n},\xi)| $.
The main step of showing exponential tightness is the following.
\begin{lemma}
\label{lem:tight}
For each fixed $ \e>0 $ and $ r<\infty $, we have that
\begin{align}
	\label{eq:tight:}
	\limsup_{n\to \infty} \
	\limsup_{N\to\infty} \frac{1}{N^2} \log \PN \big( w'(\hN,n,r)\geq \e \big)
	=
	-\infty.
\end{align}
\end{lemma}

\begin{proof}
Write $ t_i := \frac{iT}{n} $ to simplify notations.
Our goal is to bound the following probability:
\begin{align}
	\label{eq:pN}
	p_N:=
	\PN \Big( 
		\bigcup_{i=1}^n
		\bigcup_{\frac{x}{N}\in [-r,r]} 
		\Big\{ \hN(t_{i},\tfrac{x}{N})-\hN(t_{i-1},\tfrac{x}{N}) \geq \e \Big\}
	\Big).
\end{align}
Let $ m:= \lceil \frac{4r}{\e} \rceil $ and partition $ [-r,r] $ into subintervals 
$
	U_j := [\tfrac{r(j-1)}{m}, \tfrac{rj}{m}]
$,
$ j=1-m,\ldots, m $.
Since $ \hN(t)\in\Splip $, for each $ x,x'$ such that $\frac xN, \frac{x'}{N}\in U_j $, 
we have
\begin{align*}
	\big| \big(\hN(t_i,\tfrac{x}{N}) - \hN(t_{i-1},\tfrac{x}{N})\big) - 
  \big(\hN(t_i,\tfrac{x'}{N}) - \hN(t_{i-1},\tfrac{x'}{N})\big) \big|
	\leq
        2|\tfrac xN -\tfrac{x'}{N}|
	\leq
	\tfrac{2r}{m}
	\leq
	\tfrac{\e}{2}. 
\end{align*}
Consequently, if $ h(t_i,\tfrac{x}{N}) - h(t_{i-1},\tfrac{x}{N}) \geq \e $
for some $ \frac{x}{N} \in U_j $,
then $ h(t_i,\tfrac{x'}{N}) - h(t_{i-1},\tfrac{x'}{N}) \geq \frac{\e}{2} $ \emph{for all} $ \frac{x'}{N}\in U_j $.
This gives
\begin{align}
	\notag
	p_N
	&\leq
	\PN \Big( \bigcup_{i=1}^n \bigcup_{j=1-m}^m 
		\Big( \bigcup_{U_j}
			\Big\{ 
			\hN(t_{i},\tfrac{x}{N})-\hN(t_{i-1},\tfrac{x}{N}) \geq \e 
			\Big\}
		\Big)
	\Big)
\\
	\label{eq:tight:union}
	&\leq
	\sum_{i=1}^n \sum_{j=1-m}^m
	\PN \Big( 
		\bigcap_{U_j}
		\Big\{ 
			\hN(t_{i},\tfrac{x}{N})-\hN(t_{i-1},\tfrac{x}{N}) \geq \tfrac{\e}{2}
		\Big\}
	\Big)	.
\end{align}
Under the law $ \PN $, the condition $ \hN(t_i,\tfrac{x}{N})-\hN(t_{i-1},\tfrac{x}{N}) \geq \tfrac{\e}{2} $
forces the underlying Poisson clock at site $ x $ to tick at least $ N\frac{\e}{2} $ times in a time interval of length $\frac {NT}{n}$.
Using this in~\eqref{eq:tight:union} gives
\begin{align*}
	p_N \leq
	n \sum_{j=1-m}^m \Pr( X_N \geq N\tfrac{\e}{2} )^{\#(U_j\cap\frac{\Z}{N})},
\end{align*}
where $ X_N \sim \Pois(\tfrac{NT}{n}) $.
Since $ U_j $ is an interval of length $ \frac{r}{m} $, $ m:= \lceil \frac{4r}{\e} \rceil $,
we necessarily have $ \#(U_j\cap\frac{\Z}{N}) \geq \frac{\e N}{5} $, for all $ N $ large enough.
This yields
\begin{align}
	\label{eq:pNbd}
	p_N \leq
	2mn \Pr\big( X_N \geq N\tfrac{\e}{2} \big)^{\frac{{\e N}}{5}}.
\end{align}
Recall from~\eqref{eq:prf} that $ \prf(\lambda|u) $ 
denotes the large deviation rate function for Poisson variables.
In particular,
$
	\lim_{N\to\infty} 
	\frac{1}{N} \log \PN( X_N \geq N\frac{\e}{2} )
	= -\prf(\frac{\e}{2}|\frac{T}{n}).
$
Using this in~\eqref{eq:pNbd} gives
\begin{align}
	\label{eq:pNbd:}
	\limsup_{N\to\infty} \frac{1}{N^2}\log p_N 
	\leq
	- \frac{\e}{5}\prf\big(\tfrac{\e}{2}\big|\tfrac{T}{n}\big).	
\end{align}

Now, combining~\eqref{eq:pN} and \eqref{eq:pNbd:} gives
\begin{align*}
	\limsup_{N\to\infty} \frac{1}{N^2} \log \PN(w'(\hN,n,r)\geq \e) 
	\leq
	-\frac{\e}{5}
	 \prf\big(\tfrac{\e}{2}\big|\tfrac{T}{n}\big).
\end{align*}
The last expression tends to $ -\infty $ as $ n\to\infty $.
This concludes the desired result~\label{eq:tight:goal}.
\end{proof}

Given Lemma~\ref{lem:tight}, the exponential tightness follows by standard argument,
as follows.

\begin{proposition}
\label{prop:tight}
Given any $ b<\infty $, there exists a compact set $ \calK\subset\Sp $ such that
\begin{align}
	\label{eq:tight}
	\limsup_{N\to\infty} \frac{1}{N^2} \log \PN(\hN\notin \calK) \leq -b.
\end{align}
\end{proposition}

\begin{proof}
Define, for $ h\in D([0,T],C(\R)) $, 
the modulo of oscillation as
\begin{align}
	\label{eq:w}
	w(h,\d) := \inf_{\set{t_i}} \ \max_i \sup_{s\in[t_{i-1},t_{i})} \dist(h(s), h(t) ),
\end{align}	
where the infimum goes over all partitions $ \{0=t_0<t_1<\ldots<t_n=T\} $ of $ [0,T] $
such that $ t_i-t_{i-1} \geq \d $, $ i=1,\ldots,n $.
Note that $ w(h,\d) $ decreases as $ \d $ decreases.
Under these notations,
recall from \cite[Theorem~3.6.3]{ethier09} that $ \calA\subset D([0,T],C(\R)) $
is precompact if:
\begin{enumerate}
	\item \label{enu:cmp1}
		there exists compact $ \calK'\subset C(\R) $ such that
		$ h(t)\in\calK' $, $ \forall t\in[0,T] $, $ h\in\calA $;
	\item \label{enu:cmp2} 
		For each $ h\in\calA $, $ \lim_{\d\downarrow 0} w(h,\d) = \lim_{n\to\infty} w(h,\frac{T}{n}) =0 $.
\end{enumerate}
The condition~\eqref{enu:cmp1} holds automatically for any $ \calA\subset\Sp $
because $ \Splip $ is already a compact subset of $ C(\R) $.
In~\eqref{eq:w}, take the equally the spaced partition $ \{0<\frac{T}{n}<\ldots<T\} $
we obtain that, for $ h\in\Sp $ and $ k<\infty $,
\begin{align}
\label{eq:ww'}
	w(h,\tfrac{T}{n}) 
	\leq 
	\max_{i=1,\ldots,n} 
	\Vert f(\tfrac{(i-1)T}{n})-f(\tfrac{iT}{n}) \Vert_{C[-k,k]} + 2^{-k} 
	=
	w'(h,n,k) + 2^{-k}.
\end{align}
For each fixed $ k<\infty $,
using Lemma~\ref{lem:tight} with $ \e = 2^{-k} $ to bound the term~$ w'(h,n,k) $ in~\eqref{eq:ww'} gives
\begin{align*}
	\limsup_{n\to\infty}
	\limsup_{N\to\infty}
	\frac{1}{N^2} \log \PN\big( w(\hN,\tfrac{T}{n}) \geq 2^{-k+1} \big) = -\infty.
\end{align*}
Fix further $ b<\infty $. We then obtain $ n_*(b,k), N_*(b,k)<\infty $, depending only on $ b,k $,
such that
\begin{align}
	\label{eq:wtight}
	\frac{1}{N^2} \log \PN\big( w(\hN,\tfrac{T}{n}) \geq 2^{-k+1} \big) < -kb,
	\quad
	\forall n \geq n_*(b,k),
	\
	N \geq N_*(b,k).
\end{align}
Further, for each $ N\in\{1,\ldots, N_*(b,k)\} $,
it is straightforward to show that
$ \lim_{\d\downarrow 0} \PN(w(\hN,\delta) \geq 2^{-k+1}) = 0 $.
Hence, by making $ n_*(b,k) $ larger in~\eqref{eq:wtight} if necessary,
the inequality~\eqref{eq:wtight} actually holds for all $ N \geq 1 $, i.e.,
\begin{align}
	\label{eq:wtight:}
	\frac{1}{N^2} \log \PN\big( w(\hN,\tfrac{T}{n}) \geq 2^{-k+1} \big) < -kb,
	\quad
	\forall n \geq n_*(b,k),
	\
	N \geq 1.
\end{align}
Now let $ \calA := \cap_{k=1}^\infty \{h: w(h,\frac{T}{n_*(b,k)}) \geq 2^{-k+1} \} $.
By the previously stated criteria~\eqref{enu:cmp1}--\eqref{enu:cmp2},
the set $ \calA $ is precompact.
Rewriting \eqref{eq:wtight:} as $ \PN( w(\hN,\frac{T}{n_*(b,k)}) \geq 2^{-k+1} ) \leq e^{-kbN^2} $
and taking the union bound over $ k\geq 1 $,
we obtain $ \PN(\calA^c) \leq c(b) e^{-bN^2} $,
for some constant $ c(b)<\infty $ depending only on $ b $.
This concludes~\eqref{eq:tight} for $ \calK := \bar{\calA} $.
\end{proof}

We next prepare a lemma that allows us to ignore discontinuous deviations~$ g $ 
in proving Theorem~\ref{thm:main}\ref{enu:upb}.
\begin{lemma}
\label{lem:disconti}
Given any $ b<\infty $ and any $ g\in\Sp\setminus C([0,T],C(\R)) $,
i.e., discontinuous $ g $,
there exists a neighborhood $ \calO $ of $ g $, i.e., an open set with $ g\in\calO $, such that
\begin{align}
	\label{eq:discnti}
	\limsup_{N\to\infty} \frac{1}{N^2} \log \PN(\hN\in\calO) \leq -b.
\end{align}
\end{lemma}
\begin{proof}
Recall that, Skorokhod's $ J_1 $-topology is induced from the following metric
\begin{align}
	\label{eq:distt}
	\distt(g,h)
	:=
	\sup_{ v } 
	\Big\{ \Big(\sup_{t\in[0,T]} |v(t)-t|\Big)\vee  \Big(\sup_{t\in[0,T]} \dist( g(t), (h\circ v)(t) ) \Big) \Big\}.
\end{align}
Here the supremum goes over all $ v:[0,T]\to[0,T] $ that is bijective, strictly increasing and continuous.	
	
Given $ g\in\Sp\setminus C([0,T],C(\R)) $, there exists $ t\in(0,T] $, $ \xi\in\R $ and $ \e_0>0 $
such that $ g(t,\xi)-g(t^-,\xi) \geq \e_0 $.
From the expression~\eqref{eq:distt} of the Skorohod metric $ \distt(\Cdot,\Cdot) $,
we see that $ \distt(h,g)<\delta $ implies
$ h\big((t+\delta)\wedge T,\xi\big) \geq g(t,\xi)-\delta $
and $ h\big((t-\delta)\vee 0, \xi\big) \leq g(t,\xi)+\delta $.
The last two conditions gives 
\begin{align*}
	w'(h,2\delta,|\xi|) \geq g(t,\xi)-g(t^-,\xi) - 2\d
	\geq 
	\e_0-2\d.
\end{align*}
Equivalent,
\begin{align}
	\label{eq:discnt:}
	\{h:\distt(h,g)<\delta\} \subset \{ w'(h,2\delta) \geq \e_0-2\delta \}.
\end{align}
Now, for any given $ b<\infty $, by Lemma~\ref{lem:tight} there exists some small enough $ \delta>0 $ such that
\begin{align}
	\label{eq:discnt::}
	\limsup_{N\to\infty} \frac{1}{N^2} \log \PN\big( w'(\hN,2\delta,|\xi|)\geq \e_0-2\delta\big) \leq -b.
\end{align}
Combining~\eqref{eq:discnt:}--\eqref{eq:discnt::},
we see that \eqref{eq:discnti} holds for $ \calO:=\{h:\distt(h,g)<\delta\} $.
\end{proof}

We now begin the proof of Theorem~\ref{thm:main}\ref{enu:upb}.
The main ingredient is Proposition~\ref{prop:ent:lw}, which we state in the following.
To setup notations, give a continuous deviation $ g\in\Sp\cap C([0,T],\Splip) $,
we define the following tubular set around $ g $:
\begin{align}
	\label{eq:tubular}
	\calU_{a,r}(g) := \Big\{ h\in\Sp : \sup_{t\in[0,T]} \Vert h(t)-g(t) \Vert_{C[-r,r]} < a \Big\}.
\end{align}
For \emph{generic} $ a_N\downarrow 0 $ and $ r_N\uparrow \infty $,
we consider the following conditioned law:
\begin{align}
	\label{eq:QN}
	\QN := \frac{1}{\PN(\calU_{a_N,r_N}(g))} \PN|_{\calU_{a_N,r_N}(g)}.
\end{align}
Recall that, for probability laws $ Q,P $, the relative entropy of $ Q $ with respect to $ P $ 
is defined as $ H(Q|P):=\Ex_Q(\log\frac{dQ}{dP}) $ if $ Q\ll P $;
and $ H(Q|P):=-\infty $ otherwise.
\begin{proposition}
\label{prop:ent:lw}
Fix a continuous deviation $ g\in \Sp\cap C([0,T],\Splip) $,
and let $ \{\QN\}_N $ and $ \calU_{a_N,r_N}(g) $ be as in~\eqref{eq:QN},
with \emph{generic} $ a_N\downarrow 0 $ and $ r_N\uparrow \infty $.
Then 
\begin{align}
	\label{eq:ent:lw}
	- \limsup_{N\to\infty} \frac{1}{N^2} \log \PN(\hN\in\calU_{a_N,r_N}(g))
	=
	\liminf_{N\to\infty} \frac{1}{N^2} H(\QN|\PN) \geq \rfup(g).
\end{align}
\end{proposition}
\noindent
Proposition~\ref{prop:ent:lw} is proven in Section~\ref{sect:pfentlw} in the following.
Assuming this result here, we proceed to complete the proof of Theorem~\ref{thm:main}\ref{enu:upb}.
\begin{proof}[Proof of Theorem~\ref{thm:main}\ref{enu:upb}]
Recall from~\eqref{eq:distt} that $ \distt $ denotes Skorokhod's metric.
Throughout this proof we write $ B_b(h) := \{ \til h\in\Sp: \distt(h,\til h) < b \} $
for the open ball of radius $ b $ centered at a given $ h $.
First, given the exponential tightness from Proposition~\ref{prop:tight},
it suffices to prove the upper bound~\eqref{eq:upb} for \emph{compact} $ \calC $.
Fix a compact $ \calC\subset\Sp $. 
For each given radius $ b>0 $, let $ \{\calB_b(h^b_i)\}_{i=1}^{n(b)}\subset\Sp $ 
be a finite cover of $ \calC $ that consists of open balls of radius $ b $.
Choose a sequence $ b_N \downarrow 0 $ in such a way that $ \frac{1}{N^2}\log n(b_N) \to 0 $,
and write $ h^{b_N}_i := h^N_i $ to simplify notations.
We then have
\begin{align}
	\notag
	\limsup_{N\to\infty}
	\frac{1}{N^2} \log \PN(\calC)
	&\leq
	\limsup_{N\to\infty}
	\frac{1}{N^2} \log \Big( \sum_{i=1}^{n(b_B)} \PN(\hN\in\calB_{b_N}(h^N_i)) \Big) 
\\
	\label{eq:upb:limsup}
	&\leq
	\limsup_{N\to\infty}
	\max_{i=1}^{n(b_N)} \frac{1}{N^2} \log \PN(\hN\in\calB_{b_N}(h^N_i)).	 
\end{align}
In~\eqref{eq:upb:limsup},
pass to a subsequences $ N_M $ and $ i_M $ that achieves the limit,
and write $ \til{h}_M:= h^{N_M}_{i_M} $ and $ \til{b}_M := b_{N_M} $ to simply notations.
As $ \calC $ is compact, the subsequence $ \{\til{h}_M\}_{M=1}^\infty $ has a limit point $ g\in \calC $.
Hence, by refining the subsequences, we assume without lost of generality $ \til{h}_M \to g $, as $ M\to\infty $.

Consider first the case where $ g $ is continuous, i.e., $ g\in C([0,T],\Splip) $.
For such $ g $, converges to $ g $ under the $ J_1 $-topology is equivalent to convergence under the uniform topology.
This being the case, there exist $ a_N \downarrow 0 $ and $ r_N \uparrow \infty $
such that, with $ \calU_{a,r}(g) $ defined in~\eqref{eq:tubular},
$ \calB_{\til{b}_{M}}(\til{h}_{M}) \subset \calU_{a_{N_M},r_{N_M}}(g) $, for all $ M $.
This gives 
\begin{align*}
	\limsup_{N\to\infty}
	\max_{i=1}^{n(b_N)} \frac{1}{N^2} \log \PN(\hN\in\calB_{b_N}(h^N_i))	
	&=
	\lim_{M\to\infty}
	\frac{1}{N_M^2} \log \PN(\h_{M_N}\in\calB_{\til{b}_M}( \til{h}_M))
\\
	&\leq
	\limsup_{N\to\infty} \frac{1}{N^2} \log \PN( \hN\in\calU_{a_N,r_N}(g) ).
\end{align*}
The desired upper bound~\eqref{eq:upb} thus follows from~Proposition~\ref{prop:ent:lw}.

For the case of a discontinuous $ g $,
fix arbitrary $ b<\infty $.
By Lemma~\ref{lem:disconti} there exists a neighborhood $ \calO $ of $ g $ such that \eqref{eq:discnti} holds.
With $ \til h^{M} \to g $ and $ \til b_M \to 0 $,
we have $ \calB_{\til{b}_M}(\til h_M) \subset \calO $, for all $ M $ large enough.
Hence
\begin{align*}
	\limsup_{N\to\infty}
	\max_{i=1}^{n(b_N)} \frac{1}{N^2} \log \PN(\calB_{b_N}(h^N_i))	
	\leq
	\limsup_{N\to\infty} \frac{1}{N^2} \log ( \calO ) \leq -b.
\end{align*}
Letting $ b\to\infty $ gives the desired result~\eqref{eq:upb}.
\end{proof}

\subsection{Lower bound}
\label{sect:pflwb}
We begin by setting up notations and conventions.
In the following, in addition to the process $ \hN $ with initial condition $ \hic_N $ as in~\eqref{eq:ic:cnvg},
we will also consider processes with other initial conditions.
We use different notations to distinguish these processes, e.g., $ \gN $ with initial condition $ \gic_N $.
The initial conditions considered in the following are \emph{deterministic}.
This being the case,
we couple all the processes with different initial conditions together
by the \textbf{basic coupling} (see, for example, \cite{liggett13}).
That is, all the processes are driven by a common set of Poisson clocks.
Abusing notations,
we write $ \PN $ the \emph{joint} law of \emph{all} the processes with distinct initial conditions,
and write $ \PN^\g $ for the marginal law of a given process $ \g $.
It is straightforward to verify that the basic coupling preserves order, i.e.,
\begin{align}
	\label{eq:mono}
	\text{ if } \bar\h(0,x) \geq \und\h(0,x), \, \forall x\in\Z,
	\quad
	\text{ then } \bar\h(t,x) \geq \und\h(t,x), \, \forall t\in[0,NT], \, x\in\Z,
\end{align}
and that height processes are shift-invariant
\begin{align}
	\label{eq:shiftinvt}
	\text{ if } \h^1(0) = \h^2(0)+k,
	\quad
	\text{ then } \h^1(t) = \h^2(t)+k, \ \forall t\in[0,NT].
\end{align}

In the following we will often consider partition of subsets of $ [0,T]\times\R $.
We adopt the convention that the $ t $-axis is vertical, while the $ \xi $-axis is horizontal.
The direction going into larger/smaller $ t $ is referred to as upper/lower,
which the direction going to larger/smaller $ \xi $ is referred to as right/left.
For a given $ \tau=\frac{T}{\scl} $, $ \scl\in\N $,
we let $ \Sigma(\tau,b) $ denote the triangulation of $ [0,T]\times\R $ as depicted in Figure~\ref{fig:Sigma}.
Each triangle $ \triangle\in\Sigma(\tau,b) $ has a vertical edge of length $ \tau $,
and horizontal edge of length $ b $, and a hypotenuse going upper-right-lower-left.
We say a function $ h\in C([0,T]\times\R) $ is \textbf{$ \Sigma(\tau,b) $-piecewise linear}
if $ h $ is linear (i.e., $ \nabla h $ is constant) on each $ \triangle\in\Sigma(\tau,b) $.

\begin{figure}
\psfrag{T}{$ T $}
\psfrag{S}{$ t $}
\psfrag{X}{$ \xi $}
\psfrag{D}[c][c]{$ \cdots $}
\psfrag{H}{$ \tau $}
\psfrag{B}{$ b $}
\includegraphics[width=.8\textwidth]{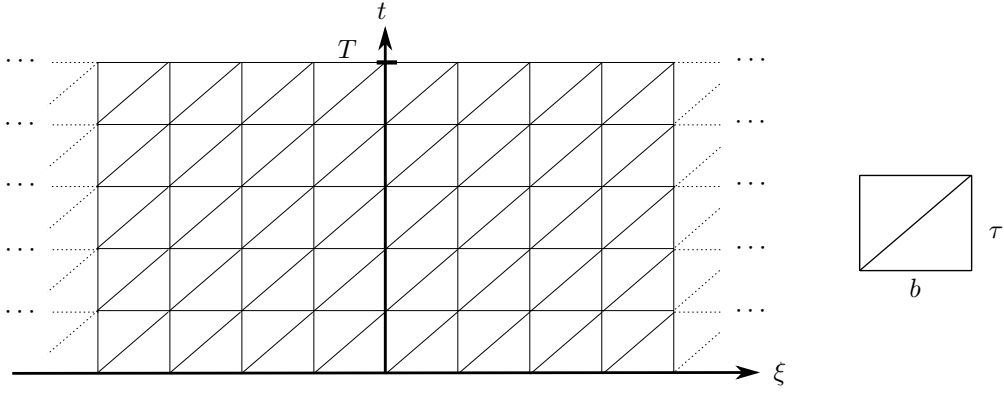}
\caption{The triangulation $ \Sigma(\tau,b) $}
\label{fig:Sigma}
\end{figure}

Recall from~\eqref{eq:tubular} that $ \calU_{a,r}(h) $
denotes a tubular set around a given deviation $ h $.
The main ingredient of the proof is the following proposition.
\begin{proposition}
\label{prop:ent:up}
Fix $ \e_*>0 $, $ r_*<\infty $; 
 $ \tau, b $ such that $ \frac{T}{\tau}, \frac{r_*}{b}\in\N $;
and a $ \Sp $-valued, $ \Sigma(\tau,b) $-piecewise linear deviation $ g $ such that
\begin{align}
	\label{eq:nondg}
	&
	0< \sup_{[0,T]\times\R} g_t <\infty,
\\
	&
	\label{eq:nondg:}
	0< \inf_{[0,T]\times\R} g_\xi \leq \sup_{[0,T]\times\R} g_\xi <1,
\end{align}
Write $ \gIC := g(0) $.
Given a \ac{TASEP} height process $ \gN $, with an initial condition $ \gic $ satisfying
\begin{align}
	\label{eq:gNic}
	\dist(\gic_N,\gIC) \longrightarrow 0,
	\quad
	\text{as } N\to\infty,
\end{align}
there exists a probability law $ \QN  $ on $ \Sp $,
supported on the trajectories of $ \gN $, such that
\begin{align}
	&
	\label{eq:lwb:apprx}
	\lim_{N\to\infty} \QN\big( \gN\in \calU_{\e_*,r_*}(g) \big) = 1,
\\
	&
	\label{eq:lwb:RD:bd}
	\sup_{N} \Ex_{\QN}\Big( \frac{1}{N^2}\log\frac{d\QN}{d\PN^\g} \Big)^2 <\infty,
\\
	&
	\label{eq:lwb:ent}
	\limsup_{N\to\infty} 
	\frac{1}{N^2} H(\QN|\PN^\g) 
	<
	\int_0^T \int_{-r_*}^{r_*} \lrf(g_t,g_\xi)dtd\xi
	+
	\int_0^T \int_{r_*\leq |\xi| \leq r^* } \prff\Big( \frac{g_t}{g_\xi(1-g_\xi)} \Big) dtd\xi,
\end{align}
where $ r^* $ is defined in terms of $ r_* $ and $ g $ as
\begin{align}
	\label{eq:r*}
	r^* &:= r_* + r_* \lceil \tfrac{T\maxlam}{r_*} \rceil,
\\
	\label{eq:maxlam}
	\maxlam &:= \sup_{[0,T]\times\R} \frac{g_t}{g_\xi(1-g_\xi)} \in (0,\infty).
\end{align}
\end{proposition}
\noindent
Proposition~\ref{prop:ent:up} is proven in Section~\ref{sect:ITASEP}--\ref{sect:speed} in the following.
Here we assume this result, and proceed to complete the proof of Theorem~\ref{thm:main}\ref{enu:lwb}.
To this end, we first prepare a few technical results.
First, using standard change-of-measure techniques, 
we have the following consequence of Proposition~\ref{prop:ent:up}:
\begin{customprop}{\ref*{prop:ent:up}}
\label{prop:ent:up:}
Let $ \gN, g, \e_*,r_*,r^* $ be as in Proposition~\ref{prop:ent:up}.
We have
\begin{align}
	\liminf_{N\to\infty} \frac{1}{N^2} \log 
	\PN\big( \gN\in \calU_{\e_*,r_*}(g) \big)
	> 
	- \int_0^T\int_{\R} \lrf(g_t,g_\xi) dtd\xi 
	\label{eq:ent:up}
	+
	\int_0^T \int_{r_*\leq |\xi| \leq r^* } \hspace{-10pt} \prff\Big( \frac{g_t}{g_\xi(1-g_\xi)} \Big) dtd\xi
	+\e_*.	
\end{align}
\end{customprop}
\begin{proof}
Let $ \{\QN\}_{N} $ be as in Proposition~\ref{prop:ent:up},
and write $ \calU := \calU_{\e_*,r_*}(g)  $ to simply notations.
Changing measures from $ \PN $ to $ \QN $, we write $ \PN(\gN\in\calU) $ as
\begin{align*}
	\PN(\gN\in\calU) 
	= 
	\Ex_{\QN} \Big( \ind_{\calU} \exp\Big( -\log \frac{d\QN}{d\PN^\g}\Big)\Big).
\end{align*}
Apply Jensen's inequality
$ \int F(X) d\mu \geq (\int d\mu) F(\frac{\int X d\mu}{\int d\mu})  $
with the convex function $ F(\xi)= \exp(-\xi) $, and with $ X=\log \frac{d\QN}{d\PN^\g} $ and $ \mu=\Ex_{\QN} (\ind_{\calU}\Cdot) $. 
We then obtain
\begin{align}
\begin{split}
	\label{eq:LDlwbd:jensen}
	\PN(\gN\in\calU)  
	&\geq 
	\QN(\calU) \exp\Big( -\frac{1}{\QN(\calU)}\Ex_{\QN}\Big(\ind_{\calU}\log \frac{d\QN}{d\PN^\g}\Big) \Big)
\\
	&=
	\QN(\calU) \exp\Big( -\frac{1}{\QN(\calU)}H(\QN|\PN^\g) 
        +\frac{1}{\QN(\calU)}\Ex_{\QN}\Big(\ind_{\calU^c}\log\frac{d\QN}{d\PN^\g}\Big) \Big).
\end{split}
\end{align}
Take $ \frac{1}{N^2}\log(\Cdot) $ on both sides of~\eqref{eq:LDlwbd:jensen},
and let $ N\to\infty $. 
We have
\begin{align}
	\notag
	&\liminf_{N\to\infty} \frac{1}{N^2}\log\PN(\gN\in\calU) 
\\
	\label{eq:LDlwbd:jensen:::}
	\geq&
	\liminf_{N\to\infty} 
	\Big(
		\frac{1}{N^2} \log \QN(\calU)
		+
		\frac{1}{\QN(\calU)}\frac{-1}{N^2} H(\QN|\PN^\g)
		+
		\frac{1}{\QN(\calU)} \Ex_{\QN}\Big(\ind_{\calU^c} \frac{1}{N^2}
		\log\frac{d\QN}{d\PN^\g}\Big)
	\Big).
\end{align}
With \eqref{eq:lwb:RD:bd} and $ \Ex_{\QN}(\calU^c) \to 0 $,
we have that
$
	\Ex_{\QN}(\ind_{\calU^c} \frac{1}{N^2}\log\frac{d\QN}{d\PN^\g}) \rightarrow 0.
$
Using this in~\eqref{eq:LDlwbd:jensen:::} gives
\begin{align}
	\liminf_{N\to\infty} \frac{1}{N^2}\log\PN(\gN\in\calU) 
	\label{eq:LDlwbd:jensen::}
	\geq
	\liminf_{N\to\infty} 
	\Big(
		\frac{1}{N^2} \log \QN(\calU)
		+
		\frac{1}{\QN(\calU)}\frac{-1}{N^2} H(\QN|\PN^\g)
	\Big).
\end{align}
Now, in~\eqref{eq:LDlwbd:jensen::},
using~\eqref{eq:lwb:apprx} to replace each~$ \QN(\calU) $ with $ 1 $, 
and then using~\eqref{eq:lwb:ent} to take limit of the last term,
we obtain the desired result~\eqref{eq:ent:up}.
\end{proof}

The next Lemma allows us to approximate $ h^*\in\Sp $ with $ \rf(h^*)<\infty $
with a piecewise linear $ g $ of the form considered in Proposition~\ref{prop:ent:up}.

\begin{lemma}
\label{lem:h*apprx}
Fix $ a,\e_*>0 $, $ r_0<\infty $, and a deviation $ h^*\in\Sp $ such that $ \rf(h^*)<\infty $.
There exist $ r_*\geq r'_*\in [r_0,\infty) $, $ \scl_*\in\N $, 
and a  $ \Sigma(\frac{T}{\scl_*},\frac{r_*}{\scl}) $-piecewise linear function $ \ling\in\Sp $,
such that 
\begin{align}
	\label{eq:ling:apprx}
	\sup_{[0,T]\times[-r'_*,r'_*]} |g-h^*| &< a,
\\
	\label{eq:ling:ent}
	\int_0^T\int_{-r_*}^{r_*} \lrf(g_t,g_\xi) dtd\xi 	
	&< 
	\rf(h_*) + \e_*,
\\
	\label{eq:ling:entail}
	\int_0^T \int_{r_*\leq |\xi| \leq r^* } \prff\Big( \frac{g_t}{g_\xi(1-g_\xi)} \Big) dtd\xi
	&<
	\e_*,
\end{align}
and satisfies~\eqref{eq:nondg}--\eqref{eq:nondg:}, and
\begin{align}
	\label{eq:ICrg+}
	g(0,r'_*) &> \Big( \sup_{[0,T]\times[-r_0,r_0]} g \Big) +\frac{a}{5},
\\
	\label{eq:ICrg-}
	\quad
	g(0,-r'_*)+r'_* &> \Big( \sup_{[0,T]\times[-r_0,r_0]} (g(t,\xi)-\xi) \Big)+\frac{a}{5},
\end{align}
where $ r^*\geq r_* $ is defined in terms of $ r_* $ and $ g $ as in~\eqref{eq:r*}--\eqref{eq:maxlam}.
\end{lemma}
\begin{remark}
Indeed, $ \Sigma(\frac{T}{\scl},\frac{r'_*}{\scl}) $-piecewise linear functions have 
derivatives everywhere except along the edges of the underlying triangulation.
Slightly abusing notations,
the supremum and infimum in~\eqref{eq:nondg}--\eqref{eq:nondg:} neglect
a set of zero Lebesgue measure where $ \nabla g $ is undefined.
We adopt this convention also in the following.
\end{remark}

\begin{remark}
Here we explain the role of this lemma and the conditions~\eqref{eq:nondg}--\eqref{eq:nondg:}, \eqref{eq:ling:apprx}--\eqref{eq:ICrg-} therein.
The idea behind Lemma~\ref{lem:h*apprx} is to approximate a \emph{generic} $ h^* $
with a \emph{specific} type of deviation $ g $, 
with various properties that facilitates the subsequent analysis. 
Indeed, \eqref{eq:ling:apprx} allows us to approximate the deviation $ h^* $ with $ g $, 
and \eqref{eq:ling:ent}--\eqref{eq:ling:entail} ensure the corresponding cost does not increase,
up to an error of $ \e_* $.
The conditions~\eqref{eq:nondg}--\eqref{eq:nondg:}
assert that $ (\nabla g) $ is bounded away from the boundary of $ (\kappa,\rho)\in[0,\infty)\times[0,1] $.
In particular, the resulting rate density $ \lrf(g_t,g_\xi) $ is uniformly bounded.
The purpose of having~\eqref{eq:ICrg+}--\eqref{eq:ICrg-} is
to incorporate a localization result from Lemma~\ref{lem:icLoc} in following. 
\end{remark}
\begin{proof}~\\
\noindent\textbf{Step 0, some properties of $ h^* $.}~
Fix $ a>0,r_0<\infty $ and $ h^*\in\Sp $ with $ \rf(h^*)<\infty $.
Note that for such $ h^* $ we must have $ h^*(0)=\hIC $.
Before starting the proof, let us first prepare a few useful properties of $ h^* $.
Since $ \hIC\in\Sp $, $ \xi\mapsto\hIC(\xi) $ is nondecreasing and $ \xi\mapsto \hIC(\xi)-\xi $ is nonincreasing.
We let
\begin{align*}
	\alpha^+ := \lim_{\xi\to\infty} \hIC(\xi) = \sup_{\R} \hIC \in \R\cup\{\infty\},
	\quad
	\alpha^- := \lim_{\xi\to-\infty}  (\hIC(\xi) -\xi) = \sup_{\xi\in\R} (\hIC(\xi) -\xi) \in \R\cup\{\infty\}.
\end{align*}
Under the current assumption $ \rf(h^*)<\infty $, we claim that
\begin{align}
	\label{eq:h*:wedgebd:}
	&\alpha^+ = \sup_{[0,T]\times\R} h^*, 
\\
	\label{eq:h*:wedgebd::}	
	&\alpha^- = \sup_{[0,T]\times\R} (h^*(t,\xi)-\xi),
\\
	\label{eq:h*supdiff}
	&
	\sup_{\xi\in\R} (h^*(T,\xi)-h^*(0,\xi))<\infty.
\end{align} 
To see why~\eqref{eq:h*:wedgebd:} should hold,
assume the contrary: $ h^*(t_0,\xi_0) = \alpha >\alpha^+ $, for some $ t_0\in(0,T] $, $ \xi_0\in\R $.
Since $ h^*(t_0)\in\Splip $, we necessarily have that $ h^*(t_0,\xi)|_{\xi \geq \xi_0} \geq \alpha $.
Using \eqref{eq:nonincr:mob:}, we write
\begin{align*}
	\rf(h^*) 
	\geq
	 \int_{\xi_0}^\infty \int_0^{t_0} \lrf(h_t^{*},h_\xi^{{*}}) dt  d\xi
	\geq
	\int_{\xi_0}^\infty \int_0^{t_0} \prff(h_t^{{*}}) dt  d\xi.
\end{align*}
Further utilizing the convexity of $ \lambda\mapsto\prff(\lambda) $ gives
\begin{align*}
	\rf(h^*) 
	\geq
	\int_{\xi_0}^\infty t_0 \prff\Big(\fint_0^{t_0} h_t^{{*}} dt\Big)  d\xi
	\geq
	\int_{\xi_0}^\infty t_0 \prff\Big(\frac{\alpha-\alpha^+}{t_0}\Big)  d\xi
	=
	\infty.	
\end{align*}
This contradicts with the assumption $ \rf(h^*)<\infty $.
Hence~\eqref{eq:h*:wedgebd:} must hold.
Likewise, if \eqref{eq:h*:wedgebd::} fails,
i.e., $ h^*(t_0,\xi_0)-\xi_0 = \alpha' >\alpha^- $, 
we must have $ (h^*(t_0,\xi)-\xi)|_{\xi \leq \xi_0} \geq \alpha' $.
The last inequality gives $ (h^*(t_0,\xi)-\hIC(\xi))|_{\xi\leq\xi_0} \geq \alpha'-\alpha^->0 $.
From here a contradiction is derived by similar calculation to the preceding.
Hence~\eqref{eq:h*:wedgebd::} must also hold.
Turning to~\eqref{eq:h*supdiff}, 
for each $ \xi_0\in\R $, using $ h^*(T),h^*(0)\in\Splip $ we write
\begin{align}
	\label{eq:nondg:v1}
	\big( h^*(T,\xi_0)-h^*(0,\xi_0) \big) - 1
	\leq
	\inf_{|\xi-\xi_0|\leq \frac12}
	\big( h^*(T,\xi)-h^*(0,\xi) \big).
\end{align}
Recall the definitions of $ \rff_n(h,\xi) $ and $ \rff(h) $ from \eqref{eq:rffn}--\eqref{eq:rff}.
With $ \lim_{\lambda\to\infty} \prff(\lambda)\lambda^{-1} =\infty $,
we have that $ \lambda \leq \prff(\lambda)+ c_0 $, $ \forall \lambda\in [0,\infty) $,
for some universal constant $ c_0<\infty $.
Using this for $ \lambda = \frac{h^*(T,\xi)-h^*(0,\xi)}{T} $ on the r.h.s.\ of~\eqref{eq:nondg:v1}, we obtain
\begin{align}
	\notag
	\big( h^*(T,\xi_0)-h^*(0,\xi_0) \big) - 1-	Tc_0
	&\leq
	\int_{|\xi-\xi_0|\leq \frac12}
	T\prff\Big( \frac{h^*(T,\xi)-h^*(0,\xi)}{T} \Big) d\xi
\\
	\label{eq:nondg:v2}
	&
	\leq
	\int_{\R}
	T\prff\Big( \frac{h^*(T,\xi)-h^*(0,\xi)}{T} \Big) d\xi
	=
	\int_{\R} \rff_1(h^*,\xi) d\xi.
\end{align}
By \eqref{eq:rff} and Lemma~\ref{lem:rff<rf},
the last expression in~\eqref{eq:nondg:v2} is bounded by $ \rff(h^*) $, which is further bounded by $ \rf(h^*) $.
Namely,
$
	\big( h^*(T,\xi_0)-h^*(0,\xi_0) \big) - 1-	Tc_0 \leq \rf(h^*).
$
Under the assumption $ \rf(h^*)<\infty $, 
taking the supremum over $ \xi_0\in\R $ gives~\eqref{eq:h*supdiff}.

\noindent\textbf{Step 1, tilting.}~
Our goal is to construct a suitable $ g $ that satisfies all the prescribed conditions.
The construction is done in three steps.
Starting with $ h^* $, 
in each step we perform a surgery on the function from the previous step.
Here, in the first step, we `tilt' $ h^* $ to obtain $ \til g $, described as follows.

Set $ \gamma_0 := \sup_{t\in[0,T]} |h^*(t,0)| <\infty $,
and let $ {r'} > (8a)\vee r_0\vee(2a\gamma_0) $ be an auxiliary parameter. 
We tilt the function $ h^* $ to get
\begin{align}
	\label{eq:tilgr}
	\til g^{r'}(t,\xi) := (1-\tfrac{a}{2{r'}})h^*(t,\xi) + \tfrac{a}{4{r'}}\xi.
\end{align}
Such a tilting ensures the $ \xi $-derivatives are bounded away from $ 0 $ and $ 1 $.
More precisely,
\begin{align}
	\label{eq:tilg:drv}
	\til g^{r'}_\xi = (1-\tfrac{a}{2r'})h^*_\xi+\tfrac{a}{4{r'}} \in [\tfrac{a}{4{r'}},1-\tfrac{a}{4{r'}}],
\end{align}
Also, $ \til g^{r'}_t = (1-\tfrac{a}{2{r'}})h^*_t \leq h^*_t $, and
\begin{align}
	\label{eq:h*tilg:apprx:}
	&
	\sup_{[0,T]\times[-r_0,r_0]} |\til g^{r'}-h^*| 
	\longrightarrow
	0,
	\text{ as } {r'}\to\infty.
\end{align}
Furthermore, evaluating $ \til g^{r'} $ at $ (t,\xi)=(0,\pm {r'}) $ gives
\begin{align}
	\label{eq:tilg:r+}
	\til g^{r'}(0,{r'})
	=
	\big( 1-\tfrac{a}{2{r'}} \big) \hIC({r'}) + \tfrac{a}{4}
	&\longrightarrow
	\alpha^+ + \tfrac{a}{4},
	\quad
	\text{as } {r'}\to\infty,
\\
	\label{eq:tilg:r-}
	\til g^{r'}(0,-{r'})+{r'}
	=
	\big( 1-\tfrac{a}{2{r'}} \big) (\hIC(-{r'})+{r'}) + \tfrac{a}{4} 
	&\longrightarrow
	\alpha^- + {\tfrac{a}{4}},
	\quad
	\text{as } {r'}\to\infty.
\end{align}

We now list a few consequence of the prescribed properties of $ \til g^{r'} $.
Recall that $ \mobb(\rho):=\rho(1-\rho) $.
From~\eqref{eq:tilg:drv}, it is straightforward to verify that $ |\til g^{r'}_\xi-\frac12| \leq |h^*_\xi-\frac12| $,
so in particular $ \mobb(\til g^{r'}_\xi) \geq \mobb(h^*_\xi) $.
Combining this with \eqref{eq:nonincr:mob}, together with $ \til g^{r'}_t \leq h^*_t $, we obtain
\begin{align}
	\label{eq:h*tilg:ent}
	\lrf(\til g^{r'}_t, \til g^{r'}_\xi)
	\leq
	\lrf(h^*_t, h^*_\xi).
\end{align}
Next, with $ |\til g^{r'}(t,\xi) -h^*(t,\xi)| \leq \frac{a}{2{r'}}|h^*(t,\xi)|+\frac{a|\xi|}{4{r'}} $,
$ |h^*(t,\xi)| \leq \gamma_0 +|\xi| $, and $ {r'} > 2a \gamma_0 $
,
we have that
\begin{align}
	\label{eq:h*tilg:apprx}
	\sup_{[0,T]\times[-{r'},{r'}]} |\til g^{r'} -h^*| 
	&\leq \frac{a}{2{r'}}(\gamma_0+{r'}) + \frac{a{r'}}{4{r'}}
	< a.
\end{align}
Further, combining~\eqref{eq:tilg:r+} with~\eqref{eq:h*:wedgebd:} and \eqref{eq:h*tilg:apprx:},
we have that
\begin{align}
	\label{eq:tilg:r+:}
	\liminf_{{r'}\to\infty} \Big( \til g^{{r'}}(0,{r'}) - \sup_{[0,T]\times[-r_0,r_0]} \til g^{{r'}} \Big)
	\geq
	\alpha^+ + \frac a4 - \sup_{[0,T]\times[-r_0,r_0]} h^* 
	\geq
	\frac{a}{4}.
\end{align}
Similarly, \eqref{eq:tilg:r-}, \eqref{eq:h*:wedgebd::} and \eqref{eq:h*tilg:apprx:} gives
\begin{align}
	\label{eq:tilg:r-:}
	\liminf_{{r'}\to\infty} 
	\Big( \til g^{{r'}}(0,-{r'})+r' 
	- \Big(\sup_{[0,T]\times[-r_0,r_0]}\til g^{r'}(t,\xi) +\xi \Big) \Big)
	\geq
	{\frac{a}{4}}.
\end{align}
In view of~\eqref{eq:h*tilg:apprx}--\eqref{eq:tilg:r-:}.
we now \emph{fix} $ r'=r'_* $, and write $ \til g^{r'_*}=: \til g $,
for large enough $ r'_* $ so that
\begin{align}
	\label{eq:h*tilg*:apprx}
	&
	\sup_{[0,T]\times[-{r'_*},{r'_*}]} |\til g -h^*| < a,
\\
	\label{eq:tilg*:r+}
	&
	\til g(0,{r'_*}) > 
	\Big(\sup_{[0,T]\times[-r_0,r_0]} \til g \Big) - \frac{a}{5}.
\\	
	\label{eq:tilg*:r-}
	&
	\til g(0,-{r'_*}) +r'_*
	>
	\Big(\sup_{[0,T]\times[-r_0,r_0]} \big( \til g (t,\xi) -\xi \big)\Big) - \frac{a}{5}.
\end{align}

\noindent\textbf{Step 2, mollification.}~
Having constructed $ \til g $,
we next mollify $ \til g $ to obtain a smooth function $ \hat g $.
To prepare for this, let us first fix the threshold $ r_* $.
From~\eqref{eq:h*tilg:ent},
we have that
\begin{align*}
	\int_0^T \int_{|\xi| \geq r} \lrf(\til g_t,\til g_\xi) dt d\xi
	\leq
	\int_0^T \int_{|\xi| \geq r} \lrf(h^*_t,h^*_\xi) dt d\xi
	\longrightarrow 
	0,
	\quad
	\text{as } r\to\infty.
\end{align*}
This being the case, we fix $ r_*\geq r'_* $ such that
$
	\int_0^T \int_{|\xi| \geq r_*} \lrf(\til g_t,\til g_\xi) dt d\xi
	<
	(\frac{a}{4r'_*})^2 \e_*.
$

Fix a mollifier $ \moll \in C^\infty(\R\times\R) $,
i.e., nonnegative, supported on the unit ball, integrates to unity.
Extend $ \til g $ to $ \R\times\R $ by setting $ \til g(t,\xi)|_{t< 0} := g(0,\xi) $ 
and $ \til g(t,\xi)|_{t> T} := g(T,\xi) $.
Under this setup, for $ \d>0 $, we mollify $ \til g $, and then tilt in $ t $, to obtain
\begin{align}
	\label{eq:mollg}
	\hat g^\d(t,\xi) 
	:= 
	\int_{\R^2}\til g(s,\zeta)\; \moll(\tfrac{t-s}{\d},\tfrac{\xi-\zeta}{\d}) \tfrac{dsd\zeta}{\d^2}
	+\d t.
\end{align}
With $ \til g $ being continuous on $ [0,T]\times[-r'_*,r'_*] $,
the properties~\eqref{eq:h*tilg*:apprx}--\eqref{eq:tilg*:r-} hold also 
for $ \hat g^\d $ in place of $ \til g $, for all $ \d $ small enough.
Further, the convexity of $ (\kappa,\rho)\mapsto \lrf(\kappa,\rho) $ gives
\begin{align*}
	&
	\int_0^T \int_{-r_*}^{r_*} \lrf(\hat g^\d_t,\hat g^\d_\xi) dtd\xi
	\leq
	\int_{-\d}^{T+\d} \int_{|\xi|\leq r_*+\d} \lrf(\til g_t+\d,\til g_\xi) dtd\xi
	\longrightarrow
	\int_{0}^{T} \int_{-r_*}^{r_*} \lrf(\til g_t,\til g_\xi) dtd\xi
	\leq
	\rf(h^*),	
\\
	&
	\int_0^T \int_{-|\xi|\geq r_*} \lrf(\hat g^\d_t,\hat g^\d_\xi) dtd\xi
	\leq
	\int_{-\d}^{T+\d} \int_{|\xi|\geq r_*+\d} \lrf(\til g_t+\d,\til g_\xi) dtd\xi
	\longrightarrow
	\int_{0}^{T} \int_{-|\xi|\geq r_*} \lrf(\til g_t,\til g_\xi) dtd\xi
	<
	\Big(\frac{a}{4r'_*}\Big)^2 \e_*,
\end{align*}
as $ \d\downarrow 0 $.
In view of these properties, we now \emph{fix} small enough $ \d=\d_*>0 $, set $ \hat g:= \hat g^{\d_*} $, so that
\begin{align}
	\label{eq:mollg:apprx}
	&
	\sup_{[0,T]\times[-{r'_*},{r'_*}]} |\hat g -h^*| < a,
\\
	&
	\label{eq:mollg:ent}
	\int_0^T \int_{-r_*}^{r_*} \lrf(\hat g_t,\hat g_\xi) dtd\xi < \rf(h^*)+\e_*,
\\
	&
	\label{eq:mollg:entail:}
	\int_0^T \int_{|\xi| \geq r_*} \lrf(\hat g_t,\hat g_\xi) dtd\xi < \Big(\frac{a}{4r'_*}\Big)^2\e_*,
\\
	\label{eq:mollg:r+}
	&
	\hat g(0,{r'_*}) > 
	\Big(\sup_{[0,T]\times[-r_0,r_0]} \hat g\Big) - \frac{a}{5}.
\\	
	\label{eq:mollg:r-}
	&
	\hat g(0,-{r'_*}) + r'_*
	>
	\Big(\sup_{[0,T]\times[-r_0,r_0]} \big( \hat g(t,\xi) -\xi \big)\Big) - \frac{a}{5}.
\end{align}

Further, since $ \hat g_\xi $ is an average of $ \til g_\xi $, from~\eqref{eq:tilg:drv} we have that
\begin{align}
	\label{eq:mollg:drvx}
	\hat g_\xi \in [\tfrac{a}{4{r'_*}},1-\tfrac{a}{4{r'_*}}].
\end{align}
As for the $ t $-derivative, we claim that, for some fixed constant $ c_*<\infty $ 
(depending on $\delta_*$),
\begin{align}
	\label{eq:mollg:drvt}
	\d_*\leq \sup_{[0,T]\times\R} \hat g_t \leq c_*.
\end{align}
Indeed, the tilting in~\eqref{eq:mollg} ensures $ \hat g_t \geq \d_* $.
To show the upper bound, we use~\eqref{eq:mollg} to write
\begin{align}
	\label{eq:mollg:t}
	\hat g_t(t,\xi) 
	=  \int_{\R^2} \til g(s,\zeta) \moll_t(\tfrac{t-s}{\d_*},\tfrac{\xi-\zeta}{\d_*}) \tfrac{dsd\zeta}{\d_*^3}
	= \int_{\R^2} \big( \til g(s,\zeta) - \til g(0,\zeta) \big) 
	\moll_t(\tfrac{t-s}{\d_*},\tfrac{\xi-\zeta}{\d_*}) \tfrac{dsd\zeta}{\d_*^3}.
\end{align}
Under the convention $ \til g(t,\xi)|_{t< 0} := g(0,\xi) $ and $ \til g(t,\xi)|_{t> T} := g(T,\xi) $,
referring back to~\eqref{eq:tilgr}, we have that
\begin{align*}
	\sup_{t\in\R} \Big( \sup_{\xi\in\R} |\til g(t,\xi)-\til g(0,\xi)| \Big)
	=
	\sup_{\xi\in\R} (\til g(T,\xi) - \til g(0,\xi))
	=
	\Big(1-\frac{a}{4r'_*}\Big)\sup_{\xi\in\R} (h^*(T,\xi) - h^*(0,\xi)).
\end{align*}
The last quantity, by~\eqref{eq:h*supdiff}, is finite,
so in particular $ |\til g(t,\xi)-\til g(0,\xi)| $ is uniformly bounded.
Using this bound in~\eqref{eq:mollg:t} gives $ \sup_{[0,T]\times\R} \hat g_t <\infty $.
This concludes~\eqref{eq:mollg:drvt}.
Also, combining \eqref{eq:mollg:drvx} with~\eqref{eq:mollg:entail:}, we have that
\begin{align}
	\label{eq:mollg:entail}
	\int_0^T \int_{|\xi| \geq r_*} 
	\prff\Big( \frac{\hat g_t}{\hat g_\xi(1-\hat g_\xi)}\Big) dt d\xi 
	=
	\int_0^T \int_{|\xi| \geq r_*} 
	\frac{1}{\hat g_\xi(1-\hat g_\xi)}
	\lrf(\hat g_t,\hat g_\xi) dtd\xi
	<
	\e_*. 
\end{align}

\noindent\textbf{Step 3, linear interpolation.}~
Given the smooth function $ \hat g $, 
we are now ready to construct the piecewise linear $ g $.
Similarly to the preceding,
the construction involves an auxiliary parameter, $ \ell\in\N $, 
which will be fixed toward the end of the proof.
For $ \ell\in\N $, consider the triangulation $ \Sigma(\frac{T}{\ell},\frac{r'_*}{\ell}) $.
Define a $ \Sigma(\frac{T}{\ell},\frac{r'_*}{\ell}) $-piecewise linear function $ g^{\scl} $
by letting $ \ling^{{\scl}}=\til g $ at each of the \emph{vertices} of the triangle 
$ \triangle\in\Sigma(\frac{T}{\ell},\frac{r'_*}{\ell}) $,
and then linearly interpolating within $ \triangle $. 

Similarly to~\eqref{eq:r*}--\eqref{eq:maxlam}, we let
\begin{align*}
	r^{*,\scl} := 
	r_* + r_* 
	\Big\lceil \frac{T}{r_*} \Big( \sup_{[0,T]\times\R} \frac{g^\scl_t}{g^\scl_\xi(1-g^\scl_\xi)} \Big) \Big\rceil.
\end{align*}
Since each $ \triangle $ has a vertical edge and a horizontal edge,
and since $ \ling^{\scl} $ is the linear interpolation of $ \til g $ on $ \triangle $,
the (constant) derivatives $ \ling^{\scl}_t|_{\triangle^\circ} $ and $ \ling^{\scl}_\xi|_{\triangle^\circ} $ 
are the averages of $ \hat g_t $ and $ \hat g_\xi $ 
along the vertical and horizontal edges of $ \triangle $, respectively.
This together with~\eqref{eq:mollg:drvx}--\eqref{eq:mollg:drvt} gives
\begin{align}
	\label{eq:lingtx:bd}
	\ling^\scl_t \in [\d_*, c_*],
	\quad
	\hat g_\xi \in [\tfrac{a}{4{r'_*}},1-\tfrac{a}{4{r'_*}}].
\end{align}
Using this bound~\eqref{eq:lingtx:bd} on the derivatives, we have
\begin{align}
	r^{*,\scl}
	\leq
	r_* + r_* 
	\Big\lceil \frac{T}{r_*} \Big( \sup_{[0,T]\times\R} \frac{c_*}{(a/4r'_*)^2} \Big) \Big\rceil
	:=
	\hat r^{*}.
\end{align}

Now, since $ \hat g \in C^\infty([0,T]\times\R) $ is smooth, we necessarily have that
\begin{align}
	\label{eq:ling:cnvg}
	\lim_{{\scl}\to\infty} \sup_{ [0,T]\times[-r'_{*},r'_{*}] } |\ling^{\scl}-\hat g| =0,
	\quad
	\lim_{{\scl}\to\infty} \sup_{ [0,T]\times[-\hat r^{*},\hat r^{*}] } 
	(|\ling^\scl_t-\hat g_t|+|\ling^\scl_\xi-\hat g_\xi|) 
	=0.
\end{align}
In view of~\eqref{eq:lingtx:bd}--\eqref{eq:ling:cnvg},
and the properties~\eqref{eq:mollg:apprx}--\eqref{eq:mollg:ent}, \eqref{eq:mollg:r+}--\eqref{eq:mollg:r-}, \eqref{eq:mollg:entail}
that $ \til g $ enjoys, we fix some large enough $ \scl=\scl_* $,
so that $ \ling := \ling^{\scl_*} $ satisfies all 
the desired conditions~\eqref{eq:ling:apprx}--\eqref{eq:ICrg-} and \eqref{eq:nondg}--\eqref{eq:nondg:}.
\end{proof}

The next lemma allows us to localize the dependence on initial conditions.
Hereafter, we adopt the convention $ \inf\emptyset :=\infty $ and $ \sup\emptyset :=-\infty $.
To setup notations,
define, for $ \f\in\Splip, b,x_0\in\Z $,
\begin{subequations}
\label{eq:icrg}
\begin{align}
	\label{eq:icrg+}
	k^+(\f,b,x_0) &:= 
		\inf\{ x\in\Z\cap[x_0,\infty) : \f(x) \geq b \}, 
\\
	\label{eq:icrg-}
	k^-(\f,b,x_0) &:= 
		\sup\{ x\in\Z\cap(-\infty,x_0] : \f(x) -x \geq b - x_0  \}, 
\\
	\label{eq:icrg:}
	\icint(\f,b,x_0) 
	&:= 
		[k^-(\f,b,x_0),k^+(\f,b,x_0)]\cap\Z.	
\end{align}
\end{subequations}
\begin{lemma}
\label{lem:icLoc}
Let $ \h^1 $ denote a generic, pre-scale \ac{TASEP} height process,
with initial condition $ \h^1(0)=\f\in\Spf $,
and let $ b,x_0\in\Z $, $ t_0\in[0,NT] $.
%
\begin{enumerate}[label=(\alph*),leftmargin=5ex]
\item \label{enu:icLoc<}
The event $ \{\h^1(t,x_0) < b\} $ depends on the initial condition $ \f $ 
only through its restriction onto $ \icint(\f,b,x_0) $.
That is, given any other process $ \h^2 $ such that 
$ \h^2(0)|_{\icint(\f,b,x_0)}=\f|_{\icint(\f,b,x_0)} $,
(under the prescribed basic coupling) we have
\begin{align*}
	\big\{ \h^1(t_0,x_0) < b \big\} = \big\{ \h^2(t_0,x_0) < b \big\}.
\end{align*}
\item \label{enu:icLoc>}
Similarly, the event $ \{ \h^1(t_0,x_0) > b \} $ dependent on $ \f $
only through its restriction onto $ \icint(\f,b,x_0) $.
\end{enumerate}
\end{lemma}
\begin{proof}
The proof of Part\ref{enu:icLoc<} and \ref{enu:icLoc>} are similar,
so we consider only the former.
The proof goes through the correspondence between surface growths and particle systems.
More precisely, let
\begin{align}
	\label{eq:particlen}
	Y_n(t) := \inf\{ y\in \tfrac12+\Z : \h^1(t,y+\tfrac12) \geq n \},
	\quad
	n\in \Z.
\end{align}
Referring back the correspondence~\eqref{eq:CGMtoTASEP}--\eqref{eq:TASEPtoCGM} 
between height profiles and particle configurations,
one readily verifies that $ \{ \ldots<Y_1(t)<Y_2(t)<\ldots \} $ gives the trajectories of the corresponding particles.
Let $ \f^* := \lim_{x\to\infty} \f(x) \in \Z\cup\{\infty\} $ 
and $ \f_* := \lim_{x\to-\infty} \f(x) \in \Z\cup\{-\infty\} $.
Note that, by definition, $ Y_n(t)\equiv\infty $, $ \forall n > \f^* $
and $ Y_n(t) \equiv -\infty $, $ \forall n\leq f_* $, 
so we allow phantom particles to be placed at $ \pm\infty $ if $ \f^*<\infty $ or $ \f_*>-\infty $.
In addition to particles, we also consider the trajectories of holes (i.e., empty sites). Let
\begin{align}
	\label{eq:holen}
	\til Y_n(t) := \sup\{ y\in \tfrac12+\Z : \h^1(t,y-\tfrac12)-(y-\tfrac12) \geq n \},
	\quad
	n\in \Z.
\end{align}
The holes $ \{ \ldots < \til Y_2(t)<\til Y_1(t) <\ldots \} $ evolve under the reverse dynamics of the particles:
each $ \til Y_n $ attempts to jump to the \emph{right} in continuous time, under the exclusion rule.

Fix $ x_0,b\in\Z $, $ t_0\in[0,NT] $.
Under the preceding setup, we have
\begin{align*}
	\{ \h^1(t_0,x_0) < b \} = \{ Y_{b}(t_0) > x_0 \} = \{ \til Y_{b-x_0}(t_0) < x_0 \}.
\end{align*}
Also, from~\eqref{eq:icrg}--\eqref{eq:holen}, it is straightforward to verify that
\begin{align}
	\label{eq:YtilYid}
	Y_b(0) +\tfrac12 = k^+(\f,b,x_0),
	\quad
	\til Y_{b-x_0}(0) - \tfrac12 = k^-(\f,b,x_0).
\end{align}
Indeed, since particles in \ac{TASEP} always jumps to the left,
all the particles $ \{Y_n\}_{n>b} $ to the right of $ Y_b $ do not affect the motion of $ Y_b $.
In particular, the event $ \{ Y_b(t_0) > x_0 \} $ is independent of $ \{ Y_n(0) \}_{n>b} $.
Translating this statement into the language of height function using~\eqref{eq:YtilYid},
we conclude that $ \{ \h^1(t_0,x_0) < b \} $ is independent of $ \f(x)|_{x>k^+(\f,b,x_0)} $.
The same argument applied to holes in places of particles shows that
$ \{ \h^1(t_0,x_0) < b \} $ is independent of $ \f(x)|_{x<k^-(\f,b,x_0)} $. 
\end{proof}

We now prove~Theorem~\ref{thm:main}\ref{enu:lwb}.
\begin{proof}[Proof of Theorem~\ref{thm:main}\ref{enu:lwb}]

Indeed, the lower bound~\eqref{eq:lwb}, is equivalent to the following statement
\begin{align}
	\label{eq:lwb:}
	\liminf_{N\to\infty} \frac{1}{N^2} \log \PN(\hN\in\calO)
	\geq
	-\rf(h_*),
	\quad
	\forall h_* \in \calO \subset \Sp,
	\
	\calO \text{ open.}
\end{align}
To show~\eqref{eq:lwb:}, we fix $ h_* \in \calO \subset \Sp $ hereafter, 
and assume without lost of generality $ \rf(h_*)<\infty $.
Under such an assumption, $ h^* $ is necessarily continuous (otherwise it is straightforward to show that $ \rff(h^*)=\infty $).
This being the case, there exist $ a>0 $ and $ r_0<\infty $ such that $ \calU_{3a,r'_*}(h^*) \subset \calO $.
Hence it suffices to show
\begin{align}
	\label{eq:lwb::}
	\liminf_{N\to\infty} \frac{1}{N^2} \log \PN\big( \hN\in\calU_{3a,r_0}(h^*) \big)
	\geq
	-\rf(h_*).
\end{align}

The step is to approximate $ h^* $ with $ g $ of the form described in Proposition~\ref{prop:ent:up}.
Fix $ \e_*\in(0,\frac{a}{7}] $.
We apply Lemma~\ref{lem:h*apprx} with the prescribed $ a $, $ r_0 $, $ \e_* $ and $ h^* $, 
to obtain a $ \Sp $-valued, $ \Sigma(\frac{T}{\scl_*},\frac{r_*}{\scl_*}) $-piecewise linear function $ g $
that satisfies \eqref{eq:ling:apprx}--\eqref{eq:ICrg-},
together with $ \scl_*\in\N $ and $ r_*,r'_0 \geq r_0 $.
Write $ \gIC := g(0) $
Next, we discretize $ \gIC $ to obtain $ \gic(x) := \lfloor N g(\tfrac{x}{N}) \rfloor $.
Indeed, with $ \gIC\in\Splip $, this defines a $ \Spf $-valued (see~\eqref{eq:Spf}) profile.
Also, one readily check that the corresponding scaled $ \gic_N $ profile does converge to $ \gIC $,
i.e.,
\begin{align}
	\label{eq:gic:cnvg}
	\lim_{N\to\infty} \dist(\gic_N,\gIC)=0.
\end{align}
Let $ \gN(t,\xi) $ denote the \ac{TASEP} height process starting from~$ \gic_N $.
We apply Proposition~\ref{prop:ent:up:} with the prescribed $ \e_*\leq\frac{a}{5} $, $ r_* \geq r_0 $ and $ g $ to get
\begin{align*}
	\liminf_{N\to\infty} \frac{1}{N^2} \log 
	\PN\Big( \gN \in \calU_{\frac{a}{5},r_*}(g)\Big)
	\geq 
	- \int_0^T\int_{\R} \lrf(g_t,g_\xi) dtd\xi 
	-
	\int_0^T \int_{r_*\leq |\xi| \leq r^* } \hspace{-10pt} \prff\Big( \frac{g_t}{g_\xi(1-g_\xi)} \Big) dtd\xi
	-\e_*.
\end{align*}
Further use~\eqref{eq:ling:ent}--\eqref{eq:ling:entail} 
to bound the r.h.s.\ by $ -\rf(h^*)-3\e_* $ from below,
we obtain
\begin{align}
	\label{eq:lwb:g:}
	\liminf_{N\to\infty} \frac{1}{N^2} \log 
	\PN\Big( \gN \in \calU_{\frac{a}{5},r_0}(g)\Big)
	\geq 
	- \rf(h^*) - 3\e_*.
\end{align}

The next step is to relate the l.h.s.\ of~\eqref{eq:lwb:g:} 
to a bound on $ \frac{1}{N^2} \log\PN( \hN \in \calU_{3,r_0}(h^*)) $.
To this end, we consider the super-process $ \barg $ and sub-process $ \undg $,
which are \ac{TASEP} height processes starting from the following shifted initial conditions:
\begin{align}
	\label{eq:barudng0}
	\bargic := \gic + \lfloor Na \rfloor,
	\quad
	\undgic := \gic- \lfloor Na \rfloor.
\end{align}
Recall from~\eqref{eq:shiftinvt} that height processes are shift-invariant,
so, in fact, $ \barg(t)=\g(t) + \lfloor Na \rfloor $ 
and $ \undg(t)=\g(t) - \lfloor Na \rfloor $, $ \forall t\in[0,NT] $.
In particular,
\begin{align}
	\label{eq:barundA}	
	\big\{ \gN \in \calU_{\frac{a}{5},r_0}(g) \big\}
	\subset
	(\bar\calA(\barg_N)\cap\und\calA(\undg_N)),
\end{align}
where
\begin{align*}
	\bar\calA(\barg_N) &:=
	\Big\{ 
		\barg_N(t,\tfrac{x}{N}) < g(t,\tfrac{x}{N})+a+\tfrac{a}{5}, \ \forall (t,\tfrac{x}{N})\in[0,T]\times[-r_0,r_0] 
	\Big\},
\\
	\und\calA(\undg_N) &:=
	\Big\{ 
		\undg_N(t,\tfrac{x}{N}) > g(t,\tfrac{x}{N})-a-\tfrac{a}{5}, \ \forall (t,\tfrac{x}{N})\in[0,T]\times[-r_0,r_0] 
	\Big\}.
\end{align*}
Furthermore, rewriting~\eqref{eq:ling:apprx} for $ t=0 $ as
$
	g(0,\xi) - a < h^*(0,\xi) < g(0,\xi) + a,
$
$ \forall \xi\in [-r'_*,r'_*] $,
and combining this with~\eqref{eq:barudng0}, \eqref{eq:ic:cnvg} and \eqref{eq:gic:cnvg},
we obtain
\begin{align}
	\label{eq:sand}
	\undgic(x) < \hic(x) < \bargic(x),	\quad \forall x\in [-Nr'_*,Nr'_*],
\end{align}
for all $ N $ large enough.

Our goal is to utilize the ordering property~\eqref{eq:mono} to \emph{sandwich}
the process $ \h(t) $ in between the super- and sub-processes.
However, in order for~\eqref{eq:mono} to apply, we need the inequality in~\eqref{eq:sand} 
to hold for all $ x\in\Z $, not just $ x\in [Nr^-,Nr^+] $.
With this in mind, 
writing $ [-Nr'_*,Nr'_*]\cap\Z = [-x'_*,x'_*] $, $ x'_*\in\N $,
we perform the following surgery on $ \barg(0) $ and $ \undg(0) $:
\begin{align*}
	\g^{*,\ic}(x) &:= 
	\left\{\begin{array}{l@{,}l}
		\bargic(x)				&\text{ for } x\in [-x'_*,x'_*],
		\\
		\bargic(x'_*)+|x-x'_*|	&\text{ for } x\in (x'_*,\infty),
		\\
		\bargic(-x'_*)			&\text{ for } x\in (-\infty,-x'_*),
	\end{array}\right.
\\
	\g^\ic_*(x) &:= 
	\left\{\begin{array}{l@{,}l}
		\undgic(x)				&\text{ for } x\in [-x'_*,x'_*],
		\\
		\undgic(x'_*)			&\text{ for } x\in (x'_*,\infty),
		\\
		\undgic(-x'_*)-|x+x'_*|	&\text{ for } x\in (-\infty,-x'_*).
	\end{array}\right.
\end{align*}
This gives $ \g^\ic_*(x) < \hic(x) < \g^{*,\ic}(x) $, $ \forall x\in \Z $.
Let $ \g^* $ and $ \g_* $ denote the height processes starting from $ \g^\ic_* $ and $ \g^{*,\ic} $, respective.
We then have
\begin{align}
	\label{eq:sand:}
	\g_*(t,x) < \h(t,x) < \g^*(t,x),	\quad \forall (t,x)\in[0,TN]\times\Z.
\end{align}
Next, recall the definition of $ k_\pm(\f,b,x) $ and $ \icint(\f,b,x) $ from~\eqref{eq:icrg}.
By Lemma~\ref{lem:icLoc}, the event $ \bar\calA(\barg_N) $ depends on $ \bargic $ only through
$ \bargic|_{\calV} $, where
\begin{align*}
	\calV := \bigcup_{t\in[0,NT]}\bigcup_{x\in[-Nr_0,Nr_0]} \icint(\bargic, \beta_N, x),
	\quad
	\beta_N(t,x) := \lceil N(g(t,\tfrac{x}{N})+a+\tfrac{a}{5} \rceil.
\end{align*}
Referring to~\eqref{eq:icrg+}--\eqref{eq:icrg-} and~\eqref{eq:ICrg+}--\eqref{eq:ICrg-},
we have that
\begin{align*}
	\lim_{N\to\infty} 
	&\sup\big\{ \tfrac{1}{N} k^+ (\bargic,\beta_N(t,x), x) : x\in[-Nr_0,Nr_0] \big\}
\\
	&=
	\inf \Big\{ \xi \geq r_0 : g(0,\xi)+a >  \Big(\sup_{[0,T]\times[-r_0,r_0]}g\Big) +a+\tfrac{a}{5} \Big\}
	<
	r'_*,
\\
	\lim_{N\to\infty} 
	&\inf\big\{ \tfrac{1}{N} k^- (\bargic,\beta_N(t,x), x) : x\in[-Nr_0,Nr_0] \big\}
\\
	&=
	\sup \Big\{ \xi \leq -r_0 : g(0,\xi)+a -\xi >  \Big(\sup_{[0,T]\times[-r_0,r_0]}g(t,\xi)-\xi\Big) +a+\tfrac{a}{5} \Big\}
	>
	-r'_*.
\end{align*}
Consequently, $ \calV \subset [-Nx'_*,Nx'_*] $, for all $ N $ large enough.
Since, by construction, $ \g^{*,\ic}|_{[-Nx'_*,Nx'_*]} = \bargic|_{[-Nx'_*,Nx'_*]} $,
we thus have $ \bar\calA(\barg_N) = \bar\calA(\g^*_N) $.
A similar argument also gives $ \und\calA(\undg_N) = \und\calA(\g_{*,N}) $. 
Referring back to~\eqref{eq:barundA}, we now have
$
	\{ \gN \in \calU_{\frac{a}{5},r_0}(h^*) \}
	\subset
	(\bar\calA(\g^*_N)\cap\und\calA(\g_{*,N})). 
$
Combining this with~\eqref{eq:sand:} gives
\begin{align}
	\label{eq:sandd}
	\big \{ \gN \in \calU_{\frac{a}{5},r_0}(g) \big \}
	\subset
	\big\{ \hN \in \calU_{(1+\frac{1}{5})a,r_0}(g)  \big\}.
\end{align}
Since $ g $ satisfies~\eqref{eq:ling:apprx} and since $ r'_*\geq r_0 $,
we have 
$ \calU_{(1+\frac{1}{5})a,r_0}(g) \subset \calU_{(2+\frac{1}{6})a,r_0}(h^*)\subset\calU_{3a,r_0}(h^*) $.
Using this to replace~$ \calU_{(1+\frac{1}{5})a,r_0}(g) $ by $ \calU_{3a,r_0}(h^*) $ in~\eqref{eq:sandd},
and inserting the result into~\eqref{eq:lwb:g:},
we arrive at
\begin{align*}
	\liminf_{N\to\infty} \frac{1}{N^2} \log \PN\big( \hN\in\calU_{3a,r_0}(h^*) \big)
	\geq
	-\rf(h_*)-3\e_*.
\end{align*}
Since $ \e_* \in (0,\frac{a}{5}] $ is arbitrary,
letting $ \e_* \downarrow 0 $ gives the desired result~\eqref{eq:lwb::}.
\end{proof}

\section{Upper Bound: Proof of Proposition~\ref{prop:ent:lw}}
\label{sect:pfentlw}

\subsection{The Conditioned Law $ \QN $}
Recall from~\eqref{eq:tubular} the definition of the tubular set $ \calU_{a,r}(g) $.
This purpose of this subsection is to prepare a few basic properties
of the conditioned law $ \QN $ as in~\eqref{eq:QN}.
Roughly speaking, 
Proposition~\ref{prop:QNIto}--\ref{prop:QN:ent} in the following
assert that the conditioned law $ \QN $ is written as a perturbed \ac{TASEP},
where the underlying Poisson clocks have rates $ \lambda(t,x,\h(t)) $
that vary over $ (t,x) $ and depend on the current configuration $ \h(t) $ at the given time.
In a finite state space setting (e.g., \ac{TASEP} on the circle $ \Z/(N\Z) $)
such a result follows at once by standard theories.
For the \ac{TASEP} on the full line $ \Z $ considered here,
as we cannot identify a complete proof of Proposition~\ref{prop:QNIto}--\ref{prop:QN:ent} in the literature,
we include a brief, self-contained treatment in this subsection.

For the rest of this subsection,
fix $ a>0, r<\infty $,  a continuous deviation $ g\in\Sp\cap C([0,T],\Splip) $,
and let $ \calU=\calU_{a,r}(g) $ denote the tubular set around $ g $.
Scaling is irrelevant in this subsection,
we often drop the dependence on $ N $, e.g., writing $ \Pr $ in place of $ \PN $.
Hereafter, for a given $ v\in\R $,  $\lceil v \rceil := \inf\{ i\in \Z: v\le x\}$ and 
$ \lfloor v \rfloor  := \sup\{ i\in \Z: \ge i\}$ denote the correspond round-up and round-down.
We write the tubular set $ \calU $ as
\begin{align}
	\label{eq:tubular:a}
	&\calU = 
	\bigcap_{t\in[0,NT]} \big\{ \h(t) \in \calB(t)  \big\},
\\
	&\label{eq:tubular:b}
	\calB(t) := \bigcap_{x\in[-k_0,k_0]}  \calB(t,x),
	\quad
	\calB(t,x) := \big\{ \f\in\Spf : \, \undenv(t,x) < \f(x) < \barenv(t,x) \big\},
\end{align}
where $ [-k_0,k_0]=[-Nr,Nr]\cap\Z $, and $ t\mapsto \barenv(t,x), \undenv(t,x) $
are the upper and lower envelops, given by 
$ \barenv(t,x) := \lim_{\e\downarrow 0}\lceil N(g(t,\frac{x}{N})+a) + \e \rceil $
and $ \undenv(t,x) := \lim_{\e\downarrow 0} \lfloor N(g(t,\frac{x}{N}) -a) - \e \rfloor $. 

Let us first step up a few notations. Define
\begin{align}
	\label{eq:Sph}
	\Sph(\hic) := \Big\{ \f\in\Spf : \f(x) \geq \hic(x),\forall x\in\Z  \Big\}.
\end{align}
Indeed, the space $ \Sph(\hic) $ contains the set of all possible configurations $ \h(t) $
of the \ac{TASEP} height process starting from $ \hic $,
(because \ac{TASEP} height function \emph{grows} in time).
We say $ F:\Sph(\hic) \to \R $ and $ G:[0,NT]\times\Sph(\hic)\to\R $ are \textbf{local}, 
with support $ \calV=[k^-,k^+] $, if,
\begin{align*}
	&
	F(\f^1) = F(\f^2),
	\quad
	\forall \f^1,\f^2
	\text{ such that }
	\f^1|_{\calV} = \f^2|_{\calV},
\\
	&
	G(t,\f^1) = G(t,\f^2),
	\quad
	\forall t\in[0,NT], 
	\
	\forall \f^1,\f^2\in
	\text{ such that }
	\f^1|_{\Z\cap[-k,k]} = \f^2|_{\calV}.
\end{align*} 
Namely, $ F,G $ are local with support $ \calV $ if they reduce to functions on 
$ \Z^{\Z\cap\calV} $ and $ [0,T]\times \Z^{\Z\cap\calV} $, respectively.
\begin{remark}
We emphasize here that our definition of local functions differs slightly from standard terminologies.
In the conventional setup,
one considers a Markov process with a state space $ \mathscr{S} $,
and functions $ F:\mathscr{S}\to\R $, $ G:[0,NT]\times\mathscr{S}\to\R $ 
that act on the \emph{entire} state space $ \mathscr{S} $.
Under such a setup, functions are local if they have finite supports, independent of the initial conditions of the process.
Here, unlike the conventional setup, we have \emph{fixed} the initial condition $ \hic $,
and consider functions $ F,G $ that act on the subspace $ \Sph(\hic) $. 
The supports of functions consider here may refer to $ \hic $ in general.
\end{remark}
\noindent
Define, for $ \f\in\Sph(\hic) $, Doob's conditioning function
\begin{align}
	\label{eq:doob}
	\doob(t,\f) 
	:= 
	\Pr \Big( \bigcap_{s\in[t,NT]} \big\{ \h(s) \in \calB(s) \big\} \Big| \h(t)=\f \Big).
\end{align}
This function is the building block of various properties of the conditioned law $ \QN $.
We begin by showing the following.

\begin{lemma}
\label{lem:doob}
The function~\eqref{eq:doob} is local,
and, $ t\mapsto \doob(t,\f) $ is Lipschitz, uniformly over $ [0,NT]\times\Sph(\hic) $.
The derivative is given by
\begin{align}
	\label{eq:doob'}
	\tfrac{d~}{dt} \doob(t,\f) = - \big( \gen q (t,\f) \big) \ind_{\calB(t)}(\f),
\end{align}
for all $ (t,\f) \in [0,NT]\times\Sph(\hic) $ where $ \tfrac{d~}{dt} \doob(t,\f) $ is defined.
\end{lemma}
\begin{proof}
We begin by showing that $ \doob(t,\f) $ is local.
With $ g $ being continuous, 
the upper and lower envelops $ t\mapsto \barenv(t,x), \undenv(t,x) $ are necessarily $ D([0,NT],\Z) $-valued.
We enumerate the discontinuity of $ t\mapsto\und{b}(t,x) $ and $ t\mapsto\bar{b}(t,x) $, $ x\in[-k_0,k_0] $,
within $ t\in[0,NT) $ as $ 0\leq t_1<t_1<\ldots<t_{n-1}< NT $, set $ t_0=0 $ and $ t_n=NT $ for consistency of notations.
Under such notations, we write
\begin{align}
	\bigcap_{s\in[t,NT]}\big\{ \h(s) \in \calB(s)\big\}
	\label{eq:doob:local}
	=
	\bigcap_{x\in[-k_0,k_0]}
	&\Big(
		\bigcap_{ i=1}^{n-1}
		\bigcap_{ s\in [t_{i-1},t_i)\cap[t,NT] } 
		\big\{ \h(s,x) > \undenv(t_{i},x) \big\}
		\cap
		\big\{ \h(s,x) < \barenv(t_{i},x) \big\}
\\
	&
	\notag
	\hphantom{	\bigcap_{x\in[-\ell,\ell]} \Big(\bigcap_{i=0}^{n-1}  }
		\cap
		\big\{ \h(t_n,x) > \undenv(t_n,x) \big\}
		\cap
		\big\{ \h(t_n,x) < \barenv(t_n,x) \big\} \Big).
\end{align}
Our goal is to show that, the probability of the event~\eqref{eq:doob:local},
conditioned on $ \h(t)=\f $, depends on $ \f $ in a local fashion.
Recall the notations~$ k_\pm(\f,b,x) $ and $ \icint(\f,b,x) $ from~\eqref{eq:icrg}.
View $ \f $ as the initial condition of the \ac{TASEP} starting at time $ t $.
Lemma~\ref{lem:icLoc} asserts that the event~$ \big\{ \h(s,x) < b \big\} $
depends on $ \f $ only through $ \f|_{\icint(\f,b,x)} $.
We say $ \f^1 \geq \f^2 \in \Spf $, if $ \f^1(x) \geq \f^2(x) $, $ \forall x\in\Z  $.
From~\eqref{eq:icrg}, one readily checks that $ \icint(\f^1,b,x) \subset \icint(\f^2,b,x) $, if $ \f^1 \geq \f^2 $.
Further, recall from~\eqref{eq:Sph} that $ \f \geq \hic $, $ \forall \f\in\Sph(\hic) $,
so in particular $ \icint(\f,b,x) \subset \icint(\hic,b,x) $.
Now, if $ \icint(\hic,b,x) $ is an unbounded interval, 
i.e., $ k^+(\hic,\barenv(s,x),x) =\infty $ or $ k^-(\hic,\barenv(s,x),x) =-\infty $,
is it straightforward to verify that $ \{ \h(s) < b, \forall s\in[0,NT] \} $ must hold.
In this case, $ \{ \h(s) < b, \forall s\in[t,NT] \} $ holds regardless of $ \f $.
Consequently, the event $ \{ \h(s) < b, \forall s\in[t,NT] \} $ depends on $ \f $
only through its restriction onto
\begin{align*}
	\icint'(\hic,b,x)
	:=
	\left\{
	\begin{array}{l@{,}l}
		\icint(\hic,b,x)	&	\text{ if } \icint(\hic,b,x) \text{ is bounded},
		\\
		\emptyset			&	\text{ otherwise.}
	\end{array}\right.
\end{align*}
A similar argument shows that $ \{ \h(s) > b, \forall s\in[t,NT] \} $ depends on $ \f $
only through its restriction onto $ \icint'(\hic,b,x) $.
Using these properties for $ b=\barenv(t_{i},x),\undenv(t_{i},x) $, $ i=1,\ldots,n $,
and $ x\in[-k_0,k_0] $ in~\eqref{eq:doob:local},
we see that the event $ \bigcap_{s\in[t,NT]}\{ \h(s) \in \calB(s)\} $ depends on $ \f $
only through $ \f|_\calV $, where $ \calV $ is the finite interval
\begin{align*}
	\calV := \bigcup_{i=1}^n \bigcup_{x\in[-k_0,k_0]}\icint'(\hic,b,x).
\end{align*}
This concludes the locality of Doob's function $ \doob(t,\f) $.

Next we turn to the Lipschitz continuity.
Fix $ t_1<t_2\in[0,NT] $.
Referring back to~\eqref{eq:doob}, we have that
\begin{align}
	\label{eq:doob:decomp}
	\doob(t_1,\f) = \Ex\big( \doob(t_2,h(t_2)) \ind_{\cap_{s\in[t_1,t_2]}\set{\h(s)\in \calB(s)} } | \h(t_1)=\f \big). 
\end{align}
Let $ V $ denote the event that none of the underlying Poisson clocks
among sites $ x\in[-k_0,k_0] $ ever ring during $ s\in[t_1,t_2] $.
On the event $ V $, we have that
$ \doob(t_2,h(t_2)) \ind_{\cap_{s\in[t_1,t_2]}\set{\h(s)\in \calB(s)} } = \doob(t_2,\h(t_1)) $.
Using this in~\eqref{eq:doob:decomp} gives
\begin{align}
	\label{eq:doob:lip}
	\doob(t_1,\f) = \doob(t_2,\f) \Pr(V) 
	+ 
	\Ex\big( \doob(t_2,\h(t_2))\ind_{\cap_{s\in[t_1,t_2]}\set{\h(s)\in \calB(s)}\cap V^c} \big| \h(t_1)=\f \big).
\end{align}
For the event $ V $,
there exists a constant $ c<\infty $, depending only on $ k_0 $,
such that $ \Pr(V) \geq 1 - c|t_2-t_1| $.
Using this in \eqref{eq:doob:lip} gives
\begin{align*}
	|\doob(t_1,\f)-\doob(t_2,\f)| \leq c \doob(t_2,\f)|t_2-t_1| \leq c|t_2-t_1|.
\end{align*}
This concludes the Lipschitz continuity of $ t\mapsto\doob(t,\f) $.

To show~\eqref{eq:doob'},
fix $ t_1 $ and let $ \sigma:= \inf\{ s\geq t_1: \h(s)\notin \calB(s) \} $
to be the first hitting time for $ \h(s) $ to be outsides of the tubular set $ \calU $.
Since, by definition, $ \doob(t,\f) =0 $ for $ \f\notin \calB(t) $,
we have that
\begin{align}
	\label{eq:doob:localized}
	\doob(t_2,\h(t_2))\ind_{\cap_{s\in[t_1,t_2]}\set{\h(s)\in \calB(s)}}
	=
	\doob(t_2\wedge\sigma, \h(t_2\wedge\sigma)).
\end{align}
Since $ \doob(t,\f) $ is local and uniformly Lipschitz in $ t $,
we have that
\begin{align}
	\label{eq:doob:mg}
	t\longmapsto \doob(t,\h(t)) - \int_{t_1}^{t} \big( \partial_t + \gen \big) \doob(t,\h(s)) ds
\end{align}
is a $ \Pr $-martingale.
Furthermore, with $ \doob(t,\f) $ being local and uniformly Lipschitz in $ t $, 
the process~\eqref{eq:doob:mg} is bounded.
Hence the localized process
\begin{align*}
	t\longmapsto 
	\doob(t\wedge\sigma,\h(t\wedge\sigma)) - \int_{t_1}^{t\wedge\sigma} 
  \big(\partial_t + \gen \big) \doob(s,\h(s)) ds
\end{align*}
is also a $ \Pr $-martingale.
Combining this with~\eqref{eq:doob:localized} and \eqref{eq:doob:decomp} gives
\begin{align}
	\label{eq:doob':}
	\Ex\Big( \int_{t_1}^{t_2\wedge\sigma} \big(\partial_t + \gen \big)\doob(s,\h(s)) ds \Big| \h(t_1)
  =\f \Big) =0.
\end{align}
Now, consider the case $ \h(t_1)=\f\in \calB(t_1) $.
In this case we necessarily have $ \sigma > t_1 $.
Hence, for fixed $ t_1 $, almost surely as $ t_2\downarrow t_1 $,
\begin{align}
	\label{eq:ascndoob}
	\frac{1}{t_2-t_1}\int_{t_1}^{t_2\wedge\sigma} \big(\partial_t + \gen \big)\doob(s,\h(s)) ds
	\longrightarrow
         \big(\partial_t + \gen \big) \doob(t_1,\h(t_1)).
\end{align}
With $ \doob(t,\f) $ being local and uniformly Lipschitz in $ t $, 
the l.h.s.\ of~\eqref{eq:ascndoob} is uniformly bounded over $ t_2\in(t_1,NT] $.
Hence the almost sure convergence~\eqref{eq:ascndoob} give convergence in expectation, i.e.,
\begin{align}
	\label{eq:ascndoob:}
	\frac{1}{t_2-t_1}\Ex\Big(\int_{t_1}^{t_2\wedge\sigma} \big(\partial_t + \gen \big)\doob(s,\h(s))
  ds\Big|\h(t_1)=\f\Big)
	\longrightarrow
	\partial_t \doob(t_1,\h(t_1)) + \gen\doob(t_1,\h(t_1)).
\end{align}
Combining~\eqref{eq:ascndoob:} with~\eqref{eq:doob':} gives~\eqref{eq:doob'},
for the case $ \f\in\calB(t_1) $.
For the case $ \f=\h(t_1)\not\in \calB(t_1) $,
since the envelops $ \barenv(t,x) $ and $ \undenv(t,x) $ are right-continuous in $ t $,
we have that $ \doob(t_2,\f)=0=\doob(t_1,\f) $, for all $ 0<t_2-t_1 $ small enough.
Hence $ \partial_t \doob(t_1,\f) =0 $ and \eqref{eq:doob'} follows.
\end{proof}

The next step is to derive the It\^{o} formula for $ \h $ under the conditioned law $ \Prr $.
To this end, define, for $ \f\in\Sph(\hic) $, the perturbed rate
\begin{align}
	\label{eq:lambdatxf}
	\lambda(t,x,\f) := \frac{\doob(t,\f^{x})}{\doob(t,\f)}.
\end{align}
Recall that $ \f^{x} := \f + \ind_\set{x} $
and recall from~\eqref{eq:mob} that $ \mob(\f,x) $ denotes the mobility function.
We consider the perturbed, time-dependent generator acting on local $ \f $:
\begin{align}
	\label{eq:genn}
	\big( \genn(t) F \big)(\f) := \left\{\begin{array}{l@{}l}
		\displaystyle \sum_{x\in\Z} \lambda(t,x,\f) \mob(\f,x) \big( F(\f^{x}) - F(\f) \big),
		&
		\text{ if } \f \in \calB(t),
		\\
		~
		\\
		0,	&	\text{ otherwise.}
	\end{array}\right.
\end{align}
Since the term $ 1/\doob(t,\f) $ is unbounded in general,
the expression~\eqref{eq:genn} could potentially cause issues when integrating $ \genn(t)F $ over $ \Ex_{\Prr} $.
We show in the next lemma that this is not the case.
\begin{lemma}
\label{lem:doob:int}
For all $ t\in[0,NT] $,
\begin{align}
	\label{eq:doob:int}
	\Ex_{\Prr} \Big( \frac{1}{\doob(t,\h(t))} \Big)
	\leq
	\frac{1}{\doob(0,\hic)}.
\end{align}
In particular, for any local, bounded $ G:[0,NT]\times\Sph(\hic)\to\R $ with support $ \calV $,
\begin{align*}
	\Ex_{\Prr}\big| \genn(t)G(t,\h(t)) \big|
	\leq
	\frac{\#(\calV\cap\Z)}{\doob(0,\hic)} \Vert G(t,\Cdot) \Vert_{\infty}.
\end{align*}
\end{lemma}
\begin{proof}
Indeed, since $ \Prr $ is the conditioned law around the tubular set $ \calU $, we have 
\begin{align}
	\label{eq:prr:cnd}
	\Ex_{\Prr} \Big( \frac{1}{\doob(t,\h(t)} \Big) 
	= 
	\frac{1}{\doob(0,\hic)} \Ex\Big( \frac{1}{\doob(t,{\h(t)})} \ind_{\cap_{s\in[0,NT]}\set{\h(s)\in \calB(s)}} \Big).
\end{align}
Let $ \filt_t $ denote the canonical filtration of $ \h(t) $.
We indeed have that
\begin{align}
	\label{eq:doob:cnd}
	\Ex\big(\ind_{\cap_{s\in[0,NT]}\set{\h(s)\in \calB(s)}}\big|\filt_t )
	&= {\ind_{\cap_{s\in[0,t]}\set{\h(s)\in \calB(s)}} \Ex\big(\ind_{\cap_{s\in[t,NT]}\set{\h(s)\in \calB(s)}}\big|\filt_t )}
	\nonumber
 \\
	&= {\ind_{\cap_{s\in[0,t]}\set{\h(s)\in \calB(s)}} \doob(t,\h(t)).}
\end{align}
Inserting~\eqref{eq:doob:cnd} into~\eqref{eq:prr:cnd} gives the desired result~\eqref{eq:doob:int}.
\end{proof}

We now derive the It\^{o} formula for $ \h $ under the conditional law $ \Prr $.
\begin{proposition}
\label{prop:QNIto}
Let $ G:[0,NT]\times\Sph(\hic) \to \R $ be a bounded local function which is Lipschitz in $ t $,
uniformly over $ [0,NT]\times\Sph(\hic) $.
We have, for each fixed $ t_1<t_2\in[0,NT] $,
\begin{align}
	\label{eq:QIto}
	\Ex_{\Prr} G(t,\h(t) ) |_{t=t_1}^{t=t_2}
	=
	\Ex_\Prr \int_{t_1}^{t_2} \Big( 
       \partial_t+\genn(t)\Big)G(t,\h(t)) dt.
\end{align}
\end{proposition}
\begin{proof}
By definition,
\begin{align}
	\label{eq:QIto:1}
	\Ex_{\Prr} \big( G(t_2,\h(t_2) ) \big| \filt_{t_1} \big)
	=
	\frac{1}{\doob(t_1,\h(t_1))} \Ex\big( G(t_2, \h(t_2)) \ind_{\cap_{t\in[t_1,NT]}\set{\h(t)\in \calB(t)} } \big| \filt_{t_1} \big).
\end{align}	
Let $ \sigma:=\inf\{ t\geq t_1: \h(t)\in \calB(t)^c \} $ be the first time that $ \h $ 
reaches outside of the tubular set $ \calU $.
Using~\eqref{eq:doob:cnd} for $ [s,t]=[t_1,t_2] $ on the r.h.s.\ of \eqref{eq:QIto:1}, together with $ \doob(\sigma,\h(\sigma))=0 $, 
we rewrite~\eqref{eq:QIto:1} as
\begin{align}
	\notag
	\Ex_{\Prr} \big( G(t_2,\h(t_2) ) \big| \filt_{t_1} \big)
	&= 
	\frac{1}{\doob(t_1,\h(t_1))}
	\Ex\big(
		\ind_{\cap_{t\in[t_1,t_2]}\set{\h(t)\in \calB(t)}}\big(qG\big)(t_2, \h(t_2)) \big| \filt_{t_1} 
	\big) 
\\
	\label{eq:QIto:2}
	&= 
	\frac{1}{\doob(t_1,\h(t_1))} \Ex\big(\ind_{\cap_{t\in[t_1,\sigma \wedge t_2]}\set{\h(t)\in \calB(t)}}
	(\doob G)(t_2\wedge\sigma, \h(t_2\wedge\sigma)) \big| \filt_{t_1} \big).
\end{align}

Our next step is to express~\eqref{eq:QIto:2} in terms of a time integral.
To this end, note that since $ (\doob G)(t,\f) $ is bounded, local, and uniformly Lipschitz in $ t $, the process
\begin{align*}
	t \longmapsto \int_{t_1}^{t\wedge\sigma} \big( \partial_t +\gen \big) \big(\doob G\big)(t,\h(t)) dt
\end{align*}
is a $ \Pr $-martingale.
Consequently,
\begin{align}
	\notag
	\Ex\big( (\doob G)(t_2\wedge\sigma, &\h(t_2\wedge\sigma)) \big| \filt_{t_1} \big)
\\
	\notag
	&=
	(\doob G)(t_1,\h(t_1)) 
	+	
	\int_{t_1}^{t_2} 
	\Ex\Big( 
		\ind_{\sigma > t_2}
		\big( \partial_t +L \big) \big(\doob G\big)(t,\h(t)) 
		\Big|
			\filt_{t_1}
	\Big) dt
\\ 
	\label{eq:QIto:3}
	&=
	(\doob G)(t_1,\h(t_1)) 
	+
	\int_{t_1}^{t_2} 
		\Ex\Big( \ind_{\cap_{s\in[t_1,t]}\set{\h(s)\in \calB(s)}} \big( \partial_t +\gen \big) \big(\doob G\big)(t,\h(t)) 
	\Big|
		\filt_{t_1}
	\Big) dt.
\end{align}
Next, in~\eqref{eq:QIto:3}, use~\eqref{eq:doob:cnd} to write
$	
	\ind_{\cap_{s\in[t_1,t]}\set{\h(s)\in \calB(s)}} 
	= 
	\frac{1}{\doob(t,\h(t))} \Ex(\ind_{\cap_{s\in[t_1,NT]}\set{\h(s)\in \calB(s)}}|\filt_t)
$,
and divide the resulting equation~\eqref{eq:QIto:3} by $ \doob(t_1,\h(t_1)) $.
We now obtain
\begin{align*}
	\frac{1}{\doob(t_1,\h(t_1))}
	\Ex\big( (\doob G)(t_2\wedge\sigma, \h(t_2\wedge\sigma)) \big| \filt_{t_1} \big)
	=
	G(t_1,\h(t_1))
	+	
	\int_{t_1}^{t_2} 
	 \Ex_\Prr \left( \Big(\frac{1}{\doob}\Big( \partial_t + \gen\Big)(qG)\Big)(t,\h(t)) 
	\Big|
	\filt_{t_1}
	\right) dt.
\end{align*}
Combining this expression with~\eqref{eq:QIto:2} gives
\begin{align*}
	\Ex_{\Prr} \big( G(t_2,\h(t_2) ) \big| \filt_{t_1} \big)
	=
	G(t_1,\h(t_1))
	+	
	\int_{t_1}^{t_2} 
	 \Ex_\Prr \left( \Big(\frac{1}{\doob}\Big( \partial_t + \gen\Big)(qG)\Big)(t,\h(t)) 
	\Big|
	\filt_{t_1}
	\right) dt.
\end{align*}
Now, move the term~$ G(t_1,\h(t_1)) $ to the l.h.s., and aver the result over $ \Ex_{\Prr} $, we arrive at
\begin{align}
	\label{eq:QIto:}
	\Ex_{\Prr} (G(t,\h(t))) |_{t=t_1}^{t=t_2}
	=
	\int_{t_1}^{t_2} \Ex_\Prr\left( \Big(\frac{1}{\doob}\Big( \partial_t + L\Big)(qG)\Big)(t,\h(t)) \right) dt.	
\end{align}
Finally, a straightforward calculation from the definition~\eqref{eq:genn},
together with the identity~\eqref{eq:doob'}, gives\\
$ \frac{1}{\doob}( \partial_t + L) (qG) = (\partial_t + \genn(t)) G $.
Inserting this into~\eqref{eq:QIto:} gives the desired result~\eqref{eq:QIto}.
\end{proof}

Recall that $ \prf(\lambda) $ denote the rate function for Poisson variables. 
We next derive an expression for the relative entropy $ H(\Prr|\Pr^\h) $.
\begin{proposition}
\label{prop:QN:ent}
The relative entropy of the conditioned $ \Prr $ with respect to $ \Pr $ is given by
\begin{align}
	\label{eq:prr:ent}
	H(\Prr|\Pr^\h) =
	\Ex_{\Prr}\Big( \int_0^{NT}  \sum_{x\in\Z} \mob(\h(t),x) \prf\big( \lambda(t,x,\h(t)) \big) dt \Big).
\end{align}
\end{proposition}
\begin{proof}
From~\eqref{eq:doob'}, we have that
\begin{align}
	\label{eq:prr:ent:1}
	\int_0^{NT} (\tfrac{1}{\doob} (\partial_t +\gen) \doob)(t,\h(t)) dt = 0,
	\quad
	\Prr\text{-a.s.}
\end{align}
Write $ \dot\doob(t,\f) := \frac{d~}{dt}\doob(t,\f) $.
Since $ \doob(t,\f) $ is local, 
the random variables $ \gen\doob =-\dot{\doob} $ are uniformly bounded, i.e.,
\begin{align}
	\label{eq:doob':bdd}
	|\gen\doob(t,\f)|, \ |\dot{\doob}(t,\f)| \leq c,
	\quad
	\forall t\in[0,NT], \ \f\in \calB(t),
\end{align}
for some $ c<\infty $ depending only on the support of $ \doob $.
Combining this with Lemma~\ref{lem:doob:int},
we see that the random variables $ \frac{\gen\doob}{\doob}(t,\h(t)) $ and $ \frac{\dot{\doob}}{\doob}(t,\h(t)) $ 
are $ L^1 $ under $ \Prr $, uniformly over $ t\in[0,NT] $.
Taking expectation $ \Prr $ in~\eqref{eq:prr:ent:1} thus gives
\begin{align}
	\label{eq:prr:ent:2}
	0=
	\Ex_{\Prr} \int_0^{NT} \frac{\dot{\doob}}{\doob}(t,\h(t)) dt 
	+
	\Ex_{\Prr} \int_0^{NT} \frac{\gen\doob}{\doob}(t,\h(t)) dt .
\end{align}
With $ \Prr $ being the conditioned law around the tubular set $ \calU $,
we have $ H(\Prr|\Pr^\h) = -\log \Pr(\calU) = - \log \doob(0,\hic) $.
Subtracting~\eqref{eq:prr:ent:2} from the previous expression gives
\begin{align}
	\label{eq:prr:ent:3}
	H(\Prr|\Pr^\h)
	=
	-\log \doob(0,\hic)
	-\Ex_{\Prr} \int_0^{NT} \frac{\dot{\doob}}{\doob}(t,\h(t)) dt 
	-\Ex_{\Prr} \int_0^{NT} \frac{\gen\doob}{\doob}(t,\h(t)) dt.
\end{align}
The next step is to apply Proposition~\ref{prop:QNIto} with the function $ G(t,\f) = \log(\doob(t,\f)) $.
However, such a function is not Lipschitz in $ t $ due to the singularity at $ \doob(t,\f)=0 $. 
We hence introduce a small threshold $ a>0 $,
and apply Proposition~\ref{prop:QNIto} with $ G(t,\f) = \log(\doob(t,\f)+a) $.
This gives
\begin{align}
\begin{split}
	\label{eq:prr:ent:4}
	0=
	\Ex_{\Prr} &\log\frac{\doob(NT,\h(NT))+a}{\doob(0,\hic)+a}
	-
	\Ex_{\Prr} \int_0^{NT} \frac{\dot{\doob}}{\doob+a}(t,\h(t)) dt
	-
	\Ex_{\Prr} \int_0^{NT} \gen \log(\doob+a) (t,\h(t)) dt.
\end{split}
\end{align}
Since $ \h(NT)\in \calB(NT) $, $ \Prr $-a.s,
the first term in~\eqref{eq:prr:ent:4} is equal to $ \frac{1+a}{\doob(0,\hic)+a} $.
Subtracting~\eqref{eq:prr:ent:4} from~\eqref{eq:prr:ent:3},
we arrive at
\begin{align}
	\label{eq:prr:ent:5}
	H(\Prr|\Pr^\h)
	=
	H_1+H_2+H_3 + \Ex_{\Prr} \int_0^{NT} \Big( \gen \log \doob - \frac{\gen\doob}{\doob} \Big)(t,\h(t)) dt,
\end{align}
where
\begin{align*}
	H_1 &:= \log \frac{\doob(0,\hic)+a}{(1+a)\doob(0,\hic)},
\\
	H_2 &:= \Ex_{\Prr} \int_0^{NT}  \Big(\frac{\dot{\doob}}{\doob+a}-\frac{\dot{\doob}}{\doob}\Big)(t,\h(t)) dt
		= \Ex_{\Prr} \int_0^{NT}  \Big(\frac{a\dot{\doob}}{(\doob+a)\doob} \Big)(t,\h(t)) dt,
\\
	H_3 &:= \Ex_{\Prr} \int_0^{NT} \gen\big(\log(\doob+a)-\log\doob\big)(t,\h(t)) dt.
\end{align*}
A straightforward calculation
shows that
$ 
	(\gen \log \doob - \frac{\gen\doob}{\doob})(t,\f) 
	= \sum_{x\in\Z} \mob(\f,x) \prf(\frac{\doob(t,\f^x)}{\doob(t,\f)}) 
	= \sum_{x\in\Z} \mob(\f,x) \prf(\lambda(t,x,\f))
$.
Refer back to~\eqref{eq:prr:ent:5}.
It now remains only to show $ H_i\to 0 $, as $ a\downarrow 0 $, for $ i=1,2,3 $.

Clearly, $ H_1\to 0 $, as $ a\downarrow 0 $.
As for $ H_2 $, using~\eqref{eq:doob':bdd} to bound $ |\dot\doob| $, we have
\begin{align*}
	|H_2| 
	\leq 
	\Big( 
		\Ex_{\Prr} \int_0^{NT} \frac{ca}{\doob(t,\h(t))(\doob(t,\h(t))+a)}dt 
	\Big)	
	=	
	c\Big( 
		\Ex_{\Prr} \int_0^{NT} \frac{1}{\doob(t,\h(t))} dt 
		- \Ex_{\Prr} \int_0^{NT} \frac{1}{\doob(t,\h(t))+a}dt 
	\Big).
\end{align*}
By Lemma~\ref{lem:doob:int} and the monotone convergence theorem,
the r.h.s.\ tends to zero as $ a\downarrow 0 $.
Turning to $ H_3 $, we let $ \calV $ be a support of $ \doob $,
and write $ H_3 $ as $ H_3 = \Ex_{\Prr} \int_0^{NT} \til{H}_3(t) dt $,
where
\begin{align}
	\label{eq:prr:ent:H3}
	\til{H}_3(t) :=  
	\sum_{x\in\calV} \mob(\h(t),x) \log\frac{\doob(t,\h^{x}(t))+a}{\doob(t,\h(t))+a}\frac{\doob(t,\h(t))}{\doob(t,\h^{x}(t))}.
\end{align}
Clearly, $ \til{H}_3(t) \to 0 $ as $ a\downarrow 0 $, and
\begin{align}
	\label{eq:prr:ent:H3:}
	|\til{H}_3(t)|
	\leq
	\sum_{x\in\calV} \log\frac{1}{\doob(t,\h^{x}(t))(\doob(t,\h(t))}
	\leq
	\sum_{x\in\calV} \Big( \frac{1}{\doob(t,\h(t))} + \frac{1}{\doob(t,\h^{x}(t))} \Big).
\end{align}
By Lemma~\ref{lem:doob:int}, the r.h.s.\ of~\eqref{eq:prr:ent:H3:} is $ L^1 $ with respect to $ \Ex_{\Prr}\int_0^{NT}dt $.
Consequently, by the dominated convergence theorem,
$ H_3 \to 0 $ as $ a\downarrow 0 $.
\end{proof}

For convenience for referencing, 
we now summary Proposition~\ref{prop:QNIto}--\ref{prop:QN:ent}
in the \emph{scaled} form as follows.
\begin{corollary}
\label{cor:RE:Ito}
Let $ \QN $ be as in~\eqref{eq:QN}, $ \lambda(t,x,\f) $ be as in~\eqref{eq:lambdatxf},
and set $ \lambda_N(t,x,\f) := N^{-1}\lambda(Nt,x,\f) $.
For each $ t_1<t_2\in[0,T] $ and $ x\in\Z $,
\begin{align}
	&
	\label{eq:QN:hN}
	\Ex_{\QN}(\hN(t_2,\tfrac{x}{N}) - \hN(t_1,\tfrac{x}{N})) 
	= 
	\Ex_{\QN} \int_{t_1}^{t_2} \mob(\h(Nt),x) \lambda_N(t,x,\h(Nt)) dt,
\\
	&
	\label{eq:QN:ent}
	\frac{1}{N^2} H(\QN|\PN^\h)
	=
	\Ex_{\QN}\Big( \frac{1}{N} \sum_{x\in\Z} \int_0^T \mob(\h(Nt),x) \prf(\lambda_N(t,x,\h(Nt))) dt \Big).
\end{align}
\end{corollary}
\begin{proof}
The identity~\eqref{eq:QN:hN} essentially follows from Proposition~\ref{prop:QNIto} for $ G(t,\f)=\f(x) $.
The only twist is that such a function is not bounded above.
(Such $ G $ is bounded below because $ \f(x) \geq \hic(x) $, $ \forall \f\in\Sph(\hic) $, by~\eqref{eq:Sph}).
We hence fix a large threshold $ r<\infty $,
and apply Proposition~\ref{prop:QNIto} with $ G(t,\f) = \f(x)\wedge r $ to obtain
\begin{align*}
	\Ex_{\QN}(\hN(t_2,\tfrac{x}{N})\wedge r) - \Ex_{\QN}(\hN(t_1,\tfrac{x}{N})\wedge r) 
	= 
	\Ex_{\QN} \int_{t_1}^{t_2} \mob(\h(Nt),x) \lambda_N(t,x,\h(Nt)) \ind_\set{\h(Nt)\leq r} dt.
\end{align*}
Referring back to~\eqref{eq:QN},
we have that $ \h(t,x) $ is bounded under $ \QN $,
so let $ r\to\infty $  gives~\eqref{eq:QN:hN}.
The identity~\eqref{eq:QN:ent} follows directly from Proposition~\ref{prop:QN:ent}.
\end{proof}

\subsection{Proof of Proposition~\ref{prop:ent:lw}}
\label{sect:pfentlw:}
To simplify notations, in the following
we often write $ \mob(x)=\mob(\f,x) $ for the mobility function,
and write $ \lambda_N=\lambda_N(t,x)=\lambda_N(t,x,\f) $ for the rate.
Recall the expression of $ \rff(g) $ from~\eqref{eq:rff}.
We consider first the degenerate case $ \rff(g)=\infty $.

\medskip
\noindent\textbf{The case $ \rff(g)=\infty $.}
~We show that, in fact,
\begin{align}
	\label{eq:ent:lw:}
	\liminf_{N\to\infty} \frac{1}{N^2} H(\QN|\PN^\h) =\infty,
\end{align}
so in particular~\eqref{eq:ent:lw} holds.
We achieve~\eqref{eq:ent:lw:} by bounding the expression~\eqref{eq:QN:ent}
of the relative entropy from below.
To this end, fixing arbitrary $ n $, 
we recall that $ \{\sigma_i^n=\frac{iT}{2^n}\}_{i=0}^{2^n} $ denotes a dyadic partition,
and rewrite~\eqref{eq:QN:ent} accordingly as
\begin{align}
	\label{eq:ent:rffinf}
	\frac{1}{N^2} H(\QN|\PN^\h)
	=
 	\frac{1}{N}\sum_{x\in\Z} 
	\sum_{i=1}^{2^{n}}
	\Ex_{\QN} \int_{\sigma_{i-1}^n}^{{\sigma_{i}^n}}
 	\mob(x)\prf\big( \lambda_N(t,x) \big) dt.
\end{align}
In~\eqref{eq:ent:rffinf}, for each fixed $ i\in\{1,\ldots,2^n\} $ and $ x\in\Z $,
view the corresponding expression as an average of $ \prf\big( \lambda_N(t,x) \big) $
over the measure $ \Ex_{\QN} \int_{\sigma_{i-1}^n}^{{\sigma_{i}^n}}(\ \Cdot\ )dt $,
with total mass $ A_{i,x} := \Ex_{\QN}\int_{\sigma^n_{i-1}}^{\sigma^n_i} \mob(\h(Nt),x) dt $.
Using the convexity of $ \lambda\mapsto\prf(\lambda) $,
followed by applying the identity~\eqref{eq:QN:hN} for $ (t_1,t_2)=(\sigma^n_{i-1},\sigma^n_i) $,
we have
\begin{align}
	\notag
	\Ex_{\QN} \int_{\sigma_{i-1}^n}^{{\sigma_{i}^n}}
 	\mob(x)\prf\big( \lambda_N(t,x) \big) dt
 	&\geq
	A_{i,x} \prf\Big( \frac{1}{A_{i,x}} \Ex_{\QN} \int_{\sigma_{i-1}^n}^{{\sigma_{i}^n}} \mob(x) \lambda_N(t,x) dt \Big)
  \\ 
  \notag
	&=
	A_{i,x} \prf\Big( \frac{1}{A_{i,x}} \Ex_{\QN}\big(\hN(t,\tfrac{x}{N})\big)\big|_{\sigma_{i-1}^n}^{{\sigma_{i}^n}} \Big)
 \\
	\label{eq:cnvx:rffinf}
	&\ge A_{i,x} \prff\Big( \frac{1}{A_{i,x}} 
 	\Ex_{\QN}\big(\hN(t,\tfrac{x}{N})\big)\big|_{\sigma_{i-1}^n}^{{\sigma_{i}^n}} \Big).
\end{align}
The mobility function $ \mob(x)=\mob(\f,x) $ is $ \{0,1\} $-valued,
so in particular $ 0\leq A_{i,x} \leq (\sigma^n_i-\sigma^n_{i-1}) = \frac{T}{2^n} $.
Using this and~\eqref{eq:nonincr:mob} (for $ \xi=A_{i,x} $) in~\eqref{eq:cnvx:rffinf},
and inserting the result back into~\eqref{eq:ent:rffinf},
we arrive at 
\begin{align}
	\label{eq:cnvx:rffinf:}
	\frac{1}{N^2} H(\QN|\PN^\h)
	\geq
 	\frac{1}{N}\sum_{x\in\Z} 	
	\frac{T}{2^n} 
	\prff\Big( \frac{1}{\sigma^n_i-\sigma^n_{i-1}} 
	\Ex_{\QN}\big(\hN(t,\tfrac{x}{N})\big)\big|_{\sigma_{i-1}^n}^{{\sigma_{i}^n}} \Big).
\end{align}
Next, with $ \QN $ being the conditioned law as in~\eqref{eq:QN},
we have that
\begin{align}
	\label{eq:QN:cnvg}
	|\Ex_{\QN}(\hN(t,\tfrac{x}{N}))-g(t,\tfrac{x}{N})| \leq a_N ,\quad \forall (t,\tfrac{x}{N})\in[0,T]\times[-r_N,r_N].
\end{align}
In particular,
$	
	\Ex_{\QN}(\hN(t,\tfrac{x}{N}))|_{\sigma_{i-1}^n}^{{\sigma_{i}^n}} 
	\geq
	g(\sigma^n_{i},\tfrac{x}{N})-g(\sigma^n_{i-1},\tfrac{x}{N})-2a_N.
$
In~\eqref{eq:cnvx:rffinf:},
using this, together with the fact that $ \lambda\mapsto\prff(\lambda) $ is nondecreasing,
we further obtain
\begin{align}
	\label{eq:upb:inf:3}
	\frac{1}{N^2} H(\QN|\PN^\h)
	&\geq
	\frac{1}{N}\sum_{\frac{x}{N}\in[-r_N\wedge r,r_N\wedge r]} 
	\sum_{i=1}^{2^{n}} {\frac{T}{2^{n}}}
	\prff\Big(\Big( \frac{g(\sigma^n_{i},\frac{x}{N})-g(\sigma^n_{i-1},\frac{x}{N})-2a_N}{\sigma^n_{i}-\sigma^n_{i-1}} \Big)_+\Big).
\end{align}
Now, fix $ r<\infty $, and let $ N\to\infty $ in~\eqref{eq:upb:inf:3}.
Under this limit $ a_N\downarrow 0 $ and $ r_N \uparrow \infty $, so
\begin{align}
	\label{eq:upb:inf:4}
	\liminf_{N\to\infty} \frac{1}{N^2} H(\QN|\PN^\h)
	&\geq
	\int_{-r}^{r}
	\sum_{i=1}^{2^{n}}{\frac{T}{2^{n}}}
 	\prff\Big( \frac{g(\sigma^n_{i},\xi)-g(\sigma^n_{i-1},\xi)}{\sigma^n_{i}-\sigma^n_{i-1}} \Big) d\xi.
\end{align}
Recall the expression of $ \rff_n(h,\xi)  $ from~\eqref{eq:rffn}.
Upon letting $ r\to\infty $, the r.h.s.\ of~\eqref{eq:upb:inf:4} gives $ \int_\R \rff_n(g,\xi) d\xi $.
Further taking the supremum over $ n $ thus gives the desired result:
\begin{align*}
	\liminf_{N\to\infty} \frac{1}{N^2} H(\QN|\PN^\h) 
	\geq 
	\sup_{n} \int_{\R}\rff_n(g,\xi)
	=
	\rff(g)
	=\infty.
\end{align*}

\medskip
\noindent\textbf{The case $ \rff(g)<\infty $.}
~
Without lost of generality, we assume $ g(0)=\hIC $.
Otherwise, if $ g(0,\xi) \neq \hIC(\xi) $, for some $ \xi\in\R $,
with $ \QN $ being the conditioned law as in~\eqref{eq:QN} and $ a_N\downarrow 0 $,
by the assumption~\eqref{eq:ic:cnvg},
we necessarily have $ \QN \not\ll \PN $, for all large enough of $ N $.
Hence $ \liminf_{N\to\infty} \frac{1}{N^2} H(\QN|\PN^\h) = \infty $.

Under the assumption $ \rff(g)<\infty $, by Lemma~\ref{lem:rf<inf} we have $ g\in\Spd $.
This together with $ g(0) = \hIC $ implies 
$ \rfup (g) = \int_0^T \int_{\R} \lrfupa(g_t,g_\xi)dtd\xi $.
Recall from~\eqref{eq:Rprtn} the partition $ R_\ell(r) $ that consists of rectangles.
The first step is to localize the function $ \rfup(g) $ 
and relative entropy $ \frac{1}{N^2} H(\QN|\PN^\h) $ onto each rectangle $ \square\in R_\ell(r) $.
To this end,
recalling the definition of $ \lrfupa(\kappa,\rho) $ from~\eqref{eq:lrfeup}.
and fixing $ \e>0 $,
we apply Lemma~\ref{lem:local:apprx} for $ h=g $,
to obtain $ r,\ell<\infty $ and $ a>0 $ such that
\begin{align}
	\label{eq:rect::}
	\sum_{\square\in R_\ell(r)} |\square|\,\lrfupa(\kappa_\square,\rho_\square) \geq \rfup(g)\wedge\e^{-1} - \e,
	\quad
	(\kappa_\square,\rho_\square)
	:=
	\big({\textstyle\fint_{\square}} g_tdtd\xi \,,\, {\textstyle\fint_{\square}} g_\xi dtd\xi\big).
\end{align}
As for the relative entropy, in~\eqref{eq:QN:ent},
we drop those terms corresponding to $ \frac{x}{N}\notin [-r,r] $, and write
\begin{align}
	\label{eq:QN:ent:}
	\frac{1}{N^2} H(\QN|\PN^\h)
	\geq
	\Ex_{\QN}\Big( \frac{1}{N}\int_0^T \sum_{\frac{x}{N}\in[-r,r]}   \mob(x) \prf(\lambda_N(t,x)) dt \Big).
\end{align}
Then, decompose the r.h.s.\ of~\eqref{eq:QN:ent:} as
\begin{align}
	\frac{1}{N^2} H(\QN|\PN^\h)
	\geq
	\sum_{\square\in R_\ell(r)} H_N(\square),
	\quad
	\label{eq:HN}
	H_N(\square)
	:=
	\Ex_{\QN}
	\Big(
		\frac{1}{N}	
		{\int \sum} \ind_{\{(t,\frac{x}{N})\in\square\}}
		\mob(x)\prf(\lambda_N(t,x)) dt 
	\Big).
\end{align}

In view of~\eqref{eq:rect::} and~\eqref{eq:HN},
the next step is to show that $ H_N(\square) $ 
approximately bound~$ |\square|\lrfupa(\kappa_\square,\rho_\square) $ from above, for each $ \square\in R_\ell(r) $.
Recall from~\eqref{eq:fluxeup} the definition of $ \mobbup_a $.
Fix $ \square\in R_\ell(r) $, and, for each $ (t,\frac{x}{N})\in \square $,
use the convexity of $ \lambda\mapsto \prff(\lambda) $ to write
\begin{align}
	\label{eq:prff:cnvx:}
	\prff(\lambda_N(t,x))
	\geq
	\prff\Big( \frac{\kappa_\square}{\mobbup_a(\rho_\square)} \Big)
	-
	\prff'\Big( \frac{\kappa_\square}{\mobbup_a(\rho_\square)} \Big)
	\Big(\frac{\kappa_\square}{\mobbup_a(\rho_\square)}-\lambda_N(t,x) \Big).
\end{align}
Set
$ A_N(\square) := \Ex_{\QN}( \frac{1}{N}\int\sum_{(t,\frac{x}{N})\in\square} \mob(x) dt) $
and
$
	B_N(\square) :=
	\Ex_{\QN} ( \frac{1}{N}\int\sum_{(t,\frac{x}{N})\in\square}  \mob(x) \lambda_N(t,x) dt ).
$
In~\eqref{eq:prff:cnvx:}, multiply both sides by $ \mob(x) $,
and apply $ \Ex_{\QN} ( \frac{1}{N}\int\sum_{(t,\frac{x}{N})\in\square}(\ \Cdot\ )dt ) $ to the result.
Using that $\prf \ge \prff$, this gives
\begin{align}
	\label{eq:HN:bd}
	H_N(\square)
	\geq
	\prff\Big( \frac{\kappa_\square}{\mobbup_a(\rho_\square)} \Big) A_N(\square)
	-
	\prff'\Big( \frac{\kappa_\square}{\mobbup_a(\rho_\square)} \Big)
	\frac{\kappa_\square}{\mobbup_a(\rho_\square)}
	A_N(\square)
	+
	\prff'\Big( \frac{\kappa_\square}{\mobbup_a(\rho_\square)} \Big) B_N(\square).
\end{align}
Further, parametrizing the rectangle as
$ \square = [\und t_\square, \bar t_\square]\times[\xi^-_\square, \xi^+_\square] $,
using \eqref{eq:QN:hN} for $ (t_1,t_2) = (\und t_\square, \bar t_\square) $,
we have
\begin{align}
	\label{eq:BNsq}
	B_N(\square) 
	&=
	\Ex_{\QN} \Big( 
		\frac{1}{N}\sum_{\frac{x}{N}\in[\xi^-_\square, \xi^+_\square]} 
		\int_{{\und{t}_\square}}^{{\bar{t}_\square}}  \mob(x) \lambda_N(t,x) dt 
	\Big)
	=
	\Ex_{\QN} \Big( 
		\frac{1}{N}\sum_{\frac{x}{N}\in[\xi^-_\square, \xi^+_\square]} 
		\hN(t,\tfrac{x}{N})|_{t={\und{t}_\square}}^{t={\bar{t}_\square}} 
	\Big).
\end{align}
Letting $ N\to\infty $ in~\eqref{eq:BNsq},
using~\eqref{eq:QN:cnvg} on the r.h.s., we obtain
\begin{align}
	\label{eq:BNsq:cnvg}
	\liminf_{N\to\infty} B_N(\square) 
	\geq
	\int_{\xi^-_\square}^{\xi^+_\square} g(t,\xi)
	\Big|_{t={\und{t}_\square}}^{t={\bar{t}_\square}}  d\xi
	=
	|\square| \, \kappa_\square.
\end{align}
Combining this with~\eqref{eq:HN:bd} gives
\begin{align}
	\label{eq:HN:bd:}
	H_N(\square)
	\geq
	\big(\prff(\lambda)-\lambda\prff'(\lambda)\big)|_{\lambda=\frac{\kappa_\square}{\mobbup_a(\rho_\square)}}
	A_N(\square)
	+
	\prff'(\tfrac{\kappa_\square}{\mobbup_a(\rho_\square)}) |\square| \, \kappa_\square
	+
	\e_N(\square),
\end{align}
for some remainder term such that $\lim_{N} |\e_N(\square)| = 0 $.
Adding and subtracting the expression
$ 
	|\square|\lrfupa(\kappa_\square,\rho_\square) 
	= |\square| \mobb_a(\rho_\square) \prff(\frac{\kappa_\square}{\mobbup_a(\rho_\square)}) 
$ 
on the r.h.s.\ of~\eqref{eq:HN:bd:}, we arrive at
\begin{align}
	\label{eq:HN:bd:1}
	H_N(\square) 
	\geq
	|\square|\lrfupa(\kappa_\square,\rho_\square)
	+
		\big(\lambda\prff'(\lambda)-\prff(\lambda)\big)|_{\lambda=\frac{\kappa_\square}{\mobbup_a(\rho_\square)}}
		\Big(|\square|\,\mobbup_a(\rho_\square)-A_N(\square) \Big)
	+ \e_N(\square).
\end{align}

The expression $ \lambda\prff'(\lambda)-\prff(\lambda) = (\lambda-1)_+ $ 
in~\eqref{eq:HN:bd:1} nonnegative.
Furthermore, with $ A_N(\square) $ defined as in the preceding and with $ \mob(t,x) := \eta(t,x)(1-\eta(t,x)) $, parameterizing
 $ \square := [\und t_\square, \bar t_\square] \times [\xi^-_\square,\xi^+_\square] $,
we have
\begin{align}
	\label{eq:nts:0}
	A_N(\square) 
	\leq 
	 \frac{1}{N} \Ex_{\QN} \int_{\und t_\square}^{\bar t_\square} dt
  \Big(\sum_{\frac{x}{N} \in\square} \eta(Nt,x) \wedge \sum_{\frac{x}{N}\in\square} (1-\eta(Nt,x)) \Big).
\end{align}
Since
\begin{align*}
	\frac{1}{N}\sum_{\frac{x}{N}\in\square} \eta(Nt,x) &=
 |\xi^+_\square-\xi^-_\square| \big( \hN(t,\xi^+_\square) - \hN(t,\xi^-_\square) \big),
\\
	\frac{1}{N}\sum_{\frac{x}{N}\in\square} (1-\eta(Nt,x)) &=
        |\xi^+_\square-\xi^-_\square| \big(  1-\hN(t,\xi^+_\square) - \hN(t,\xi^-_\square) \big).
\end{align*}
It follows that
\begin{align*}
  	A_N(\square) 
	\leq \Ex_{\QN} \int_{\und t_\square}^{\bar t_\square} dt 
          |\xi^+_\square-\xi^-_\square| \mobbup ( \hN(t,\xi^+_\square) - \hN(t,\xi^-_\square) )
\end{align*}
With $ \QN $ being the conditioned law as in~\eqref{eq:QN}, we have
\begin{align}
	\label{eq:nts:1}
	\limsup_{N\to\infty}
	A_N(\square) 
	\leq 
	\int_{\und t_\square}^{\bar t_\square} |\xi^+_\square-\xi^-_\square|
       \mobbup ( g(t,\xi^+_\square) - g(t,\xi^-_\square) ) dt
	= |\square|\mobbup( \rho_\square),
\end{align}

Combining this with~\eqref{eq:HN:bd:1} and the fact that $ \lambda\prff'(\lambda)-\prff(\lambda) \geq 0 $,
we arrive at
\begin{align*}
	\liminf_{N\to\infty}
	H_N(\square) 
	\geq
	|\square|\lrfupa(\kappa_\square,\rho_\square).
\end{align*}

This gives the desired bound on each rectangle $ \square\in R_\ell(r) $.
Referring back to \eqref{eq:rect::} and~\eqref{eq:HN}, we now have
\begin{align*}
	\liminf_{N\to\infty} H_N(\QN|\PN^\h) 
	\geq
	\rfup(g)\wedge \e^{-1} - \e.
\end{align*}
The proof is completed upon letting $ \e\downarrow 0 $.

\section{Lower Bound: Inhomogeneous \ac{TASEP}}
\label{sect:ITASEP}
The remaining of this article, Section~\ref{sect:ITASEP}--\ref{sect:speed},
are devoted to proving Proposition~\ref{prop:ent:up}.
To this end, 
hereafter we fix $ \e_*>0 , r_*<\infty $, $ \tau,b $ such that $ \frac{T}{\tau},\frac{r_*}{b}\in\N $
as in Proposition~\ref{prop:ent:up}.
To simplify notations, we write $ \Sigma = \Sigma(\tau,b) $ for the triangulation.
Fix further a $ \Sp $-valued, $ \Sigma $-piecewise linear function $ g $
that satisfies~\eqref{eq:nondg}--\eqref{eq:nondg:},
write $ \gIC:=g(0) $,
and fix a \ac{TASEP} height process $ \gN $ with initial condition that satisfies~\eqref{eq:gNic},
as in Proposition~\ref{prop:ent:up}.
Let $ r^*,\maxlam $ be given as in~\eqref{eq:r*}--\eqref{eq:maxlam}.
We write $ (\linkap,\linrho) := (g_t,g_\xi)|_{\triangle^\circ} $
for the constant derivatives of $ g $ on a given $ \triangle\in\Sigma $,
and let $ \linlam := \frac{\linkap}{\linrho(1-\linrho)} $.
With $ g $ satisfying the properties~\eqref{eq:nondg}--\eqref{eq:nondg:},
we have
\begin{align}
	\label{eq:nondg:lin}
	0< \inf_{\triangle\in\Sigma} \linkap \leq \sup_{\triangle\in\Sigma} \linkap < \infty,
	\quad
	0< \inf_{\triangle\in\Sigma} \linrho \leq \sup_{\triangle\in\Sigma} \linrho < 1,
	\quad
	0 <\inf_{\triangle\in\Sigma} \linlam \leq \sup_{\triangle\in\Sigma} \linlam \leq \maxlam <\infty.
\end{align}
Let $ \Sigma_* := \{ \triangle\in\Sigma : \triangle\subset[0,T]\times[-r_*,r_*] \} $
denote the restriction of the triangulation onto $ [-r_*,r_*] $,
and similarly $ \Sigma^* := \{ \triangle\in\Sigma : \triangle\subset[0,T]\times[-r^*,r^*] \} $.
%

Proving~Proposition~\ref{prop:ent:up} amounts to \emph{constructing} probability laws $ \{\QN\}_N $
that satisfies~\eqref{eq:lwb:apprx}--\eqref{eq:lwb:ent}.
We will achieve this using \textbf{inhomogeneous \ac{TASEP}}, defined as follows.
We say $ \speed: [0,T)\times \R\to (0,\maxlam] $ is a \textbf{speed function}
if $ \speed $ is Borel measurable, positive, and bonded by $ \maxlam $ from above.
We say $ \speed $ is a \textbf{simple speed function}
if it a speed function that takes the following form
\begin{align}
	\label{eq:simple}
	\speed: &[0,T)\times\R\to (0,\infty),
	\quad
	\speed(t,\xi) := \sum_{i=1}^n\ind_{[t_{i-1},t_i)}(t) \speed_i(\xi),
	\quad
	0=t_0<t_2<\ldots < t_n=T,
\end{align}
where each $\speed_i:\R\to (0,\infty) $
is lower semi-continuous, piecewise constant, with finitely many discontinuities, and $ \lim_{|\xi|\to\infty} \speed_i(\xi) = 1 $.
Now, given a \emph{simple} speed function $ \speed $, 
we define the associated inhomogeneous \ac{TASEP} similarly to the \ac{TASEP},
starting from the initial condition $ \gic_N $ (as fixed in the preceding),
but, instead of having unit-rate Poisson clocks at each $ x\in\Z $, we let the rate be $ \speed(\frac{t}{N},\frac{x}{N}) $.
We do not define the value of $ \speed $ at $ t=T $ for convenience of notations,
and these values $ \speed(T,\xi) $ do not pertain to the dynamics of the inhomogeneous \ac{TASEP}, 
define for $ t\in[0,NT] $.
We write $ \QN^\speed $ for the law of the inhomogeneous \ac{TASEP} with a simple speed function $ \speed $.

For a time-homogeneous (i.e., $ \speed(t,\xi)=\speed(\xi)$, $ \forall t\in[0,T) $) simple speed function,
the corresponding inhomogeneous \ac{TASEP} sits within the scope studied in \cite{georgiou10}.
For simple speed functions of the form \eqref{eq:simple} considered here,
the associated inhomogeneous \ac{TASEP} is constructed inductively in time from the time homogeneous process.
A key tool from~\cite{georgiou10} in our proof is the hydrodynamic limit.
To state this result precisely,
For given $ f\in\Splip $, and a speed function $ \speed $,
we define the Hopf--Lax function $ \hl{\speed}{f} $ via the following variational formula: 
\begin{align}
	\label{eq:hl:frml}
	\hl{\speed}{f}: [0,T]\times\R \to \R,
	\quad
	&
	\hl{\speed}{f}(t,\xi) 
	:= 
	\inf_{w\in W(t,\xi)} \big\{ \HLf_{0,t}(w;\speed) + f(w(0)) \big\},
\end{align}
where $ W(t,\xi) $ is the set of piecewise $ C^1 $ paths $ w:[0,t]\to\R $ connected to $ (t,\xi) $, i.e.,
\begin{align}
	\label{eq:W}
	&
	\quad\hphantom{\text{ where }}
	W(t,\xi) := \{ w:[0,t]\to \R : w \text{ piecewise }C^1, w(t)=\xi \},
\end{align}
and $ \HLf_{t_1,t_2}(w;\speed) $ is a functional on $ (w,\speed) $, defined as
\begin{align}
	\label{eq:HL}
	\HLf_{t_1,t_2}(w;\speed)
	:=
	\int_{t_1}^{t_2} \speed(s,w(s)) &\hlf\Big(\frac{w'(s)}{\speed(s,w(s))}\Big) ds,
\\
	\label{eq:hl}
	&\hlf(\xi) := \left\{\begin{array}{l@{,}l}
		0					&\text{ for } \xi \leq -1,
	\\
		\frac14 (\xi+1)^2	&\text{ for } \xi\in(-1,1),
	\\
		\xi					&\text{ for } \xi \geq 1.
	\end{array}\right.
\end{align}
As we show in Lemma~\ref{lem:hl}\ref{enu:hl:sp} in the following,
the variational formula \eqref{eq:hl:frml} does define a $ \Sp\cap C([0,T],\Splip) $-valued height function.
Such a height function can be viewed as the viscosity solution of the inhomogeneous Burgers equation:
\begin{align*}
	h_t(t,\xi) = \speed h_\xi(1-h_\xi), \quad h(0) = f.
\end{align*}
We will, however, operate entirely with the variational formula~\eqref{eq:HL} and avoid referencing 
to the PDE.

The following is the hydrodynamic result from~\cite{georgiou10}.
\begin{proposition}[\cite{georgiou10}]
\label{prop:inhomo}
Fix a time-homogeneous, simple speed function $ \speed $.
For each fixed $ (t,\xi)\in[0,T]\times\R $,
the random variable $ \gN(t,\xi) $ converges to $ \hl{\speed}{\gIC}(t,\xi) $,  $ \QN^\speed $-in probability.
\end{proposition}

Proposition~\ref{prop:inhomo} is readily generalized to the time-inhomogeneous setting considered here.
To see this, we first prepare a simple lemma that leverages pointwise convergence into uniform convergence.
\begin{lemma}
\label{lem:pt:uniform}
Let $ \{h_N\}\subset \Sp $ be a sequence that converges to $ h\in C([0,T],\Splip) $ pointwisely,
i.e., $ h_N(t,\xi)\to h(t,\xi) $, $ \forall (t,\xi)\in[0,T]\times\R $.
Then, in fact, $ \sup\limits_{t\in[0,T]}\dist(h_N(t),h) \to 0 $ holds.
\end{lemma}
\begin{proof}
Fix arbitrary $ \e>0 $ and $ r<\infty $,
and consider the partition $ R_\ell(r) $ as in~\eqref{eq:Rprtn}.
As $ h $ is continuous, there exists large enough $ \ell<\infty $ such that,
on each of the rectangle $ \square\in R_\ell(r) $,
\begin{align}
	\label{eq:gu:ocs}
	|h(t,\xi)-h(s,\zeta)| \leq \e,
	\quad
	\forall (t,\xi), (s,\zeta) \in \square.
\end{align}
Fix a rectangle $ \square\in R_\ell(r) $ and
parametrize it as $ [\und t_\square,\bar t_\square]\times[\xi^-_\square,\xi^+_\square] $.
With $ h_N $ being nondecreasing in $ t $ and $ \xi $, for each $ t,\xi\in\square $, we have
\begin{subequations}
\label{eq:gu:oscc}
\begin{align}
	&
	h(t,\xi)-h_N(t,\xi) 
	\leq 
	h(t,\xi) - h_N(\und t_\square,\xi^-_\square)
	=
	\big( h(t,\xi)-h(\und t_\square,\xi^-_\square)\big) + h(\und t_\square,\xi^-_\square) - h_N(\und t_\square,\xi^-_\square),
\\	
	&
	h(t,\xi)-h_N(t,\xi) 
	\geq 
	h(t,\xi) - h_N(\bar t_\square,\xi^+_\square)
	=
	\big( h(t,\xi) - h(\bar t_\square,\bar\xi_\square) \big) + h(\bar t_\square,\bar\xi_\square) - h_N(\bar t_\square,\xi^+_\square).
\end{align}
\end{subequations}
Let $ \calV $ denote the set of all vertices of the rectangles in $ R_\ell(r) $.
Using~\eqref{eq:gu:ocs} in~\eqref{eq:gu:oscc} gives,
\begin{align*}
	h(t,\xi)-h_N(t,\xi) 
	\leq 
	\e + \sup_{\calV} (h - h_N),
	\quad
	h(t,\xi)-h_N(t,\xi) 
	\geq 
	-\e + \inf_{\calV} (h - h_N).
\end{align*}
Equivalently, $ \sup_{[0,T]\times[-r,r]} |h_N-h| \leq \max_{ \calV } |h_N-h| + \e $.
As $ \calV $ is a fixed, finite set, letting $ N\to\infty $ 
gives $ \limsup_{N\to\infty} \sup_{[0,T]\times[-r,r]} |h_N-h| \leq \e $.
With $ \e>0 $ and $ r<\infty $ being arbitrary,
this concludes the desired result~$ \sup_{t\in[0,T]}\dist(h_N(t),h) \to 0 $.
\end{proof}

The following Corollary generalizes Proposition~\ref{prop:inhomo} to the time-inhomogeneous setting considered here.
\begin{corollary}
\label{cor:inhomo}
For any given simple speed function $ \speed $,
\begin{align*}
	\sup_{t\in[0,T]} \dist( \gN(t), \hl{\speed}{\gIC}(t) ) \longrightarrow 0,
	\quad
	\QN^\speed\text{-in probability.}
\end{align*}
\end{corollary}
\begin{proof}
Let $ 0:=t_0<t_1<\ldots<t_n=T $ denote the discontinuities of $ \speed $.
Combining Proposition~\ref{prop:inhomo} and Lemma~\ref{lem:pt:uniform} 
gives $ \sup_{t\in[t_0,t_1]}\dist(\gN(t),\hl{\speed}{f}(t)) \to 0 $, $ \QN^\speed $-in probability.
In particular $ \dist(\gN(t_1),h(t_1)) \to 0 $, $ \QN^\speed $-in probability.
This allows us to progress onto $ [t_1,t_2] $.
The proof is completed by inductively applying Proposition~\ref{prop:inhomo} and Lemma~\ref{lem:pt:uniform}
for $ [t_1,t_2] $, $ \ldots $,  $ [t_{n-1},t_n] $.
\end{proof}

In addition to the hydrodynamic result Corollary~\ref{cor:inhomo},
to the end of proving Proposition~\ref{prop:ent:up},
we also need a formula for the Radon--Nikodym derivative.
Using the Feynman--Kac formula, it is standard to show that
\begin{align}
	\label{eq:RNdrv}
	\frac{d\QN^\speed}{d\PN^\g}
	=
	\exp\Big( \sum_{x\in\Z} \int_0^{NT} 
		\Big(
			\log\speed\big(Nt,\tfrac{x}{N}\big) d\h(t,x) 
			- \mob(\h(t),x)\big(\speed\big(Nt,\tfrac{x}{N}\big)-1\big) dt
		\Big)
	\Big).
\end{align}
In particular, with $ \prf(\xi):=\xi\log\xi-(\xi-1) $, 
taking $ \Ex_{\QN^{\speed}}(\Cdot) $ in~\eqref{eq:RNdrv} gives
\begin{align}
	\label{eq:ent}
	\frac{1}{N^2}H(\QN^\speed|\PN^\g)	
	=
	\frac{1}{N}
	\sum_{x\in\Z}
	\Ex_{\QN}\Big(  \int_0^{T} 
		\mob(\hN(t),x)\prf\Big(\speed\big(t,\tfrac{x}{N}\big)\Big) dt
	\Big).
\end{align}

Our strategy of proving Proposition~\ref{prop:ent:up}
is to construct a simple speed function $ \speed $, 
so that, $ \QN^{\speed} $ satisfies~\eqref{eq:lwb:apprx}--\eqref{eq:lwb:ent}.
In view of Corollary~\ref{cor:inhomo},
achieving~\eqref{eq:lwb:apprx} amounts to constructing $ \speed $
in such a way that $ \hl{\speed}{\gIC} $ well approximates $ \ling $.
To this end, it is more convenient to consider piecewise constant speed functions that are not necessarily simple.
In Section~\ref{sect:speedd},
we will first construct a speed function $ \Speedd $ that is not simple,
and in Section~\ref{sect:speed},
we obtain the desired simple speed function $ \Speed $ as an approximate of $ \Speedd $.
As the functions $ \Speedd $ and $ \Speed $
depend on the two auxiliary parameters $ \scll,\sclll $ (introduced in the sequel),
hereafter we write $ \Speedd=\Speedd_{\scll,\sclll} $ and $ \Speed=\Speed_{\scll,\sclll} $
to emphasize such dependence.

\section{Lower Bound: Construction of $ \Speedd_{\scll,\sclll} $}
\label{sect:speedd}

\subsection{Overview of the Construction}
\label{sect:speedd:ov}
To motivate the technical construction in the sequel,
in this subsection we give an overview. 
The discussion here is \emph{informal}, and does not constitute any part of the proof. 

Corollary~\ref{cor:inhomo} asserts that $ \gN $ converges to $ \hl{\speed}{\gIC} $ under $ \QN^{\speed} $.
In order to achieve~\eqref{eq:lwb:apprx},
it is desirable to construct construct $ \Speedd_{\scll,\sclll} $ 
so that $ \hl{\Speedd_{\scll,\sclll}}{\gIC} $ approximates $ \ling $ on $ [0,T]\times[-r^*,r^*] $, i.e.,
\begin{align*}
	\sup_{[0,T]\times[-r^*,r^*]} \big| \hl{\Speedd_{\scll,\sclll}}{\ling} - \ling \big| \approx 0.
\end{align*}

Indeed, it is well-known that the Burgers equation~\eqref{eq:burgers} is solved by characteristics,
which a linear trajectories of speed $ 1-2g_\xi $.
We generalize the idea of characteristic velocity to the inhomogeneous setting considered here,
and call $ \Speedd_{\scll,\sclll}(t,\xi)(1-2g_\xi(t,\xi)) $ the characteristic velocity at a given point $ (t,\xi) $.
Informally speaking, the Hopf--Lax function $ h=\hl{\Speedd_{\scll,\sclll}}{\gIC} $
corresponds to a solution of the inhomogeneous equation
$
	h_t = \Speedd_{\scll,\sclll} h_\xi(1-h_\xi)
$
with initial condition $ \gIC $.
As $ \ling $ is a $ \Sigma $-piecewise linear function,
a natural, \emph{preliminary proposal} is to set 
$ \Speedd_{\scll,\sclll}|_{\triangle^\circ} := \linlam $ 
on each triangle $ \triangle\in\Sigma $,
so that $ \ling $ solves the aforementioned inhomogeneous Burgers equation.
One then hopes that (after extending $ \Speedd_{\scll,\sclll} $ onto the edges of the triangulation $ \Sigma $ in a suitable way),
the resulting Hopf--Lax function $ \hl{\Speedd_{\scll,\sclll}}{\gIC} $ matches $ \ling $. 
This is false in general.
To see why, assume $ \hl{\Speedd_{\scll,\sclll}}{\gIC} = \ling $ were the case.
Then, on each $ \triangle\in\Sigma $, characteristic velocity is constant $ \linlam $.
%
Along vertical or diagonal edges of the triangulation $ \Sigma $,
characteristics may: merge, semi-merge, refract, semi-refract, or diverge,
as illustrated in Figure~\ref{fig:char:}.
While the first four scenarios are admissible,
the Hopf--Lax function $ \hl{\Speedd_{\scll,\sclll}}{\gIC} $ does \emph{not} permit diverging characteristics
as depicted in Figure~\ref{fig:div}.
\begin{figure}[h]
\begin{subfigure}{0.15\textwidth}
\includegraphics[width=\textwidth]{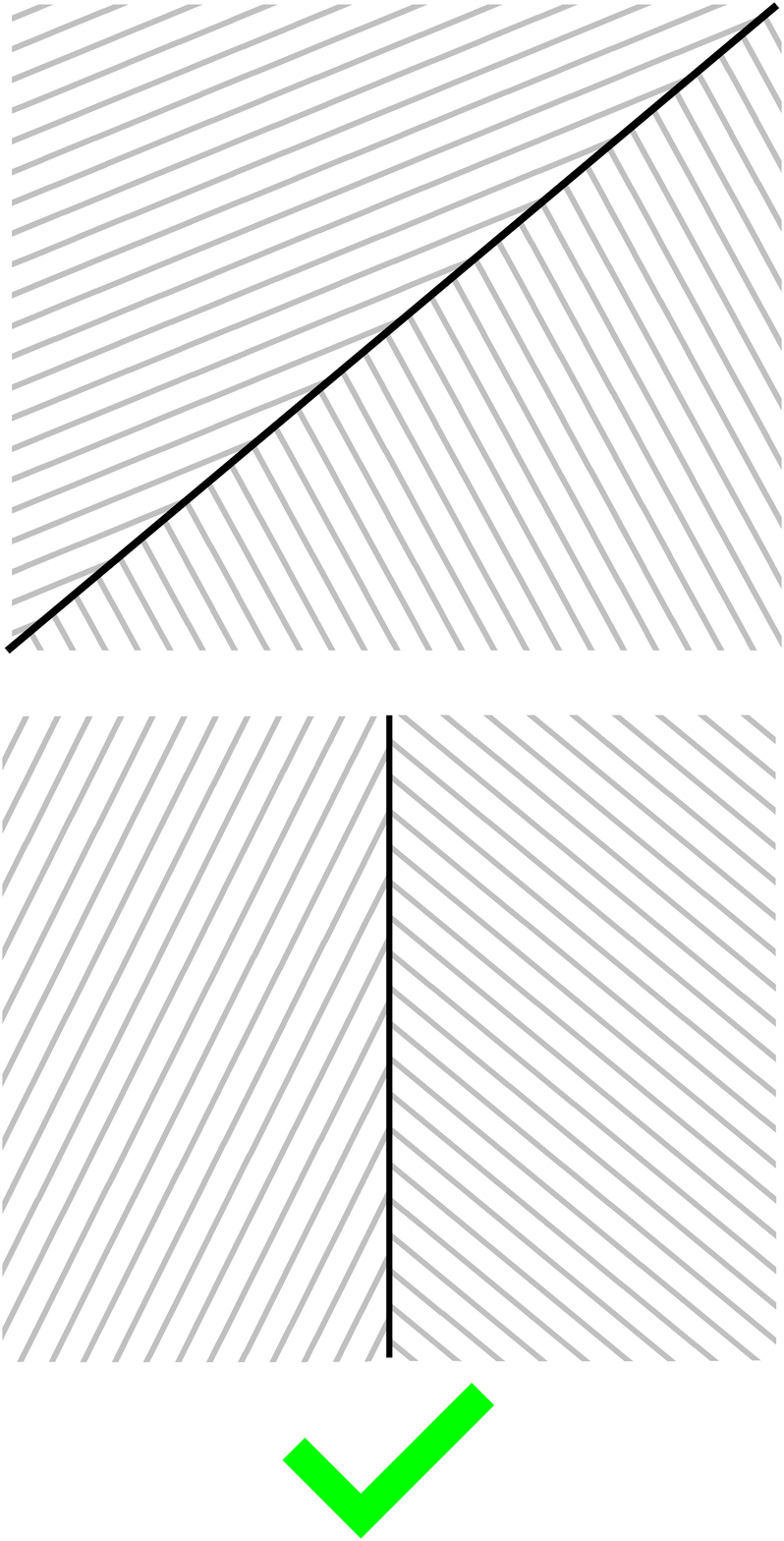}
\caption{\scriptsize{}Merging}
\end{subfigure}
\hfill
\begin{subfigure}{0.15\textwidth}
\includegraphics[width=\textwidth]{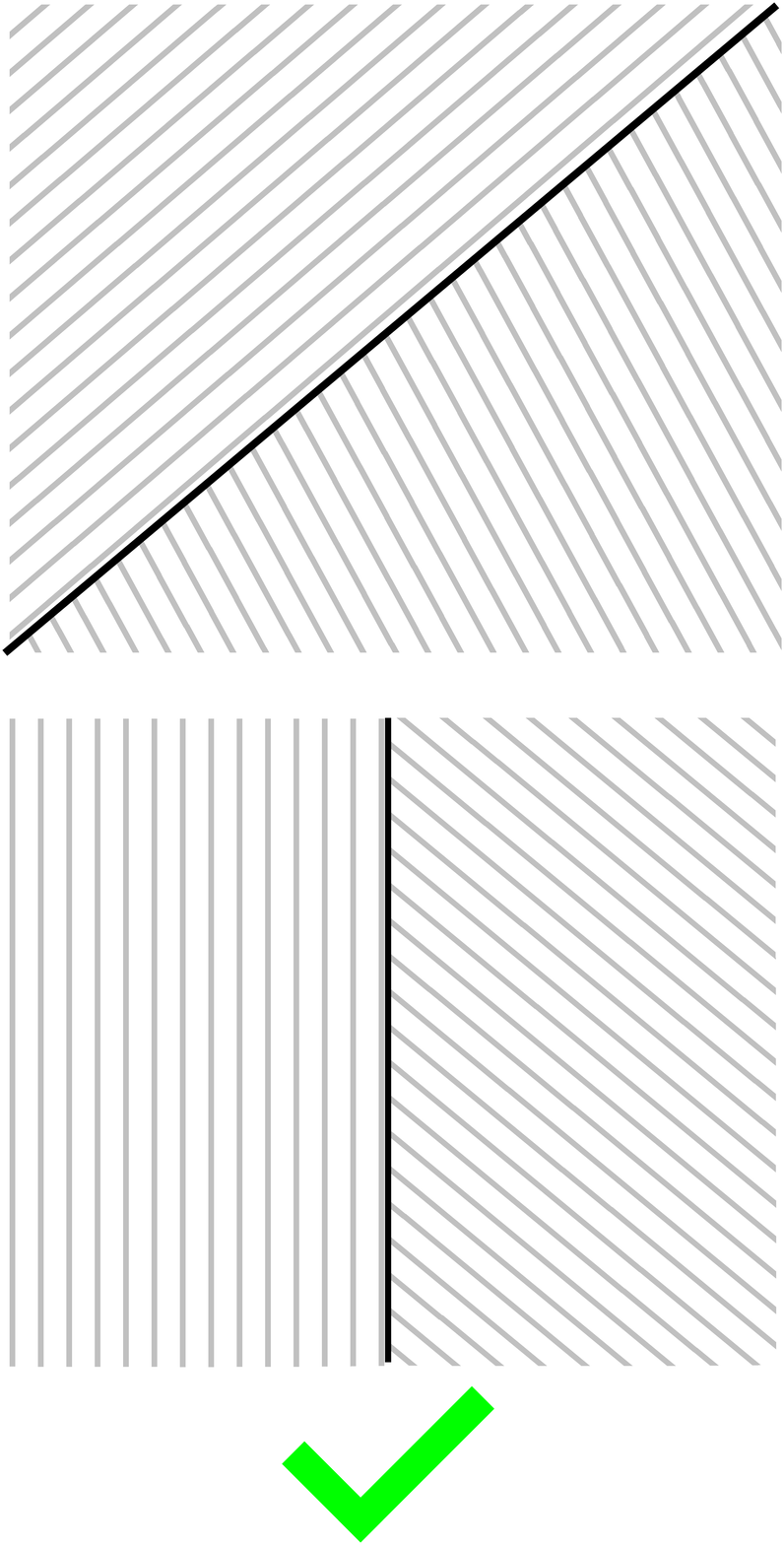}
\caption{\scriptsize{}Semi-merging}
\end{subfigure}
\hfill
\begin{subfigure}{0.15\textwidth}
\includegraphics[width=\textwidth]{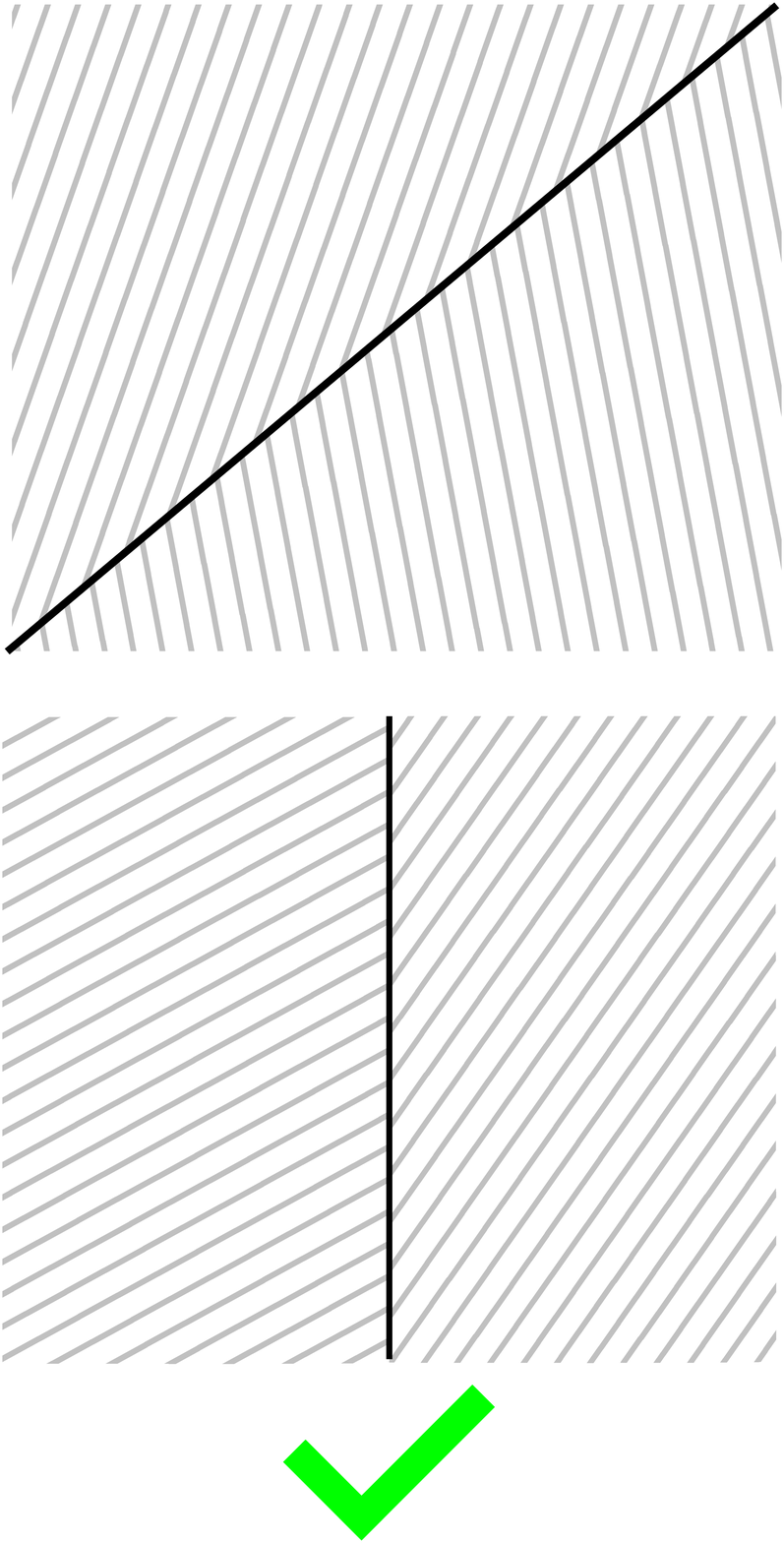}
\caption{\scriptsize{}Refracting}
\end{subfigure}
\hfill
\begin{subfigure}{0.15\textwidth}
\includegraphics[width=\textwidth]{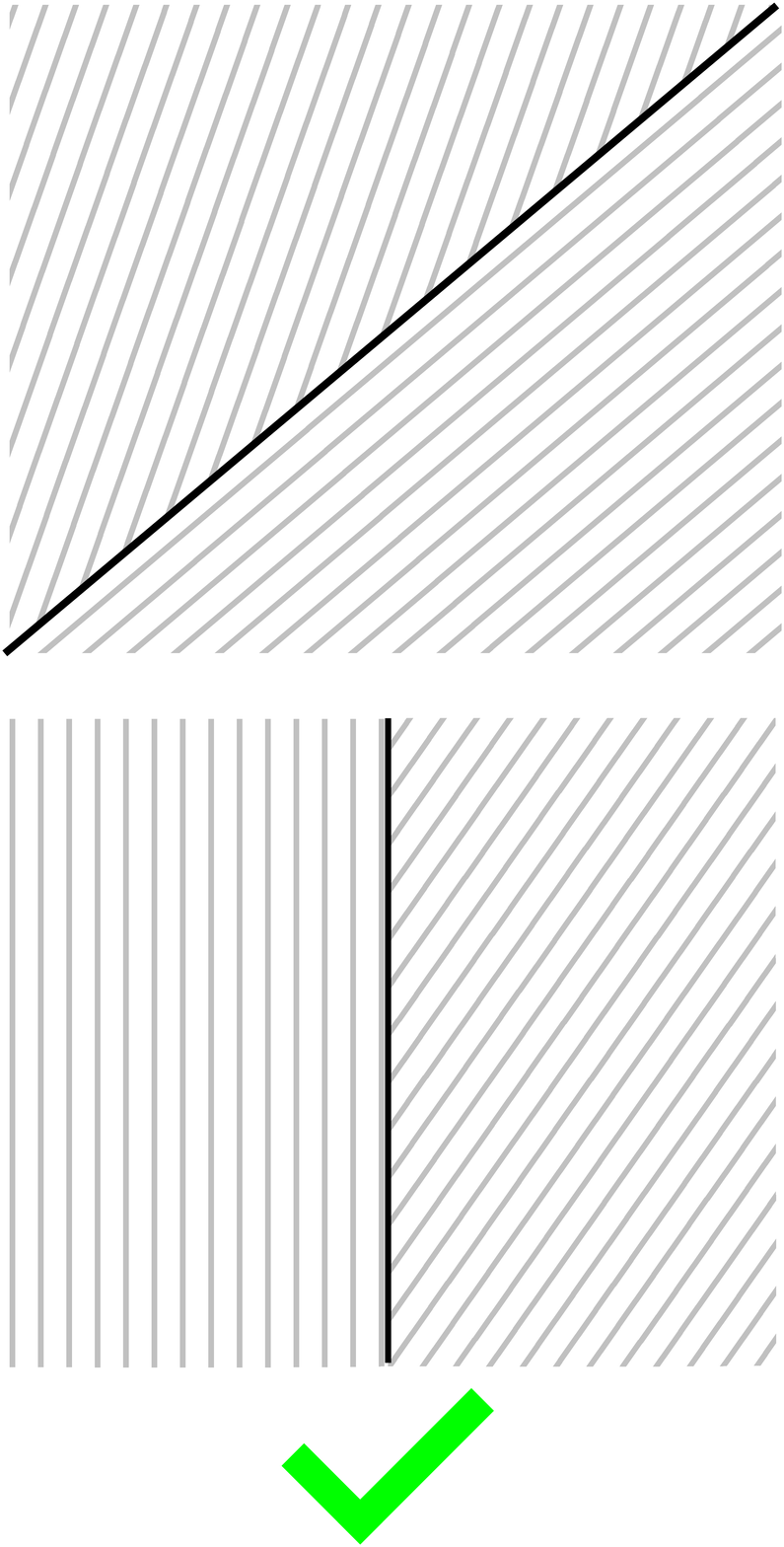}
\caption{\scriptsize{}Semi-refracting}
\end{subfigure}
\hfill
\begin{subfigure}{0.15\textwidth}
\includegraphics[width=\textwidth]{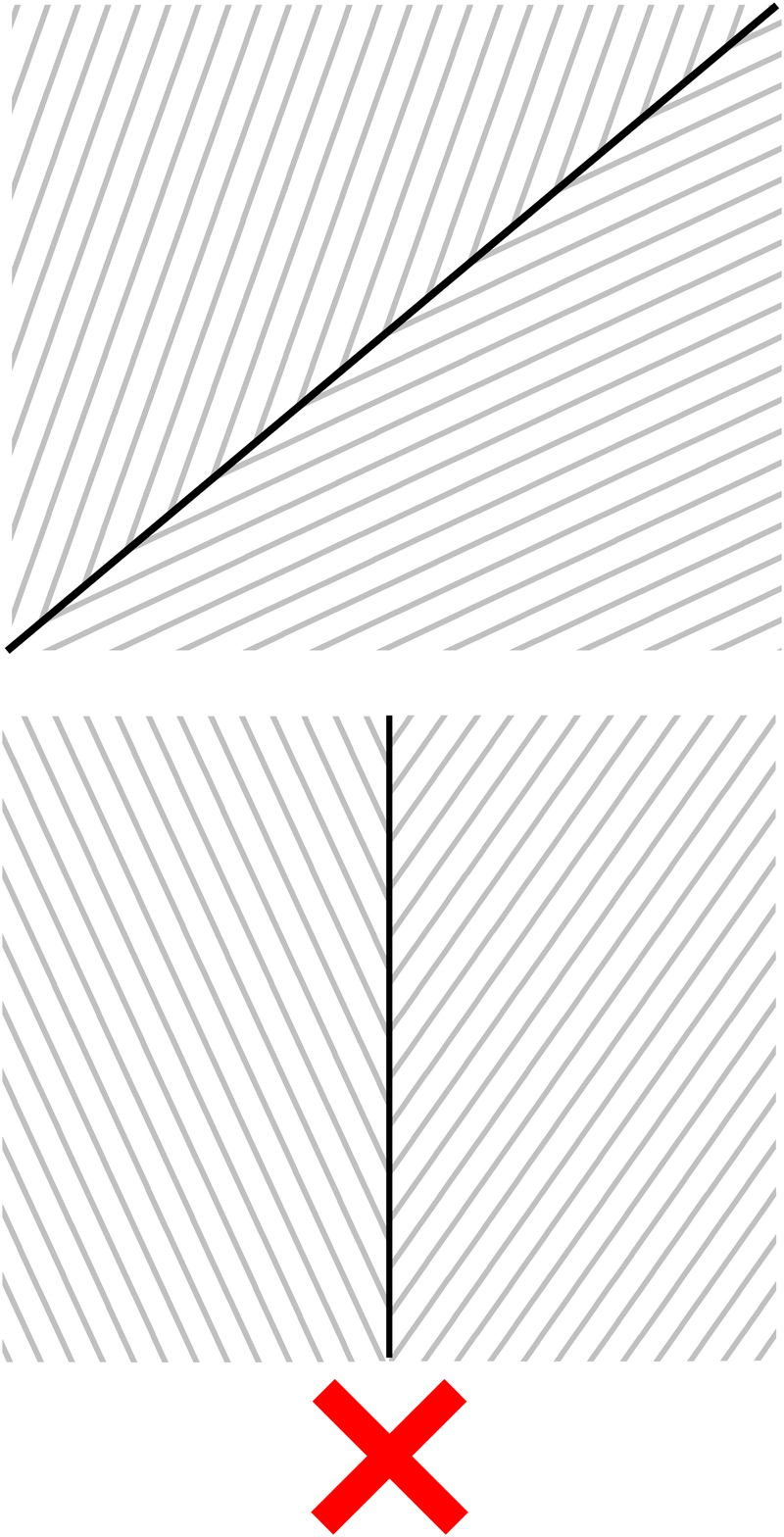}
\caption{\scriptsize{}Diverging}
\label{fig:div}
\end{subfigure}
\caption{Configurations at a vertical or diagonal edge}
\label{fig:char:}
\end{figure} 

We circumvent this problem by introducing buffer zones around vertical and diagonal edges.
These zones are thin stripes of width $ O(\frac{1}{\scll}) $.
If, the neighboring triangles of a given (vertical or diagonal) edge 
demand diverging characteristics as depicted in Figure~\ref{fig:div},
we tune $ \Speedd_{\scll,\sclll} $ on the buffer zone, in such a way that 
characteristics run \emph{parallel} to the edge on in the zone, as depicted in Figure~\ref{fig:buffering}.
This way, instead of diverging characteristics,
along the sides of the buffer zone we have semi-refracting characteristics.
As $ \scll\to\infty $, buffer zones become effectively invisible,
and the resulting $ \hl{\Speedd_{\scll,\sclll}}{\gIC} $ should well-approximate $ \ling $.
\begin{figure}[h]
\psfrag{W}{$ O(\frac{1}{\scll}) $}
\includegraphics[width=.5\textwidth]{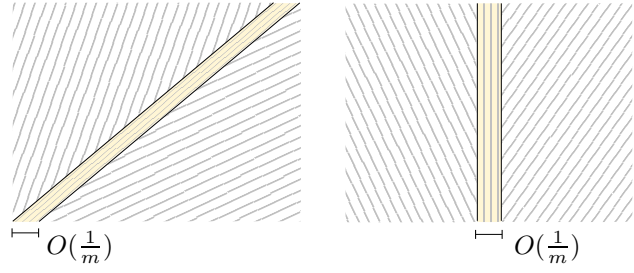}
\caption{Buffer zones (yellow) in action.}
\label{fig:buffering}
\end{figure}

The preceding construction achieves~$ \hl{\Speedd_{\scll,\sclll}}{\gIC}\approx\ling $,
but is not cost efficient in terms of entropy.
A few modifications are in place to improve the entropy cost.
Recall the definition of $ r^* $ and $ \maxlam $ from~\eqref{eq:r*}--\eqref{eq:maxlam}.
First, to avoid the entropy being infinite, we truncation $ \Speedd_{\scll,\sclll} $,
by setting it to unity outsides of $ [-r^*,r^*] $, i.e.,
$ \Speedd_{\scll,\sclll}|_{|\xi|>r^*} := 1. $
Refer to the formula~\eqref{eq:ent} for relative entropy:
the prescribed truncation ensures the cost from $ \{|\xi|>r^*\} $ is zero.
Further, such a truncation does not 
change the value of $ \hl{\Speedd_{\scll,\sclll}}{\gIC}(t,\xi) $ for $ |\xi|\leq r_* $.
To see why, recall that $r^* \ge r_* + T\maxlam$, and note that, with $ 1-2g_\xi\in[-1,1] $ and $ \Speedd_{\scll,\sclll}\leq\maxlam $,
characteristic velocity is always bounded by $ \maxlam $ in magnitude.
This being the case, the value of $ \Speedd_{\scll,\sclll} $ in $ \{|\xi|>r^*\} $
does not affect $ \hl{\Speedd_{\scll,\sclll}}{\gIC}|_{|\xi|\leq r_*} $,
because characteristics starting from $ \{|\xi|>r^*\} $ at $ t=0 $ cannot reach $ \{|\xi|\leq r_*\} $ within $ [0,T] $.

Next, recalling the the discussion in Section~\ref{sect:intro:heu},
we see that on those triangles $ \triangle $ with $ \linlam < 1 $,
having $ \Speedd_{\scll,\sclll} = \linlam $ is too cost ineffective.
Instead, we should perform the intermittent constriction as sketched in Section~\ref{sect:intro:heu}.
On each of the triangle $ \triangle $ with $ \linlam < 1 $,
we place thin vertical stripes of width $ O(\frac{1}{\sclll^2\scll}) $, every distance $ O(\frac{1}{\sclll\scll}) $ apart.
We then set $ \Speedd_{\scll,\sclll} = \linlam $ on those thin stripes,
and set $ \Speedd_{\scll,\sclll} =1 $ for the rest of the triangles.
As explained in Section~\ref{sect:intro:heu},
as $ \sclll\to\infty $, the prescribed construction
should produce approximately the desired linear function on $ \triangle $, at effectively zero cost.

This concludes our overview of the construction of $ \Speedd_{\scll,\sclll} $.
The precise construction is carried out in Section~\ref{sect:speedd:cnst} in the following,
and in Section~\ref{sect:speedd:hl} we verify that the resulting Hopf--Lax function 
$ \hl{\Speed_{\scll,\sclll}}{\gIC} $ converges to $ \ling $, under a limit procedure.
Even though the preceding heuristic discussion invokes inhomogeneous Burgers equation
as a motivation for constructing $ \Speedd_{\scll,\sclll} $,
our analysis in the following completely bypasses references to PDEs.
Instead, we work directly with the variational formula~\eqref{eq:hl:frml} of Hopf and Lax.
To prepare for this, in Section~\ref{sect:speedd:hlprpty}
we establish some elementary properties of the Hopf--Lax function.

\subsection{Properties of the Hopf--Lax function}
\label{sect:speedd:hlprpty}

Let us first setup the notations.
For a given set $ \calA\subset[0,T]\times\R $ and $ h\in\Sp\cap C([0,T],\Splip) $,
we define the localization $ \hll{\speed}{f}{\calA} $ of~\eqref{eq:hl:frml} onto $ \calA $ as follows:
\begin{align}
	\label{eq:hlfrml:loc}
	\hll{\speed}{h}{\calA}: \calA\to\R,
	\quad
	\hll{\speed}{h}{\calA}(t,\xi)
	:=
	\inf_{ w\in W_\calA(t,\xi) } \big\{ \HLf_{t_w,t}(w;\speed) + h(t_w,w(t_w)) \big\},
\end{align}
where $ W_\calA(t,\xi) $ denotes the set of piecewise $ C^1 $ paths 
that lie within $ \calA^\circ $
and connect $ (t,\xi) $ to the boundary $ \partial\calA := \bar{\calA}\setminus\calA^\circ $ of $ \calA $, i.e.,
\begin{align}
	\label{eq:W:local}
	W_\calA(t,\xi) := \{ w:[t_w,t]\to\R : \, & w \text{ piecewise } C^1, \, (s,w(s))|_{s\in(t_w,t)}\in\calA^\circ,
	\\	
	&
	\notag
	w(t)=\xi, \, (t_w,w(t_w))\in \partial\calA \}.
\end{align}
The expression~\eqref{eq:hlfrml:loc} depends on $ (\speed,h) $ 
only through $ (\speed|_{\calA^\circ},h|_{\partial\calA}) $,
and is hence referred to as the localization onto $ \calA $.
For the special case $ \calA := [s_0,T]\times\R $, $ s_0\in[0,T] $, 
sightly abusing notations, we write
\begin{align}
	\label{eq:hlfrmlt}
	\hll{\speed}{f}{s_0}: [t_0,T]\times\R\to\R,
	\quad
	\hll{\speed}{f}{s_0}(t,\xi)
	:=
	\hspace{-5pt}
	\inf_{ w\in W_{s_0}(t,\xi) } \big\{ \HLf_{s_0,t}(w;\speed) + f(w(s_0)) \big\},
\end{align}
where $ f\in\Splip $ and $ W_{s_0}(t,\xi):=\{ w:[s_0,T]\to\R : w $ piecewise $ C^1, w(t)=\xi \} $.

Recall that, by definition, each speed function $ \speed $ is bounded by $ \maxlam $.
We hence refer to $ \maxlam $ as the light speed, and let
\begin{align}
	\label{eq:cone}
	\calC(t_0,\xi_0) := \big\{ (t,\xi): t\in[0,t_0], |\xi-\xi_0| \leq \maxlam(t_0-t) \big\}
\end{align}
denote the light cone going backward in time from $ (t_0,\xi_0) $.

The following lemma contains the elementary properties of the Hopf--Lax function that will be used in the sequel.
\begin{lemma}
\label{lem:hl}
Let $ \speed,\speed_1,\speed_2 $ be speed functions, $ f\in\Sp $.
\begin{enumerate}[label=(\alph*),leftmargin=5ex]
\item \label{enu:hl:loc}
Let $ \calA\subset[0,T]\times\R $.
The Hopf--Lax function~\eqref{eq:hl:frml} localizes onto $ \calA $ as
\begin{align*}
	\hl{\speed}{f}\big|_\calA = \hll{\speed}{\hl{\speed}{f}}{\calA}.
\end{align*}
Similarly, let $ s_0\leq s_1\in[0,T] $. We have
\begin{align*}
	&
	\hl{\speed}{f}|_{[s_0,T]\times\R} = \hll{\speed}{f_0}{s_0},
	\quad
	\text{where }
	f_0(\Cdot) := \hl{\speed}{f}(s_0,\Cdot),
\\
	&
	\hll{\speed}{f_0}{s_0}|_{[s_1,T]\times\R} = \hll{\speed}{f_1}{s_1},
	\quad
	\text{where }
	f_1(\Cdot) := \hl{\speed}{f_0}(s_1,\Cdot).
\end{align*}

\item \label{enu:hl:sp}
We have 
\begin{align}
	\label{eq:hl:time}
	&
	0\leq \hl{\speed}{f}(t'_0,\xi_0) - \hl{\speed}{f}(t_0,\xi_0) \leq \tfrac{\maxlam}{4}(t'_0-t_0),&
	&
	\forall t_0\leq t'_0\in[0,T], \ \xi_0\in\R,
\\
	\label{eq:hl:lip}
	&
	0\leq \hl{\speed}{f}(t_0,\xi'_0) - \hl{\speed}{f}(t_0,\xi_0) \leq \xi'_0-\xi_0,&
	&
	\forall t_0\in[0,T], \ \xi_0\leq \xi'_0\in\R.
\\	
	\label{eq:hls:time}
	&
	0\leq \hll{\speed}{f}{s_0}(t'_0,\xi_0) - \hll{\speed}{f}{s_0}(t_0,\xi_0) 
	\leq 
	\tfrac{\maxlam}{4}(t'_0-t_0),&
	&
	\forall t_0\leq t'_0\in[{s_0},T], \ \xi_0\in\R,
\\
	\label{eq:hls:lip}
	&
	0\leq \hll{\speed}{f}{s_0}(t'_0,\xi_0) - \hll{\speed}{f}{s_0}(t_0,\xi_0) \leq \xi'_0-\xi_0,&
	&
	\forall t_0\in[{s_0},T], \ \xi_0\leq \xi'_0\in\R.
\end{align}
In particular $ \hl{\speed}{f} \in \Sp\cap C([0,T],\Splip) $.
\item\label{enu:hl:chaV}
Given a piecewise $ C^1 $ path $ w:[s_0,t_0]\to\R $ and any $ t'_0\in [s_0,t_0] $,
there exists a piecewise $ C^1 $ path $ v:[s_0,t_0]\to\R $ such that
\begin{align}
	&
	\label{eq:charV:goal:}
	v|_{[t'_0,t_0]}=w|_{[t'_0,t_0]},
\\
	&
	\label{eq:charV:goal::}
	(t,v(t))|_{t\in[s_0,t'_0]}\in \calC(t'_0,w(t'_0))
\\
	&
	\label{eq:charV:goal}
	\HLf_{s_0,t_0}(v;\speed) + f(t_0,v(t_0))
	\leq
	\HLf_{s_0,t_0}(w;\speed) + f(t_0,w(t_0)).
\end{align}
Namely, without making the functional $ \HLf_{s_0,t_0}(w;\speed) + f(t_0,w(t_0)) $ larger,
we can replace $ w $ with a path that:
agrees with $ w $ on $ [t'_0,t_0] $; 
and lies within a the light cone $ \calC(t'_0,w(t'_0)) $ for $ t\in[s_0,t'_0] $.
\item\label{enu:hl:match}
Let $ f_1,f_2\in\Splip $, and $ s_0\in[0,T] $.
For any given $ (t_0,\xi_0)\in[s_0,T]\times\R $, 
let
\begin{align}
	\label{eq:cone:}
	\calC'(s_0,t_0,\xi_0) := \calC(t_0,\xi_0) \cap \big((s_0,t_0)\times\R\big)
\end{align}
denote the restriction of $ \calC(t_0,\xi_0) $ onto $ t\in(s_0,t_0) $.
If $ \speed_1|_{\calC'(s_0,t_0,\xi_0)}=\speed_2|_{\calC'(s_0,t_0,\xi_0)} $
and $ f_1(\xi)|_{(s_0,\xi)\in\calC(t_0,\xi_0)} = f_2(\xi)|_{(s_0,\xi)\in\calC(t_0,\xi_0)} $,
then
\begin{align*}
	\hll{\speed_1}{f_1}{s_0}(t_0,\xi_0)=\hll{\speed_2}{f_2}{s_0}(t_0,\xi_0).
\end{align*}
\item\label{enu:hl:Linf}
Let $ f_1,f_2\in\Splip $, and $ s_0\in[0,T] $. We have
\begin{align*}
	&\big| \hl{\speed}{f_1}(t_0,\xi_0)-\hl{\speed}{f_2}(t_0,\xi_0) \big|
	\leq
	\sup_{|\xi-\xi_0| \leq t_0\maxlam} |f_1(\xi)-f_2(\xi)|,&
	&
	\forall (t_0,\xi_0)\in[0,T]\times\R.
\\
	&\big| \hll{\speed}{f_1}{s_0}(t_0,\xi_0)-\hll{\speed}{f_2}{s_0}(t_0,\xi_0) \big|
	\leq
	\sup_{|\xi-\xi_0| \leq (t_0-s_0)\maxlam} |f_1(\xi)-f_2(\xi)|,&
	&
	\forall (t_0,\xi_0)\in[s_0,T]\times\R.
\end{align*}
\end{enumerate}
\end{lemma}
\begin{proof}
\ref{enu:hl:loc}
We prove only the statement for $ \hll{\speed}{f}{\calA} $, as the other statements follow similarly.
Fix $ (t_0,\xi_0)\in\calA $.
Under the convention, for any given $ w\in W(t_0,\xi_0) $, 
consider its first exist time $ s_\star := \inf\{s\in[t_w,t_0]: (s,w(s))\in\calA^\circ\}\wedge t_0 $ from $ \calA $,
and cut $ w $ into two pieces accordingly as: $ w_1:[0,s_\star]\to \R $ and $ w_2:[s_\star,t_0]\to \R $.
Under this set up we have
\begin{align}
	\label{eq:hl:local:cut}
	\HLf_{0,t_0}(w;\speed) + g(0,w(0))
	=
	\HLf_{s_\star,t_0}(w_2;\speed) + \HLf_{0,s_\star}(w_1;\speed) + g(0,w_1(0)).
\end{align}
For such $ w $, the resulting paths $ w_2 $ and $ w_1 $ 
are $ W_\calA(t_0,\xi_0) $-valued and $ W(s_\star,w(s_\star)) $-valued, respectively.
Conversely, given $ w_2\in W_\calA(t_0,\xi_0) $ and $ w_1 \in W(t_{w_2},w_2(t_{w_2})) $,
the joint path $ w(t) := w_1(t)\ind_{[0,t_{w_2})}(t) + w_2(t)\ind_{[t_{w_2},t_0]}(t) $ is $ W(t_0,\xi_0) $-valued.
Hence, in~\eqref{eq:hl:local:cut}, 
taking infimum over $ w\in W(t_0,\xi_0) $ is equivalent to taking 
infimum over $ w_2\in W_\calA(t_0,\xi_0) $ and $ w_1 \in W(t_{w_2},w_2(t_{w_2})) $.
This concludes the desired result $ \hl{\speed}{f}(t_0,\xi_0)=\hll{\speed}{\hl{\speed}{f}}{\calA} $.

\medskip
\noindent\ref{enu:hl:sp}
We prove only \eqref{eq:hl:time}--\eqref{eq:hl:lip}, 
as \eqref{eq:hls:time}--\eqref{eq:hls:lip} follow similarly.
To this end, we note the following useful properties of $ \hlf(\Cdot) $,
which are readily verified from the definition~\eqref{eq:hl}:
\begin{align}
	\label{eq:hl:id}
	&
	u \hlf\big(\tfrac{\beta}{u}\big) = \beta_+,
	\quad
	\forall \alpha\in\R, u\in(0,\infty) \text{ with } |\tfrac{\alpha}{u}| \geq 1,
\\
	\label{eq:hl:ndecr}
	&
	u\mapsto u\hlf(\tfrac{\alpha}{u}) \text{ is nondecreasing in } u\in[0,\infty),
	\
	\forall\alpha\in\R.
\end{align}
Fixing $ t_0\leq t'_0\in[0,T] $ and $ \xi_0\leq\xi'_0\in\R $,
we consider a generic path $ w\in W(t_0,\xi_0) $.
For small $ \d>0 $ we perform a surgery on $ w $ to obtain $ w_\d\in W(t_0,\xi'_0) $:
\begin{align*}
	w_\d(t) := w(t)\ind_{[0,t_0-\d]}(t) + \tfrac{\xi'_0-w(t_0-\d)}{\d}(t-t_0+\d) \ind_{[t_0-\d,t_0]}(t).
\end{align*}
That is, we let $ w_\d $ follow $ w $ for $ t\in[0,t_0-\d] $ and then linearly connect $ w_\d(t_0-\d) $ to $ (t_0,\xi'_0) $.
Recall that speed functions are bounded by $ \maxlam $.
Under this assumption, using \eqref{eq:hl:ndecr} gives
\begin{align*}
	\HLf_{0,t_0}(w_\d;\speed)
	=
	\HLf_{0,t_0-\d}(w;\speed) + \HLf_{t_0-\d,t_0}(w;\speed)
	\leq
	\HLf_{0,t_0-\d}(w;\speed) + \d\maxlam \hlf(\tfrac{\xi'_0-w(t_0-\d)}{\d\maxlam}).
\end{align*}
Letting $ \d\downarrow 0 $ using~\eqref{eq:hl:id} for $ \beta=\xi'_0-w(t_0-\d) $,
we obtain
\begin{align*}
	\limsup_{\d\downarrow 0}
	\HLf_{0,t_0}(w_\d;\speed)
	\leq
	\HLf_{0,t_0}(w;\speed) + (\xi'_0-w(t_0))_+ = \HLf_{0,t_0}(w;\speed) + \xi'_0-\xi_0.
\end{align*}
Adding $ f(w(0))=f(w_\d(0)) $ to both sides gives
\begin{align*}
	\hl{\speed}{f}(t_0,\xi'_0)
	\leq
	\limsup_{\d\downarrow 0} \big\{ \HLf_{0,t_0}(w_\d;\speed) + f(w_\d(0)) \big\}
	\leq
	\HLf_{0,t_0}(w;\speed) + f(w(0)) + \xi'_0-\xi_0.
\end{align*}
Since $ w\in W(t_0,\xi_0) $, further taking infimum over $ w $ gives 
$ \hl{\speed}{f}(t_0,\xi'_0)-\hl{\speed}{f}(t_0,\xi_0) \leq\xi'_0-\xi_0 $.
This proves one half of~\eqref{eq:hl:lip}.
The other half, $ \hl{\speed}{f}(t_0,\xi_0)-\hl{\speed}{f}(t_0,\xi'_0) \geq-(\xi'_0-\xi_0) $,
is proven similarly, by performing the same type of surgery on any given $ w\in W(t_0,\xi'_0) $.
We omit repeating the argument here.

We now turn to showing~\eqref{eq:hl:time}.
First, for any given $ w\in W(t_0,\xi_0) $, continuing the path vertically gives
$ v(t) := w(t)\ind_{t\in[0,t_0]} + \xi_0 \ind_{t\in(t_0,t'_0]}\in W(t'_0,\xi_0) $.
Referring to the definition~\eqref{eq:hl} of $ \HLf_{t_1,t_2}(w;\speed) $,
we have that $ \HLf_{t_0,t'_0}(v;\speed) \leq \maxlam (t'_0-t_0) \hlf(0)= \maxlam (t'_0-t_0)\frac14  $.
This gives $ \hl{\speed}{f}(t'_0,\xi_0)-\hl{\speed}{f}(t_0,\xi_0) \leq \frac14 \maxlam (t'_0-t_0) $.
This settles one half of~\eqref{eq:hl:time}.
To show the other half, we apply Part\ref{enu:hl:loc} with $ s_0=t_0 $
to localize the expression $ \hl{\speed}{f} $ onto $ t\in[t_0,T] $ as
\begin{align}
	\label{eq:hl:lip:1}
	\hl{\speed}{f}(t'_0,\xi_0)
	=
	\hll{\speed}{f_0}{t_0}(t'_0,\xi_0)
	=
	\inf_{w\in W_{t_0}(t'_0,\xi_0)} \big\{ \HLf_{t_0,t_0'}(w;\speed) + \hl{\speed}{f}(t_0,w(t_0)) \big\}.
\end{align}
The convexity of $ \xi\mapsto \hlf(\xi) $ gives
\begin{align}
	\notag
	\HLf_{t_1,t_2}(w;\speed)
	\geq&
	(t_2-t_1)u \hlf\Big( \frac{1}{u} \int_{t_1}^{t_2} w'(t) dt \Big)\Big|_{u=\fint_{t_1}^{t_2} \speed(t,w(t)) dt}
\\
	\label{eq:hl:cnvx}
	&=
	(t_2-t_1)u \hlf\Big( \frac{w(t_2)-w(t_1)}{(t_2-t_1)u} \Big)\Big|_{u=\fint_{t_1}^{t_2} \speed(t,w(t)) dt}.
\end{align}
Also, by \eqref{eq:hl:lip} we have 
$ 
	\hl{\speed}{f}(t_0,w(t_0)) \geq \hl{\speed}{f}(t_0,\xi_0) -(w(t_0)-\xi_0)_- 
$.
Using this and \eqref{eq:hl:cnvx} for $ (t_1,t_2)=(t_0,t'_0) $ in~\eqref{eq:hl:lip:1} gives
\begin{align}
	\label{eq:hl:lip:1.5}
	\hl{\speed}{f}(t'_0,\xi_0)
	\geq
	\inf_{w(t_0)\in\R}
	\inf_{u\in(0,\maxlam]} 
	\Big\{ (t'_0-t_0)u \hlf\Big( \frac{w(t'_0)-w(t_0)}{(t'_0-t_0)u} \Big)  + \hl{\speed}{f}(t_0,\xi_0) -(w(t_0)-\xi_0)_-  \Big\}.
\end{align}
By~\eqref{eq:hl:ndecr}, the infimum over $ u\in(0,\maxlam] $ in \eqref{eq:hl:lip:1.5} occurs at $ u\downarrow 0 $.
Taking such a limit $ u\downarrow 0 $ using~\eqref{eq:hl:id} for $ \beta=w(t'_0)-w(t_0) $ gives 
\begin{align*}
	\lim_{u\downarrow 0} (t'_0-t_0)u 
	\hlf\Big(\frac{w(t'_0)-w(t_0)}{(t'_0-t_0)u}\Big) = (w(t'_0)-w(t_0))_+ = (\xi_0-w(t_0))_+.
\end{align*}
From this we then obtain
\begin{align*}
	\hl{\speed}{f}(t'_0,\xi_0)
	\geq
	\inf_{w(t_0)\in\R} 
	\big\{ (\xi_0-w(t_0))_+  + \hl{\speed}{f}(t_0,\xi_0) -(w(t_0)-\xi_0)_- \big\}
	\geq
	\hl{\speed}{f}(t_0,\xi_0).
\end{align*}
This completes the proof of~\eqref{eq:hl:time}.

\medskip
\noindent\ref{enu:hl:chaV} 
Fixing $ s_0\leq t'_0\leq t_0 \in [0,T] $, a piecewise $ C^1 $ path $ w:[s_0,t_0]\to \R $,
we write $ \calC:=\calC(t_0,w(t_0)) $ to simplify notations.
Our goal is to construction $ v\in W(t_0,w(t_0)) $
that satisfies~\eqref{eq:charV:goal:}--\eqref{eq:charV:goal}.
To this end, without lost of generality assume $ (t,w(t))\notin \calC $, for some $ t\in[s_0,t'_0) $,
otherwise simply take $ v=w $.
For such a path $ w $ let $ s_\star:= \inf\{ t\in[s_0,t'_0]: (t,w(t))\notin \calC \} $
denote the first exists time from $ \calC $.
Let $ (s_0,\xi^\pm_0) $ denote the intersection of $ \partial\calC $ and $ \set{s_0}\times\R $,
i.e., $ \xi^\pm_0:=w(t'_0) \pm \maxlam (t'_0-s_0) $.
We set
\begin{align*}
	v(t) 
	&:= 
	w(t)\ind_{[s_0,s_\star)\cup(t'_0,t_0]}(t) 
	+ 
	\Big( (t-s_\star)\frac{w(t'_0)-\alpha}{t'_0-s_\star} + \alpha \Big) \ind_{[s_\star,t'_0]}(t),
\\
	\alpha
	&:=
	\left\{\begin{array}{l@{,}l}
		w(s_\star)	&	\text{ if } s_\star >s_0,
	\\
		\xi^+_0		&	\text{ if } s_\star =s_0, \text{ and } w(s_0)\in(\xi^+_0,\infty),
	\\
		\xi^-_0		&	\text{ if } s_\star =s_0, \text{ and } w(s_0)\in(-\infty,\xi^-_0).
	\end{array}\right.
\end{align*}
Such a path $ v $ indeed satisfies~\eqref{eq:charV:goal:}--\eqref{eq:charV:goal::}.
To verify the last condition~\eqref{eq:charV:goal},
using $ w|_{[s_0,s_\star)\cup(t'_0,t_0]}=v|_{[s_0,s_\star)\cup(t'_0,t_0]} $, we write
\begin{align}
	\notag
	\big( &\HLf_{s_0,t_0}(w;\speed) + f(w(s_0) \big)
	- \big( \HLf_{s_0,t_0}(v;\speed) + f(v(s_0)) \big)
\\
	\label{eq:charV:red}
	&=
	\left\{\begin{array}{l@{,}l}
		\HLf_{s_\star,t'_0}(w;\speed)-\HLf_{s_\star,t'_0}(v;\speed)	&\text{ if } s_\star >s_0,
	\\
		\big( \HLf_{s_\star,t'_0}(w;\speed) + f(w(s_0)) \big)
		- \big( \HLf_{s_\star,t'_0}(v;\speed) + f(\alpha) \big)
				&\text{ if } s_\star =s_0.
	\end{array}\right.
\end{align}
Next, apply \eqref{eq:hl:cnvx} with $ (t_1,t_2)=(s_\star,t'_0) $ to get
\begin{align}
	\label{eq:charV::}
	\HLf_{s_\star,t'_0}(w;\speed)
	\geq
	(t'_0-s_\star)u_\star \hlf\Big( \frac{w(t_0)-w(s_\star)}{(t'_0-s_\star)u_\star} \Big)
	\quad
	\text{where }u_\star:=\fint_{s_\star}^{t'_0} \speed(t,w(t)) dt.
\end{align}
Since $ \speed\leq \maxlam $, we have $ u_\star \leq \maxlam $,
so
$
	\frac{w(t_0)-w(s_\star)}{(t'_0-s_\star)u_\star}
	\geq
	|\frac{w(t'_0)-w(s_\star)}{(t_0-s_\star)\maxlam}| =1.
$
With this property,
using~\eqref{eq:hl:id} for $ \beta=\frac{w(t_0)-w(s_\star)}{t_0-s_\star} $ and 
$ u=u_\star $ in~\eqref{eq:charV::} gives
\begin{align}
	\label{eq:charVw}
	\HLf_{s_\star,t'_0}(w;\speed) \geq (w(t'_0)-w(s_\star))_+.
\end{align}
As for $ \HLf_{s_\star,t'_0}(v;\speed) $,
since the path $ v $ has a \emph{constant} derivative $ v'(t) = \pm \maxlam $ for $ t\in(s_\star,t'_0) $,
using~\eqref{eq:hl:id} for $ \beta=v'(t) $ and $ u=\speed(t,v(t)) $, 
and integrating the result over $ t\in(s_\star,t'_0) $, we obtain 
\begin{align}
	\label{eq:charV}
	\HLf_{s_\star,t'_0}(v;\speed) 
	=
	\int_{s_\star}^{t'_0} (v'(t))_+ dt
	=
	(v(t'_0)-v(s_\star))_+ = (w(t'_0)-v(s_\star))_+.
\end{align}

Given~\eqref{eq:charV:red},
Our goal of showing~\eqref{eq:charV:goal} amounts to showing the r.h.s.\ of~\eqref{eq:charV:red} is nonnegative.
For the case $ s_\star>s_0 $, 
we have $ v(s_\star)=w(s_\star) $, so combining \eqref{eq:charVw}--\eqref{eq:charV}
gives $ \HLf_{s_\star,t'_0}(w;\speed) -\HLf_{s_\star,t'_0}(v;\speed) \geq 0 $.
Inserting this into~\eqref{eq:charV:red} gives the desired result.
For the case $ s_\star=s_0 $, 
we consider further the sub-cases $ w(s_0)\in(-\infty,\xi^-_0) $ and $ w(s_0)\in(\xi^+_0,\infty) $, as follows:
\begin{itemize}[leftmargin=5ex]
	\item
	if $ w(s_0)\in (\xi^+_0,\infty) $, 
	we have $ f(w(s_0)) \geq f(\xi^+_0)=f(\alpha) $.
	Using this and \eqref{eq:charVw}--\eqref{eq:charV} gives
	\begin{align*}
		\big( \HLf_{s_\star,t'_0}(w;\speed) + f(w(s_0)) \big)
		- \big( \HLf_{s_\star,t'_0}(v;\speed) + f(\alpha) \big)
		\geq
		(w(t'_0)-w(s_0))_+ - (w(t'_0)-\xi^+_0)_+ = 0.
	\end{align*}	
	\item
	if $ w(s_0)\in (-\infty,\xi^-_0) $, 
	$ f(w(s_0)) \geq f(\xi^-_0) - (\xi^-_0-w(s_0)) = f(\alpha) - (\xi^-_0-w(s_0)) $.
	Using this and \eqref{eq:charVw}--\eqref{eq:charV}
	gives
	\begin{align*}
		\big( \HLf_{s_\star,t'_0}(w;\speed) &+ f(w(s_0)) \big)
		- \big( \HLf_{s_\star,t'_0}(v;\speed) + f(\alpha) \big)
	\\
		&\geq 
		(w(t'_0)-w(s_0))_+ - (w(t'_0)-\xi^-_0)_+- (\xi^-_0-w(s_0)) = 0.
	\end{align*}
\end{itemize}
The preceding discussions verify that \eqref{eq:charV:red} is nonnegative.

\medskip
\noindent\ref{enu:hl:match}
Fix $ (t_0,\xi_0)\in[s_0,T]\times\R $.
Applying Part\ref{enu:hl:chaV} with $ (s_0,t'_0,t_0)=(s_0,t_0,t_0) $,
have that
\begin{align}
	\label{eq:hl:match:}
	\hll{\speed}{f}{s_0}(t_0,\xi_0)
	&:=
	\big\{ \HLf_{s_0,t_0}(w;\speed)+f(w(s_0)):
		w\in W_{s_0}(t_0,\xi_0), \, (t,w(t))|_{t\in[s_0,t_0]}\in\calC(t_0,\xi_0)
	\big\}
\\
	\notag
	&=
	\big\{ \HLf_{s_0,t_0}(w;\speed)+f(w(s_0)):
		w\in W_{s_0}(t_0,\xi_0), \, (t,w(t))|_{t\in(s_0,t_0)}\in\calC'(s_0,t_0,\xi_0)
	\big\}.
\end{align}
The last expression depends on $ \speed $ and $ f $
only through $ \speed|_{\calC'(s_0,t_0,\xi_0)} $ and $ f(\xi)|_{\xi: (s_0,\xi)\in\calC(t_0,\xi_0)} $.
From this we conclude the desired result.

\medskip
\noindent\ref{enu:hl:Linf} 
Similarly to~\eqref{eq:hl:match:},
applying Part\ref{enu:hl:chaV} with $ (s_0,t'_0,t_0)=(0,t_0,t_0) $ gives
\begin{align}
	\label{eq:hl:match::}
	\hl{\speed}{f}(t_0,\xi_0)
	:=
	\big\{ \HLf_{0,t_0}(w;\speed)+f(w(s_0)):
		w\in W(t_0,\xi_0), \, (t,w(t))|_{t\in[0,t_0]}\in\calC(t_0,\xi_0)
	\big\}.
\end{align}
The desired results follow immediately
by comparing the expressions~\eqref{eq:hl:match:}--\eqref{eq:hl:match::} for $ f=f_1 $ and for $ f=f_2 $.
\end{proof}

In view of the overview given in Section~\ref{sect:speedd:ov},
to prepare for the construction of $ \Speedd_{\scll,\sclll} $,
here we solve explicitly the variational formula~\eqref{eq:hl:frml} of Hopf and Lax,
for a few piecewise constant speed functions $ \speed $ and piecewise linear initial conditions $ f $.
To setup notations,
fix $ \kappa,\kappa^-,\kappa^+\in(0,\infty) $, $ \rho,\rho^-,\rho^+\in(0,1) $,
and set $ \lambda:=\frac{\kappa}{\rho(1-\rho)} $, $ \lambda^\pm := \frac{\kappa^\pm}{\rho^\pm(1-\rho^\pm)} $.
We assume $ \lambda,\lambda^\pm \in (0,\maxlam] $.
Fix further $ \zeta_0\in\R $ and $ s_0\in[0,T] $,
we divide the region $ [s_0,T)\times\R $ into two parts:
through a vertical cut into 
\begin{align}
	\label{eq:vrtlcut}
	\calA^- := [s_0,T)\times(-\infty,\zeta_0),
	\quad 
	\calA^+:= [s_0,T)\times(\zeta_0,\infty); 
\end{align}
or through a diagonal cut into
\begin{align}
	\label{eq:diagcut}
	\calB^- := \{ (t,\xi) : \xi< \zeta_0+\tfrac{b}{\tau}(t-s_0), \ t\in[s_0,T) \},
	\quad 
	\calB^+ := \{ (t,\xi) : \xi> \zeta_0+\tfrac{b}{\tau}(t-s_0) \ t\in[s_0,T) \}.
\end{align}
Under these notations,
consider a pair $ (\speed,f) $ of speed function and $ \Splip $-valued profile, 
of the following form:
\begin{enumerate}[label=(\alph*),leftmargin=5ex]
\item \label{enu:explicit:homo} 
	constant $ \speed := \lambda $, and linear $ f\in\Splip $ with $ f':=\rho $;
\item \label{enu:explicit:ver} 
	piecewise constant $ \speed $
	with $ \speed|_{{\calA^\pm}} := \lambda^\pm $
	and unique extension onto $ [s_0,T)\times\R $ by lower semi-continuity,
	and a piecewise linear $ f\in\Splip $ with $ f'|_{(-\infty,\zeta_0)}=\rho^- $ and $ f'|_{(\zeta_0,\infty)}=\rho^+ $;
\item \label{enu:explicit:dia} 
	piecewise constant $ \speed $
	with $ \speed|_{{\calB^\pm}} := \lambda^\pm $
	and unique extension onto $ [s_0,T)\times\R $ by lower semi-continuity,
	and a piecewise linear $ f\in\Splip $ with $ f'|_{(-\infty,\zeta_0)}=\rho^- $ and $ f'|_{(\zeta_0,\infty)}=\rho^+ $;
\item \label{enu:explicit:shk} 
	constant speed function $ \speed :=1 $,
	and a piecewise linear $ f\in\Splip $ with $ f'|_{(-\infty,\zeta_0)}=\rho^- $ and $ f'|_{(\zeta_0,\infty)}=\rho^+ $;
\end{enumerate}
See Figure~\ref{fig:4types} for an illustration.
\begin{figure}[h]
\psfrag{A}[c][c]{\ref{enu:explicit:homo}}
\psfrag{B}[c][c]{\ref{enu:explicit:ver}}
\psfrag{C}[c][c]{\ref{enu:explicit:dia}}
\psfrag{D}[c][c]{\ref{enu:explicit:shk}}
\psfrag{U}[c][c]{$ s_0 $}
\psfrag{Z}[c][c]{$ \zeta_0 $}
\psfrag{L}[c][c]{$ \speed=\lambda $}
\psfrag{M}[c][c]{$ \speed=\lambda^- $}
\psfrag{N}[c][c]{$ \speed=\lambda^+ $}
\psfrag{R}[c][c]{$ f'=\rho $}
\psfrag{S}[c][c]{$ f'=\rho^- $}
\psfrag{T}[c][c]{$ f'=\rho^+ $}
\includegraphics[width=\textwidth]{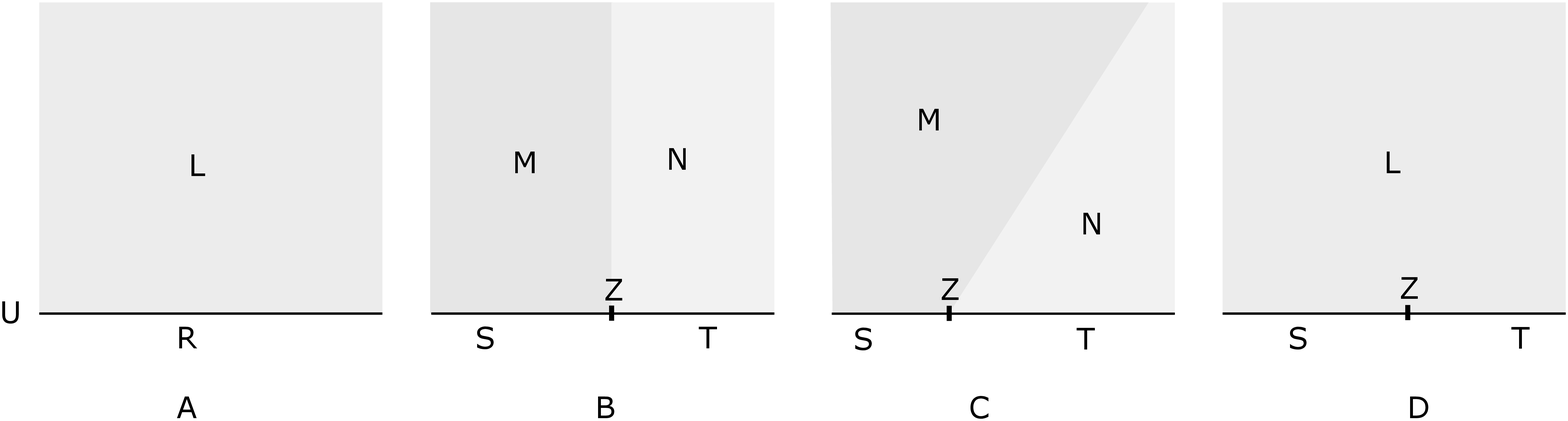}
\caption{Four types of $ (\speed,f) $}
\label{fig:4types}
\end{figure}

For each of the case~\ref{enu:explicit:ver}--\ref{enu:explicit:shk} in the preceding,
we assume the following condition:
\begin{align}
	\label{eq:RH:ver}
	&\text{\ref{enu:explicit:ver}}\quad \kappa^-=\kappa^+;
\\
	\label{eq:RH:diag}
	&\text{\ref{enu:explicit:dia}}\quad 
          \kappa^- + \tfrac{b}{\tau}\rho^- = \kappa^+ + \tfrac{b}{\tau}\rho^+,
\\
	\label{eq:RH:shk}
	&\text{\ref{enu:explicit:shk}}\quad \kappa^- + \kappa^-=\kappa^+.
\end{align}
Under these assumptions,
for each of the case~\ref{enu:explicit:homo}--\ref{enu:explicit:shk} in the preceding, 
we consider a piecewise linear function $ \Gamma $, specified by its derivatives
and value at $ (s_0,0) $, as follows:
\begin{align}
	\notag
	\Gamma \in & C([s_0,T]\times\R),
	\quad
	\Gamma(s_0,\R) := f(0),
\\
	\notag
	&\text{\ref{enu:explicit:homo}}\quad 
	\nabla\Gamma(t,\xi) := (\kappa,\rho), \ \forall (t,\xi) \in [s_0,T]\times\R,
\\
	\label{eq:explicit:Gamma}
	&\text{\ref{enu:explicit:ver}}\quad 
	\nabla\Gamma(t,\xi) := (\kappa^\pm,\rho^\pm), \ \forall (t,\xi) \in \calA^\pm;
\\
	\notag
	&\text{\ref{enu:explicit:dia}}\quad 
	\nabla\Gamma(t,\xi) := (\kappa^\pm,\rho^\pm), \ \forall (t,\xi) \in \calB^\pm;
\\
	\notag
	&\text{\ref{enu:explicit:shk}}\quad 
	\nabla\Gamma(t,\xi) := (\kappa^\pm,\rho^\pm), \ \forall (t,\xi) \in \calA^\pm.
\end{align}
Indeed, given $ f(0) $, \eqref{eq:explicit:Gamma} admits \emph{at most} one such function $ \Gamma $.
The conditions~\eqref{eq:RH:ver}--\eqref{eq:RH:shk} ensures the existence of such $ \Gamma $.
The following Lemma shows that, under suitable conditions,
the Hopf--Lax function $ \hl{\speed}{f} $ is given by the piecewise linear $ \Gamma $
for each of the case~\ref{enu:explicit:homo}--\ref{enu:explicit:shk}.
\begin{lemma}
\label{lem:explicit}
Let $ \kappa,\kappa^\pm,\rho,\rho^\pm, \lambda,\lambda^\pm $
be as in the preceding, and assume $ \lambda,\lambda^\pm \in (0,\maxlam] $.
Consider $ (\speed,f) $ of the form~\ref{enu:explicit:homo}--\ref{enu:explicit:shk} as in the preceding.
For each of the cases, we assume \eqref{eq:RH:ver}--\eqref{eq:RH:shk} and, additionally:
\begin{align}
	\label{eq:char:ver}
	&\text{\ref{enu:explicit:ver}}\quad 2\rho^--1 \geq 0, \text{ or } 2\rho^+-1 \leq 0;
\\
	\label{eq:char:diag}
	&\text{\ref{enu:explicit:dia}}\quad 
	\lambda^-(2\rho^--1) \geq \tfrac{b}{\tau}, \text{ or } \lambda^+(2\rho^+-1) \leq \tfrac{b}{\tau};
\\
	\label{eq:char:shk}
	&\text{\ref{enu:explicit:shk}}\quad \rho^-+\rho^+=1, \quad \rho^- \geq \rho^+.
\end{align}
Then, the Hopf--Lax function $ \hl{\speed}{f} $ matches the piecewise linear $ \Gamma $ as in~\eqref{eq:explicit:Gamma}:
\begin{align*}
	\hl{\speed}{f} = \Gamma.
\end{align*}
\end{lemma}
\begin{remark}
In the language of Figure~\ref{fig:char:},
the conditions~\eqref{eq:char:ver}--\eqref{eq:char:diag}
amount to saying that characteristics
must not diverge along the discontinuity of $ \speed $.
As for~\eqref{eq:char:shk}, under the assumption~\eqref{eq:RH:shk} 
the first condition $ \rho^-+\rho^+=1 $ together with \eqref{eq:RH:shk} ensures $ \lambda^\pm = 1 $,
which is consistent with the form of $ \speed $ as in \ref{enu:explicit:shk}.
This being the case, we must have $ \rho^- \geq \rho^+ $ to avoid diverging characteristics.
\end{remark}
\begin{proof}
Assume without lost of generality $ s_0,\zeta_0=0 $ and $ f(0)=0 $.
Fixing arbitrary $ (t_0,\xi_0)\in[0,T]\times\R $,
we proceed by solving the variational problem:
\begin{align*}
	\inf_{ w\in W(t_0,\xi_0) }
	\big\{ \HLf_{0,t_0}(w;\speed) +f(w(0)) \big\}
\end{align*}
for each of the cases \ref{enu:explicit:homo}--\ref{enu:explicit:shk}.

\ref{enu:explicit:homo}
Applying~\eqref{eq:hl:cnvx} with $ (t_1,t_2)=(0,t_0) $ gives
$ \HLf_{0,t_0}(w;\lambda) \geq t_0\lambda \hlf(\frac{\xi-w(0)}{t_0\lambda}) $.
Add $ \rho w(0) $ to both sides of the inequality,
and further optimize over $ w(0) $. 
We obtain
\begin{align*}
	\hl{\lambda}{f}(t_0,\xi_0)
	&\geq 
	\big(  
		t_0\lambda \hlf(\tfrac{\xi-w(0)}{t_0\lambda}) +\rho w(0)
	\big)
	\big|_{w(0)=\xi-(2\rho-1)\lambda t_0}
	=
	\kappa t_0+\rho\xi_0.
\end{align*}
Conversely, the linear path $ \til w(t):= \xi_0-(2\rho-1)\lambda (t_0-t) $, $ \til w\in W(t_0,\xi_0) $,
does yield the desired value, i.e., $ \HLf_{0,t_0}(w_0;\lambda) + \rho \til w(0) =  \kappa t_0+\rho\xi_0 $.
This concludes the desired result:
\begin{align}
	\label{eq:var:homo}
	\hl{\lambda}{\rho\xi}(t_0,\xi_0)
	=
	\inf_{w\in W(t_0,\xi_0)} 
	\big\{ \HLf_{0,t_0}(w;\lambda) + \rho w(0) \big\}
	=
	\kappa t_0 + \rho \xi_0,
	\quad
	\text{for } \lambda := \frac{\kappa}{\rho(1-\rho)}.
\end{align}

\ref{enu:explicit:ver}
Assume $ \lambda^-\leq\lambda^+ $ for simplicity of notations.
The proof of the other scenario $ \lambda^->\lambda^+ $ is similar.
We consider first the case when $ (t_0,\xi_0) $ sits on where $ \speed $ is discontinuous, i.e., $ \xi_0=\zeta_0:=0 $,
and prove $ \hl{\speed}{f}(t_0,0) = \kappa^- t_0 $.
To this end,
given any $ w\in W(t_0,0) $, with $ \speed(t_0,w(t_0))  = \lambda^- $,
let $ s_0 := \inf\{ [0,t_0] : \speed(s,w(s))=\lambda^- \} $ 
be the entrance time of $ w $ into the region $ [0,T]\times(-\infty,0] $.

\begin{itemize}[leftmargin=5ex]
\item 
If $ s_0 =0 $, we have $ w(s_0)\leq 0 $ and $ \speed(s,w(s))|_{s\in[0,t_0]}=\lambda^- $.
The last two conditions give 
$
	\HLf_{0,t_0}(w;\speed) + f(w(0)) 
	=
	\HLf_{0,t_0}(w;\lambda^-) + \rho^-w(0).
$
Combining this with \eqref{eq:var:homo} for $ \xi_0=0 $ and $ (\kappa,\rho)=(\kappa^-,\rho^-) $ gives
\begin{align*}
	\HLf_{0,t_0}(w;\speed) + f(w(0)) 
	=
	\HLf_{0,t_0}(w;\lambda^-) + \rho^-w(0)	
	\geq 
	\kappa^- t_0.
\end{align*}
\item
If $ s_0>0 $,
decompose $ \HLf_{0,t_0}(w;\speed)+f(w(0)) $ 
as $ \HLf_{s_0,t_0}(w;\speed)+(\HLf_{0,s_0}(w;\speed)+\rho^+w(0)) $.
For the first term apply~\eqref{eq:hl:cnvx} with $ (t_1,t_2)=(s_0,t_0) $ to get
\begin{align}
	\notag
	\HLf_{s_0,t_0}(w;\speed) 
	&\geq (t_0-s_0)\lambda^- \hlf(0) 
	= (t_0-s_0)\tfrac{1}{4}\lambda^-	
\\
	\label{eq:explicit:stat:1:}
	&\geq 
	(t_0-s_0)\rho^-(1-\rho^-)\lambda^- 
	=
	(t_0-s_0)\kappa^-. 	
\end{align}
For the second term, with $ \speed(0,w(s))|_{s\in[0,s_0)} =\lambda^+ $,
applying~\eqref{eq:hl:cnvx} with $ (t_1,t_2)=(0,s_0) $ gives
\begin{align}
	\label{eq:explicit:stat:1}
	\HLf_{0,s_0}(w;\speed) + \rho^+w(0) 
	\geq 
	s_0 \lambda^+\hlf(\tfrac{-w(0)}{s_0\lambda^+}) + \rho^+ w(0).
\end{align}
Let $ w_\star(t):= w(0)-t\frac{w(0)}{t_0} $ denote the linear path that joins
$ (0,w(0)) $ and $ (t_0,0) $.
Applying~\eqref{eq:var:homo} with $ (t_0,\xi_0;\kappa,\rho)=(s_0,0;\kappa^+,\rho^+) $
gives
\begin{align*}
	s_0 \lambda^+\hlf(\tfrac{-w(0)}{s_0\lambda^+}) + \rho^+ w(0) 
	=
	\HLf_{0,s_0}(w_\star;\lambda^+) + \rho^+ w_\star(0)
	\geq 
	\kappa^+t_0.
\end{align*}
Inserting this into~\eqref{eq:explicit:stat:1}, 
and combining the result with~\eqref{eq:explicit:stat:1:},
we obtain
\begin{align*}
	\HLf_{0,t_0}(w;\speed) + f(w(0)) \geq \kappa^-(t_0-s_0) + \kappa^+ s_0  =\kappa^- t_0.
\end{align*}
\end{itemize}
The preceding argument gives $ \hl{\speed}{f}(t_0,0) \geq \kappa^-t_0 $.
Conversely, for the linear paths 
\begin{align*}
	\til w^-(t) &:= \lambda^-(2\rho^--1)(t-t_0),
\\
	\til w_\d^+(t) &:= \lambda^+\big((2\rho^+_\d-1)(t-t_0),
	\quad
	\rho^+_\d := \rho^+ \wedge (1-\d),
\end{align*}
it is straightforward to verify that
\begin{align*}
	& \HLf_{0,t_0}(\til w^-;\speed) + f(\til w^-(0)) 
	= (\rho^-)^2 t_0\lambda^- - \rho^-\Big| \lambda^-(2\rho^--1)(-t_0) \Big|
	= \kappa^-t_0,&
	&
	\text{if } 2\rho^--1  \geq 0,
\\
	&\HLf_{0,t_0}(\til w_\d^+;\speed) + f(\til w^\d_2(0))
	= (\rho^+_\d)^2 t_0\lambda^+ + \rho^+\Big|\lambda^+(2\rho^+_\d-1)(-t_0)\Big|
	\xrightarrow[]{\d\downarrow 0} \kappa^+t_0,&
	&
	\text{if } 2\rho^+-1  \leq 0.
\end{align*}
That is, under the assumption~\eqref{eq:char:ver},
one of the linear path $ w_1 $ or $ w^\d_2 $ (under a limiting procedure $ \d\downarrow 0 $) 
does yield the value $ \kappa^-t_0=\kappa^+t_0 $.

So far we have shown $ \hl{\speed}{f}(t_0,0)=\kappa^-t_0 $, $ \forall t_0\in[0,T] $.
Next, fix $ \xi_0<0 $.
Applying Lemma~\ref{lem:hl}\ref{enu:hl:loc} with $ \calA=[0,T]\times(-\infty,0) $ and $ h=\hl{\speed}{f} $, 
we localize the expression $ \hl{\speed}{f}(t_0,\xi_0) $ onto $ [0,T]\times(-\infty,0) $ as
\begin{align}
	\label{eq:explicit:stat:2}
	\hl{\speed}{f}(t_0,\xi_0)
	&=
	\inf_{ w\in W_\calA(t_0,\xi_0) } \big\{ \HLf_{t_w,t}(w;\speed) + \hl{\speed}{f}(t_w,w(t_w)) \big\}.
\end{align}
Similarly,
applying Lemma~\ref{lem:hl}\ref{enu:hl:loc} with $ \calA=[0,T]\times(-\infty,0) $ and $ h=\hl{\lambda^-}{\rho^-\xi} $ gives
\begin{align}
	\label{eq:explicit:stat:3}
	\hl{\lambda^-}{\rho^-\xi}(t_0,\xi_0)
	&=
	\inf_{ w\in W_\calA(t_0,\xi_0) } \big\{ \HLf_{t_w,t}(w;\lambda^-) + \hl{\lambda^-}{\rho^-\xi}(t_w,w(t_w)) \big\}.
\end{align}
In~\eqref{eq:explicit:stat:3},
further using~\eqref{eq:var:homo} for $ (\kappa,\lambda)=(\kappa^-,\lambda^-) $ 
to replace $ \hl{\lambda^-}{\rho^-\xi}(t,\xi) $ with $ \kappa^-t+\rho^-\xi $,
we rewrite~\eqref{eq:explicit:stat:3} as
\begin{align}
	\label{eq:explicit:stat:3:}
	\kappa^-t_0+\rho^-\xi_0
	=
	\inf_{ w\in W_\calA(t_0,\xi_0) } 
	\big\{ \HLf_{t_w,t}(w;\lambda^-) + \kappa^-t_w+\rho^-w(t_w) \big\}.
\end{align}
The r.h.s.\ of \eqref{eq:explicit:stat:2}
depends on $ \speed $ and $ \hl{\speed}{f} $ only through
$ \speed|_{[0,T]\times(-\infty,0)} $,
$ \hl{\speed}{f}(\Cdot,0) $ and $ \hl{\speed}{f}(0,\xi)|_{\xi\leq 0} $.
Since $ \speed|_{[0,T]\times(-\infty,0)}=\lambda^- $,
$ \hl{\speed}{f}(t,0)=\kappa^-t $, $ \hl{\speed}{f}(0,\xi)|_{\xi\leq 0} = \rho^-\xi $,
we conclude that the r.h.s.\ of 
\eqref{eq:explicit:stat:2} and \eqref{eq:explicit:stat:3:} must be the same.
This gives $ \hl{\speed}{f}(t_0,\xi_0)= \kappa^-t_0+\rho^-\xi_0 = \kappa^-t_0+f(\xi_0) $ for $ \xi_0<0 $.
The case $ \xi_0>0 $ follows by the same localization and matching procedures.

\ref{enu:explicit:dia}
Assume $ \lambda^-\leq\lambda^+ $ for simplicity of notations.
The proof of the other scenario $ \lambda^->\lambda^+ $ is similar.
We consider first the case when $ (t_0,\xi_0) $ sits on where $ \speed $ is discontinuous, i.e., $ \xi_0=\frac{b}{\tau}t_0 $,
and prove $ \hl{\speed}{f}(t_0,\xi_0) = (\kappa^-+\frac{b}{\tau}\rho^-) t_0 $.
To this end, for any given $ w\in W(t_0,\xi_0) $, with $ \speed(t_0,w(t_0))  = \lambda^- $,
we let $ s_0 := \inf\{ [0,t_0] : \speed(s,w(s))=\lambda^- \} $ 
be the entrance time of $ w $ into the region $ \{(t,\xi): \xi \leq t\frac{b}{\tau}\} $.
\begin{itemize}[leftmargin=5ex]
\item 
If $ s_0 =0 $, namely $ w(s_0)\leq 0 $ and $ \speed(s,w(s))|_{s\in[0,t_0]}=\lambda^- $,
we have 
$
	\HLf_{0,t_0}(w;\speed) + f(w(0)) 
	=
	\HLf_{0,t_0}(w;\lambda^-) + \rho^-w(0).
$
Combining this with \eqref{eq:var:homo} for $ (\kappa,\rho)=(\kappa^-,\rho^-) $ gives
\begin{align*}
	\HLf_{0,t_0}(w;\speed) + f(w(0)) 
	=
	\HLf_{0,t_0}(w;\lambda^-) + \rho^-w(0)
	\geq 
	\kappa^- t_0 + \rho^-\xi_0
	=
	(\kappa^-+\tfrac{b}{\tau}\rho^-)t_0.
\end{align*}
\item
If $ s_0>0 $,
decompose $ \HLf_{0,t_0}(w;\speed) + f(w(0)) $ 
as $ \HLf_{s_0,t_0}(w;\speed)+(\HLf_{0,s_0}(w;\speed)+\rho^+w(0)) $.
For the first term applying~\eqref{eq:hl:cnvx} with $ (t_1,t_2)=(s_0,t_0) $ 
gives
\begin{align}
	\label{eq:explicit:move:1}
	\HLf_{s_0,t_0}(w;\speed) 
	\geq (t_0-s_0)\lambda^- \hlf(\tfrac{b}{\tau\lambda^-}).
\end{align}
For the second term, with $ \speed(0,w(s))|_{s\in[0,s_0)} =\lambda^+ $,
applying~\eqref{eq:hl:cnvx} with $ (t_1,t_2)=(0,s_0) $ gives
\begin{align}
	\label{eq:explicit:move:2}
	\HLf_{0,s_0}(w;\speed)+\rho^+w(0) 
	\geq 
	s_0 \lambda^+\hlf(\tfrac{w(s_0)-w(0)}{s_0\lambda^+})+\rho^+w(0). 
\end{align}
Further, letting
$ w_\star(t):= w(0)+t\frac{w(s_0)-w(0)}{s_0} $ denote the linear path
that joins $ (0,w(0)) $ and $ (s_0,w(s_0)) $,
applying~\eqref{eq:var:homo} with $ (t_0,\xi_0;\kappa,\rho)=(s_0,w(s_0);\kappa^+,\rho^+) $,
we obtain
\begin{align}
	\label{eq:explicit:move:2::}
	s_0 \lambda^+\hlf(\tfrac{w(s_0)-w(0)}{s_0\lambda^+})+\rho^+w(0) 
	=
	\HLf_{0,s_0}(w_\star;\lambda^+) + \rho^+w_\star(0)
	\geq
	\kappa^+ s_0 + w(s_0)\rho^+.
\end{align}
Use $ w(s_0)=\frac{b}{\tau}s_0 $ and $ \kappa^++\frac{b}{\tau}\rho^+=\kappa^-+\frac{b}{\tau}\rho^- $
in the last expression in~\eqref{eq:explicit:move:2::},
and then combine the result with \eqref{eq:explicit:move:2}. 
We have
\begin{align}
	\label{eq:explicit:move:2:}
	\HLf_{0,s_0}(w;\speed)+\rho^+w(0) 
	\geq
	(\kappa^-+\tfrac{b}{\tau}\rho^-) s_0.
\end{align}
Combining~\eqref{eq:explicit:move:1} and \eqref{eq:explicit:move:2:} gives
\begin{align}
	\notag
	\HLf_{0,t_0}(w;\speed)+\rho^+w(0) 
	&\geq
	(t_0-s_0)\lambda^- \hlf(\tfrac{b}{\tau\lambda^-}) + (\kappa^-+\tfrac{b}{\tau}\rho^-) s_0
\\
	\label{eq:explicit:move:3}
	&\geq
	t_0 \min\big\{ \lambda^- \hlf(\tfrac{b}{\tau\lambda^-}), \kappa^-+\tfrac{b}{\tau}\rho^- \big\}.
\end{align}
Let $ w_{\star\star}(t)=\frac{b}{\tau}t $ denote the linear path that goes along the discontinuity of $ \speed $.
Using~\eqref{eq:var:homo} for $ (t_0,\xi_0;\kappa,\rho)=(t_0,w_{\star\star}(t_0);\kappa^-,\rho^-) $ we have
\begin{align*}
	t_0\lambda^- \hlf(\tfrac{b}{\tau\lambda^-}) 
	=
	\HLf_{0,t_0}(w_{\star\star};\lambda^-)+\rho^-w_{\star\star}(0) \geq \kappa^-t_0 + \rho^-\xi_0 
	=
	t_0 (\kappa^-+\tfrac{b}{\tau}\rho^-).
\end{align*}
Inserting this into~\eqref{eq:explicit:move:3} gives
\begin{align*}
	\HLf_{0,t_0}(w;\speed)+\rho^+w(0) 
	\geq
	t_0 (\kappa^-+\tfrac{b}{\tau}\rho^-).
\end{align*}
\end{itemize}
The preceding argument gives $ \hl{\speed}{f}(t_0,\xi_0) \geq (\kappa^-+\frac{b}{\tau}\rho^-)t_0 $.
Conversely, for the linear paths 
\begin{align*}
	w^-(t) &:= \lambda^-(2\rho^--1)(t-t_0) + \tfrac{b}{\tau} t_0,
\\
	w^+_\d(t) &:= \lambda^+(2\rho^+_\d-1)(t-t_0)+ \tfrac{b}{\tau} t_0,
	\quad
	\rho^+_\d := \rho^+ \wedge \big(\tfrac12\big(\tfrac{b}{\tau\lambda^+}+1-\d\big)\big),
\end{align*}
it is straightforward to verify that
\begin{align*}
	&\HLf_{0,t_0}(\til w^-;\speed)+ f(w^-(0))  = (\kappa^-+\tfrac{b}{\tau}\rho^-)t_0,&
	&
	\text{if } \lambda^-(2\rho^--1)  \geq \tfrac{b}{\tau},
\\
	&\lim_{\d\downarrow 0} \HLf_{0,t_0}(\til w^+_\d;\speed)+ f(w^+_\d(0))  
	= (\kappa^++\tfrac{b}{\tau}\rho^+)t_0 
	= (\kappa^-+\tfrac{b}{\tau}\rho^-)t_0,&
	&
	\text{if } \lambda^+(2\rho^+-1)  \leq \tfrac{b}{\tau}.
\end{align*}
That is, under the assumption~\eqref{eq:char:diag},
one of the linear paths $ w^\d_i $ (under a limiting procedure) does yield the value 
$ t_0 (\kappa^-+\tfrac{b}{\tau}\rho^-) $.

So far we have shown $ \hl{\speed}{f}(t,t\frac{b}{\tau})=(\kappa^-+\frac{b}{\tau})t $.
The desired result $ \hl{\speed}{f}(t,\xi)= f(\xi-\frac{b}{\tau}t) + (\kappa^-+\frac{b}{\tau})t $
follows by the same localization and matching procedures as in Part\ref{enu:explicit:ver}. 

\ref{enu:explicit:shk}
We consider first the case $ (t_0,\xi_0) $ sits on where $ \speed $ is discontinuous, i.e., $ \xi_0=0 $,
and prove $ \hl{\speed}{f}(t_0,0) = \kappa t_0 $.
Fix a generic $ w\in W(t_0,0) $.
Since $ f(\xi) = \rho^-\xi\ind_{\xi<0}+ \rho^+\xi\ind_{\xi\geq 0} $, depending on where $ w(0) $ sits,
we have
\begin{align}
	\label{eq:explicit:shock1}
	\HLf_{0,t_0}(1;w) + f(w(0)) 
	=
	\left\{\begin{array}{l@{,}l}
		\HLf_{0,t_0}(1;w) + \rho^+ w(0)	& \text{ if } w(0) \geq 0,
		\\
		\HLf_{0,t_0}(1;w) + \rho^- w(0)	& \text{ if } w(0) \leq 0,
	\end{array}
	\right.
\end{align}
By~\eqref{eq:var:homo} for $ \xi_0=0 $ and $ (\kappa,\rho)=(\rho^\pm(1-\rho^\pm),\rho^\pm) $
(where $ \kappa:=\rho^\pm(1-\rho^\pm) $ so that $ \lambda:= \frac{\kappa}{\rho^\pm(1-\rho^\pm)}=1 $),
the r.h.s.\ of~\eqref{eq:explicit:shock1} is bounded blew by
\begin{align*}
	\left\{\begin{array}{l@{,}l}
		\rho^-(1-\rho^-)	& \text{ if } w(0) \geq 0,
		\\
		\rho^+(1-\rho^+)	& \text{ if } w(0) < 0.
	\end{array}
	\right.
\end{align*}
Under the assumption $ \rho^-+\rho^+=1 $ from~\eqref{eq:char:shk}, we have $ \rho^-(1-\rho^-)=\rho^+(1-\rho^+) $.
This being the case, taking the infimum over $ w\in W(t_0,0) $ gives
$ \hl{1}{f}(t_0,0)\geq \rho^-(1-\rho^-)t_0 $.
Conversely,
under the assumption $ \rho^-\geq \rho^+ $ from~\eqref{eq:char:shk},
the linear paths $ \til w^\pm(t) := (2\rho^\pm-1)(t-t_0) $ both give the optimal value $ \rho^\pm(1-\rho^\pm) t_0 $.
That is,
\begin{align*}
	\HLf_{0,t_0}(1;\til w^-) + f(\til w^-(0))
	=
	\rho^-(1-\rho^-) t_0
	=
	\rho^+(1-\rho^+) t_0
	=
	\HLf_{0,t_0}(1;\til w^+) + f(\til w^+(0)).
\end{align*}

So far we have shown $ \hl{1}{f}(t,0)=\rho^-(1-\rho^-)t $.
The desired result $ \hl{1}{f}(t,\xi)= f(\xi) + \rho^-(1-\rho^-)t $
follows by the same localization and matching procedures as in Part\ref{enu:explicit:ver}. 
\end{proof}

\subsection{Constructing $ \Speedd_{\scll,\sclll} $}
\label{sect:speedd:cnst}

First, we set $ \Speedd_{\scll,\sclll} $ to unity out side of $ [0,T]\times[-r^*,r^*] $, i.e.,
\begin{align}
	\label{eq:speedd:>r*}
	\Speedd_{\scll,\sclll}(t,\xi)|_{|\xi|>r^*} :=1.
\end{align}
Recall that each $ \triangle\in\Sigma $ has height $ \tau $ and width $ b $,
such that $ \frac{T}{\tau}, \frac{r_*}{b}\in \N $ (the latter implies $ \frac{r^*}{b}\in\N $).
We write $ \scl_* := \frac{T}{\tau}\in\N $.
Given the auxiliary parameter $ \scll\in\N $, we divide $ \tau,b $ into $ \scll $ parts,
and introduce the scales:
\begin{align}
	\label{eq:scll}
	\tau'_{\scll}:=\tfrac{\tau}{\scll}, \quad b'_{\scll} := \tfrac{b}{\scll}.
\end{align}
Under these notations,
we divide the region $ [0,T]\times[-r^*,r^*] $ into $ \scl_* $ horizontal slabs,
each has height $ \tau - 6\tau'_\scll $:
\begin{align}
	\label{eq:slab}
	&
	\slab_i := [\undt_i, \bart_i]\times[-r^*,r^*],
	\quad
	i=1,\ldots,\scl_*,
\\
	\label{eq:undbart}
	&
	\undt_i := (i-1)\tau + 3\tau'_\scll,
	\quad
	\bart_i := i\tau - 3\tau'_\scll.
\end{align}
We omit the dependence of $ \slab_i $, $ \undt_i $ and $ \bart_i $ on $ \scll $ to simplify notations.
Such a convention is frequently practiced in the sequel.
In between the slabs $ \slab_i $ are thin, horizontal stripes of height $ 6\tau'_\scll $ or $ 3\tau'_\scll $:
\begin{align}
	\label{eq:tz}
	\tz_i := \big( [\bart_{i},\undt_{i+1}]\cap[0,T] \big) \times [-r^*,r^*],
	\quad
	i=0,\ldots, \scl_*.
\end{align}
We refer to these regions $ \tz_i $ as the \textbf{transition zones},
transitioning from one slab to another.
See Figure~\ref{fig:zone_transi}.
We set $ \Speedd_{\scll,\sclll} $ to unity within the interior $ \tz_i^\circ $ of each transition zone:
\begin{align}
	\label{eq:speedd:tz}
	\Speedd_{\scll,\sclll}|_{\tz_i^\circ} :=1,
	\quad
	i=0,\ldots,\scl_*.
\end{align}
\begin{figure}
\psfrag{A}{$ t $}
\psfrag{B}{$ T $}
\psfrag{D}[c][c]{$ \vdots $}
\psfrag{E}{$ r^* $}
\psfrag{F}{$ -r^* $}
\psfrag{S}{$ \slab_1 $}
\psfrag{R}{$ \slab_2 $}
\psfrag{Q}{$ \slab_{\scl_*} $}
\psfrag{T}{$ \tz_0 $}
\psfrag{U}{$ \tz_1 $}
\psfrag{V}{$ \tz_{2} $}
\psfrag{W}{$ \tz_{\scl_*} $}
\psfrag{I}{$ 3\tau'_\scll $}
\psfrag{H}{$ 6\tau'_\scll $}
\psfrag{X}{$ \xi $}
\includegraphics[width=.8\textwidth]{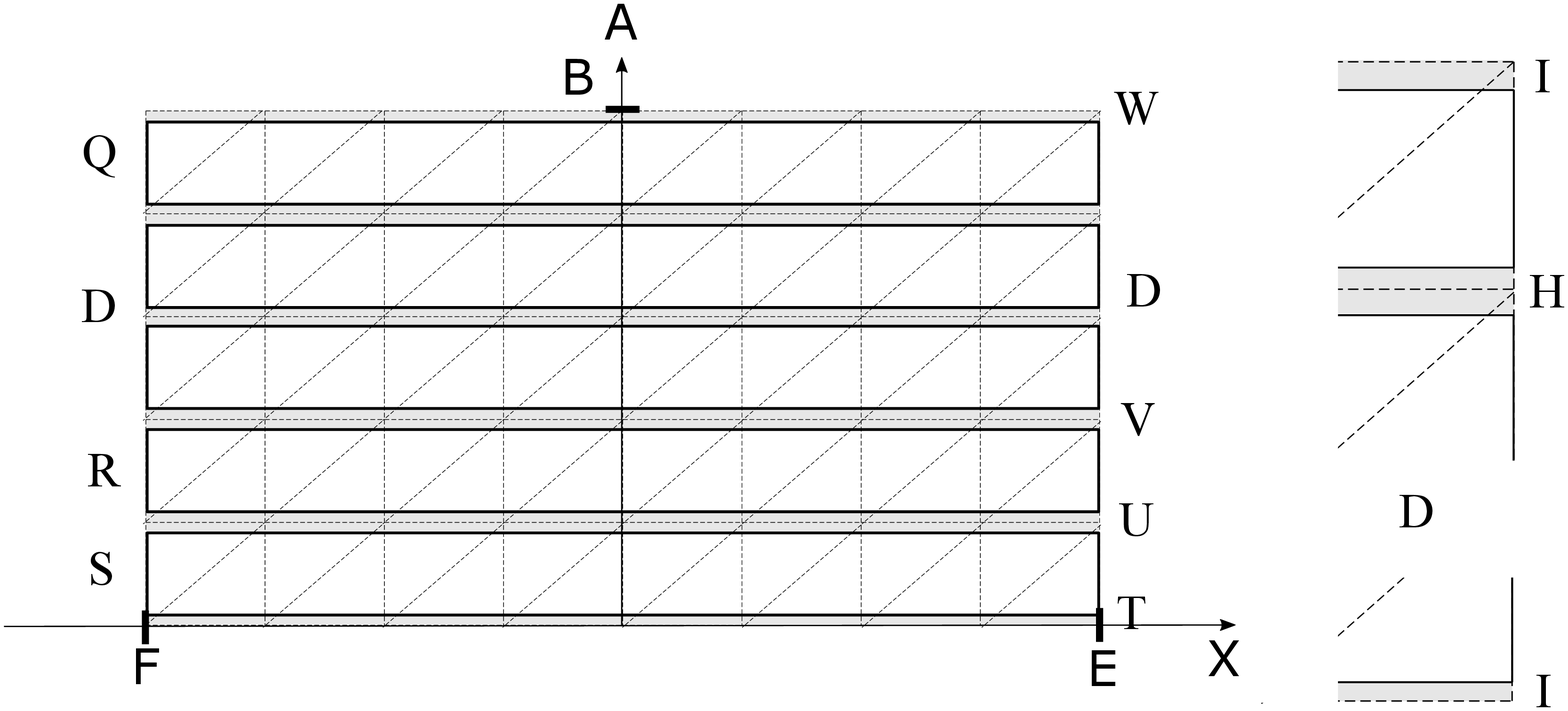}
\caption{The Slabs $ \slab_i $ (white boxes) and transition zones $ \tz_i $ (gray)}
\label{fig:zone_transi}
\end{figure}

Fixing $ i\in\{1,\ldots,\scl_*\} $, 
we now focus on constructing $ \Speedd_{\scll,\sclll} $ within the slab $ \slab_i $.
To this end, we will first construct a partition $ \prtn_i $ of $ \slab_i $,
and then, build $ \Speedd_{\scll,\sclll} $ as a piecewise constant function on $ \slab_i $
according to this partition $ \prtn_i $.

\noindent\textbf{Constructing the partition $ \prtn_i $.}
To setup notations, we write
\begin{align*}
	(t_1,\xi_1)\edge(t_2,\xi_2) := \{ (t,\tfrac{\xi_2-\xi_1}{t_2-t_1}(t-t_1) \}_{t\in[t_1,t_2]}
\end{align*}
for the line segment joining $ (t_1,\xi_1) $ and $ (t_2,\xi_2) $,
and  consider the sets of vertical and diagonal edges from $ \Sigma $ that intersect $ \slab_i $:
\begin{align}
	\notag
	\VE_i
	&:= 
	\Big\{ e = ((i-1)\tau,jb)\edge(i\tau,jb) :  
	j=-\tfrac{r^*}{b},\ldots,\tfrac{r^*}{b} \Big\},
\\
	\label{eq:DE}
	\DE_i
	&:= 
	\Big\{ e = ((i-1)\tau,(j-1)b)\edge(i\tau,jb)
	: 
	j=-\tfrac{r^*}{b}+1,\ldots,\tfrac{r^*}{b} \Big\}.
\end{align}
Around each vertical or diagonal edge $ e\in\VE_j\cup\VE_j $,
we introduce a \textbf{buffer zone} of width $ 2b'_\scll $ or $ b'_\scll $, as depicted in Figure~\ref{fig:buff_tri}.
More explicitly, for $ \ve = ((i-1)\tau,jb)\edge(i\tau,jb) \in \VE_i $,
\begin{align*}
	\buffer_{e} := [\undt_i,\bart_i]\times
		\Big( [jb-b'_{\scll},jb+b'_{\scll}] \cap[-r^*,r^*] \Big),
\end{align*}
and for $ \de = ((i-1)\tau,(j-1)b)\edge(i\tau,jb) \in \DE_i $
\begin{align}
	\label{eq:bufferDE}
	\buffer_{e} := \{ (t,\xi): |(\xi-(j-1)b)-(t-(i-1)\tau)|\leq b'_{\scll}, \  t\in [\undt_i,\bart_i] \}.
\end{align}
We call $ \buffer_e $ a \textbf{vertical buffer zone} if $ e\in\VE_i $,
and likewise call $ \buffer_e $ a \textbf{diagonal buffer zone} if $ e\in\DE_i $.
Referring to Figure~\ref{fig:buff_tri}, 
the buffer zones $ \buffer_e $ and the transition zones $ \tz_i $ 
shrink the triangle $ \triangle\in\Sigma $, resulting in trapezoidal regions.
Despite the trapezoidal shapes, we refer to these regions as \textbf{reduced triangles},
use the symbol $ \rTri $ to denote them, 
and let $ \RTri_i $ be the collection of all reduced triangles within the slab $ \slab_i $.
Each reduced triangle $ \rTri $ is uniquely contained in triangle $ \triangle\in\Sigma $.
Under such a correspondence, we set $ (\kappa_\rtri,\rho_\rtri,\lambda_\rtri) := (\linkap,\linrho,\linlam) $.
\begin{figure}[h]
\psfrag{W}{$ \scriptscriptstyle 2b'_\scll $}
\psfrag{X}{$ \scriptscriptstyle b'_\scll $}
\psfrag{D}[c][c]{$ \cdots $}
\includegraphics[width=\textwidth]{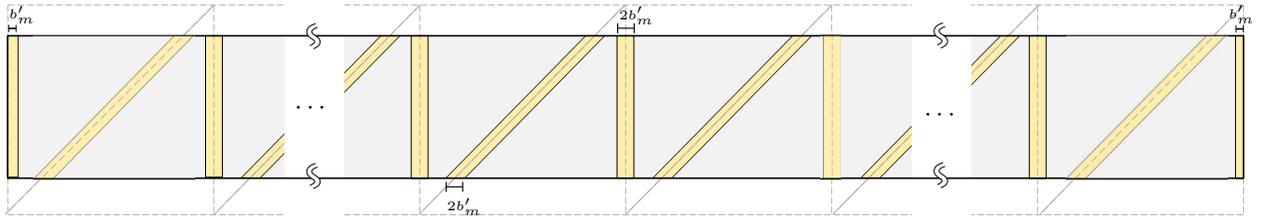}
\caption{Buffer zones (yellow) and reduced triangles (gray)}
\label{fig:buff_tri}
\end{figure}

As mentioned in Section~\ref{sect:speedd:ov},
those $ \rTri\in\RTri_i $ with $ \lambda_{\rtri} < 1 $ need an intermittent construction.
To this end, we divide the slab $ \slab_i $ into thinner slabs, each of height $ \tau'_\scll $, as
\begin{align}
	\label{eq:slab:}
	&
	\slab_{i,i'} := [\undt_{i,i'},(i-1)\tau+i'\bart_{i,i'}]\times[-r^*,r^*],
	\quad
	i'=4,\ldots,\scll-3,
\\
	\label{eq:undbart:}
	&
	\undt_{i,i'} := (i-1)\tau + (i'-1)\tau'_{\scll} = \undt_i +(i-4)\tau'_{\scll},
	\quad
	\bart_{i,i'} := (i-1)\tau + i'\tau_{\scll} = \undt_i +(i-3)\tau'_{\scll}.
\end{align}
With $ \sclll\in\N $ being an auxiliary parameter,
we divide the scales $ \tau'_{\scll}, b'_{\scll} $ (as in~\eqref{eq:scll}) into $ \scll^2 $ parts,
and introduce the finer scales 
\begin{align}
	\label{eq:sclll}
	\tau''_{\sclll} := \tfrac{\tau'_{\scll}}{\sclll^2},
	\quad
	b''_{\scll,\sclll} := \tfrac{b'_{\scll}}{\sclll^2}.
\end{align}
Now, fix $ \rTri\in\RTri_i $ with $ \lambda_{\rtri} < 1 $, and fix $ i'\in\{4,\ldots,\scll-3\} $.
Referring to Figure~\ref{fig:zone_iterm}, on $ \rTri\cap\slab_{i,i'} $, 
we place a vertical stripe $ \intm_{i',j''}(\rTri) $ of width $ b''_{\scll,\sclll} $, 
every distance $ (m-1)b''_{\scll,\sclll} $ apart.
These stripes start from the vertical edge of $ \rTri $,
and continue until reaching distance $ b'_\scll $ from the hypotenuse.
Making a vertical cut at distance $ b'_\scll $ from the hypotenuse,
we denote the region beyond by $ \intm_{i',\star}(\rTri) $; see Figure~\ref{fig:zone_iterm}
We refer to $ \intm_{i',j''}(\rTri) $ and $ \intm_{i',\star}(\rTri) $ as the \textbf{intermittent zones}.
\begin{figure}[h]
\psfrag{T}[c][c]{$ \intm_{i',\star}(\rTri) $}
\psfrag{I}[c][c]{$ \intm_{i',j''}(\rTri) $}
\psfrag{J}[c][c]{$ \intm_{i',j''}(\rTri) $}
\psfrag{Z}[r][c]{$ \scriptscriptstyle j''=1 $}
\psfrag{A}[l][c]{$ \scriptscriptstyle 1=j'' $}
\psfrag{B}[c][c]{$ \scriptscriptstyle 2 $}
\psfrag{C}[c][c]{$ \scriptscriptstyle 3 $}
\psfrag{D}[c][c]{$ \scriptscriptstyle \cdots $}
\psfrag{E}[c][c]{$ \cdots $}
\psfrag{S}[l][c]{$ 4=j'' $}
\psfrag{R}[l][c]{$ 5 $}
\psfrag{Q}[l][c]{$ \rTri\cap\slab_{i,i'} $}
\psfrag{V}[c][c]{$ \vdots $}
\psfrag{P}[l][c]{$ \scll-3 $}
\psfrag{G}[r][c]{$ j''=4 $}
\psfrag{H}[r][c]{$ 5 $}
\psfrag{K}[r][c]{$ \rTri\cap\slab_{i,i'} $}
\psfrag{L}[r][c]{$ \scll-3 $}
\psfrag{W}{$ \scriptstyle b'_\scll $}
\psfrag{X}{$ \scriptscriptstyle b''_\sclll $}
\psfrag{Y}[r][c]{$ \scriptscriptstyle (\sclll-1)b''_\sclll $}
\includegraphics[width=\textwidth]{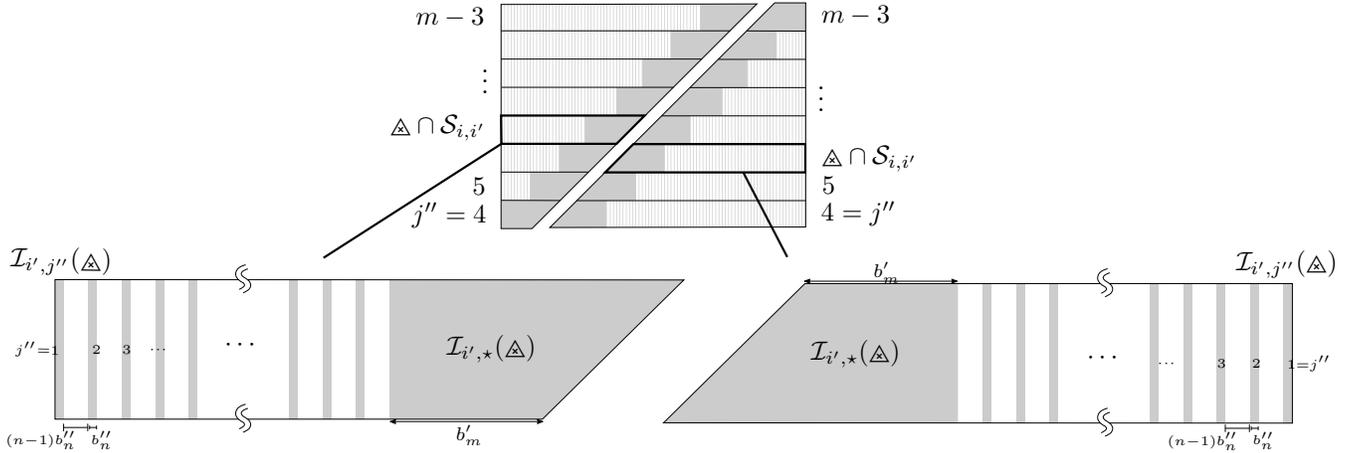}
\caption{Intermittent zones (gray) on a given $ \rTri $ with $ \lambda_{\rtri}<1 $}
\label{fig:zone_iterm}
\end{figure}

Outside of the intermittent zones on $ \rTri\cap\slab_{i,i'} $
are stripes of width $ (m-1)b''_{\scll,\sclll} $.
We enumerate these regions as $ \resd_{i',j''}(\rTri) $, as depicted in Figure~\ref{fig:zone_resd:}.
We further divide each of these regions $ \resd_{i',j''}(\rTri) $ into two parts,
$ \resd^1_{i',j''}(\rTri) $ on the left and $ \resd^2_{i',j''}(\rTri) $ on the right, 
one of width $ r^1_{\scll,\sclll}(\rTri) $ and width $ r^2_{\scll,\sclll}(\rTri) $, respectively, 
as depicted in Figure~\ref{fig:zone_resd12}.
The values of $ r^1_{\scll,\sclll}(\rTri) $ and $ r^2_{\scll,\sclll}(\rTri) $ 
are given in~\eqref{eq:r1:resd}--\eqref{eq:r2:resd} in the following.
The regions $ \resd_{i',j''}(\rTri) $, $ \resd^1_{i',j''}(\rTri) $ and $ \resd^2_{i',j''}(\rTri) $
are referred to as \textbf{residual regions}.
For convenient of notations, in the following we do not explicitly specify the range of the indice $ i',j'' $
in $ \intm_{i',j''}(\rTri) $, $ \resd^1_{i',j''}(\rTri) $, etc.,
under the conscientious that it alway runs through admissible values as described in the preceding.
\begin{figure}[h]
\psfrag{I}[c][c]{$ \resd_{i',j''}(\rTri) $}
\psfrag{J}[c][c]{$ \resd_{i',j''}(\rTri) $}
\psfrag{Z}[r][c]{$ \scriptscriptstyle j''=1 $}
\psfrag{A}[l][c]{$ \scriptscriptstyle 1=j'' $}
\psfrag{B}[c][c]{$ \scriptscriptstyle 2 $}
\psfrag{C}[c][c]{$ \scriptscriptstyle 3 $}
\psfrag{D}[c][c]{$ \scriptscriptstyle \cdots $}
\psfrag{E}[c][c]{$ \cdots $}
\psfrag{M}[c][c]{$ \resd^1 $}
\psfrag{N}[c][c]{$ \resd^2 $}
\psfrag{W}[c][c]{$ \scriptscriptstyle(\sclll-1)b''_{\scll,\sclll} $}
\psfrag{R}[c][c]{$ \scriptscriptstyle r^1_{\scll,\sclll}(\rTri) $}
\psfrag{S}[c][c]{$ \scriptscriptstyle r^2_{\scll,\sclll}(\rTri) $}
\begin{subfigure}{0.65\textwidth}
\includegraphics[width=\textwidth]{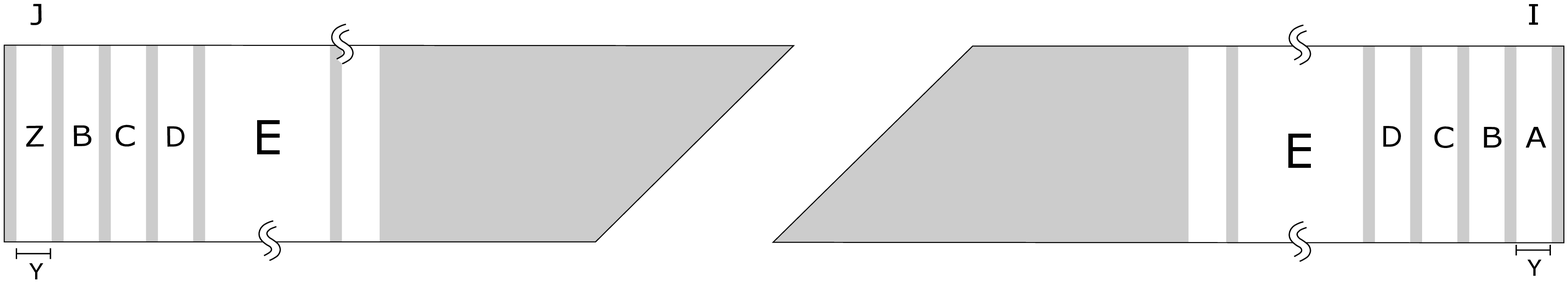}
\caption{Residual regions $ \resd_{i',j''}(\rTri) $ (white areas) on $ \rTri\cap\slab_{i,i'} $}
\label{fig:zone_resd:}
\end{subfigure}
\hfill
\begin{subfigure}{0.3\textwidth}
\includegraphics[width=\textwidth]{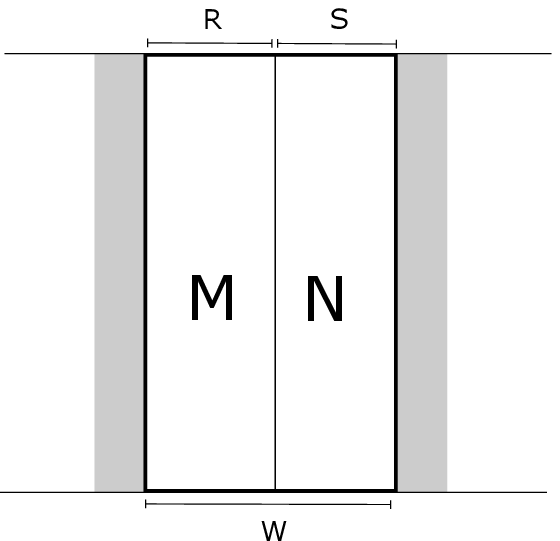}
\caption{\scriptsize{}The region $ \resd=\resd_{i',j''}(\rTri) $, 
is further divided into $ \resd^1=\resd_{i',j''}^1(\rTri) $ and $ \resd^2=\resd_{i',j''}^1(\rTri) $.}
\label{fig:zone_resd12}
\end{subfigure}
\caption{Residual regions}
\label{fig:zone_resd}
\end{figure}

Collecting the regions introduced in the preceding,
we define the partition $ \prtn_i $ of the slab $ \slab_i $ as
\begin{align*}
	\prtn_i
	&:=
	 \big\{ \buffer_e: e\in\VE_i\cup\DE_i \big\} \cup \big\{ \rTri\in\RTri_i : \lambda_{\rtri} \geq 1 \big\}
\\
	&\cup
	\big\{ \intm_{i',\star}(\rTri), \intm_{i',j''}(\rTri), \resd^1_{i',j''}(\rTri), \stripp^2_{i',j''}(\rTri) 
	:\rTri\in\RTri_i, \lambda_{\rtri} < 1 \big\}.
\end{align*}
Further, collecting these partitions $ \prtn_i $, $ i=1,\ldots,\scl_* $,
the transition zones $ \tz_i $ (as in~\eqref{eq:tz}),
and the `outer regions' $ [0,T]\times[r^*,\infty) $ and $ [0,T]\times(-\infty,r^*] $,
we obtain a partition $ \prtnU $  of the entire domain $ [0,T]\times\R $:
\begin{align}
	\label{eq:prtnU}
	\prtnU := 
	\big\{ [0,T]\times[r^*,\infty),[0,T]\times(-\infty,-r^*] \big\}	
	\cup
	\bigcup_{i=0}^{\scl_*}\set{\tz_i}
	\cup
	\bigcup_{i=1}^{\scl_*}\prtn_i.
\end{align}
The edges of $ \region\in\prtnU $ collectively gives rise to a graph,
and we call the collection of these edges the \textbf{skeleton} $ \ske $.
More precisely,
\begin{align}
	\label{eq:ske}
	\ske := 
	\Big\{ 
		\skeE=\region_1 \cap \region_2 : 
		\region_1 \neq \region_2\in\prtnU, 
		\ 
		\skeE \text{ is not a point} \Big\}.
\end{align}

In the following we will also consider the coarser version $ \prtnC $ of $ \prtn $:
\begin{align}
	\label{eq:prtnC}
	\prtnC_i
	:=
	\big\{ \buffer_e: e\in\VE_i\cup\DE_i \big\} \cup \RTri_i.
\end{align}
That is, we dismiss the intermittent construction on those $ \rTri\in\RTri $, $ \lambda_\rtri <1 $,
and replace the regions $ \{\intm_{i',\star}(\rTri), \intm_{i',j''}(\rTri), \resd^1_{i',j''}(\rTri), \stripp^2_{i',j''}(\rTri) \} $
simply by $ \{\rTri\} $ itself.

\medskip

Having constructed the partition $ \prtn_i $,
we proceed to define $ \Speedd_{\scll,\sclll} $ on each region $ \region\in\prtn_i $ of $ \prtn_i $.
To do this in a streamline fashion, 
in the following we assign a triplet $ (\kappa_{\region},\rho_\region, \lambda_\region) $ to each $ \region\in\prtn_i $.
To this end, let us first prepare a simple result regarding $ (\kappa_{\triangle},\rho_\triangle, \lambda_\triangle)_{\triangle\in\Sigma} $.
\begin{lemma}
\label{lem:ranHug}.
Let $ e\in\VE_i\cup\DE_i $, $ i=1,\ldots,\scl_* $, be a vertical or diagonal edge, 
and $ \triangle^-,\triangle^+\in\Sigma $ be the neighboring triangles of $ e $.
If $ e $ is vertical, we have $ \kappa_{\triangle^-} = \kappa_{\triangle^+} $;
if $ e $ is diagonal, we have
$ 	
	\kappa_{\triangle^-}+\tfrac{b}{\tau}\rho_{\triangle^-}
	=
	\kappa_{\triangle^+}+\tfrac{b}{\tau}\rho_{\triangle^+}
$
.
\end{lemma}
\begin{proof}
Parametrize $ e $ as $ e=(\und t,\und \xi)\edge(\bar t,\bar\xi) $
and consider the difference of $ \ling $ across the two ends of $ e $.
With $ \ling $ being piecewise linear on $ \triangle^- $ and on $ \triangle^+ $, we have
\begin{align*}
	\left.\begin{array}{l@{,}l}
		\kappa_{\triangle^-}	& \text{ if } e\in\VE_i
		\\
		\kappa_{\triangle^-}+ \frac{b}{\tau}\rho_{\triangle^-} &\text{ if } e\in\DE_i
	\end{array}\right\}
	=
	\frac{ \ling(\bar t,\und \xi)-\ling(\bar t,\bar\xi) }{ \bar t- \und t }
	=
	\left\{\begin{array}{l@{,}l}
		\kappa_{\triangle^+}	& \text{ if } e\in\VE_i
		\\
		\kappa_{\triangle^+}+ \tfrac{b}{\tau}\rho_{\triangle^+} &\text{ if } e\in\DE_i
	\end{array}\right.
\end{align*}
\end{proof}

\noindent
Previously, we have already associated the triplet 
$ (\kappa_{\rtri},\rho_\rtri, \lambda_\rtri):=(\kappa_{\triangle},\rho_\triangle, \lambda_\triangle) $
to each $ \rTri\in\RTri $,
where $ \triangle\supset\rTri $ is the unique triangle that contains $ \rTri $.
We now proceed to do this for each other region $ \region\in\prtn_i $.
\medskip

\noindent\textbf{Defining the triplet $ (\kappa_{\region},\rho_\region, \lambda_\region) $, for $ \region\in\prtn_i $.}
\begin{itemize}[leftmargin=3ex, itemsep=3pt]
\item For a vertical buffer zone $ \buffer_{e} $:
	\begin{itemize}[leftmargin=2ex]
	\item[]
	We let $ \triangle^- $ and $ \triangle^+ $ be the left and right neighboring triangles of $ e $,
	set $ \kappa_e = \kappa_{\triangle^-}=\kappa_{\triangle^+} $,
	($ \kappa_{\triangle^-}=\kappa_{\triangle^+} $ by Lemma~\ref{lem:ranHug}),
	and set $ (\kappa_e,\rho_e,\lambda_e) := (\kappa_e ,\frac12, 4\kappa_e ) $.
	\end{itemize}
\item
For a diagonal buffer zone $ \buffer_{e} $:
	\begin{itemize}[leftmargin=2ex]
	\item[]
	We let $ \triangle^- $ and $ \triangle^+ $ be the left and right neighboring triangles of $ e $.
	To define $ (\kappa_e,\rho_e,\lambda_e) $,
	we consider the two cases separately, as follows.

\begin{itemize}[leftmargin=3ex] 
\item If the condition holds:
	\begin{align}
		\label{eq:de:case}
		(2\rho_{\triangle^-}-1)\lambda_{\triangle^-} < \tfrac{b}{\tau} < (2\rho_{\triangle^+}-1)\lambda_{\triangle^+}.
	\end{align}
	By Lemma~\ref{lem:ranHug},
	$ 
		\kappa_{\triangle^-}+\frac{b}{\tau}\rho_{\triangle^-}
		  = \kappa_{\triangle^+}+\frac{b}{\tau}\rho_{\triangle^+} 
	$. 
	We let
	$ 
		\alpha := \kappa_{\triangle^-}+\frac{b}{\tau}\rho_{\triangle^-}
	$
	denote this quantity,
	and let $ F(\rho) := (2\rho - 1) \frac{\alpha-\frac{b}{\tau}\rho}{\rho(\rho-1)} $.
	Under the condition~\eqref{eq:de:case}, we necessarily have that $ \rho^+ > \frac12 $,
	and therefore $ \alpha > \frac{b}{\tau}\rho^+ > \frac{b}{2\tau} $.
	It is then straightforward to verify that $ F $ is increasing on $ \rho\in [\frac12,1) $
	and that $ F([\frac12,1))=[0,\infty) $.
	Further, using
	$ \lambda_{\triangle^\pm}=\frac{\kappa_{\triangle^\pm}}{\rho_{\triangle^\pm}(1-\rho_{\triangle^\pm})} $
	and
	$
		\kappa_{\triangle^\pm} =\alpha - \frac{b}{\tau}\rho_{\triangle^\pm} 
	$
	in~\eqref{eq:de:case},
	we have that $ F(\rho_{\triangle^-}) < \frac{\tau}{b} < F(\rho_{\triangle^+}) $.
	From these properties we see that $ F(\rho)=\frac{\tau}{b} $ has a unique solution in $ (\frac12,1) $.
	We let $ \rho_e $ be this solution,
	and set $ \lambda_e := \frac{b/\tau}{(2\rho_e-1)} $ and $ \kappa_e := \lambda_e\rho_\de(1-\rho_e) $.
	To summarize, $ (\kappa_e,\rho_e,\lambda_e)\in(0,\infty)\times(\frac12,1)\times(0,\infty) $
	is the unique solution of the following equations
	\begin{align*}
		& \kappa_e = \lambda_e \rho_e(1-\rho_e), 
	\\
		&(2\rho_e-1)\lambda_e = \tfrac{b}{\tau},
	\\
		&\kappa_e + \tfrac{\tau}{b} \rho_e 
		= \alpha 
		:=\kappa_{\triangle^-}+\tfrac{\scll}{\tau}\rho_{\triangle^-}
		= \kappa_{\triangle^+}+\tfrac{\scll}{\tau}\rho_{\triangle^+}.
	\end{align*}
\item 
	 If the condition holds:
	\begin{align*}
		(2\rho_{\triangle^-}-1)\lambda_{\triangle^-} \geq \tfrac{b}{\tau},
		\text{ or } 
		(2\rho_{\triangle^+}-1)\lambda_{\triangle^+} \leq \tfrac{b}{\tau}.
	\end{align*}
	In this case we set $ (\kappa_e,\rho_e,\lambda_e)=(\kappa_{\triangle^-},\rho_{\triangle^-},\lambda_{\triangle^-}) $.
\end{itemize}
\end{itemize}

\item For the intermittent zones $ \intm=\intm_{i',\star}(\rTri),\intm_{i',\star}(\rTri) $, with $ \lambda_\rtri <1 $:
	\begin{itemize}[leftmargin=2ex]
	\item []
	We let 
	$ (\kappa_{\intm},\rho_{\intm},\lambda_{\intm}):=(\kappa_{\rTri},\rho_\rTri, \lambda_\rTri) $.
	\end{itemize}
	
\item For the residual regions $ \resd^1 = \resd^1_{i',j''}((\rTri)) $ 
	and $ \resd^2=\resd^2_{i',j''}(\rTri) $, with $ \lambda_\rtri <1 $:
	\begin{itemize}[leftmargin=2ex]
	\item []
	Since $ \lambda_\rtri <1 $
	and $ \kappa_\rtri = \lambda_{\rtri}\rho_\rtri(1-\rho_\rtri) $,
	there are two solutions $ \rho_1,\rho_2 $ of the equation $ \kappa_\rtri = \rho(1-\rho) $.
	We order them as $ \rho_1 > \rho_2 \in[0,1] $.
	Under these notations,
	we set
	\begin{align*}
		(\kappa_{\resd^1},\rho_{\resd^1},\lambda_{\resd^1})
		:=(\kappa_{\rtri},\rho_1, 1),
	\quad
		(\kappa_{\resd^2},\rho_{\resd^2},\lambda_{\resd^2})
		:=(\kappa_{\rtri},\rho_2, 1).
	\end{align*}
	Note that, with $ \rho_1>\rho_2 $ solving the equation $ \kappa_\rtri = \rho_i(1-\rho_i) $,
	we necessarily have $ \rho_1+\rho_2=1 $.
	\end{itemize}
		
	Recall that $ r^1_{\scll,\sclll}(\rTri) $ and $ r^2_{\scll,\sclll}(\rTri) $
	denote the (yet to be specified) widths of the residual regions
	$ \resd^1=\resd^1_{i',j''}(\rTri) $ and $ \resd^2=\resd^2_{i',j''}(\rTri) $.
	We now define
	\begin{align}
	\label{eq:r1:resd}
		r^1_{\scll,\sclll}(\rTri) 
		&:= 
		({\sclll}b''_{\scll,\sclll}-1) 
		\frac{ \rho_\rtri-\rho_1 }{ \rho_{1}-\rho_{2} }
		=
		({\sclll}b''_{\scll,\sclll}-1) 
		\frac{ \rho_\rtri-\rho_{\resd^1_{i',j''}(\rTri)} }{ \rho_{\resd^1_{i',j''}(\rTri)}-\rho_{\resd^2_{i',j''}(\rTri)} },
	\\
	\label{eq:r2:resd}
		r^2_{\scll,\sclll}(\rTri) 
		&:= 
		({\sclll}b''_{\scll,\sclll}-1) 
		\frac{ \rho_2-\rho_\rtri }{ \rho_{1}-\rho_{2} }
		=
		({\sclll}b''_{\scll,\sclll}-1) 
		\frac{ \rho_{\resd^2_{i',j''}(\rTri)}-\rho_\rtri }{ \rho_{\resd^1_{i',j''}(\rTri)}-\rho_{\resd^2_{i',j''}(\rTri)} }.
	\end{align}
\end{itemize}

We next list a few important properties of $ (\kappa_\region,\rho_\region,\lambda_\region)_{\region\in\prtn} $.
These properties are readily verified from the preceding construction.
First,
\begin{align}
	\label{eq:region:lkr}
	\kappa_\region = \lambda_\region\rho_\region(1-\rho_\region),
	\quad
	\forall\region\in\prtn_i.
\end{align}
Next, recall from~\eqref{eq:slab} that $ \slab_i $ denotes a slab.
Consider vertical or diagonal edges $ \skeE\in\ske $ in the skeleton 
that is not on the boundary of the slabs $ \slab_i $, $ i=1,\ldots,\scl_* $:
\begin{align*}
	\ske'_\text{v} := 
	\Big\{ \skeE\in\ske : \skeE \not\subset\bigcup\nolimits_{i=1}^{\scl_*}\partial\slab_i, \ \skeE\text{ vertical} \Big\},
	\quad
	\ske'_\text{d} := 
	\Big\{ \skeE\in\ske : \skeE \not\subset\bigcup\nolimits_{i=1}^{\scl_*}\partial\slab_i, \ \skeE\text{ diagonal} \Big\}.
\end{align*}
Given $ \skeE\in\ske'_\text{v}\cup\ske'_\text{d} $,
letting $ \region^+,\region^-\in\prtn $ denote, respectively, the right and left neighboring regions of $ \skeE $,
we have
\begin{align}
	\label{eq:RH:ver:}
	&
	\kappa_{\region^-}=\kappa_{\region^+},&
	&
	\text{if } \skeE\in\ske'_\text{v},
\\
	\label{eq:RH:diag:}
	&
	(\kappa_{\region^-}+\tfrac{b}{\tau}\rho_{\region^-})
	=
	(\kappa_{\region^+}+\tfrac{b}{\tau}\rho_{\region^+}),&
	&
	\text{if } \skeE\in\ske'_\text{d},
\\
	\label{eq:char:ver:}
	&
	(1-2\rho_{\region^-}) \geq 0 \text{ or } (1-2\rho_{\region^+}) \leq 0,&
	&
	\text{if } \skeE\in\ske'_\text{v},
\\
	\label{eq:char:diag:}
	&
	(2\rho_{\region^-}-1)\lambda_{\region^-} \geq \tfrac{b}{\tau},
	\text{ or } 
	(2\rho_{\region^+}-1)\lambda_{\region^+} \leq \tfrac{b}{\tau},&
	&
	\text{if } \skeE\in\ske'_\text{d}.
\end{align}
Also, for a given pair of residual regions $ \resd_{i',j''}^1(\rTri) $ and $ \resd_{i',j''}^2(\rTri) $,
we have
\begin{align}
	\label{eq:resd:r1>2}
	\rho_{\resd_{i',j''}^1(\rTri)}+\rho_{\resd_{i',j''}^2(\rTri)}=1.
	\quad
	\rho_{\resd_{i',j''}^1(\rTri)}>\rho_{\resd_{i',j''}^2(\rTri)},
\end{align}
and, with $ r^1_{\scll,\sclll}(\rTri) $ and $ r^2_{\scll,\sclll}(\rTri) $
defined as in~\eqref{eq:r1:resd}--\eqref{eq:r2:resd},
\begin{align}
\label{eq:resd:r12}
	r^1_{\scll,\sclll}(\rTri) \rho_1
	+
	r^2_{\scll,\sclll}(\rTri) \rho_2
	=
	({\sclll}b''_{\scll,\sclll}-1) \rho_\rTri.
\end{align}	
%

\medskip

Now, for the $ \{\lambda_\region\}_{\region\in\prtn_i} $ defined in the preceding,
we set
\begin{align}
	\label{eq:speedd:<prtn}
	\Speedd_{\scll,\sclll}|_{\region^\circ} := \lambda_\region, \quad \region\in\prtn_i,
	\quad
	i=1,\ldots,\scl_*.
\end{align}
This together with \eqref{eq:speedd:>r*} and \eqref{eq:speedd:tz}, defines $ \Speedd_{\scll,\sclll} $
$
	([0,T)\times\R) \setminus ( \bigcup_{ \skeE\in\ske} \skeE),
$
i.e., everywhere expect along edges of the skeletons.
To complete the construction, for any given $ (t,\xi)\in \bigcup_{ \skeE\in\ske} \skeE $,
we define
\begin{align}
	\notag
	\Speedd_{\scll,\sclll}(t,\xi) 
	:= 
	\lim_{\d\downarrow 0}
	\ 
	\inf\Big\{ 
		\Speedd_{\scll,\sclll}(s,\zeta) \ : \
		(s,\zeta)\in ([0,T)\times\R) \setminus &\big( \bigcup\nolimits_{ \skeE\in\ske} \skeE\big),
\\
	\label{eq:ext}
		&
		s \geq t, 
		\,
		|s-t|+|\xi-\zeta| <\d
	\Big\}.
\end{align}
That is, we extend the value of $ \Speedd_{\scll,\sclll} $ 
onto the edges of the skeletons in such way that
$ \xi\mapsto \Speedd_{\scll,\sclll}(t,\xi) $ is lower-semicontinuous for each $ t\in[0,T) $,
and $ t\mapsto \Speedd_{\scll,\sclll}(t,\xi) $ is right-continuous for each $ \xi\in\R $.

This completes the construction of the speed function $ \Speedd_{\scll,\sclll} $.
We summarizes a few properties of $ \Speed_{\scll,\sclll} $ that will be useful in the sequel.
These properties are readily verified from the preceding construciton.
\begin{align}
	&
	\label{eq:speedd:rg}
	0< \minlam=:\inf_{\triangle\in\Sigma} \linlam \leq 
	\Speedd_{\scll,\sclll}(t,\xi) \leq \sup_{\triangle\in\Sigma} \linlam \leq \maxlam, 
	\quad \forall (t,\xi)\in [0,T)\times\R,
\\
	\label{eq:speedd:limit}
	&
	\lim_{\scll\to\infty}
	\lim_{\sclll\to\infty}
	\sum_{\triangle\in\Sigma^*}
	\int_{\triangle}
	\big|\Speedd_{\scll,\sclll}- (\linlam\vee 1) \big| dt d\xi
	=0.
\end{align}

\subsection{Estimating $ \hl{\Speedd_{\scll,\sclll}}{\gIC} $}
\label{sect:speedd:hl}
Having constructed $ \Speedd_{\scll,\sclll} $,
in this subsection, 
we verify that the resulting Hopf--Lax function $ \hl{\Speedd_{\scll,\sclll}}{\gIC} $ 
does approximate the piecewise linear function $ \ling $.
More precisely, we show in Proposition~\ref{prop:speeddhl} in the following that,
under the iterated limit $ \sclll\to\infty $, $ \scll\to\infty $,
$ \hl{\Speedd_{\scll,\sclll}}{\gIC} $ converges to $ g $.

Recall from~\eqref{eq:slab} and \eqref{eq:slab:} 
the definitions of the slabs $ \slab_{i} $ and $ \slab_{i,i'} $, 
together with the corresponding $ \undt_i $, $ \bart_i $, $ \undt_{i,i'} $, $ \bart_{i,i'} $
from~\eqref{eq:undbart} and \eqref{eq:undbart:}.
Closely related to $ \Speedd_{\scll,\sclll} $ 
is the piecewise linear function $ \lingi^{i,i'}_{\scll,\sclll}:\slab_{i,i'}\to\R $,
defined by
\begin{subequations}
\label{eq:lingi}
\begin{align}
	\label{eq:lingi:}
	\lingi^{i,i'}_{\scll,\sclll}& \in C(\slab_{i,i'},\R),
\\
	&
	\label{eq:lingi::}
	\nabla\lingi^{i,i'}_{\scll,\sclll}
	\big|_{\region^\circ\cap\slab_{i,i'}}
	=
	(\kappa_\region,\rho_\region),
	\quad
	\forall \region\in \prtn_i,
\\
	\label{eq:lingi:::}
	&\lingi^{i,i'}_{\scll,\sclll}(\undt_{i},0)= \ling(\undt_{i},0).
\end{align}
\end{subequations}
Indeed, \eqref{eq:lingi} admits \emph{at most one} such $ \lingi^{i,i'}_{\scll,\sclll} $ .
On the other hands,
The identities~\eqref{eq:RH:ver:}--\eqref{eq:RH:diag:} 
guarantee the existence of $ \lingi^{i,i'}_{\scll,\sclll} $ that satisfies \eqref{eq:lingi}.

Recall from~\eqref{eq:cone} that $ \calC(t,\xi) $ denote the light cone going back from $ (t,\xi) $.
In the following we will often work with on domain
\begin{align}
	\label{eq:trapD}
	\trapD := \{ (t,\xi) : t\in[0,T], \xi \in [-(T-t)\maxlam-r_*,r_*+(T-t)\maxlam] \}.
\end{align}
This is the smallest region in $ [0,T]\times\R $ that contains $ [0,T]\times[-r_*,r_*] $
and enjoys:
\begin{align}
	\label{eq:trapD:cone}
	&
	\calC(t,\xi) \subset \trapD,	\quad \forall (t,\xi)\in\trapD,
\end{align}
Note also that $ [0,T]\times [-r_*,r_*]\subset\trapD\subset [0,T]\times[-r^*,r^*] $.

The following result shows that, the Hopf--Lax function $ \hll{\Speedd_{\scll,\sclll}}{f}{\undt_{i,i'}} $
actually coincides with the piecewise linear function $ \lingi^{i,i'}_{\scll,\sclll} $,
provided that the initial condition $ f $ agrees with $ \lingi^{i,i'}_{\scll,\sclll} $.
\begin{lemma}
\label{lem:lingiSpeedd}
	Fix $ i\in\{ 1,\ldots,{\scl}_*\} $, $ i'\in\{4,\ldots,{\scll}-3\} $ and $ f\in\Splip $.
	If
	$
		f(\xi) = \lingi^{i,i'}(\undt_{i,i'},\xi),
	$
		$\forall (\undt_{i,i'},\xi) \in \trapD$,	
	then
	\begin{align*}
		\hll{\Speedd_{\scll,\sclll}}{f}{\undt_{i,i'}} \big|_{\slab_{i,i'}\cap\trapD}
		=
		\lingi^{i,i'}_{\scll,\sclll}\big|_{\slab_{i,i'}\cap\trapD}.
	\end{align*}
\end{lemma}
\begin{proof}
To simplify notations, throughout this proof we write $ \lingi^{i,i'}_{\scll,\sclll}=\lingi $.
Let
\begin{align}
	\label{eq:s*}
	t_\star := \sup\big\{ 
		s\in[\undt_{i,i'},\bart_{i,i'}] : 
		\hll{\Speedd_{\scll,\sclll}}{f}{\undt_{i,i'}}(t,\xi)=\lingi(t,\xi), 
		\,
		\forall (t,\xi)
		\in
		\slab_{i,i'}\cap\trapD
	\big\}
\end{align}
denote the first time when the desired property fails. 
Our goal is to show $ t_\star=\bart_{i,i'} $.
To this end, we advance $ t_\star $ by the small amount
$
	\sigma_\star := b''_{\scll}/(\maxlam+\frac{b}{\tau})
$
and consider a generic point $ (t_0,\xi_0) \in \trapD\cap([t_\star,(t_\star+\sigma_\star)\wedge \bart_{i,i'}]\times\R) $.
Let $ f_\star(\xi) := \hll{\Speedd_{\scll,\sclll}}{f}{s_{i,i'}}(t_\star,\xi) $
denote the profile at time $ t_\star $.
Apply Lemma~\ref{lem:hl}\ref{enu:hl:loc} for $ (s_0,s_1)=(\undt_{i,i'},t_\star) $, we write
\begin{align}
	\label{eq:hlmatching} 
	\hll{\Speedd_{\scll,\sclll}}{f}{\undt_{i,i'}}(t_0,\xi_0) 
	= 
	\hll{\Speedd_{\scll,\sclll}}{f_\star}{t_\star}(t_0,\xi_0).
\end{align}
Recall the notation $ \calC'(s_0,t_0,\xi_0) $ from~\eqref{eq:cone:},
and write $ \calC' := \calC(t_\star,t_0,\xi_0) $ to simplify notations.
Let $ \calX := \{\xi: (\undt_{i,i'},\xi)\in\calC(t_0,\xi_0)\} $ denote the intersection
of the light cone with the lower boundary of $ \slab_{i,i'} $.
As shown in Lemma~\ref{lem:hl}\ref{enu:hl:match},
the r.h.s.\ of~\eqref{eq:hlmatching} depends on $ (\Speedd_{\scll,\sclll},f_\star) $
only through $ (\Speed_{\scll,\sclll}|_{\calC'},f_\star|_\calX) $.
Our next step is to utilize this localization of dependence to evaluate the expression~\eqref{eq:hlmatching}.
First, With $ t_\star $ defined in~\eqref{eq:s*}, we necessarily have
$ f_\star(\xi) = \lingi(t_\star,\xi) $, $ \forall (t_\star,\xi)\in \trapD $.
Also, by~\eqref{eq:trapD:cone}, $ \set{\undt_{i,i'}}\times\calX\subset\trapD $.
Consequently,
\begin{align*}
	f_\star|_{\calX} = \lingi(t_\star,\Cdot)|_{\calX}.
\end{align*}
Next, recall the definition of the skeleton $ \ske $ from~\eqref{eq:ske}.
We claim that $ \calC' $ intersects with at most one edge of $ \ske $, i.e.,
\begin{align}
	\label{eq:intersect}
	\# \big\{ \calC'\cap\skeE \neq\emptyset : \skeE\in\ske \big\} \leq 1. 
\end{align}
To see why, first note that, since $ \calC' \subset (\undt_{i,i'},\bart_{i,i'})\times\R $,
the restricted cone $ \calC' $ does not intersect with horizontal edges of $ \ske $,
and it suffices to consider vertical and diagonal edges of $ \ske $ within the slab $ \slab_{i,i'} $.
From the preceding construction of $ \region\in\prtn_i $,
we see that vertical and diagonal edges in $ \ske $ are at least horizontally distance $ b''_{\scll,\sclll} $ apart. 
Viewed as spacetime trajectories, vertical edges travel at zero velocity,
and diagonal edges travel at velocity $ \frac{b}{\tau} $.
Since the cone $ \calC' $ goes backward in time at a speed of at most $ \maxlam $,
the time span of $ \calC' $ has to be more than $ \frac{b''_{\scll,\sclll}}{\maxlam+\frac{b}{\tau}}=:\sigma_\star $
for $ \calC' $ to intersect with two vertical or diagonal edges in $ \ske $.
This, with $ \calC'\subset[\undt_{i,i'},\undt_{i,i'}+\sigma_\star] $, does not happen,
so \eqref{eq:intersect} follows.

Recall the four special types \ref{enu:explicit:homo}--\ref{enu:explicit:shk} of $ (\speed,f) $
from before Lemma~\ref{lem:explicit}, in Section~\ref{sect:speedd:hlprpty}.
With \eqref{eq:intersect} being the case,
the pair $ (\Speedd_{\scll,\sclll},\lingi(t_\star)) $,
when restricted to $ \calC'\times\calX $, coincides with $ (\speed,f) $ of the form considered in 
Section~\ref{sect:speedd:hlprpty} for $ s_0=t_\star $, i.e.,
\begin{align}
	\label{eq:matching:pair}
	\big( \Speedd_{\scll,\sclll}|_{\calC'},\lingi(t_\star)|_{\calX}\big)
	= 
	\big(\speed|_{\calC'}, f|_{\calX}\big).
\end{align}
The condition~\eqref{eq:RH:ver}--\eqref{eq:RH:shk} holds thanks to~\eqref{eq:RH:ver:}--\eqref{eq:RH:diag:}.
Given~\eqref{eq:matching:pair},
we apply Lemma~\ref{lem:hl}\ref{enu:hl:match} with 
$ (\speed_1,f_1;\speed_2,f_2)=(\Speedd_{\scll,\sclll},f_\star;\speed_0,f_0) $,
to replace $ (\Speedd_{\scll,\sclll},f_\star) $ with $ (\speed,f) $ in~\eqref{eq:hlmatching} .
This yields
\begin{align}
	\hll{\Speedd_{\scll,\sclll}}{f}{\undt_{i,i'}}(t_0,\xi_0) 
	=
	\label{eq:hl:matching}
	\hll{S}{f}{\undt_{i,i'}}(t_0,\xi_0).
\end{align}
Further, thanks to~\eqref{eq:region:lkr}, \eqref{eq:char:ver:}--\eqref{eq:resd:r12},
the conditions~\eqref{eq:char:ver}--\eqref{eq:char:shk} hold.
This being the case, we apply Lemma~\ref{lem:explicit} for $ s_0=t_\star $ 
to conclude $ \hll{S}{f}{\undt_{i,i'}}(t_0,\xi_0)= \lingi(t_0,\xi_0) $.
This together with~\eqref{eq:hl:matching} gives
\begin{align*}
	\hll{\Speedd_{\scll,\sclll}}{f}{\undt_{i,i'}}(t_0,\xi_0) 
	=
	\lingi(t_0,\xi_0).
\end{align*}
As this holds for all $ (t_0,\xi_0)\in ([t_\star,(t_\star+\sigma_\star)\wedge \bart_{i,i'}]\times\R)\cap\trapD $,
we must have that $ t_\star \geq (t_\star +\sigma_\star)\wedge \bart_{i,i'} $.
This forces the desired result $ t_\star =\bart_{i,i'} $ to be true.
\end{proof}

Given Lemma~\ref{lem:lingiSpeedd},
our next step is to show that $ \lingi_{\scll,\sclll}^{i,i'} $ approximates $ \ling $.
The this end, it is convenient to consider an analog $ \lingiC^i_\scll $ of $ \lingi_{\scll,\sclll}^{i,i'} $,
defined as follows.
Recall from~\eqref{eq:prtnC} that $ \prtnC_i $ denotes the coarser version of the partition $ \prtn_i $.
We consider unique the piecewise linear function $ \lingiC^i_\scll:\slab_i\to\R $
with gradient given by $ (\kappa_{\region},\rho_{\region}) $ on each $ \region\in\prtnC_i $, i.e.,
\begin{subequations}
\label{eq:lingiC}
\begin{align}
	\label{eq:lingiC:}
	\lingiC^{i}_{\scll}:& \in C(\slab_i,\R),
\\
	&
	\label{eq:lingiC::}
	\nabla\lingiC^{i}_{\scll}\big|_{\region^\circ}
	=
	(\kappa_\region,\rho_\region),
	\quad
	\forall \region\in \prtnC_i,
\\
	\label{eq:lingiC:::}
	&\lingiC^{i}_{\scll}(\undt_{i},0)= \ling(\undt_{i},0),
\end{align}
\end{subequations}

\begin{lemma}
\label{lem:lingi}
For fixed $ i\in\{ 1,\ldots,{\scl}_*\} $,
\begin{align}
	\label{eq:gi:gii}
	&
	\lim_{\sclll\to\infty} \Big(\sup_{\slab_{i,i'}} |\lingi^{i,i'}_{\scll,\sclll}-\lingiC^{i}_\scll| \Big) = 0,
	\quad
	i'=4,\ldots,\scll-3,
	\quad
	\text{for each fixed } \scll<\infty,
\\	
	\label{eq:g:gi}
	&
	\lim_{\scll\to\infty} \Big(\sup_{\slab_{i}} |\lingiC^{i}_{\scll}-\ling| \Big) = 0.
\end{align}
\end{lemma}
\begin{proof}
We first establish~\eqref{eq:gi:gii}.
Since the partition $ \prtnC_i $ differs from $ \prtn_i $ only on 
those reduced triangles $ \rTri $ with $ \lambda_{\rTri}<1 $, we have
\begin{align}
	\label{eq:drmatch}
	\nabla \lingi^{i,i'}_{\scll,\sclll}\big|_{\region^\circ} = \nabla \lingiC^{i}_\scll\big|_{\region^\circ},
	\quad
	\forall \region\in \prtnC\setminus\{\rTri\in\RTri_i: \lambda_\rTri <1\}.
\end{align}
On each $ \rTri $ with $ \lambda_\rTri<1 $, 
the  $ \prtn $ invokes the intermittent zones
$ \intm_{i',\star}(\rTri) $, $ \intm_{i',j''}(\rTri) $
and residual regions $ \resd_{i',j''}^1(\rTri) $, $ \resd_{i',j''}^2(\rTri) $;
see Figure~\ref{fig:zone_iterm}--\ref{fig:zone_resd}.
Referring to the definition of $ (\kappa_\region,\rho_\region)_{\region\in\prtn_i} $ in the preceding, we have that
\begin{align}
	\label{eq:drmatch:}
	\kappa_{\rTri}
	= \kappa_{\intm_{i',\star}(\rTri)}
	= \kappa_{\intm_{i',j''}(\rTri)}
	= \kappa_{\resd^1_{i',j''}(\rTri)}
	= \kappa_{\resd^2_{i',j''}(\rTri)},
	\quad
	\rho_{\rTri}
	= \rho_{\intm_{i',\star}(\rTri)}
	= \rho_{\intm_{i',j''}(\rTri)},
\end{align}
for all relevant $ i',j'' $.
Also, the identity~\eqref{eq:resd:r12} implies that,
for each $ \resd =\resd_{i',j''}(\rTri) $,
\begin{align}
	\label{eq:drmatch::}
	\int_{\xi^-_{\resd}}^{\xi^+_\resd} (\lingi^{i,i'}_{\scll,\sclll})_{\xi}(t,\xi) d\xi
	=
	\int_{\xi^-_{\resd}}^{\xi^+_\resd} (\lingiC^i_\scll)_{\xi}(t,\xi) d\xi,
	\
	\forall t\in[\und t_{\resd},\bar t_\resd],
	\quad
	\text{where }
	\resd = [\und t_{\resd},\bar t_\resd]\times[\xi^-_\resd,\xi^+_\resd].
\end{align}
That is, the integrals of $ (\lingi_{\scll,\sclll})_\xi^i $ and $ (\lingiC^i_\scll)_{\xi} $
along any horizontal line segment passing through $ \resd $ do match.
To briefly summarize,
\eqref{eq:drmatch}--\eqref{eq:drmatch:} shows that
the derivatives of $ \lingi^{i,i'}_{\scll,\sclll} $ and $ \lingiC^i_{\scll} $ match everywhere they are defined, 
except for the $ \xi $-derivatives in $ \resd_{i',j''}(\rTri) $,
and \eqref{eq:drmatch:} gives a matching of the $ \xi $-derivatives in $ \resd_{i',j''}(\rTri) $
in an integrated sense.
These properties together with~\eqref{eq:lingi:::} and~\eqref{eq:lingiC:::} gives that
\begin{align}
	\label{eq:gi:gii:}
	\lingi^{i,i'}_{\scll,\sclll}(t,\xi) = \lingiC^{i}_{\scll}(t,\xi),
	\quad
	\forall (t,\xi) \in \slab_i \setminus 
	\bigcup\Big\{ (\resd_{i',j''}(\rTri))^\circ : \lambda_\rTri<1, \, \text{ relevant }i',j'' \Big\}.
\end{align}
Since each $ \resd_{i',j''}(\rTri) $
has a width of $ (\sclll-1)b''_{\scll,\sclll} = \frac{\sclll-1}{\sclll^2}b'_\scll $,
and since $ \lingiC^i_\scll $ is continuous,
letting $ \sclll\to\infty $ in~\eqref{eq:gi:gii:} gives~\eqref{eq:gi:gii}.

Next, to prove~\eqref{eq:g:gi}, fix arbitrary $ (t_0,\xi_0)\in\slab_{i} $,
and express $ \lingiC^{i}_{\scll}(t_0,\xi_0) $ and $ \ling(t_0,\xi_0) $
in terms of the integral of their derivatives
along the vertical line segment $ (\undt_i,0)\edge(t_0,0) $
and the horizontal line segment $ (t_0,0)\edge(t_0,\xi_0) $, i.e.,
\begin{align}
	\label{eq:lingic:dr}
	\lingiC^i_\scll(t_0,\xi_0)
	&=
	\lingiC^i_\scll(\undt_i,0) + \int_{\undt_i}^{t_0} (\lingiC^i_\scll)_t(t,0) dt
	+
	\int_0^{\xi_0} (\lingiC^i_\scll)_\xi(t_0,\xi) d\xi,
\\
	\label{eq:ling:dr}
	\ling(t_0,\xi_0)
	&=
	\ling(\undt_i,0) + \int_{\undt_i}^{t_0} \ling_t(t,0) dt
	+
	\int_0^{\xi_0} \ling_\xi(t_0,\xi) d\xi.	
\end{align}
Note that the line segment $ (\undt_i,0)\edge(t_0,0) $ 
sits within a vertical line segment of the triangulation $ \Sigma $; see Figure~\ref{fig:Sigma}.
Even though $ \ling $ is in general not smooth along edges of $ \Sigma $,
$ g_t $ does exist along \emph{vertical} edges of $ \Sigma $.
More explicitly, letting $ \triangle_\star\in\Sigma $ 
be a neighboring triangle of the line segment $ (\undt_i,0)\edge(t_0,0) $,
we have that $ \ling_t(t,0)|_{t\in(\undt_i,t_0)} = \kappa_{\triangle_*} $.
Likewise, letting $ \buffer_\star $ be the (unique) buffer zone that contains 
$ (\undt_i,0)\edge(t_0,0) $, we have $ (\lingiC^i_\scll)_t(t,0)|_{t\in(\undt_i,t_0)} = \kappa_{\buffer_*} $.
For any triangle $ \triangle\in\Sigma $ that intersects with the slab $ \slab_i $, 
let $ e\in\VE_i $ denote its neighboring vertical edge,
and let $ \rTri\subset\triangle $ denote the corresponding reduced triangle.
Referring the definition of $ (\kappa_\region,\rho_\region)_{\region\in\prtnC_i} $ in the preceding,
we have that
\begin{align}
	\label{eq:drmatch:C}
	(\linkap,\linrho)
	=
	(\kappa_\rTri,\rho_\rTri)
	=
	(\kappa_{\buffer_e},\rho_{\buffer_{e}}).
\end{align}
In \eqref{eq:lingic:dr}--\eqref{eq:ling:dr}, 
use~\eqref{eq:drmatch:C} for $ (\triangle,\buffer_e)=(\triangle_\star,\buffer_\star) $ 
to match the $ t $-derivatives $ (\lingiC^i_\scll)_t $ and $ \ling_t $,
use~\eqref{eq:drmatch:C} to match the $ \xi $-derivatives 
$ (\lingiC_\scll)_\xi $ and $ \ling_\xi $ on those reduced triangles 
$ \rTri $ along the line segment $ (t_0,0)\edge(t_0,\xi_0) $
(recall that $ \ling_\xi|_{\triangle^\circ}=\linrho $, $ \forall \triangle\in\Sigma $),
and take the difference of the result, using~\eqref{eq:lingiC:::}.
We arrive at
\begin{align*}
	|\lingiC^i_\scll(t_0,\xi_0)-\ling(t_0,\xi_0)|
	=	
	\Big|
		\int_0^{\xi_0} 
		\big( (\lingiC^i_\scll)_\xi- \ling_\xi\big)(t_0,\xi) \ind_\set{(t_0,\xi)\notin\rTri,\forall\rTri\in\RTri_i} d\xi
	\Big|
	\leq
	(\scl*+1)2b'_\scll
	\Vert (\lingiC^i_\scll)_\xi- \ling_\xi \Vert_\infty.
\end{align*}
With $ (\lingiC^i_\scll)_\xi, \ling_\xi $ being $ [0,1] $-valued,
and with $ b'_\scll = \frac{b}{\scll} $,
letting $ \scll\to\infty $ gives~\eqref{eq:g:gi}.
\end{proof}

A useful consequence of Lemma~\ref{lem:lingiSpeedd}--\eqref{lem:lingi} is the following result.
It controls the deviation of the Hopf--Lax function 
$ \hll{f_{\scll,\sclll}}{\Speedd_{\scll,\sclll}}{\undt_{i}} $ from $ \ling $
in terms of the deviation of a given initial condition $ f_{\scll,\sclll} $.
\begin{lemma}
\label{lem:speeddhl}
Let $ \{f_{\scll,\sclll}\}_{\scll,\sclll} \subset \Splip $.
For any fixed $ i\in\{ 1,\ldots,{\scl}_*\} $,
\begin{align}
	\label{eq:speeddhl}
	\limsup_{\scll\to\infty} \limsup_{\sclll\to\infty} 
	\Big(
		\sup_{\slab_{i}\cap\trapD} 
		\big| \hll{\Speedd_{\scll,\sclll}}{f_{\scll,\sclll}}{\undt_{i}}-\ling\big|
	\Big)
	\leq
	\limsup_{\scll\to\infty} \limsup_{\sclll\to\infty} 
	\sup_{(\undt_{i},\xi)\in\slab_{i}\cap\trapD} |f_{\scll,\sclll}(\xi)-\ling(\undt_{i},\xi)|.		
\end{align}
\end{lemma}
\begin{proof}
Throughout this proof, to simplify notations, 
we write $ G:= \hll{\Speedd_{\scll,\sclll}}{f_{\scll,\sclll}}{\undt_{i}} $.
Let us first setup a few notations.
For $ i'=4,\ldots,\scll-3 $, 
let $ f^{i'} := G(\undt_{i,i'}) $ denote the fixed time profile of the Hopf--Lax function at $ \undt_{i,i'} $.
Consider also the fixed time profile $ \gamma^{i'} := \lingi^{i,i'}_{\scll,\sclll}(\undt_{i,i'}) $ 
of $ \lingi^{i,i'}_{\scll,\sclll} $ at time $ \undt_{i,i'} $.
The function $ \gamma^{i'} $ is defined on
\begin{align*}
	\Xi^{i'}:=\{ \xi: (\undt_{i,i'},\xi)\in\slab_{i,i'}\cap\trapD \},
\end{align*}
and we extend the function beyond $ \Xi^{i'} $ in such away that $ \gamma^{i'}\in\Splip $.
The precise way of extending $ \gamma^{i'} $ does not matter,
as long as the result is $ \Splip $-valued.
We omit the dependence of $ G $, $ f^{i'} $ and $ \gamma^{i'} $ on $ \scll,\sclll $ to simplify notations.

Instead of showing~\eqref{eq:speeddhl}, we show 
\begin{align}
	\label{eq:speeddhl:}
	\limsup_{\sclll\to\infty} 
	\Big(
		\sup_{\slab_{i}\cap\trapD} 
		\big| G-\lingiC^{i}_{\scll}\big|
	\Big)
	\leq
	\limsup_{\sclll\to\infty} 
	\sup_{(\undt_{i},\xi)\in\slab_{i}\cap\trapD} |f_{\scll,\sclll}(\xi)-\lingiC^{i}_{\scll}(\undt_{i},\xi)|.		
\end{align}
By \eqref{eq:g:gi}, 
the function $ \lingiC^{i}_{\scll} $ uniformly approximates $ \ling $ on $ {\slab_{i}\cap\trapD} $ as $ \scll\to\infty $.
This being the case, 
the desired result~\eqref{eq:speeddhl} follows by letting $ \scll\to\infty $ in~\eqref{eq:speeddhl:}.

To prove~\eqref{eq:speeddhl:},
we fix $ i'=4,\ldots,\scll-3 $,
and proceed to bound the difference $ | G-\lingiC^{i}_{\scll}| $
on each $ \slab_{i,i'}\cap\trapD $.
First, by Lemma~\ref{lem:hl}\ref{enu:hl:loc} for $ (s_0,s_1)=(\undt_i,\undt_{i,i'}) $,
the Hopf--Lax function $ G $ localizes onto $ \slab_{i,i'} $ as
$  
	G|_{\slab_{i,i'}} = \hll{\Speedd_{\scll,\sclll}}{f^{i'}}{\undt_{i,i'}}|_{\slab_{i,i'}}.
$
Given this property,
applying Lemma~\ref{lem:hl}\ref{enu:hl:Linf} with $ (f_1,f_2) = (f^{i'},\gamma^{i'}) $ and $ s_0=\undt_{i,i'} $, 
we obtain
\begin{align}
	\notag
	&\big| 
		G(t_0,\xi_0)
		-
		\hll{\Speedd_{\scll,\sclll}}{\gamma^{i'}}{\undt_{i,i'}}(t_0,\xi_0)
	\big|
\\
	=&
	\label{eq:speeddhl:4}
	\big| 
		\hll{\Speedd_{\scll,\sclll}}{f^{i'}}{\undt_{i,i'}}(t_0,\xi_0)
		-
		\hll{\Speedd_{\scll,\sclll}}{\gamma^{i'}}{\undt_{i,i'}}(t_0,\xi_0)
	\big|
	\leq
	\sup_{(\undt_{i,i'},\xi)\in\calC(t_0,\xi_0)} |f^{i'}(\xi)-\gamma^{i'}(\xi)|,
\end{align}
for all $ (t_0,\xi_0)\in\slab_{i,i'} $.
Recall the definition of $ \Xi^{i'} $ from the preceding.
By~\eqref{eq:trapD:cone}, for each $ (t_0,\xi_0)\in\slab_{i,i'}\cap\trapD $,
we have that $ \{\xi:(\undt_{i,i'},\xi)\in\calC(t_0,\xi_0)\}\subset\Xi^{i'} $.
Using this property, 
we take the supremum of~\eqref{eq:speeddhl:4} over $ (t_0,\xi_0)\in\slab_{i,i'}\cap\trapD $ to get
\begin{align*}
	\sup_{\slab_{i,i'}\cap\trapD}
	\big| 
		G
		-
		\hll{\Speedd_{\scll,\sclll}}{\gamma^{i'}}{\undt_{i,i'}}
	\big|
	&\leq
	\sup_{\Xi^{i'}} |f^{i'}-\gamma^{i'}|.
\end{align*}
Next, using Lemma~\ref{lem:lingiSpeedd} for $ f=\gamma^{i'} $,
we replace the expression $ \hll{\Speedd_{\scll,\sclll}}{\gamma^{i'}}{\undt_{i,i'}} $ 
with $ \lingi^{i,i'}_{\scll,\sclll} $, and write
\begin{align}
	\label{eq:speeddhl:4::}
	\sup_{\slab_{i,i'}\cap\trapD}
	\big| 
		G
		-
		\lingi^{i,i'}_{\scll,\sclll}
	\big|
	\leq
	\sup_{\Xi^{i'}} |f^{i'}-\gamma^{i'}|
	=
	\sup_{\Xi^{i'}} 
	\big|
		G(\undt_{i,i'})
		-\lingi^{i,i'}_{\scll,\sclll}(\undt_{i,i'})
	\big|.
\end{align}
Further, by \eqref{eq:gi:gii}, the function $ \lingi^{i,i'}_{\scll,\sclll} $ 
uniformly approximates $ \lingiC^{i}_\scll $ on $ \slab_{i,i'}\cap\trapD $.
This being the case, we let $ \sclll\to\infty $ in~\eqref{eq:speeddhl:4::},
and replace each $ \lingi^{i,i'}_{\scll,\sclll} $ with $ \lingiC^{i}_\scll $ to get
\begin{align}
	\label{eq:speeddhl:4:::}
	\limsup_{\sclll\to\infty}
	\sup_{\slab_{i,i'}\cap\trapD}
	\big| 
		G
		-
		\lingiC^{i}_{\scll}
	\big|
	\leq
	\limsup_{\sclll\to\infty} 
	\sup_{\Xi^{i'}} |G(\undt_{i,i'})-\lingiC^{i}_{\scll}(\undt_{i,i'})|,
	\quad
	i'=4,\ldots,\scll-3.
\end{align}
Indeed, since $ \set{\undt_{i,i'+1}}\times\Xi^{i'+1} \subset \slab_{i,i'}\cap\trapD $, we have
$	
	\sup_{\Xi^{i'+1}} | G(\undt_{i,i'})-\lingiC^{i}_{\scll}(\undt_{i,i'}) |
	\leq
	\sup_{\slab_{i,i'}\cap\trapD} | G-\lingiC^{i}_{\scll} |.
$
Given this property, inductively applying~\eqref{eq:speeddhl:4:::} for $ i'=4,\ldots,\scll-3 $
gives the desired result~\eqref{eq:speeddhl:}.
\end{proof}

We now show that $ \hl{\Speedd_{\scll,\sclll}}{\gIC} $ 
uniformly approximates $ \ling $ over $ [0,T]\times[-r_*,r_*] $. More precisely, 
\begin{proposition}
\label{prop:speeddhl}
We have that
$ \displaystyle \quad
	\limsup_{\scll\to\infty} \limsup_{\sclll\to\infty}
	\Big( \sup_{[0,T]\times[-r_*,r_*]} \big| \hl{\Speedd_{\scll,\sclll}}{\gIC}-\ling \big| \Big) = 0.
$
\end{proposition}
\begin{proof}
Throughout this proof, to simplify notations, 
we write $ G:= \hl{\Speedd_{\scll,\sclll}}{\gIC} $ for the Hopf--Lax function.
Recall from Section~\ref{sect:speedd:cnst} that
$ [0,T]\times[-r^*,r^*] $ is divided into a stalk of slabs $ \slab_i $,
with transition zones $ \tz_i $, $ i=0,\ldots,\scl_* $ in between the slabs;
see Figure~\ref{fig:zone_transi}.
By Lemma~\ref{lem:hl}\ref{enu:hl:loc} for $ s_0=\undt_i $,
the function $ G $ localizes onto $ \slab_i $ as 
$ G|_{\slab_i} = \hll{\Speedd_{\scll,\sclll}}{G(\undt_{i})}{\undt_{i}}|_{\slab_i} $.
With this property, we rewrite Lemma~\ref{lem:speeddhl} for $ f_{\scll,\sclll}=G $ as
\begin{align}
	\label{eq:speeddhl:re}
	\limsup_{\scll\to\infty} \limsup_{\sclll\to\infty} 
	\Big(
		\sup_{\slab_{i}\cap\trapD} 
		\big| G-\ling\big|
	\Big)
	\leq
	\limsup_{\scll\to\infty} \limsup_{\sclll\to\infty} 
	\sup_{(\undt_{i},\xi)\in\slab_{i}\cap\trapD} |G(\undt_i,\xi)-\ling(\undt_{i},\xi)|,
\end{align}
for each $ i=1,\ldots,\scl_* $.
Next, fix a transition zone $ \tz_i $, $ i\in\{0,\ldots,\scl_*\} $ as in~\eqref{eq:tz}, 
and it write as
\begin{align*}
	\tz_i := \big( [\bart_{i},\undt_{i+1}]\cap[0,T] \big) \times [-r^*,r^*]
	= \big[ (\bart_{i})_+,\, \undt_{i+1}\wedge T \big]\times[-r^*,r^*].
\end{align*}
By~\eqref{eq:nondg:lin}, $ \ling $ is uniformly Lipschitz;
and by Lemma~\ref{lem:hl}\ref{enu:hl:sp} the Hopf--Lax function $ G $ is also uniformly Lipschitz.
Consequently, there exists a fixed constant $ c<\infty $, such that
\begin{align*}
	|G(t,\xi) - G(s_1,\xi)| \leq c|t-(\bart_{i})_+| \leq 6c \tau'_\scll, 
	\quad
	|g(t,\xi) - g(s_1,\xi)| \leq c|t-(\bart_{i})_+| \leq 6c \tau'_\scll,
	\quad
	\forall (t,\xi) \in \tz_i.
\end{align*}
This gives
\begin{align*}
	\sup_{\tz_i\cap\trapD} |G-g|
	\leq
	\sup_{((\bart_{i})_+,\xi)\in\tz_i\cap\trapD} \big|G((\bart_{i})_+,\xi)-g((\bart_{i})_+,\xi)\big|	
	+
	12 \tau'_\scll.
\end{align*}
Taking the iterated limit $ \sclll\to\infty $, $ \scll\to\infty $, we obtain
\begin{align}
	\label{eq:speeddhl:re:}
	\limsup_{\scll\to\infty} \limsup_{\sclll\to\infty} 
	\Big(
		\sup_{\tz_{i}\cap\trapD} 
		\big| G-\ling\big|
	\Big)	
	\leq
	\limsup_{\scll\to\infty} \limsup_{\sclll\to\infty} 
	\sup_{((\bart_{i})_+,\xi)\in\tz_i\cap\trapD} \big|G((\bart_{i})_+,\xi)-g((\bart_{i})_+,\xi)\big|,
\end{align}
for each $ i=0,\ldots,\scl_* $.
As mentioned earlier, $ [0,T]\times[-r_*,r_*] $
is decomposed into a stalk of transition zones and the slabs,
from bottom to top as $ \tz_0\cup\slab_1\cup\tz_1\cup\ldots\slab_{\scl_*}\cup\tz_{\scl_*} $.
This being the case, applying~\eqref{eq:speeddhl:re}--\eqref{eq:speeddhl:re:}
inductively, similarly to the proof of Lemma~\ref{lem:speeddhl},
we obtain
\begin{align*}
	\limsup_{\scll\to\infty} \limsup_{\sclll\to\infty} 	
	\sup_{\trapD} \big| G-\ling\big|
	\leq
	\limsup_{\scll\to\infty} \limsup_{\sclll\to\infty} 	
	\sup_{(0,\xi)\in\trapD} \big| G(0,\xi)-\ling(0,\xi)\big|.	
\end{align*}
Indeed, since $ G(0,\xi) = \hl{\Speedd_{\scll,\sclll}}{\gIC}(0,\xi)=\ling(0,\xi) $,
the r.h.s.\ is zero.
This together with $ \trapD\supset [0,T]\times[-r_*,r_*] $ completes the proof.
\end{proof}

\section{Lower Bound: Construction of $ \Speed_{\scll,\sclll} $ and Proof of Proposition~\ref{prop:ent:up}}
\label{sect:speed}

\subsection{Constructing $ \Speed_{\scll} $ and estimating $ \hl{\Speed_{\scll,\sclll}}{\gIC} $}
\label{sect:speed:cnst}

We now construct the simple speed function $ \Speed_{\scll,\sclll} $ 
as an approximation of $ \Speedd_{\scll,\sclll} $.
Recall that $ \tau''_{\scll,\sclll} $ and $ b''_{\scll,\sclll} $ are the scales defined in~\eqref{eq:sclll}.
Consider the following partition of $ [0,T]\times\R $ that consists
of rectangles of height $ \tau''_{\scll,\sclll} $ and base $ b''_{\scll,\sclll} $:
\begin{align}
	\label{eq:Pi''}
	\Pi''_{\scll,\sclll} 
	:=
	\big\{ \square:= [(i''-1)\tau''_{\scll,\sclll},i''\tau''_{\scll,\sclll}]\times
	[(j''-1)b''_{\scll,\sclll},j''b''_{\scll,\sclll}] 
	: i''=1,\ldots,{\scl}_*\scll\sclll^2, j''\in\Z \big\}.
\end{align}
Recall from~\eqref{eq:prtnU} that $ \prtnU $ denotes a partition of $ [0,T]\times\R $,
and recall from~\eqref{eq:ske} the induced skeleton $ \ske $.
One readily check that, each rectangle $ \square\in\Pi''_{\scll,\sclll} $
is either contained in a region $ \region\in\prtnU $,
or intersects with a diagonal edge $ \skeE\in\ske $.
In the latter case, the edge $ \skeE $ goes through the upper-right and lower-left vertices of $ \square $.
Having noted these properties, we now define
\begin{align*}
	\Speed_{\scll,\sclll}\big|_{\square^\circ} := \Speedd_{\scll,\sclll}\big|_{\square^\circ},
	\quad
	\text{ if } \square\subset\region,
	\text{ for some } \region\in\prtnU;
\end{align*}
and if $ \square $ intersects with a diagonal edge $ \skeE\in\ske $,
we let $ \region^\pm\in\prtnU $ denote the neighboring regions of $ \skeE $, and set
\begin{align*}
	\Speed_{\scll,\sclll}\big|_{\square^\circ} := 
	(\lambda_{\region^-} \wedge \lambda_{\region^+})
	=	
	\inf_{\square^\circ}\Speedd_{\scll,\sclll}.
\end{align*} 
So far, we have defined the values of $ \Speed_{\scll,\sclll} $ on $ [0,T)\times\R $
except along edges of the rectangles $ \square\in\Pi''_{\scll,\sclll}  $.
To complete the construction, 
we extend the value of $ \Speed_{\scll,\sclll} $ onto $ [0,T)\times\R $ in the same way as in~\eqref{eq:ext}.
This defines an simple speed function, i.e., a function of the form~\eqref{eq:simple}.
Further, from~\eqref{eq:speedd:>r*}, \eqref{eq:speedd:rg}--\eqref{eq:speedd:limit}, we have
\begin{align}
	&
	\label{eq:speed:tail}
	\Speed_{\scll,\sclll}(t,\xi)|_{|\xi| > r^*} = 1,
\\
	&
	\label{eq:speed:rg}
	\Speed_{\scll,\sclll}\in [\minlam,\maxlam],
\\
	\label{eq:speed:limit}
	&
	\lim_{\scll\to\infty}
	\lim_{\sclll\to\infty}
	\sum_{\triangle\in\Sigma^*}
	\int_{\triangle} 
	\big|\Speed_{\scll,\sclll}-(\linlam\vee 1) \big| dt d\xi
	=0.
\end{align}

The following result gives the necessarily control on the Hopf-Lax function
$ \hl{\Speed_{\scll,\sclll}}{\gIC} $.
\begin{proposition}
\label{prop:speedhl}
For each fixed $ \scll<\infty $,
\begin{align}
	\label{eq:speedhl}
	\lim_{\sclll\to\infty} \sup_{[0,T]\times\R} \Big| \hl{\Speed_{\scll,\sclll}}{\gIC} - \hl{\Speedd_{\scll,\sclll}}{\gIC} \Big|
	=0.
\end{align}
In particular, by Proposition~\ref{prop:speeddhl},
\begin{align*}
	\limsup_{\scll\to\infty} \limsup_{\sclll\to\infty}
	\Big( \sup_{[0,T]\times[-r_*,r_*]} \big| \hl{\Speed_{\scll,\sclll}}{\gIC}-\ling \big| \Big) = 0.
\end{align*}
\end{proposition}
\begin{proof}
%
Consider a generic diagonal buffer zone $ \buffer_e $, $ e\in\cup_{i=1}^{\scl_*}\DE_i $,
and parametrize the zone as~\eqref{eq:bufferDE}.
Under the notations of~\eqref{eq:bufferDE},
we let
$
	\partial^\pm\buffer_e := 
	(\undt_i,(j-1)b\pm b'_\scll)\edge(\bart_i,kb\pm b'_\scll).
$
denote the right/left boundary of the buffer zone $ \buffer_e $,
and let
\begin{align*}
	\calU^\pm_e := \{ (t,\xi) : |\xi+(j-1)b\pm b'_\scll-\tfrac{b}{\tau}(t-\undt_i)| \leq b''_{\scll,\sclll}, \, t\in[\undt_i,\bart_i] \} 
\end{align*}
denote the regions of width $ 2b''_{\scll,\sclll} $ around $ \partial^\pm\buffer_e  $.
Referring the preceding definition of $ \Speed_{\scll,\sclll} $,
we see that $ \Speed_{\scll,\sclll}\neq\Speedd_{\scll,\sclll} $ only within the regions $ \calU_{e} $, $ e\in\DE $.
Showing~\eqref{eq:speedhl} hence amounts to showing
that such a discrepancy does not affect the resulting Hopf--Lax function as $ \sclll\to\infty $.

Let $ \sigma_{\scll} := b'_\scll(2(\maxlam+\frac{b}{\tau})) $.
Divide each $ \calU^\pm_e $ into smaller parts $ \calU^\pm_{e,i'} $, 
each of height at most $ \sigma_\scll $:
\begin{align*}
	\calU^\pm_{e,i'} := \calU^\pm_{e} \cap 
	([(i'-1)\sigma_\scll,i'\sigma_\scll)\times\R).
\end{align*}
Let $ \mathscr{U} := \{\calU^\pm_{e,i'}: e\in\cup_{i=1}^{\scl_*}\DE_i,i'=1,2,\ldots\} $
denote the collection of these regions, and enumerate them as $ \mathscr{U} = \{ \calU_1,\ldots,\calU_{\#\mathscr{U}} \} $.
We replace the value of $ \Speed_{\scll,\sclll} $ on each $ \calU_k $ by that of $ \Speedd_{\scll,\sclll} $ sequentially, i.e.,
\begin{align*}
	\speed^0_{\scll,\sclll} := \Speed_{\scll,\sclll},
	\quad
	\speed^{k}_{\scll,\sclll} = \speed^{k-1}_{\scll,\sclll}\ind_{\calU_{k}^c} + \Speedd_{\scll,\sclll}\ind_{\calU_{k'}}.
\end{align*}
Under these notations,
we telescope the difference $ \hl{\Speed_{\scll,\sclll}}{\gIC}-\hl{\Speedd_{\scll,\sclll}}{\gIC} $ accordingly as
\begin{align}
	\label{eq:speedhl:tele}
	\hl{\Speed_{\scll,\sclll}}{\gIC}-\hl{\Speedd_{\scll,\sclll}}{\gIC}
	=
	\sum_{k=1}^{\#\mathscr{U}}
	\big( \hl{\speed^{k}_{\scll,\sclll}}{\gIC}-\hl{\speed^{k-1}_{\scll,\sclll}}{\gIC} \big).
\end{align}

Given the decomposition, we now fix $ k\in\{1,\ldots,\#\mathscr{U}\} $
and $ (t_0,\xi_0)\in[0,T]\times\R $, and proceed to bound the quantity
$
	|\hl{\speed^{k}_{\scll,\sclll}}{\gIC}-\hl{\speed^{k-1}_{\scll,\sclll}}{\gIC}|.
$
Let us prepare a few notations for this.
Parametrize $ \calU_k \in \mathscr{U} $ as
\begin{align}
	\label{eq:calUk}
	\calU_k = 
	\{ (t,\xi) : 
	|\xi+\zeta_0-\tfrac{b}{\tau}(t-\und s)| \leq b''_{\scll,\sclll}, \, t\in[\und s,\bar s) \},
\end{align}
where $ \und s := (i'-1)\sigma_\scll $, $ \bar s := i't\sigma_\scll $, for some $ i'\in \N $,
and $ \zeta_0=(j-1)b\pm b'_\scll $, for some $ j'\in \N $.
In addition to $ \calU_k $, we consider also the region
\begin{align*}
	\calV^+
	&:= 
	\{ (t,\xi) : 
	0\leq \big(\xi+\zeta_0-\tfrac{b}{\tau}(t-\und s)\big) \leq b'_{\scll}, \, t\in[\und s,\bar s) \},
\\
	\calV^-
	&:= 
	\{ (t,\xi) : 
	-b'_{\scll}\leq \big(\xi+\zeta_0-\tfrac{b}{\tau}(t-\und s)\big) \leq 0, \, t\in[\und s,\bar s) \};
\end{align*}
see Figure~\ref{fig:calUV}.
\begin{figure}[h]
\psfrag{S}[c][c]{$ \und s $}
\psfrag{T}[c][c]{$ \bar s $}
\psfrag{K}[c][c]{$ b'_{\scll} $}
\psfrag{X}[c][c]{$ b'_{\scll} $}
\psfrag{Y}[c][b]{$ b''_{\scll,\sclll} $}
\psfrag{U}[c][c]{$ \scriptstyle\calU_k $}
\psfrag{V}[c][c]{$ \calV^- $}
\psfrag{W}[c][c]{$ \calV^+ $}
\psfrag{L}[r][c]{$ \partial^- \buffer_e $}
\psfrag{R}[l][c]{$ \partial^+ \buffer_e $}
\includegraphics[width=\textwidth]{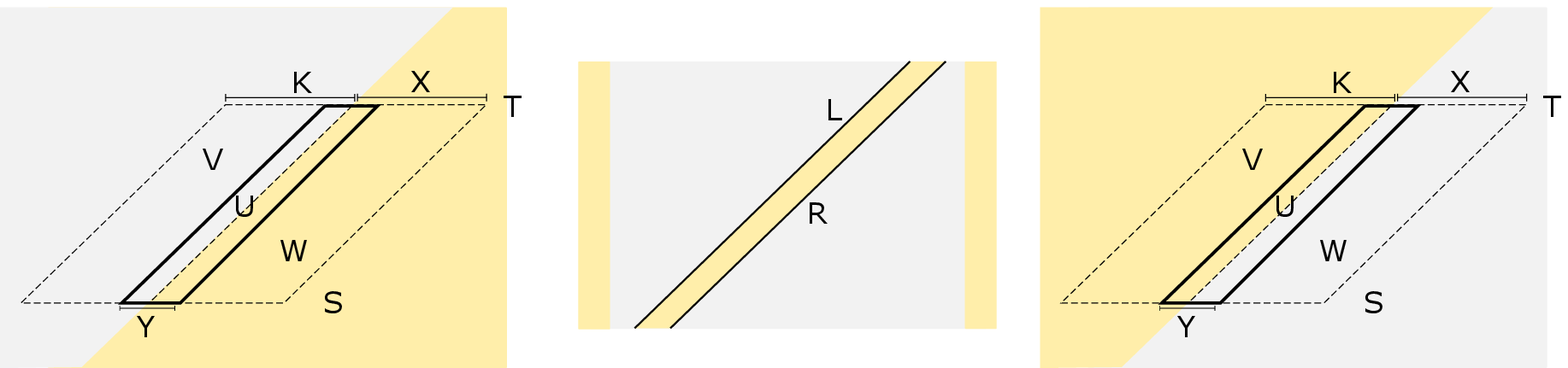}
\caption{The regions $ \calU_k $, $ \calV^+ $ and $ \calV^- $}
\label{fig:calUV}
\end{figure}
Recall from~\eqref{eq:prtnC} that $ \prtnC_i $ denotes the coarser partition,
and let $ \region^-,\region^+\in\cup_{i=1}^{\scl_*}\prtnC_i $ denote the two regions from these partitions
that intersect with $ \calU_k $, (i.e., $ \region^\pm\cap\calU_k\neq\emptyset $),
with $ \region^- $ on the left and $ \region^+ $ on the right.
Referring to Figure~\ref{fig:lamConst}, we see that
\begin{align}
	&
	\label{eq:Speedd:2values}
	\Speedd_{\scll,\sclll}|_{(\calV^\pm)^\circ} = \lambda_{\region^\pm},
\\	
	&
	\label{eq:Speeddk:2values}
	\Speedd^{k}_{\scll,\sclll}|_{(\calV^\pm\setminus\calU_k)^\circ} = \lambda_{\region^\pm},
	\quad
	\Speedd^{k-1}_{\scll,\sclll}|_{(\calV^\pm\setminus\calU_k)^\circ} = \lambda_{\region^\pm}.
\end{align}
\begin{figure}[h]
\psfrag{X}[c][c]{$ 2b'_{\scll} $}
\psfrag{W}[c][c]{$ b'_{\scll} $}
\includegraphics[width=.3\textwidth]{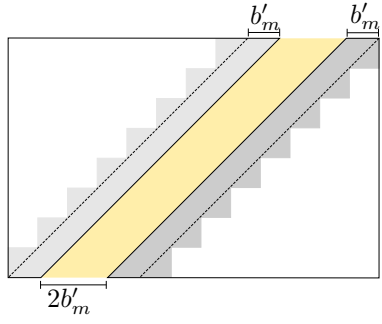}
\caption{%
On each reduced triangle $ \rTri $,
the function $ \Speedd_{\scll,\sclll} $ takes value $ \lambda_{\rTri} $ on the gray regions,
no matter $ \lambda_{\rTri} \geq 1 $ or $ \lambda_{\rTri} < 1 $.
The gray regions stretch a distance $ b'_\scll $ from the buffer zone into the reduced triangles.
}
\label{fig:lamConst}
\end{figure}

We now begin to bound 
$ |\hl{\speed^{k}_{\scll,\sclll}}{\gIC}-\hl{\speed^{k-1}_{\scll,\sclll}}{\gIC}| $.
To this end, we assume $ \lambda_{\region^+} \geq \lambda_{\region^-} $ for simplicity of notations.
The other scenario is proven by the same argument.
The functions $ \speed^{k}_{\scll,\sclll},\speed^{k-1}_{\scll,\sclll} $ differ only on $ \calU_k $,
and by~\eqref{eq:Speedd:2values},
these functions $ \speed^{k}_{\scll,\sclll},\speed^{k-1}_{\scll,\sclll} $ takes
two values $ \{\lambda_{\region^+},\lambda_{\region^-}\} $ on $ \calU_k $.
Under this property, we define the upper and lower envelopes of $ \speed^{k}_{\scll,\sclll},\speed^{k-1}_{\scll,\sclll} $ as
\begin{align*}
	\speed^\star := \speed^{k}_{\scll,\sclll} \ind_{([0,T)\times\R)\setminus \calU_k} + \lambda_{\region^+}\ind_{\calU_k},
	\
	\speed_\star := \speed^{k}_{\scll,\sclll} \ind_{([0,T)\times\R)\setminus \calU_k} + \lambda_{\region^-}\ind_{\calU_k},
\end{align*}
so that $ \speed_\star \leq \speed^{k}_{\scll,\sclll}, \speed^{k-1}_{\scll,\sclll} \leq \speed^\star $.
Combining this with~\eqref{eq:hl:ndecr} gives
\begin{align}
	\label{eq:speedhl:enve}
	\big| \hl{\speed^{k}_{\scll,\sclll}}{\gIC} - \hl{\speed^{k-1}_{\scll,\sclll}}{\gIC} \big|
	\leq
	\hl{\speed^\star}{\gIC} - \hl{\speed_\star}{\gIC}.	
\end{align}
The next step is to bound the r.h.s.\ of \eqref{eq:speedhl:enve}.
We do so by appealing to the variational formulation:
\begin{align}
	\notag
	\hl{\speed^{\star}}{\gIC}(t_0,\xi_0)
	&:=
	\inf_{w\in W(t_0,\xi_0)}
	\big\{ \HLf_{0,t_0}(\speed^{\star};w) + \gIC(w(0)) \big\},
\\
	\label{eq:speed**hl}
	\hl{\speed^{\star}}{\gIC}(t_0,\xi_0)
	&:=
	\inf_{v\in W(t_0,\xi_0)}
	\big\{ 	\HLf_{0,t_0}(\speed^{\star};v) + \gIC(v(0))  \big\}.
\end{align}
\emph{Fix} a generic $ w\in W(t_0,\xi_0) $.
Indeed, because $ \speed^{\star} \geq \speed_{\star} $ and because of~\eqref{eq:hl:ndecr},
\begin{align*}
	\HLf_{0,t_0}(\speed^{\star};w) + \gIC(w(0)) \geq \HLf_{0,t_0}(\speed_{\star};w) + \gIC(w(0)).
\end{align*}
Our goal is to perform surgery on the path $ w $ to obtain a new path $ v $,
so that the reverse inequality holds for $ v $, up to an error of order $ \sclll^{-1} $.
Consider the last time $ t_\star := \inf\{ t\in[\und s,\bar s]: w(t)\in \calU_k \} $ when $ w $ lies within $ \calU_k $.
If $ t_\star = \infty $, i.e., $ w $ never lies within $ \calU_k $,
taking $ v=w $~\eqref{eq:speed**hl} gives
\begin{align*}
	\HLf_{0,t_0}(\speed^{\star};w) + \gIC(w(0)) 
	= 
	\HLf_{0,t_0}(\speed_{\star};w) + \gIC(w(0))
	\geq
	\hl{\speed^{\star}}{\gIC}(t_0,\xi_0).		
\end{align*}
Otherwise, applying Lemma~\ref{lem:hl}\ref{enu:hl:chaV} with $ (s_0,t'_0,t_0)=(0,t_\star,t_0) $
and with $ \speed=\speed_\star $,
we obtained a modified path $ \til w\in W(t_0,\xi_0) $, such that
\begin{align}
	\label{eq:witlw<w}
	\HLf_{0,t_0}(\speed_{\star};\til w) + \gIC(\til w(0)) 
	\leq
	\HLf_{0,t_0}(\speed_{\star};\til w) + \gIC(w(0)),
\end{align}
that $ (t,\til w)|_{[0,t_\star]} \in \calC(t_\star,w(t_\star)) $,
and that $ \til w|_{[t_\star,t_0]}=w_{[t_\star,t_0]} $.
The last property ensures that $ (t,\til w(t))|_{(t_\star,t_0]} \notin \calU_k $,
and we already have $ (t,\til w(t))|_{[0,\und s)} \notin \calU_k $ (see~\eqref{eq:calUk}).
Since $ \calU_k $ is the only region where $ \speed^\star $ and $ \speed_\star $ differ,
our next step is to modify the path $ \til w(t) $ for $ t\in[\und s,t_\star] $.

Fix a small parameter $ \d>0 $.
We set
\begin{align*}
	v_\d(t) &:=
	\left\{\begin{array}{l@{,}l}
		\til w(t) &\text{ for } t\in [0, \und s]\cup [t_\star,t_0],
	\\
		\til w(t) - 3b''_{\scll,\sclll}
			&\text{ for } t\in [\und s+\d,t_\star-\d],
	\\
		\til w(\und s) + 3b''_{\scll,\sclll}\d^{-1}(t-\und s) 
			&\text{ for } t\in (\und s, \und s+\d),
	\\
		\til w(t_\star) - 3b''_{\scll,\sclll}\d^{-1}(t_\star-t) 
			&\text{ for } t\in (t_\star-\d, t_\star).
	\end{array}\right.
\end{align*}
That is, we shift the part of $ \til w $ within $ t\in[\und s+\d, -\d+t_\star] $,
by distance $ 3b''_{\scll,\sclll} $ to the left.
Within the intervals $ (\und s,\und s+\d) $ and $ (-\d+t_\star,t_\star) $,
we linearly joint the path to ensure $ v_\d\in W(t_0,\xi_0) $. 
For such a path $ v_\d $, evaluate the corresponding functional
$
	\HLf_{0,t_0}(\speed^{\star};v_\d) + \gIC(v_\d(0)) 	
$,
and let $ \d\downarrow 0 $ to get
\begin{subequations}
\label{eq:vd}
\begin{align}
	\label{eq:vd1}
	\lim_{\d\downarrow 0}
	\big( \HLf_{0,t_0}(\speed^{\star};v_\d) + \gIC(v_\d(0)) \big)
	=&
	\HLf_{0,\und s}(\speed^{\star};\til w) + \HLf_{t_\star,t_0}(\speed^{\star};\til w)  + \gIC(\til w(0))
\\
	\label{eq:vd2}
	&
	+\HLf_{\und s,t_\star}
	\big( \speed^{\star};\til w-3b''_{\scll,\sclll} \big)
\\	
	\label{eq:vd3}
	&
	+
	\lim_{\d\downarrow 0}\HLf_{\und s,\und s+\d}(\speed^{\star};v_\d)
	+
	\lim_{\d\downarrow 0}\HLf_{-\d+t_\star,t_\star}(\speed^{\star};v_\d).
\end{align}
\end{subequations}
We now analyze the expressions on~\eqref{eq:vd1}--\eqref{eq:vd3} each by each.
\begin{itemize}[leftmargin=3ex]
\item 
As mentioned earlier, for $ t\in [0,\und s) $ and for $ t\in(t_\star,t_0] $,
$ \til w(t) $ sits entirely within the region where $ \speed^\star=\speed_\star $,
so we replace $ \speed^\star $ by $ \speed_\star $ on the r.h.s.\ of~\eqref{eq:vd1}. 

\item
Next, for $ t\in(\und s,t_\star) $,
given the properties
\begin{align*}
	(t_\star,\til w(t_\star))\in \calU_k,
	\quad
	(t,\til w(t))|_{t\in[0,t_\star]} \in \calC(t_\star,w(t_\star)),
	\quad
	t_\star-\und s \leq \sigma_\scll,
\end{align*}
using the same speed-counting argument below~\eqref{eq:intersect}, we have that
\begin{align}
	\label{eq:tilwincalC}
	\speed^\star\big( t,\til w(t)- 3b''_{\scll,\sclll} \big)\big|_{t\in(\und s,t_\star)}
	\leq
	\speed_\star\big( t,\til w(t) \big)\big|_{t\in(\und s,t_\star)},
\end{align}
for all $ \sclll $ large enough such that $ 4b''_{\scll,\sclll} < b'_{\scll} - \sigma_\scll(\maxlam+\frac{b}{\tau}) $.
Further, by~\eqref{eq:Speeddk:2values}, we have
$ \speed^\star|_{(\calV^-\setminus\calU_k)^\circ} = \lambda_{\region^-} $.
This together with~\eqref{eq:tilwincalC} and $ \lambda_{\region^-} \leq \lambda_{\region^+} $ gives
\begin{align*}
	\speed^\star\big( t,\til w(t)- 3b''_{\scll,\sclll} \big)\big|_{t\in(\und s,t_\star)}
	\leq
	\speed_\star\big( t,\til w(t) \big)\big|_{t\in(\und s,t_\star)},
\end{align*}
for all $ \sclll $ large enough, and therefore
$ \HLf_{\und s,t_\star}( \speed^{\star};\til w-3b''_{\scll,\sclll}) \leq \HLf_{\und s,t_\star}( \speed_{\star};\til w) $.

\item
For $ t\in(\und s,t_\star) $ and for $ t\in(-\d+t_\star,t_\star) $,
the path $ v_\d $ has constant velocity $ v_\d = \pm \d^{-1}3b''_{\scll,\sclll} $.
Using this and~\eqref{eq:hl:id} gives
\begin{align*}
	\lim_{\d\downarrow 0}\HLf_{\und s,\und s+\d}(\speed^{\star};v_\d)
	=
	0,
	\quad
	\lim_{\d\downarrow 0}\HLf_{-\d+t_\star,t_\star}(\speed^{\star};v_\d)
	=
	3b''_{\scll,\sclll}.
\end{align*}
\end{itemize}
Combining the preceding discussions with~\eqref{eq:vd} gives
\begin{align*}
	\lim_{\d\downarrow 0}
	\big( \HLf_{0,t_0}(\speed^{\star};v_\d) + \gIC(v_\d(0)) \big)
	\leq
	\HLf_{0,t_0}(\speed^{\star};\til w) + \gIC(\til w(0)) +3b''_{\scll,\sclll}.	
\end{align*}
Further combining this is with~\eqref{eq:witlw<w} and~\eqref{eq:speed**hl},
we arrive at
$
	\hl{\speed_\star}{\gIC}(t_0,\xi_0)
	\geq
	\HLf_{0,t_0}(\speed^{\star};\til w) + \gIC(\til w(0)) +3b''_{\scll,\sclll}.	
$
As this holds for all $ w\in W(t_0,\xi_0) $ and all $ (t_0,\xi_0) $, we conclude
\begin{align*}
	\hl{\speed_\star}{\gIC} \geq \hl{\speed^\star}{\gIC} - 3 b''_{\scll,\sclll}.
\end{align*}
Inserting this into~\eqref{eq:speedhl:enve} thus gives
\begin{align}
	\label{eq:speedhl:enve:}
	\big| \hl{\speed^{k}_{\scll,\sclll}}{\gIC} - \hl{\speed^{k-1}_{\scll,\sclll}}{\gIC} \big|
	\leq
	3 b''_{\scll,\sclll}.	
\end{align}

Applying the bound~\eqref{eq:speedhl:enve:} within the decomposition~\eqref{eq:speedhl:tele} gives
\begin{align}
	\label{eq:speedhl:tele:}
	\big| \hl{\Speed_{\scll,\sclll}}{\gIC}-\hl{\Speedd_{\scll,\sclll}}{\gIC} \big|
	\leq
	3\#\mathscr{U} b''_{\scll,\sclll}.
\end{align}
Referring to the preceding definition of~$ \mathscr{U} $,
we see that $ \#\mathscr{U} $ depends on $ \scl_* $ and $ \scll $ only,
and in particular does not depend on $ \sclll $.
Hence letting $ \sclll\to\infty $ in~\eqref{eq:speedhl:tele:} completes the proof.
\end{proof}

\subsection{Estimating the relative entropy}
\label{sect:speed:ent}
Recall that $ \QN^{\speed} $ denotes the law of the inhomogeneous \ac{TASEP} with a simple speed function $ \speed $.
Having constructed $ \Speed_{\scll,\sclll} $, in this subsection we estimate the relative entropy
$ \frac{1}{N^2} H(\QN^{\Speed_{\scll,\sclll}}|\PN^\g) $.
First, from the explicit formula \eqref{eq:ent} and \eqref{eq:speed:tail}, we have
\begin{align}
	\label{eq:ent:speed}
	\frac{1}{N^2}H(\QN^{\Speed_{\scll,\sclll}}|\PN^\g)	
	=
	\frac{1}{N}
	\sum_{|x|\leq Nr^*}
	\Ex_{\QN^{\Speed_{\scll,\sclll}}}\Big(  \int_0^{T} 
		\mob(\h(Nt),x)\prf\Big(\Speed_{\scll,\sclll}\big(t,\tfrac{x}{N}\big)\Big) dt
	\Big).
\end{align}
Let us divide the r.h.s.\ of~\eqref{eq:ent:speed} into two sums over $ |x|\leq Nr_* $ and over $ Nr_*<|x|\leq Nr^* $,
and write resulting sums as $ H^1_{\scll,\sclll,N} $ and $ H^2_{\scll,\sclll,N} $, respectively.
More explicitly,
\begin{align}
	\label{eq:lwb:H1}
	H^1_{\scll,\sclll,N} 
	&:= 	
	\frac{1}{N}
		\sum_{|x|\leq Nr_*}
		\Ex_{\QN^{\Speed_{\scll,\sclll}}}\Big(  \int_0^{T} 
			\mob(\h(Nt),x)\prf\Big(\Speed_{\scll,\sclll}\big(t,\tfrac{x}{N}\big)\Big) dt
		\Big),
\\	
	\label{eq:lwb:H2}
	H^2_{\scll,\sclll,N} 
	&:= 	
	\frac{1}{N}
		\sum_{Nr_*<|x|\leq Nr^*}
		\Ex_{\QN^{\Speed_{\scll,\sclll}}}\Big(  \int_0^{T} 
			\mob(\h(Nt),x)\prf\Big(\Speed_{\scll,\sclll}\big(t,\tfrac{x}{N}\big)\Big) dt
		\Big).
\end{align}

Recall that $ \Sigma^* $ denotes the restriction of the triangulation $ \Sigma $ onto $ [0,T]\times[-r^*,r^*] $
and that $ \Sigma_* $ denotes the restriction of $ \Sigma $ onto $ [0,T]\times[-r_*,r_*] $.
We begin with a bound on $ H^{2}_{\scll,\sclll,N} $.
\begin{lemma}
\label{lem:lwbH2bd}
We have that
\begin{align*}
	\limsup_{\scll\to\infty} \limsup_{\sclll\to\infty} \limsup_{N\to\infty} H^2_{\scll,\sclll,N} 
	\leq
	\sum_{\triangle\in\Sigma^*\setminus\Sigma_*} \prff(\linlam)|\triangle|
	=
	\int_0^T \int_{r_*<|\xi| < r^*} \prff\Big(\frac{g_t}{g_\xi(1-g_\xi)}\Big) dt d\xi.
\end{align*}
\end{lemma}
\begin{proof}
Since the mobility function $ \mob(\f,x) $ is bounded by $ 1 $, 
we bound the expression~\eqref{eq:lwb:H2} as
\begin{align}
	\label{eq:lem:lwbH2bd:}
	H^2_{\scll,\sclll,N} 
	\leq 	
	\frac{1}{N}
		\sum_{Nr_*<|x|\leq Nr^*}
		\int_0^{T} 
			\prf\big(\Speed_{\scll,\sclll}(t,\tfrac{x}{N})\big) dt.
\end{align}
Since $ \Speed_{\scll,\sclll} $ is piecewise constant,
and since $ \prf(\Speed_{\scll,\sclll}) $ is bounded (thanks to \eqref{eq:speed:rg}),
letting $ N\to\infty $ in~\eqref{eq:lem:lwbH2bd:},
the discrete sum in~\eqref{eq:lem:lwbH2bd:} converges to an integral, giving
\begin{align}
	\label{eq:lem:lwbH2bd::a}
	\limsup_{N\to\infty} H^2_{\scll,\sclll,N} 
	\leq 	
	\int_{0}^T \int_{r_*\leq |\xi|\leq r^*}
	\prf(\Speed_{\scll,\sclll}(t,\xi)) dtd\xi.
\end{align}
With $ \prf(\Speed_{\scll,\sclll}) $ being bounded,
in~\eqref{eq:lem:lwbH2bd::a}, letting $ \sclll\to\infty $ and $ \scll\to\infty $ in order,
together with \eqref{eq:speed:limit}, we obtain
\begin{align*}
	\limsup_{\scll\to\infty} \limsup_{\sclll\to\infty} \limsup_{N\to\infty} H^2_{\scll,\sclll,N} 
	&\leq
		\lim_{\scll\to\infty}
		\lim_{\sclll\to\infty} \int_{0}^T \int_{r_*\leq |\xi|\leq r^*}
	\prf(\Speed_{\scll,\sclll}(t,\xi)) dtd\xi
	=
	\sum_{\triangle\in\Sigma^*\setminus\Sigma_*} \prff(\linlam)|\triangle|.
\end{align*}
This completes the proof.
\end{proof}

We next establish a bound on $ H^1_{\scll,\sclll,N} $.
\begin{lemma}
\label{lem:lwbH1bd}
We have that
\begin{align}
	\label{eq:lwbH1bd}
	\lim_{\scll\to\infty} \lim_{\sclll\to\infty} \lim_{N\to\infty} 
	H^1_{\scll,\sclll,N} 
	=
	\sum_{\triangle\in\Sigma_*} \linrho(1-\linrho)\prff(\linlam)
	=
	\int_0^T \int_{-r^*}^{r^*} \lrf(g_t,g_\xi) dt d\xi.
\end{align}
\end{lemma}
\begin{proof}
Throughout this proof,
we use $ o_{\scll,\sclll,N}(1) $ and $ u_{\scll,\sclll} $, 
to denote \emph{generic}, deterministic quantities that may change from line to line, 
but satisfy
\begin{align*}
	\lim_{\scll\to\infty} \lim_{\sclll\to\infty} \lim_{N\to\infty} |o_{\scll,\sclll,N}(1)| = 0,
	\quad
	\limsup_{\scll\to\infty} \limsup_{\sclll\to\infty} |u_{\scll,\sclll}| \leq 1.
\end{align*}

Recall from~\eqref{eq:Pi''} that $ \Pi''_{\scll,\sclll} $ denote a partition of $ [0,T]\times\R $ consisting of rectangles,
and that $ \Speed_{\scll,\sclll} $ is constant within the interior $ \square^\circ $ of each rectangle $ \square\in \Pi''_{\scll,\sclll} $.
This being the case,
letting $ \Pi''_{*,\scll,\sclll} $ denote the restriction of $ \Pi''_{\scll,\sclll} $ onto $ [0,T]\times[-r_*,r_*] $,
we parametrize each 
$ \square\in\Pi''_{\scll,\sclll} $ as $ [\und t_\square, \bar t_\square]\times[\xi^-_\square,\xi^+_\square] $,
and express~\eqref{eq:lwb:H1} as 
\begin{align}
	\notag
	H^1_{\scll,\sclll,N} 
	&=
	\sum_{\square\in\Pi''_{*,\scll,\sclll}} \prf\big( \Speed_{\scll,\sclll}|_{\square^\circ} \big)
	\Ex_{\QN^{\Speed_{\scll,\sclll}}} 
		\int_0^T
		\frac{1}{N}\sum_{(t,x)\in\square^\circ} \mob(\g(Nt),x) dt
\\
	\label{eq:lwbH1bd:}
	&=
	\sum_{\square\in\Pi''_{*,\scll,\sclll}} \prf\big( \Speed_{\scll,\sclll}|_{\square^\circ} \big)
	\frac{1}{N}\sum_{\frac{x}{N}\in[\xi^-_\square,\xi^+_\square]}
	\Ex_{\QN^{\Speed_{\scll,\sclll}}} 
	\int_{\und t_\square}^{\bar t_\square}
	\mob(\g(Nt),x) dt.
\end{align}
Here, unlike in Lemma~\ref{lem:lwbH2bd}, using $ \mob(\g(Nt),x) \leq 1 $ does \emph{not} yields 
a good enough bound for our purpose.
Instead, we use the following analog of~\eqref{eq:QN:hN}
for inhomogeneous \acp{TASEP}: 
\begin{align}
	\label{eq:lwbH1bd::}
	\Ex_{\QN^{\Speed_{\scll,\sclll}}} 
	\Big( 
		\int_{t_1}^{t_2}
		\Speed_{\scll,\sclll}(t,\tfrac{x}{N}) \mob(\g(Nt),x) dt
	\Big)
	=
	\Ex_{\QN^{\Speed_{\scll,\sclll}}} \Big( \gN(t_1,\tfrac{x}{N})-\gN(t_2,\tfrac{x}{N}) \Big),
\end{align}
$ \forall t_1\leq t_2\in[0,T]$, $ x\in\Z $.
Applying~\eqref{eq:lwbH1bd::} with $ (t_1,t_2)=(\und t_\square, t_\square) $ in~\eqref{eq:lwbH1bd:},
we obtain the following expression for~$ H^1_{\scll,\sclll,N}  $:
\begin{align}
	\label{eq:lwbH1bd:::}
	H^1_{\scll,\sclll,N} 
	=
	\sum_{\square\in\Pi''_{*,\scll,\sclll}} \prf\big( \Speed_{\scll,\sclll}|_{\square^\circ} \big)
	\frac{1}{\Speed_{\scll,\sclll}|_{\square^\circ}}
	\frac{1}{N}\sum_{\frac{x}{N}\in[\xi^-_\square,\xi^+_\square]}
	\Ex_{\QN^{\Speed_{\scll,\sclll}}} \big(\gN(t,\tfrac{x}{N})\big|_{\und t_\square}^{\bar t_\square}\big).
\end{align}

Write $ G_{\scll,\sclll} :=\hl{\Speed_{\scll,\sclll}}{\gIC} $ for the Hopf--Lax function.
Recall from Corollary~\ref{cor:inhomo} that,
$ \gN $ converges to $ G_{\scll,\sclll} $, $ \QN^{\Speed_{\scll,\sclll}} $-in probability.
In order to approximate the r.h.s.\ of~\eqref{eq:lwbH1bd:::} in term of $ G_{\scll,\sclll} $,
our next step to leverage the convergence in probability into convergence in $ L^1 $.
Recall from~\eqref{eq:speed:rg} that $ \Speed_{\scll,\sclll} \leq \maxlam $.
Consequently,  under the law, $ \QN^{\Speed_{\scll,\sclll}} $,
$ \g(t_2,x)-\g(t_1,x) $ is stochastically dominated by $ \Pois((t_2-t_1)\maxlam) $,
$ \forall t_1\leq t_2\in[0,T] $.
In particular, 
\begin{align}
	\label{eq:lwbH1bd:L2bd}
	\sup
	\Big\{ \Ex_{\QN^{\Speed_{\scll,\sclll}}}
  \big( \gN(t,\tfrac{x}{N})|_{t_1}^{t_2} \big)^2 : N\in\N, \, x\in\Z, \, [t_1,t_2]\subset [0,T] \Big\}
	<
	\infty.
\end{align}
The $ L^2 $ boundedness~\eqref{eq:lwbH1bd:L2bd},
together with the converges in probability, Corollary~\ref{cor:inhomo}, gives
\begin{align}
	\label{eq:cnvg:QNscl}
	\lim_{N\to\infty}
	\frac{1}{N}\sum_{\frac{x}{N}\in[\xi^-_\square,\xi^+_\square]}
	\Ex_{\QN^{\Speed_{\scll,\sclll}}} \big(\gN(t,\tfrac{x}{N})\big|_{\und t_\square}^{\bar t_\square}\big)
	=
	\int_{\xi^-_\square}^{\xi^+_\square} G_{\scll,\sclll}(t,\tfrac{x}{N})\big|_{\und t_\square}^{\bar t_\square} d\xi.
\end{align}
Further, by Lemma~\ref{lem:hl}\ref{enu:hl:sp}, the function $ (t,\xi)\mapsto G_{\scll,\sclll}(t,\xi) $ is uniformly Lipschitz.
This allows us to rewrite the r.h.s.\ of~\eqref{eq:cnvg:QNscl} as $ \int_{\square} \partial_t G_{\scll,\sclll} dtd\xi $.
On this note, combining~\eqref{eq:cnvg:QNscl} and \eqref{eq:lwbH1bd:::} gives 
\begin{align}
	\label{eq:lwbH1bd:N:}
	\lim_{N\to\infty}
	H^1_{\scll,\sclll,N} 
	=
	\int_0^T \int_{-r_*}^{-r^*} 
	\frac{\prf\big( \Speed_{\scll,\sclll} \big)}{\Speed_{\scll,\sclll}} 
	\partial_t G_{\scll,\sclll}
	dtd\xi.
\end{align}

Next, recall from~\eqref{eq:speed:limit} that
$ \Speed_{\scll,\sclll} $ converges in $ L^1 $ to $ (\linlam\vee 1) $ on each $ \triangle\in\Sigma^* $
under the relevant limit.
Further, $ \Speed_{\scll,\sclll} $ is bounded away from zero and infinity (by~\eqref{eq:speed:rg}),
and $ \partial_t G_{\scll,\sclll} $ is uniformly bounded (by Lemma~\ref{lem:hl}\ref{enu:hl:sp}).
Under these properties,
we rewrite the r.h.s.\ of~\eqref{eq:lwbH1bd:N:} as
$
	\sum_{\triangle\in\Sigma_*}
	\int_{\triangle} 
	\partial_t G_{\scll,\sclll}\prf( \Speed_{\scll,\sclll} )/\Speed_{\scll,\sclll} 	
	dtd\xi,
$
and for on each $ \triangle\in\Sigma_* $, 
replace $ \Speed_{\scll,\sclll} $ with its limiting value $ \linlam\vee 1 $.
This, together with 
$
	\prf( \lambda\vee 1 )/(\lambda\vee 1) 
	= \prff( \lambda )/\lambda, 
$ 
gives
\begin{align}
	\label{eq:lwbH1bd:N::}
	\lim_{\scll\to\infty}\lim_{\scll\to\infty}
	\lim_{N\to\infty}
	H^1_{\scll,\sclll,N} 
	=
	\sum_{\triangle\in\Sigma_*} 
	\frac{\prff\big( \linlam\big)}{\linlam}
	\lim_{\scll\to\infty}\lim_{\sclll\to\infty}
	\int_{\triangle}
	\partial_t G_{\scll,\sclll}
	dtd\xi,
\end{align}
provided that the limit 
$ 
	(\lim_{\scll\to\infty}\lim_{\sclll\to\infty}
	\int_{\triangle}
	\partial_t G_{\scll,\sclll}
	dtd\xi)
$ 
exists, for each $ \triangle\in\Sigma_* $.
Fixing~$ \triangle\in\Sigma_* $,
We next show that the corresponding limit does exist, and calculate its value.
To this end, we parametrize the triangle as 
$ 
	\triangle = \{ 
		(t,\xi) :
		t\in[\und t_\triangle(\xi), \bar t_\triangle(\xi)],
		\,
		\xi\in[\xi^-_\triangle,\xi^+_\triangle]
	\} 
$,
and write
\begin{align}
	\label{eq:eachlimit}
	\int_{\triangle} \partial_t G_{\scll,\sclll}
	=
	\int_{\xi^-_\triangle}^{\xi^+_\triangle} G_{\scll,\sclll}(t,\xi)|_{\und t_\triangle(\xi)}^{\bar t_\triangle(\xi)} d\xi.
\end{align}
By Proposition~\ref{prop:speedhl},
the function $ G_{\scll,\sclll} $ converges uniformly to $ \ling $ on $ [0,T]\times[-r_*,r_*] $
under the relevant iterated limit.
Using this the take limit~\eqref{eq:eachlimit} gives
\begin{align}
	\label{eq:eachlimit:}
	\lim_{\scll\to\infty}\lim_{\sclll\to\infty}	
	\int_{\triangle} \partial_t G_{\scll,\sclll}
	=
	\int_{\xi^-_\triangle}^{\xi^+_\triangle} g(t,\xi)|_{\und t_\triangle(\xi)}^{\bar t_\triangle(\xi)} d\xi
	=
	\int_{\triangle} g_t dt d\xi
	=
	|\triangle| \linkap.	
\end{align}
Inserting~\eqref{eq:eachlimit:} into~\eqref{eq:lwbH1bd:N::},
together with $ \linkap/\linlam = \linrho(1-\linrho) $,
we conclude the desired result~\eqref{eq:lwbH1bd}.
\end{proof}

Combining Lemma~\ref{lem:lwbH2bd}--\ref{lem:lwbH1bd} immediately yields:
\begin{corollary}
\label{cor:lwbH}
We have that
\begin{align*}
	\limsup_{\scll\to\infty} \limsup_{\sclll\to\infty} \limsup_{N\to\infty}
	\frac{1}{N^2}H(\QN^{\Speed_{\scll,\sclll}}|\Pr_{\gIC})
	\leq
	\int_0^T \int_{-r^*}^{r^*} \lrf(g_t,g_\xi) dt d\xi
	+
	\int_0^T \int_{r_*<|\xi| < r^*} \prff\Big(\frac{g_t}{g_\xi(1-g_\xi)}\Big) dt d\xi.
\end{align*}
\end{corollary}

\subsection{Proof of Proposition~\ref{prop:ent:up}}
\label{sect:lwb:pf}
With $ \e_*>0 $ being given and fixed,
we apply Proposition~\ref{prop:speedhl} and Corollary~\ref{cor:lwbH},
to obtain \emph{fixed} $ \scll_*,\sclll_* \in \N $ such that
\begin{align}
	&
	\label{eq:lwb:Apprx}
	\sup_{[0,T]\times[-r_*,r_*]} \big| \hl{\Speed_{\scll_*,\sclll_*}}{\gIC}-\ling \big|
	< 
	\e_*,
\\
	&
	\label{eq:lwbH:}
	\limsup_{\sclll\to\infty} \limsup_{N\to\infty}
	\frac{1}{N^2}H(\QN^{\Speed_{\scll_*,\sclll_*}}|\Pr_{\gic})
	<
	\int_0^T \int_{\R} \lrf(g_t,g_\xi) dtd\xi + 
          \int_0^T \int_{{r_*<|\xi| < r^*}} \prff\Big(\frac{g_t}{g_\xi(1-g_\xi)}\Big) dt d\xi.
\end{align}
Given such $ \scll_*,\sclll_* \in \N $, we set $ \QN := \QN^{\Speed_{\scll_*,\sclll_*}} $.
The inequality verifies~\eqref{eq:lwbH:} the condition~\eqref{eq:lwb:ent}.
Next, combining~\eqref{eq:lwb:Apprx} with Corollary~\ref{cor:inhomo},
we see that  the condition~\eqref{eq:lwb:apprx} holds.

It remains only to check the condition~\eqref{eq:lwb:RD:bd}.
From the explicit formula~\eqref{eq:RNdrv} and by~\eqref{eq:speed:tail},
we have that
\begin{align}
	\notag
	\frac{1}{N^2} &\log \frac{d\QN}{d\PN^\g}
\\
	\label{eq:lwb:RD}
	&=
	\frac{1}{N}\sum_{\frac{x}{N}\in[-r^*,r^*]} \int_0^{T} 
		\Big(
			\log\Speed_{m_*,n_*}\big(Nt,\tfrac{x}{N}\big) d\hN(t,\tfrac{x}{N}) 
			- \mob(\h(Nt),x)\big(\Speed_{m_*,n_*}\big(Nt,\tfrac{x}{N}\big)-1\big) dt
		\Big).
\end{align}
From~\eqref{eq:speed:rg},
we have the bounds $ |\log(\Speed_{m_*,n_*})| \leq |\log\minlam|+|\log\maxlam| := A $
and $ |\Speed_{m_*,n_*}-1| \leq \maxlam+1 $.
Using these bounds in~\eqref{eq:lwb:RD} gives
\begin{align}
	\label{eq:lwb:RD:}
	\Big| \frac{1}{N^2} \log \frac{d\QN}{d\Pr_\gic}\Big|
	\leq
	\frac{1}{N}\sum_{\frac{x}{N}\in[-r^*,r^*]}
	\Big( A \hN(t,\tfrac{x}{N})|_{t=0}^{t=T} + (\maxlam+1)T \Big).
\end{align}
Taking $ (\Ex_{\QN}(\Cdot)^2) $ on both sides of~\eqref{eq:lwb:RD:} 
and using Jensen's inequality,
we arrive at
\begin{align}
\label{eq:lwb:RD::}
	\Ex_{\QN}\left(\Big[ \frac{1}{N^2} \log \frac{d\QN}{d\PN^\g}\Big]^2 \right)
	\leq
	\frac{4 r^* A}{N} \sum_{|x|\leq Nr^*}
	\Ex_{\QN}\left(\Big[ \hN(t,\tfrac{x}{N})|_{t=0}^{t=T}\Big]^2\right)  + 4 r^* (\maxlam+1)^2T^2 .
\end{align}
Using the $ L^2 $ bound~\eqref{eq:lwbH1bd:L2bd} on the r.h.s.\ of~\eqref{eq:lwb:RD::},
we concludes the desired condition~\eqref{eq:lwb:RD:bd}
and hence complete the proof.

\appendix
\section{Super-exponential One-block Estimate and Smaller Deviations}
\label{append:smalldevi}

This is section we gives some discussions about deviations at speed $ N^{-a} $, $ 1<a<2 $.
The purpose is to investigate deviations in a finer topology, where also the oscillations of $ h_\xi $ are taken into account.
This gives rise to the discussions about measure-valued solutions and Young measures in the following.
The results in this section are not used in the rest of the article.

Recall the definition of $ \mob $ from~\eqref{eq:mob} and recall that $ \mobb(\rho):= \rho (1-\rho) $.
To simply notation, hereafter we write $ \mobb(\rho)=\Phi(\rho) $.
In this section, we change the underlying space from the full-line $ \Z $
to a discrete torus $ \Tn:=\Z/(N\Z) \simeq \{0,1,\ldots,N-1\} $ of $ N $ sites,
as the latter is more convenient for our discussion here.
Correspondingly we let $ \T := \R/\Z \simeq $ denotes the continuous torus.
The height function $ \h $ and consequently the space $ \Sp $ need to be modified accordingly,
with heights being understood under suitable modulation.
Additionally, we will also view \ac{TASEP} as a Markov process of the occupation variables.
That is, we  consider the occupation variables
$ \eta(t,x) := \h(t,x+1)-\h(t,x) $, which is the indicator function of having a particle at $ x-\frac12 $,
and view \ac{TASEP} as a Markov process with state space $ \{0,1\}^{\Tn}=(\eta=(\eta(x))_{x\in\Tn}) $.
It is known that, for each fixed $ \rho\in[0,1] $, 
\ac{TASEP} has i.i.d.\ Bernoulli invariant distribution $ \mu_\rho := \text{Ber}(\rho)^{\otimes\Tn} $.
Under this setup, we write $ \mob(\eta,x) = \eta(x-1)(1-\eta(x)) $,
and $ \eta^\ic(x) := \hic(x+1)-\hic(x) $, and
\begin{align}
	\label{eq:bareta}
	\bar{\eta}_{k}(x) := \frac{1}{k}\sum_{|x'-x+\frac12|\leq\frac{k}{2}} \eta(x'), 
\end{align}
for the local average around $ x $ of width $ k $.

The first observation here is that the classical proof of super-exponential one-block estimate \cite{kipnis89}
extends to speed $ N^{2-\delta} $, $ \delta>0 $, for \ac{TASEP}.
This can be seen as a slight improvement of in \cite[Theorem~3.2]{varadhan04}.

\begin{lemma}
\label{lem:sexp1block}
For a given bounded $ G(t,\xi):[0,\infty)\times\T\to\R $, define
\begin{align}
	\label{eq:VofG}
	V_{N,k}(G,\eta,t) := \sum_{x\in\TN} G(t,\tfrac{x}{N}) \big( \mob({\eta},Nx) - \Phi(\bar{\eta}_{k}(Nx)) \big).
\end{align}
For any fixed $ a,\delta\in(0,1) $, $ T<\infty $, and a deterministic initial condition $ \eta^\ic $, we have
\begin{align*}
   \limsup_{k\to\infty}\limsup_{N\to\infty} 
   N^{-2+\delta} 
   \log \Pr_{\eta^\ic} \Big(\int_0^T V_{N,k}(G,\eta(t{N}),t) dt \ge aN \Big) = -\infty.
\end{align*}
\end{lemma}
\begin{proof}
Fixing $ a,\delta>0 $, $ T<\infty $, a bounded $ G $,
and a deterministic initial condition $ \eta^\ic $,
throughout this proof we write $ C=C(\alpha,\delta,G,T) $ and $ V_{N,k}(G,\eta,t)=V(\eta,t) $ to simplify notations.

Set $ F(\eta) := \Pr_{\eta} (\int_0^T V_{N,k}(\eta(tN^2),t) dt \ge aN ) $.
Our aim is to bound $ F(\eta^\ic) $.
To this end, fixing the reference measure $ \mu_* := \mu_{\frac12} $,
we claim that, instead of $ F(\eta^\ic) $, it suffices to consider $ \int F(\eta)d\mu_*(\eta) $.
Indeed, the measure $ \mu_*(\eta^\ic) $ assigns weight $ 2^{-N} $ to $ \eta^\ic $, i.e., $ \Pr_{\mu_*}(\eta^\ic) = 2^{-N} $.
From this we have
$ \int F(\eta)d\mu_*(\eta) \geq 2^{-N} F(\eta^\ic) $.
Take logarithm on both sides and divide the result by $ N^{2-\delta} $. We have
\begin{align*}
	N^{-2+\delta} \log \int F(\eta) d\mu_*(\eta)
	\geq
	-N^{-1+\delta} \log 2 + N^{-2+\delta} \log F (\eta^\ic).
\end{align*}
Given this, with $ \delta<1 $,
it suffices to show that 
\begin{align}
	\label{eq:goal}
	\limsup_{k\to\infty}\limsup_{N\to\infty} N^{-2+\delta} \log \int F(\eta) d\mu_*(\eta) = -\infty.
\end{align}

To begin, fix $ \ell>0 $, and set
\begin{align}
	\label{eq:fk1}
 	u(\eta,T) := \Ex_{\eta} \Big( \exp\Big(\int_0^T N^{1-\delta}\ell V(\eta,t) dt \Big) \Big).	
\end{align}
Exponential Chebichef inequality gives
$ 
	F(\eta) 
	\leq
	e^{- a\ell N^{-2+\delta}} u(\eta,T).
$
Integrating both sides with respect to $ d\mu_* $, followed by using $ \int f d\mu_* \leq (\int f^2 d\mu_*)^{1/2}  $, we obtain
\begin{align}
	\label{eq:expchebichef}
	\int f(\eta) d\mu_*(\eta)
	\leq 
	e^{- a\ell N^{-2+\delta}} \Big(\int u^2(\eta,T) d\mu_*(\eta) \Big)^{1/2}.
\end{align}
We proceed to bound the last integral of $ u^2 $. 
Consider the solution of 
\begin{align}
\label{eq:fk}
	\partial_s u(\eta,s) = N \gen u(\eta,s) +\ell N^{1-\delta} V(\eta,T-s)u(\eta,s), \qquad 0\le s\le T,
\end{align}
with initial condition $ u(\eta,0) = 1 $. The Feynman--Kac formula asserts that $ u(\eta,T) $ is given by \eqref{eq:fk1}.

Here $ \gen f(\eta) := \sum_{x\in\Tn} \mob(\eta,x) (f(\eta^{x,x+1})-f(\eta)) $ denotes the generator of \ac{TASEP},
where $ \eta^{x,x+1} $ is the particle configuration obtained by swapping particles at sites $ x $ and $ x+1 $.
, we consider the Dirichlet form (corresponding to $ \gen $)
\begin{align}
	\label{eq:diri}
	\diri_x(f) := \int ( f(\eta^{x,x+1}) - f(\eta))^2 \; d\mu_*,
	\quad
	\diri(f) := \sum_{x\in\Tn} \diri_x(f).
\end{align}
From~\eqref{eq:fk}, we calculate
\begin{align*}
	\frac12\frac{d~}{dt}\int u(\eta,t)^2 d\mu_* 
	= 
	\int u(\eta,t) (N \gen +N^{1-\delta}\ell V(\eta,T-t)) u (\eta,t) \; d\mu_*.
\end{align*}
Further using the property
$ 	
	\int f \gen f \; d\mu_*
	=
	\diri(f),
$
we bound
\begin{align}
	\label{eq:fkbd}
	\frac12\frac{d~}{dt} \int u(\eta,t)^2 d\mu_* 
	= 
	&\Gamma\int u(\eta,t)^2 d\mu_*,
\\
	\label{eq:Gamma}
	&\Gamma := \sup_{f: \| f\|_2 =1,t\leq T} \left\{ \int N^{1-\delta}\ell  V(\eta,T-t)) f^2(\eta) d\mu_*(\eta) - N \diri(f) \right\}.
\end{align}
Here $ \| f\|_2 := \int f^2 d\mu_* $ denotes the $ L^2 $-norm with respect to the reference measure $ \mu_* $.
Applying Gronwall's inequality to~\eqref{eq:fkbd} gives
\begin{align}
	\label{eq:fkbd:}
	\int u(\eta,T)^2 d\mu_* \le e^{2T\Gamma}. 
\end{align}

We proceed to bound $ \Gamma $.
To this end, divide $ \Tn $ into blocks of length $ k $:
$ I_0 := \{x\in\Tn: |x-\frac12| \leq k\} $, and $ I_j := jk+I_0 $.
Let $ \tau_j \eta(x) := \eta(x+j) $, $ \tau_j f(\eta) := f(\tau_j\eta) $ denote the shift operator,
and set $ W_k = |\sum_{x=1}^k (\eta(x)(1-\eta(x-1))  - \Phi(\bar\eta_{k}(0)))| $.
We have
$
	V(\eta,t)
	\leq
	C \Vert G \Vert_{L^\infty(\R)}
	(
		\sum_{|j|\leq Nk^{-1}}
		W_k(\tau_{jk}\eta) + C k
	).
$
This gives
\begin{align}
\label{eq:Gamma:bd}
	N^{-2+\delta}\Gamma(t) 
	\leq
	&\sup_{f: \|f\|_2 =1} 
	\frac{1}{N}
	\hspace{-5pt}
	\sum_{|j|\leq Nk^{-1}}
	\hspace{-5pt}
	\Big(
		C \ell \int W_k(\eta) \big(\tau_{jk}f^2\big)(\eta) d\mu_*(\eta) 
	- N^{\delta} \sum_{x\in I_0}\diri_x(\tau_{jk} f) 
	\Big)
	+ C N^{-1-\delta} a k.
\end{align}
Consider the space of functions $ \calF_k := \{ g: \{0,1\}^{I_0} \to\R \} $
that depends only on the configurations within $ I_0 $.
Our plan is to localized the r.h.s.\ of \eqref{eq:Gamma:bd} onto $ \calF_k $.
To this end, decompose the product Bernoulli measure $ \mu_* = \text{Ber}(1/2)^{\otimes\Tn} $
into $ \mu_k := \text{Ber}(1/2)^{\otimes I_0} $ and $ \mu'_k := \text{Ber}(1/2)^{\otimes (\Tn\setminus I_0)} $.
Consider the average $ \til{f}_j := (\int( \tau_{jk}f^2) d\mu'_k)^{1/2} $, which is a function in $ \calF_k $.
Indeed, since $ W_k(\eta) \in \calF_k $, we have
$ \int W_k \big(\tau_{jk}f^2\big) d\mu_* = \int W_k \til{f}^2_j d\mu_k $.
Further, from the known variational formula for Dirichlet form (see~\cite[Theorem~10.2, Appendix~1]{kipnis13}),
we have $ \sum_{x\in I_0}\diri_x(\tau_{jk} f) \geq {\diri}_k(\til{f}_j) $, where, for $ g\in\calF_k $,
\begin{align*}
	\diri_k(g) := \sum_{x=0}^{k-1} \int ( g(\eta^{x,x+1}) - g(\eta))^2 \; d\mu_{k}.
\end{align*}
From these discussion we have
\begin{align*}
	N^{-2+\delta}\Gamma(t) 
	&\leq
	\sup_{g\in\calF_k, \|g\|_2 =1} 
	\Big(
		\frac{C\ell}{k} \int W_k g^2 d\mu_k
	- N^{\delta} \diri_k(g) 
	\Big)
	+ C N^{-1-\delta} a k
\\
	&\leq
	C\sup_{\begin{subarray}{}g\in\calF_k, \|g\|_2 =1 \\ \diri_k(g) \leq CaN^{-\delta}\end{subarray}}
		\frac{C\ell}{k} \int W_k g^2 d\mu_k
	+ C N^{-1-\delta} a k,
\end{align*}
where, the last inequality follows by using $ \int W_k g d\mu_k \leq C $.
Sending $ N\to\infty $ gives
\begin{align*}
	\limsup_{N\to\infty} N^{-2+\delta}\Gamma(t) 
	\leq
	C\ell
	\sup\Big\{ 
		\int \frac{1}{k}W_k g^2 d\mu_k:
		g\in\calF_k, \diri_k(g)=0, \|g\|_2 =1
	\Big\}.
\end{align*}
Now, the condition~$ \diri_k(g)=0 $ forces $ g $ to be a constant on the hyperplane $ \{\sum_{x\in I_0}\eta(x)=i\} $, $ i=0,\ldots,k $.
From this, it is standard (see, e.g., \cite[Chapter~5.4]{kipnis13}) to show that
\begin{align*}
	\limsup_{k\to\infty}
	\limsup_{N\to\infty} N^{-2+\delta}\Gamma(t) 
	\leq
	C\ell
	\limsup_{k\to\infty}
	\,
	\sup\Big\{ 
		\int \frac{1}{k}W_k g^2 d\mu_k:
		g\in\calF_k, \diri_k(g)=0, \|g\|_2 =1
	\Big\}
	=0.
\end{align*}
Combining this with \eqref{eq:expchebichef} and \eqref{eq:fkbd:},
with $ \ell >0 $ being arbitrary,
we conclude the desired result~\eqref{eq:goal}.
\end{proof}

We now proceed to define measure-valued solutions of Burgers equation.
Our definition differs slightly from the standard one (see, e.g, \cite[Section~8]{kipnis13}).
Let $ \probm(\Omega) $ denote the set of probability measures on a metric space $ \Omega $.
Fix a time horizon $ T\in(0,\infty) $ hereafter.
Consider the space 
\begin{align*}
	\young : =\{ \nu(t,\xi;d\rho): (t,x)\longmapsto\nu(t,\xi;d\rho) \in \probm[0,1] \}.
\end{align*}
of measurable maps $ [0,T]\times\T \to \probm[0,1] $.
Recall the definition of $ \Spd $ from~\eqref{eq:Spd}.
Consider
\begin{align*}
	\Spdm : = \Big\{ (h,\nu)\in\Spd\times\young : h_\xi(t,\xi) = \int_{[0,1]} \rho \nu(t,\xi;d\rho), \ \text{a.e. } (t,\xi)\in[0,T]\times\T \Big\}.
\end{align*}
We say $ (h,\nu)\in\Spdm $ is a \textbf{measure-valued} solution of~\eqref{eq:burgers} equation if
\begin{align}
	\label{eq:mburgers}
	h_t(t,\xi) = \int_{[0,1]} \Phi(\rho) \nu(t,\xi;d\rho), \ \text{a.e. } (t,\xi)\in[0,T]\times\T.
\end{align}

\begin{remark}
Every such solution of~\eqref{eq:burgers} is a measure-valued solution.
That is, for any $ h\in\Spd $ such that $ h_t \leq \Phi(h_\xi) $ a.e.,
there exists $ \nu\in\young $ such that $ (h,\nu)\in\Spdm $ and~\eqref{eq:mburgers} holds.
To see this, note that
for each given $ \kappa\in[0,\infty) $ and $ \bar\rho\in[0,1] $ such that $ \kappa \leq \Phi(\bar\rho) $,
there exists $ \nu_*(\kappa,\bar\rho;d\rho)\in\probm[0,1] $ such that
\begin{align*}
	\int_{[0,1]} \rho \nu_*(\kappa,\bar\phi;d\rho) = \bar\rho,
	\quad
	\int_{[0,1]} \Phi(\rho) \nu_*(\kappa,\bar\phi;d\rho) = \kappa.
\end{align*}
For example, one can consider measures of the type
\begin{align*}
	\pi(\rho_1,\rho_2,\bar\rho;d\rho)
	:=
	\tfrac{\rho_2-\bar\rho}{\rho_2-\rho_1} \delta_{\rho_2}(\rho)
	+
	\tfrac{\bar\rho-\rho_1}{\rho_2-\rho_1} \delta_{\rho_1}(\rho),
	\quad
	\rho_1 \leq \bar\rho \leq \rho_2 \in[0,1].
\end{align*}
Indeed, $ \int_{[0,1]} \rho \pi(\rho_1,\rho_2,\bar\rho;d\rho) = \bar\rho $,
and by varying $ \rho_1,\rho_2 $ across $ \rho_1\leq \bar\rho\leq \rho_2\in[0,1] $,
$ \int_{[0,1]} \Phi(\rho) \pi(\rho_1,\rho_2,\bar\rho;d\rho) $ exhausts all values within $ [0,\Phi(\bar\rho)] $. 
Hence there exists $ \rho_1,\rho_2 $ such $ \nu_*(\kappa,\bar\rho;d\rho):=\pi(\rho_1,\rho_2,\bar\rho;d\rho) $
satisfies the prescribed property.
\end{remark}

We endow the space $ \Spd $ with uniform norm.
For each $ \nu(t,\xi;d\rho)\in\young $, 
we view $ \frac{1}{T} \nu(t,\xi;d\rho) dtd\xi \in \probm([0,T]\times\T\times[0,1]) $ as a probability measure,
and endow $ \young $ with the weak-* topology, which is mobilizable on $ \probm([0,T]\times\T\times[0,1]) $.
The space $ \Spdm $ hence inherits a product metric from $ \Spd\times\young $.
Under this setup, for $ (h,\nu)\in\Spd $, we let 
$ \calU_r(h,\nu)\in\Spd\times\young $ denote the corresponding open ball centered at $ (h,\nu) $ with radius $ r $.
Recall the notation $ \bareta_{k}(x) $ from \eqref{eq:bareta}.
Taking into account the $ t $ dependence and scaling, 
we set $ \bareta_{k,N}(t,\xi) := \bareta_{k}(Nt,N\xi) $, which is defined for $ t\in[0,T] $ and $ \xi\in \TN $,
and linearly interpolate in $ \xi $ onto $ \T $ so that $ \bareta_{k,N}\in\young $.

\begin{proposition}
\label{prop:2-d}
Fix $ (h,\nu)\in\Spdm $ that is not a measure-valued solution of Burgers equation,
and fix $ \delta\in(0,1) $.
Then
\begin{align*}
	\limsup_{r\downarrow 0}
	\limsup_{k\to\infty} 
	\limsup_{N\to\infty}
	\,
	N^{-2+\d}
	\log  \Pr_{\eta^\ic} \Big( (\hN,\bareta_{k,N})\in \calU_r(h,\nu) \Big)
	=
	-\infty.
\end{align*}
\end{proposition}
\begin{proof}
Fix $ (h,\nu)\in\Spdm $ that is not a measure-valued solution of Burgers equation.
Indeed, since $ h\in\Spd $, we have $ \int_0^T\int_{\T} |h_t|dtd\xi = \int_{\T} (h(T,\xi)-h(0,\xi)) d\xi <\infty $.
In particular, $ h_t\in L^1([0,T]\times\T) $.
Also, since $ \Phi(\rho)=\rho(1-\rho) $ is a bounded function on $ \rho\in[0,1] $,
$ \int_{[0,1]}\Phi(\rho)\nu(t,\xi;d\rho) $ is bounded.
Granted the preceding properties, and that $ (h,\nu) $ is not a measure-valued solution of Burgers equation,
there must exists $ \til{z}(t,\xi) \in C^1([0,T]\times\T) $ such that
\begin{align*}
	\int_0^T \int_\T \til{z}(t,\xi) \Big( h_t(t,x) - \int_{[0,1]} \Phi(\rho) \nu(t,\xi;d\rho) \Big) dtd\xi \, := a \in \R \setminus\set{0}.
\end{align*}
Fix arbitrary $ \ell \in\Z_{>0} $, and set $ z(t,\xi) := \frac{\ell}{a} \til{z}(t,\xi)\in C^1([0,T]\times\T) $, we then have
\begin{align}
	\label{eq:burgersell}
	\int_0^T \int_\T z(t,\xi) \Big( h_t(t,x) - \int_{[0,1]} \Phi(\rho) \nu(t,\xi;d\rho) \Big) dtd\xi = \ell.
\end{align}

Now, consider the exponential martingale: $ \calF := \exp(\calF(T)) $, where
\begin{align}
	\label{eq:expMG}
	\calF(\bart)
	:=
	\hspace{-5pt}
	\sum_{x\in\TN} \hspace{-5pt} \Big( N^{1-\d} z(t,x)\hN(t,x)\Big|_{t=0}^{t=\bar{t}}
	- 
	N \int_0^{\bart} \big( 
		N^{-\d}z_t(t,x)\hN(t,x) 
		+
 		(e^{N^{-\d} z(t,x)}-1) \mob_N(t,x)
 	\big)dt\Big),
\end{align}
where $ \mob_N(t,x) := \mob(\eta(Nt),Nx) $.
To extra useful information from this exponential martingale,
we next derive a few approximations of terms within~\eqref{eq:expMG}.
For convenience of notation, hereafter we use $ C<\infty $
to denote generic universal constants, 
which may change from line to line but are in particular independent of $ r,\ell,N $ and $ t\in[0,T], \xi\in\T $.
 
First, through Taylor-expansion we have
$ (e^{N^{-\d} z(t,x)}-1) \mob_N(t,x) = N^{-\d}z(t,x)\mob_N(t,x) + N^{-2\d}r_1(t,x) $,
for some remainder term $ r_1 $ such that
$ |r_1(t,x)| \leq c\,\Vert z { \Vert_{L^\infty}^2}  $.
Set
\begin{align*}
	\Omega(k,N) := \Big\{ \Big|\int_0^T V_{N,k}(z,\eta(tN^2),t) dt\Big| \leq N \Big\},
\end{align*}
we then have
\begin{align}
\label{eq:1bapprox}
	\int_0^T \sum_{x\in\TN} (e^{N^{-\d} z(t,x)}-1) \mob_N(t,x) dt
	=
	\int_0^T \sum_{x\in\TN}z(t,x) \Phi(\bareta_{k,N}(t,x) ) dt
	+
	r_2(t,x),
\end{align}
for some remainder term $ r_2 $ such that $ |r_2(t,x)| \leq N^{1-\delta} + c\, N^{1-2\delta}\Vert z \Vert_{L^\infty}^2 $.
By definition, on $ (\hN,\bareta_{k,N})\in\calU_r(h,\nu) $ we have
\begin{align}
	\label{eq:happrox}
	&
	\sup_{[0,T]\times\T}|\hN - h| < r,
\\
	\label{eq:barlam:approx}
	&
	\Big|
		\int_0^T \frac 1N \sum_{x\in\TN}z(t,x) \Phi(\bareta_{k,N}(t,x) )
		- 
		\int_0^T\int_{\T} z(t,\xi) dtd\xi\int_{[0,1]} \Phi(\rho) \nu(t,\xi;d\rho)
	\Big|
	< r. 
\end{align}
Using~\eqref{eq:1bapprox}--\eqref{eq:barlam:approx} to approximate the corresponding terms gives
\begin{align}
	\label{eq:ibp1}
	\calE &\ind_{\calU_r(h,\nu)\cap\Omega(k,N)}
	=
	\exp\Bigg(
		N^{2-\d} \int_{\T} \Big( z(t,\xi)h(t,\xi)d\xi\Big|_{t=0}^{t=T}
		- 
		\int_0^T \Big( z_t(t,\xi)h(t,\xi) dt \Big) d\xi
\\	
	\label{eq:ibp2}
		&-N^{2-\d}\int_0^T\int_\T z(t,\xi) \int_{[0,1]} \Phi(\rho) \nu(t,\xi;d\rho) \Big) dtd\xi
		+
		r_3
	\Bigg),
\end{align}
for some remainder term $ r_3 $ such that
\begin{align*}
	|r_3(t,x)| 
	\leq  
	c\, r N^{2-\delta} \, \big( \Vert z \Vert_{L^\infty}+\Vert z_t \Vert_{L^\infty} \big)
	+
	N^{2-\delta} 
	+ 
	c\, N^{2-2\delta}\Vert z \Vert_{L^\infty}^2.
\end{align*}
Integrate by part in $ t $ in~\eqref{eq:ibp1}, and combine the result with \eqref{eq:burgersell}.
This gives
\begin{align*}
	\calE \ind_{\calU_r(h,\nu)\cap\Omega(k,N)}
	=
	\exp\big(
		-N^{2-\d} \ell+r_3.
	\big).
\end{align*}
Given that $ \Ex(\calE)=1 $, we now have
\begin{align}
\label{eq:ibp3}
\begin{split}	
	N^{-2+\d}\log\Big( \Pr_{\eta^\ic}&\Big( \big\{ (\hN,\bareta_{k,N}) \in\calU_r(h,\nu) \big\} \cap\Omega(k,N) \Big) \Big)
\\
	&\leq
	-\ell+ c\,\big( r\Vert z \Vert_{L^\infty}+r\Vert z_t \Vert_{L^\infty}+N^{-\d}\Vert z \Vert^2_{L^\infty} \big) +1 .
\end{split}
\end{align}
Lemma~\ref{lem:sexp1block} asserts that $ \limsup\limits_{k\to\infty} \limsup\limits_{N\to\infty} \Pr_{\eta^\ic}(\Omega(k,N))=-\infty $.
Using this, and letting $ N\to\infty $, $ k\to\infty $, and $ r\downarrow 0 $ in~\eqref{eq:ibp3} in order, we now have 
\begin{align*}
	\limsup_{r\downarrow 0}
	\limsup_{k\to\infty} 
	\limsup_{N\to\infty}
	\,
	N^{-2+\d}
	\log \Big( \Pr_{\eta^\ic}\big( (\hN,\bareta_{k,N}) \in\calU_r(h,\nu) \big) \Big) 
	\leq
	-\ell + 1.
\end{align*}
With $ \ell\in\Z_{>0} $ being arbitrary, we conclude the desired result.
\end{proof}

\bibliographystyle{alphaabbr}
\bibliography{tasepLDP}
\end{document}